\documentclass[11pt,twoside,reqno,a4paper]{amsart}

\usepackage{microtype}
\usepackage[OT1]{fontenc}
\usepackage{lmodern}
\usepackage{mathtools}
\usepackage{amssymb}
\usepackage{stmaryrd}
\SetSymbolFont{stmry}{bold}{U}{stmry}{m}{n}
\usepackage[shortlabels]{enumitem}
\usepackage{xcolor}
\usepackage{tikz}
\usepackage{aliascnt}
\usepackage{array}
\usepackage{supertabular}
\makeatletter
\@ifpackagelater{supertabular}{2025/06/26}{}{%
  \@ifundefined{@startpbox@action}{}{%
    \RequirePackage{etoolbox}
    \patchcmd{\x@supertabular}
      {\let\org@startpbox=\@startpbox
       \let\org@endpbox=\@endpbox
       \let\@startpbox=\ST@astartpbox}
      {\let\org@startpbox@action=\@startpbox@action
       \let\@startpbox@action=\ST@astartpbox
       \let\org@endpbox=\@endpbox}
      {}{\PackageError{supertabular-compat}{Cannot patch table initialisation}
           {Use supertabular v4.2d or newer.}}
    \patchcmd{\ST@restore}
      {\let\@startpbox\org@startpbox}
      {\let\@startpbox@action\org@startpbox@action}
      {}{\PackageError{supertabular-compat}{Cannot patch table restoration}
           {Use supertabular v4.2d or newer.}}}}
\makeatother
\newcolumntype{L}[1]{>{\raggedright\arraybackslash}p{#1}}
\usepackage[a4paper,margin=1in]{geometry}
\usepackage{hyperref}
\hypersetup{colorlinks=true,linkcolor=black,citecolor=black,urlcolor=black}
\usepackage[capitalise,noabbrev,nameinlink]{cleveref}

\numberwithin{equation}{section}
\allowdisplaybreaks

\crefformat{equation}{(#2#1#3)}
\Crefformat{equation}{(#2#1#3)}
\crefformat{section}{#2\S#1#3}
\crefformat{subsection}{#2\S#1#3}
\Crefformat{section}{#2\S#1#3}
\Crefformat{subsection}{#2\S#1#3}
\crefformat{appendix}{#2Appendix~#1#3}
\crefformat{subappendix}{#2Appendix~#1#3}
\Crefformat{appendix}{#2Appendix~#1#3}
\Crefformat{subappendix}{#2Appendix~#1#3}
\crefname{part}{Part}{Parts}
\Crefname{part}{Part}{Parts}

\theoremstyle{plain}
\newtheorem{theorem}{Theorem}[section]
\newaliascnt{corollary}{theorem}
\newtheorem{corollary}[corollary]{Corollary}
\aliascntresetthe{corollary}
\newaliascnt{proposition}{theorem}
\newtheorem{proposition}[proposition]{Proposition}
\aliascntresetthe{proposition}
\newaliascnt{lemma}{theorem}
\newtheorem{lemma}[lemma]{Lemma}
\aliascntresetthe{lemma}
\newaliascnt{conjecture}{theorem}

\aliascntresetthe{conjecture}
\theoremstyle{definition}
\newaliascnt{definition}{theorem}
\newtheorem{definition}[definition]{Definition}
\aliascntresetthe{definition}
\newaliascnt{example}{theorem}

\aliascntresetthe{example}

\theoremstyle{remark}
\newaliascnt{remark}{theorem}
\newtheorem{remark}[remark]{Remark}
\aliascntresetthe{remark}

\crefname{corollary}{Corollary}{Corollaries}      \Crefname{corollary}{Corollary}{Corollaries}
\crefname{proposition}{Proposition}{Propositions} \Crefname{proposition}{Proposition}{Propositions}
\crefname{lemma}{Lemma}{Lemmas}                   \Crefname{lemma}{Lemma}{Lemmas}
\crefname{conjecture}{Conjecture}{Conjectures}    \Crefname{conjecture}{Conjecture}{Conjectures}
\crefname{definition}{Definition}{Definitions}    \Crefname{definition}{Definition}{Definitions}
\crefname{example}{Example}{Examples}             \Crefname{example}{Example}{Examples}
\crefname{remark}{Remark}{Remarks}                \Crefname{remark}{Remark}{Remarks}

\newcommand{\restatedhead}{}
\theoremstyle{plain}
\newtheorem*{restatedbody}{\restatedhead}
\newenvironment{restate}[1]
 {\renewcommand{\restatedhead}{\Cref{#1} of \cref{part:introduction}, restated}\begin{restatedbody}}
 {\end{restatedbody}}

\newcommand{\bbP}{\mathbb{P}}
\newcommand{\bbE}{\mathbb{E}}
\newcommand{\bbN}{\mathbb{N}}
\newcommand{\bbZ}{\mathbb{Z}}
\newcommand{\bbR}{\mathbb{R}}
\newcommand{\bbB}{\mathbb{B}}
\newcommand{\cA}{\mathcal{A}}
\newcommand{\cC}{\mathcal{C}}
\newcommand{\cD}{\mathcal{D}}
\newcommand{\cE}{\mathcal{E}}
\newcommand{\cF}{\mathcal{F}}
\newcommand{\cI}{\mathcal{I}}
\newcommand{\cJ}{\mathcal{J}}
\newcommand{\cL}{\mathcal{L}}
\newcommand{\cM}{\mathcal{M}}
\newcommand{\cN}{\mathcal{N}}
\newcommand{\cR}{\mathcal{R}}
\newcommand{\cS}{\mathcal{S}}
\newcommand{\cT}{\mathcal{T}}
\newcommand{\cU}{\mathcal{U}}
\newcommand{\cX}{\mathcal{X}}
\newcommand{\eps}{\varepsilon}
\newcommand{\troot}{\emptyset}
  
\renewcommand{\leq}{\leqslant} \renewcommand{\geq}{\geqslant}
\DeclareMathOperator{\diam}{diam}
\DeclareMathOperator{\Aut}{Aut}
\DeclareMathOperator{\supp}{supp}
\DeclareMathOperator{\rep}{rep}
\DeclareMathOperator{\shrink}{sh}
\DeclarePairedDelimiter{\set}{\lbrace}{\rbrace}
\DeclarePairedDelimiter{\abs}{\lvert}{\rvert}
\DeclarePairedDelimiter{\floor}{\lfloor}{\rfloor}

\makeindex

\ExplSyntaxOn
\NewDocumentCommand{\leanbreaks}{m}
 {
  \tl_set:Nn \l_tmpa_tl {#1}
  \tl_replace_all:Nnn \l_tmpa_tl { . } { .\allowbreak }
  \tl_replace_all:Nnn \l_tmpa_tl { / } { /\allowbreak }
  \tl_use:N \l_tmpa_tl
 }
\ExplSyntaxOff

\NewDocumentCommand{\leancomment}{m t\;}{\unskip}

\usepackage{subfiles}

\begin{document}

\title{The quasi-isometry classes of Galton--Watson trees}

\author[J.~S.~Athreya]{Jayadev S.~Athreya}
\address[J.~S.~Athreya]{Fordham University, Mathematics Department, 113 West 60th Street, Room 813
New York, NY 10023}
\email{jathreya@fordham.edu}
\author[S.~Troscheit]{Sascha Troscheit}
\address[S.~Troscheit]{Department of Mathematics, Box 524, 751 20 Uppsala University, Sweden}
\email{sascha.troscheit@math.uu.se}

\subjclass[2020]{Primary 60J80; Secondary 51F30, 30L10, 28A80, 60K35, 68V20}
\keywords{Galton--Watson tree, quasi-isometry, coarse geometry, random tree, tree-indexed
  Markov chain, quasisymmetric map, fractal percolation, formalised mathematics}

\dedicatory{Dedicated to the memory of Krishna B.\ Athreya (1939--2023).}

\begin{abstract}
  We classify, up to quasi-isometry, the large-scale geometry of Galton--Watson trees for every
  finitely supported offspring distribution. Apart from the trivial finite diameter regimes, we
  condition on infinite diameter. We find that the remaining classes are the \emph{ray class}, the
  \emph{full tree class} (which includes the deterministic binary tree), one \emph{chain class}
  $(\mathrm{C}_\Lambda)$ for every possible branching semigroup $\Lambda$, and the \emph{bushy}
  class. Two independent trees almost surely admit a root-preserving quasi-isometry when their
  offspring distributions belong to the same class and are almost surely \textit{not}
  quasi-isometric when they belong to different classes. Any two survival-conditioned supercritical
  realisations with finitely supported offspring laws nevertheless almost surely admit
  quasi-isometric embeddings in both directions. 
  
  We prove that the class of a conditioned
  infinite realisation depends only on the support of the offspring distribution, not its specific
  distribution. For offspring distributions supported on $\set{1,2}$, we additionally obtain an explicit
  exponential tail bound (in the constant $D$) for the probability of non-existence of a
  root-preserving $D$-quasi-isometry. The classification also implies that the corresponding random
  Cantor boundaries are almost surely quasisymmetrically equivalent.
  
  Conditioned on nonextinction,
  this applies to general strongly separated fractal percolation in the nontrivial supercritical
  phase, even when the underlying self-similar iterated function systems and retention parameters
  differ. We also classify two families of binary trees with continuous random branching times. Our
  proofs use new automorphism-matching theorems for random graph labellings of trees, based on
  contraction estimates for mismatch potentials. These matching results for Markov labellings on the
  binary tree are fairly general and may be of independent interest. The quasi-isometry
  classification, mutual embeddability and both matching theorems are formally verified in Lean.
\end{abstract}

\maketitle

\vspace{-1em}
\setcounter{tocdepth}{1}
\makeatletter
\begingroup
\small
\setlength{\baselineskip}{12.0pt}
\renewcommand{\l@part}{\@tocline{-1}{2pt plus2pt}{0pt}{}{\bfseries}}
\tableofcontents
\par\removelastskip
\endgroup
\vskip28pt plus4pt\relax
\makeatother

\vspace{-1em}

\part{Introduction and main results}\label{part:introduction}

\section{Introduction}

\paragraph*{\bf Random branching} Galton--Watson trees are among the canonical models of random branching. They arise throughout
probability, from population dynamics and percolation to random graphs and random fractal
constructions, and often serve as local models for more complicated random objects. Their theory
of extinction, generation growth, and limiting population size is classical. We refer to
Harris~\cite{Harris1963} and Athreya--Ney~\cite{AthreyaNey1972} for the foundations of the
subject, and to Lyons--Peres~\cite{LyonsPeres2016} for a modern account of branching processes and
probability on trees.

\paragraph*{\bf Large-scale geometry} By contrast, much less is known about the large-scale metric geometry of an entire realisation.
With the graph metric, a Galton--Watson tree is a random metric space, and its geometry can
be studied up to \emph{quasi-isometry}, the equivalence relation that retains distances
at large scales while discarding uniformly bounded errors. Quasi-isometry is central to geometric
group theory, where it organises spaces, in particular Cayley graphs of finitely generated groups,
by their coarse geometry, a perspective initiated by, among others, Gromov. See de la
Harpe~\cite{DeLaHarpe2000} and Bridson--Haefliger~\cite{BridsonHaefliger1999} for more background.

\paragraph*{\bf Quasi-isometries and bi-Lipschitz equivalence} The deterministic theory of branching
trees (and their interpretation, for regular trees,  as Cayley graphs of free groups)
suggests a considerable flexibility. All regular trees
of degree at least three are bi-Lipschitz equivalent by Papasoglu~\cite{Papasoglu1995}, and all
\emph{bushy} trees of bounded valence belong to one quasi-isometry class by
Mosher--Sageev--Whyte~\cite[\S2.1]{MosherSageevWhyte2003}. This geometric use of \emph{bushy}
differs from the bushy class (B) in our classification. Whyte~\cite{Whyte1999} further showed
that, for non-amenable uniformly discrete spaces of bounded geometry, every quasi-isometry lies a
bounded distance from a bi-Lipschitz map. Lindquist~\cite{Lindquist2018} obtained a related result
for spaces satisfying linear isoperimetric inequalities and applied it to regularly branching
trees and hyperbolic fillings of compact Ahlfors-regular spaces. These results cover the homogeneous
case in which branching occurs at a uniformly bounded distance from every vertex.
The distinction between quasi-isometry and bi-Lipschitz equivalence remains important in familiar
coarse models. For $n\geq2$, every separated net in $\bbR^n$ is quasi-isometric to $\bbZ^n$, but
Burago--Kleiner~\cite{BuragoKleiner1998} and McMullen~\cite{McMullen1998} constructed nets that
are not bi-Lipschitz equivalent to $\bbZ^n$. Their obstruction comes from density fluctuations
across scales, a phenomenon excluded in Whyte's non-amenable setting. For more recent results in
this direction (focusing on Delone sets linked to substitution tilings), see
Smilansky--Solomon~\cite{SmilanskySolomon2026}.

\paragraph*{\bf Nonhomogeneity}
Galton--Watson trees need not be homogeneous in the sense of large-scale geometry. For example, if
the probability of having a single offspring is positive, infinite trees almost surely produce
arbitrarily long chains without branching, while if the probability of having zero offspring is
positive, leaves will produce finite bushes of arbitrarily large depth. Neither feature is confined
to a bounded part of the tree, and each changes its large-scale geometry. The classical theory
therefore does not decide whether two independent realisations from the same offspring law are
quasi-isometric, still less whether realisations from different laws can have the same coarse
geometry.

\paragraph*{\bf Random metric spaces} The question for Galton--Watson trees
is an instance of a wider problem for random metric spaces: when do two independent
realisations have the same large-scale geometry almost surely? Benjamini~\cite{Benjamini2013}
discusses this question in several settings. For random subsets of the line,
Peled~\cite{Peled2010} obtained quantitative finite-scale estimates, and
Basu--Sly~\cite{BasuSly2014} proved that two independent Poisson processes are almost surely
roughly isometric. Related embedding questions for independent, identically distributed (i.i.d.) fields were studied by
Basu--Sidoravicius--Sly~\cite{BasuSidoraviciusSly2018}. By contrast,
Li--Yu--Zheng~\cite{LiYuZheng2026} recently proved that two independent random subsets of a
product of regular trees are almost surely not quasi-isometric. Galton--Watson trees provide a
natural rank-one family: their branching favours flexibility, but their chains and finite bushes
create inhomogeneities of flatness on every scale.

\paragraph*{\bf Quasisymmetry} There is also a boundary version of the same geometric problem.
A locally finite tree is hyperbolic, and a
quasi-isometry induces a quasisymmetric map between its visual boundaries. This is part of the
general correspondence between the coarse geometry of hyperbolic spaces and the quasisymmetric
geometry of their boundaries developed by Bonk--Schramm~\cite{BonkSchramm2000} and described by
Buyalo--Schroeder~\cite{BuyaloSchroeder2007}, inspired by classical results of Mostow. For a non-ray
Galton--Watson tree conditioned on having infinite diameter, the boundary is a random Cantor space.
In standard geometric constructions such boundaries are coding spaces of random fractals. Thus a
classification of the trees produces maps between random fractals, rather than only an equality
between numerical invariants, something we will further discuss below.

\paragraph*{\bf Classification} In this paper we give a complete quasi-isometry classification for every finitely supported
offspring distribution. The answer is determined by the occurrence of leaves and unary offspring,
together with one arithmetic invariant generated by the possible branching excesses. In
particular, only the support of the offspring distribution is of relevance for the classification,
not the exact probabilities.
\subsection{The classification}\label{sec:classification}

Subcritical Galton--Watson trees and trees with a non-deterministic critical offspring law are
finite almost surely. The same is true of every extinct realisation of a supercritical law. Since
all finite metric spaces are quasi-isometric, these realisations form the single class
$(\mathrm{Fin})$ and it remains to classify infinite realisations. The only nonstandard invariant in
the classification is the \emph{branching semigroup}:

\begin{definition}[Branching semigroup]\label{def:branching-semigroup}
  Suppose that an offspring distribution $\theta$ satisfies $\theta_0=0$. Its
  \emph{branching semigroup} is the additive submonoid
  \[
    \Lambda_\theta
    \coloneqq\left\langle k-1 \colon k\geq1\text{ and }\theta_k>0\right\rangle
    \subseteq\bbN_0.
  \]
\end{definition}

\paragraph*{\bf Branching, clustering, and semigroups}
A vertex with $k$ children contributes the excess $k-1$. Thus $\Lambda_\theta$ records the possible total
excesses that branch points can realise. If $\theta_2 >0$, then
$\Lambda_{\theta} = \bbN_0$. Conversely, every additive submonoid $\Lambda$ of $\bbN_0$ generated by
a nonempty finite set of positive integers is the branching semigroup of some finitely supported
offspring law with $\theta_0=0<\theta_1<1$. We call these the \emph{possible branching
semigroups}. The branching semigroup $\Lambda_{\theta}$ is a coarse invariant only when long unary
chains can isolate such clusters (that is, when $\theta_0=0$, $\theta_1>0$). For each such 
$\Lambda$, we write $(\mathrm{C}_\Lambda)$ for the class of distributions with
branching semigroup $\Lambda$. The following theorem gives the complete 
\emph{root-preserving} quasi-isometry classification.

\begin{theorem}[Complete quasi-isometry classification]\label{thm:trichotomy}
  All finite Galton--Watson realisations form the quasi-isometry class $(\mathrm{Fin})$, and no
  finite tree is quasi-isometric to an infinite tree.

  Let $\theta$ and $\theta'$ be finitely supported offspring distributions, each either
  supercritical or with $\theta_1=1$, and let $\cT,\cT'$ be independent Galton--Watson trees
  with these offspring distributions, conditioned on having infinite diameter. Every such
  distribution belongs to exactly one of the following classes:
  \begin{itemize}
    \item[(R)] $\theta_1=1$, the \emph{ray} class, in which $\cT$ is the ray.
    \item[(F)] $\theta_0=\theta_1=0$, the \emph{full tree} class, in which $\cT$ admits
      a quasi-isometry to the binary tree.
    \item[$(\mathrm{C}_\Lambda)$] $\theta_0=0<\theta_1<1$ and
      $\Lambda_\theta=\Lambda$, the \emph{chain} class indexed by the branching
      semigroups $\Lambda$.
    \item[(B)] $\theta_0>0$, the \emph{bushy} class.
  \end{itemize}
  Two independent realisations belonging to the same class almost surely admit a root-preserving
  quasi-isometry.
  If they belong to different classes, then they are almost surely not quasi-isometric.
\end{theorem}

\paragraph*{\bf Regimes} We retain (R), (F), (C), and (B) for the four \emph{coarse} regimes used to organise the proof, with
regime (C) denoting the union of the classes $(\mathrm{C}_\Lambda)$. See \cref{fig:classes} for illustrative representatives of these classes.

\begin{figure}[tp]
  \centering
  \includegraphics[width=.41\textwidth]{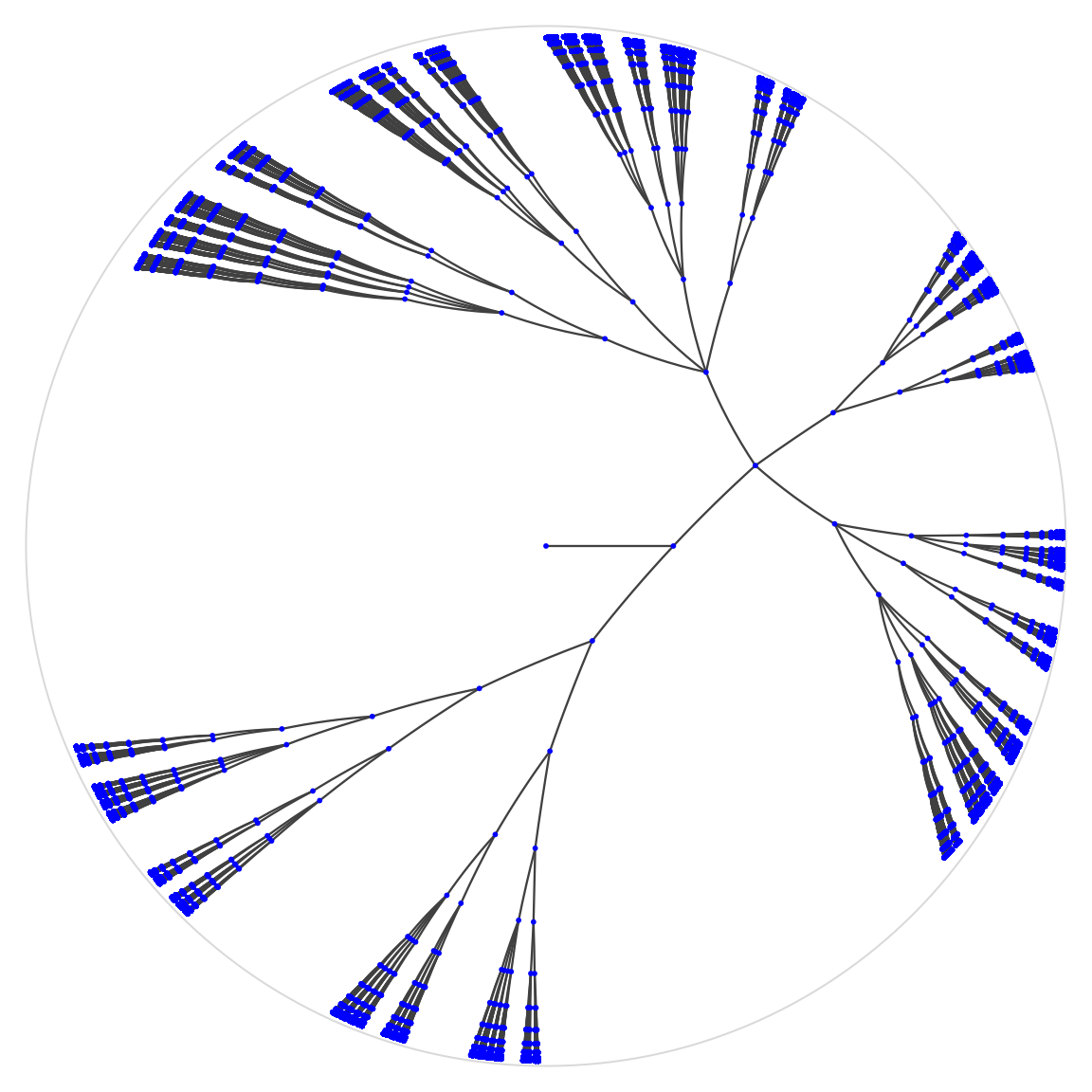}\hfill
  \includegraphics[width=.41\textwidth]{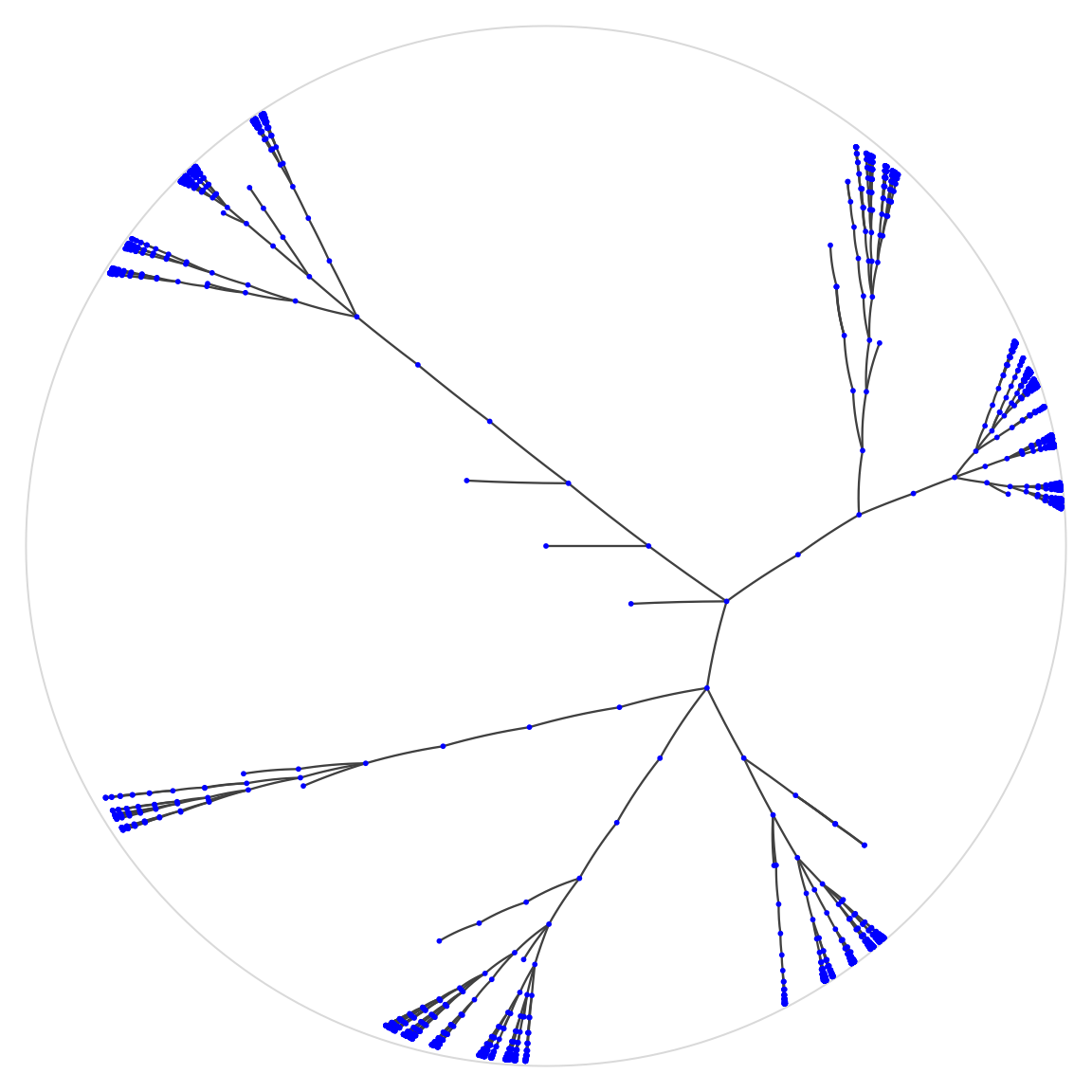}

  \includegraphics[width=.41\textwidth]{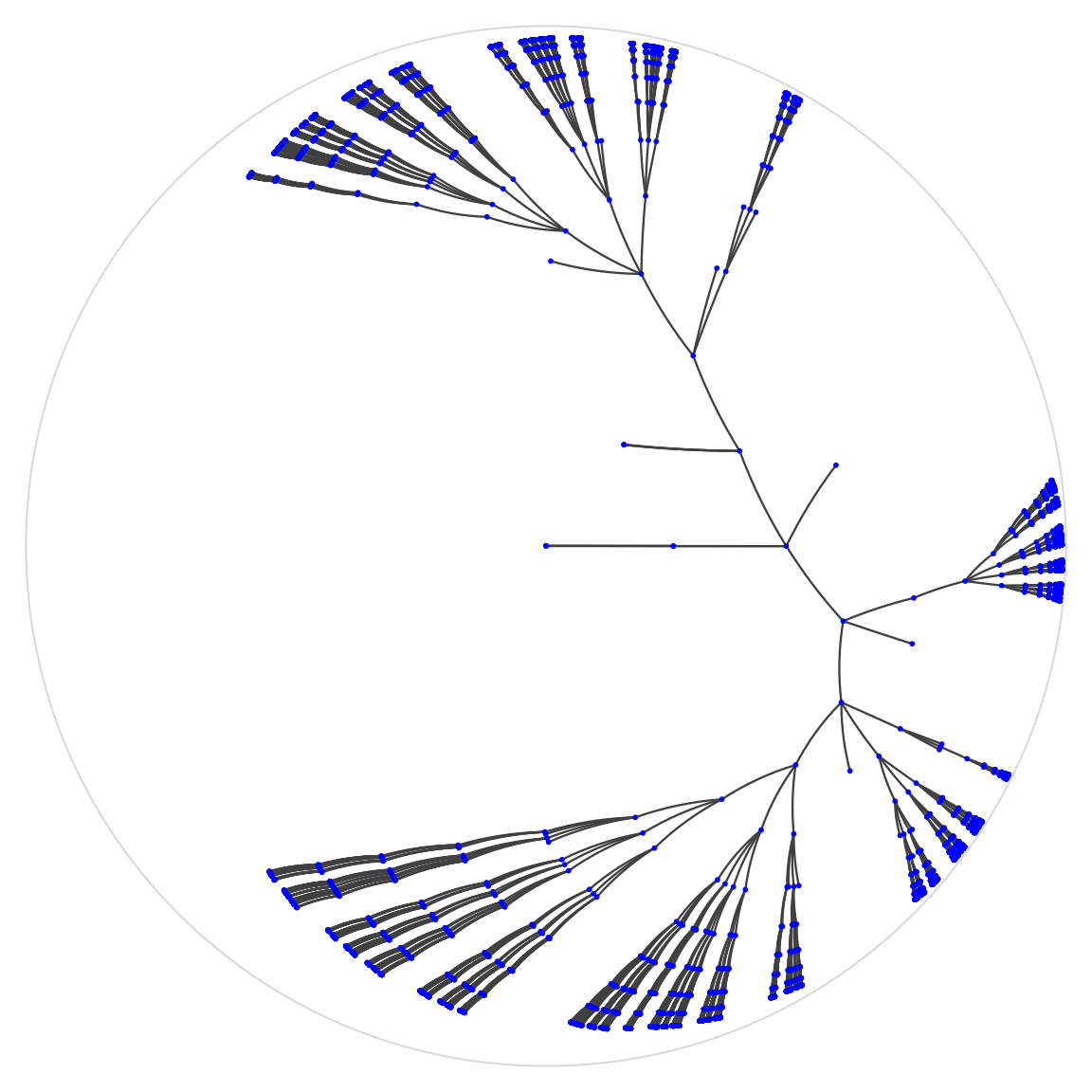}\hfill
  \includegraphics[width=.41\textwidth]{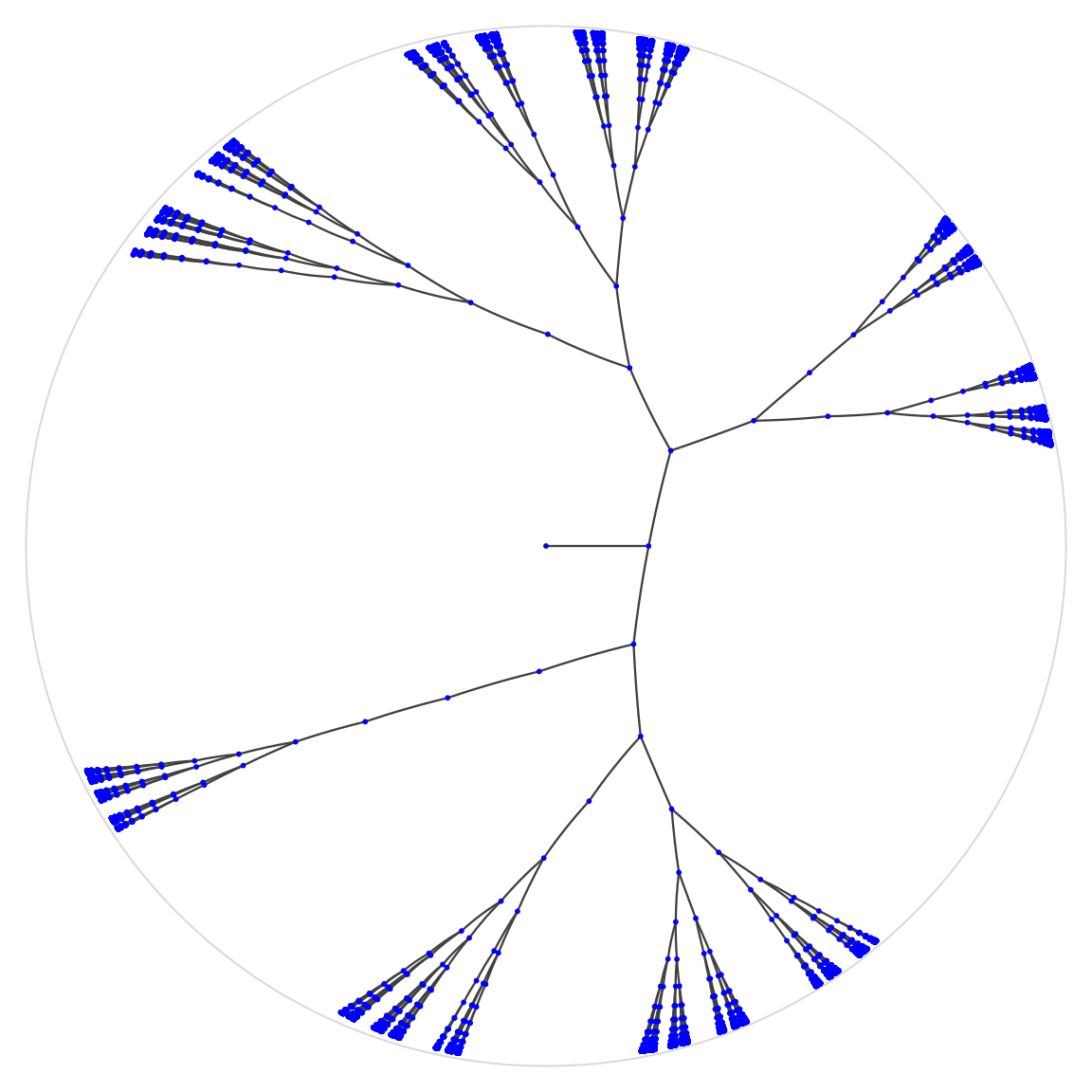}

  \includegraphics[width=.41\textwidth]{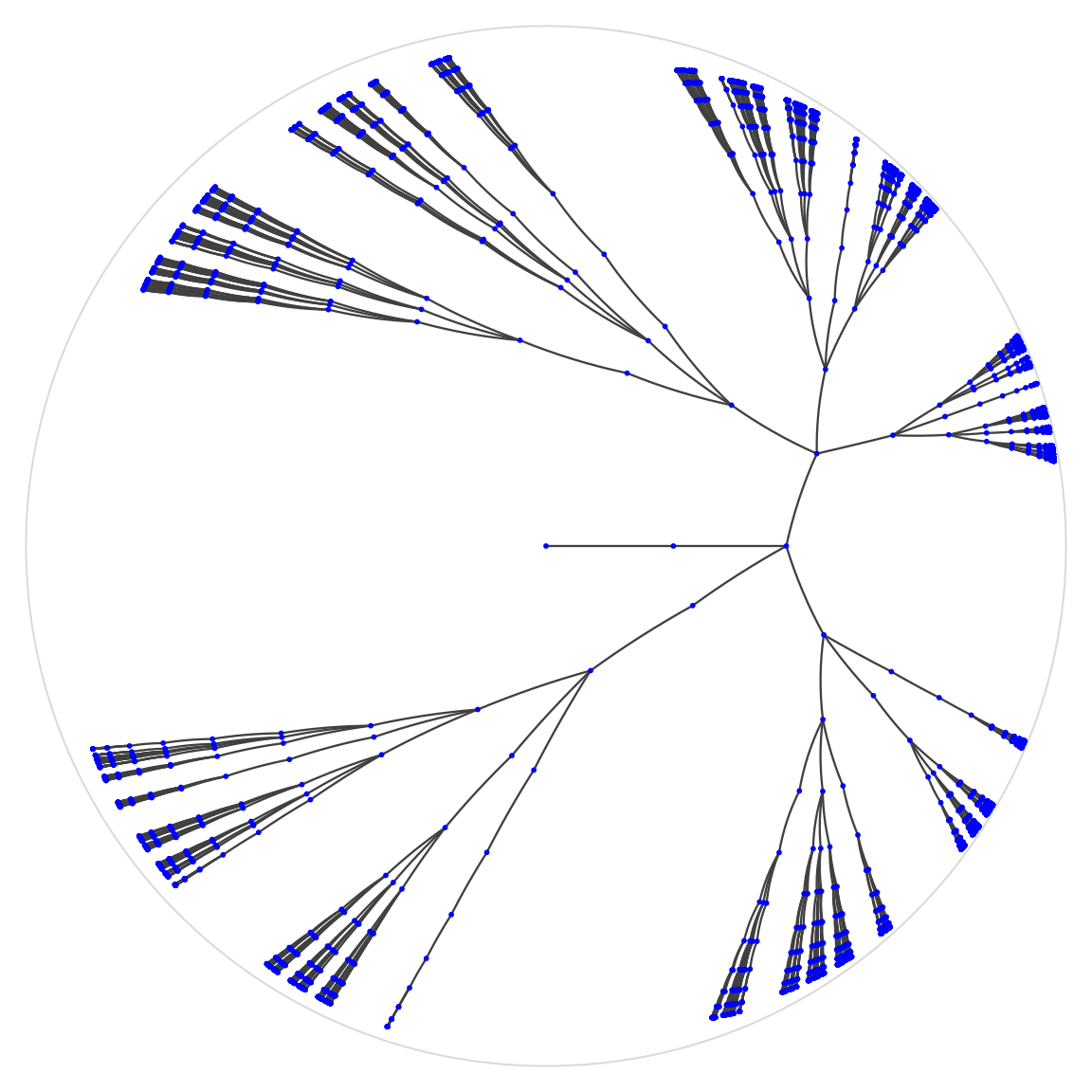}\hfill
  \includegraphics[width=.41\textwidth]{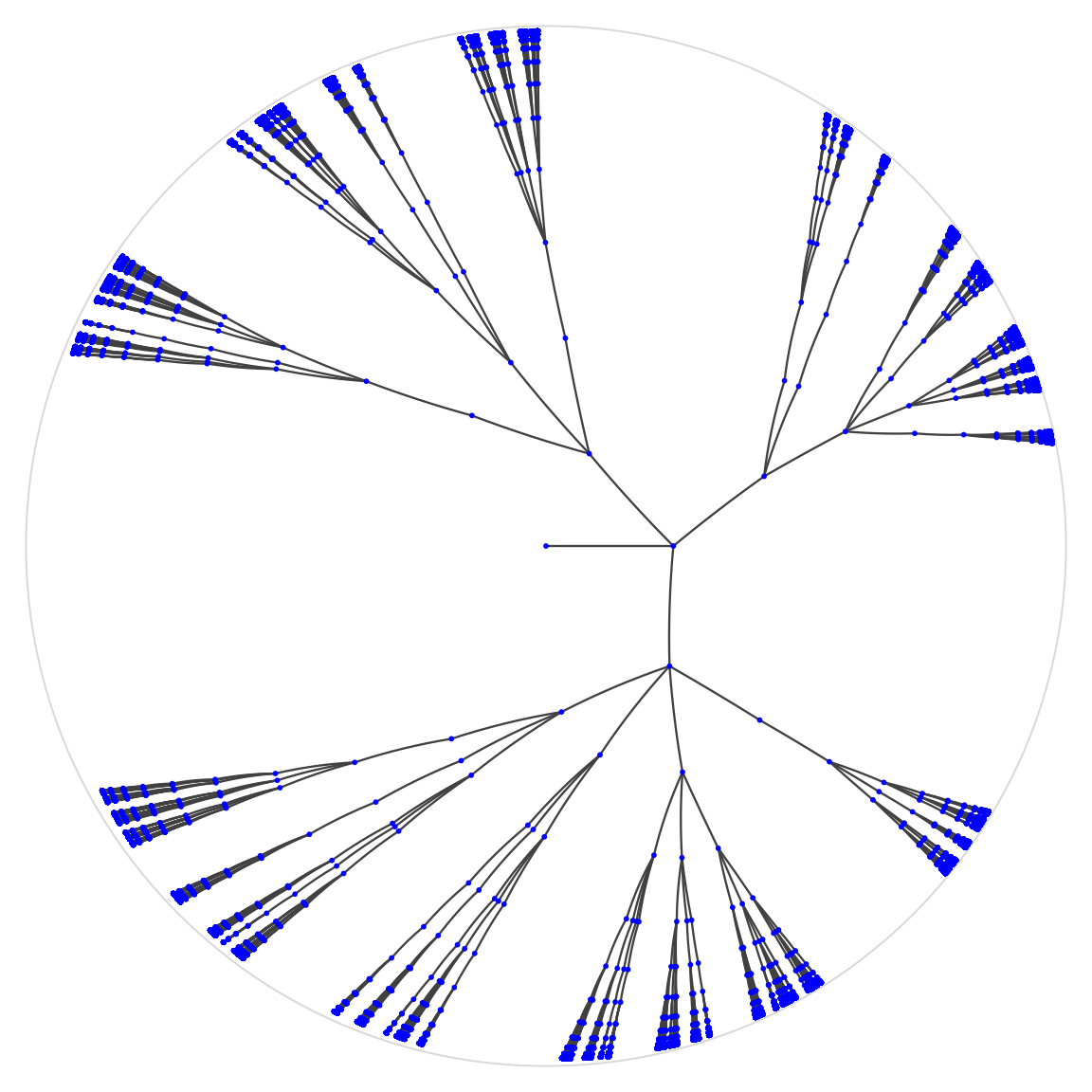}
  \caption{
    Finite truncations illustrating the three non-ray infinite regimes, embedded in the
    Poincar\'e disc. Regime (R) is omitted. The top row shows regime (F), with support
    $\set{2,3,4}$, and regime (B), with support $\set{0,1,2,3}$. The middle row shows regime (B),
    with support $\set{0,1,2,3,4}$, and $(\mathrm{C}_{\langle1\rangle})$, with support $\set{1,2}$.
    The bottom row shows $(\mathrm{C}_{\langle2\rangle})$, with support $\set{1,3}$, and
    $(\mathrm{C}_{\langle2,3\rangle})$, with support $\set{1,3,4}$. Apart from the two examples of
    realisations in the bushy regime, no realisations in these regimes are quasi-isometrically
    equivalent, almost surely.
  }
  \label{fig:classes}
\end{figure}

\paragraph*{\bf Independent realisations}
The classes in \cref{thm:trichotomy} completely resolve the classification problem for independent
realisations of every finitely supported offspring law. In particular, two independent realisations
from the same law, conditioned on having infinite diameter when necessary, almost surely admit a
root-preserving quasi-isometry. The cross-law statement is more rigid due to the class structure.
Full trees have the deterministic class of the binary tree, while the bushy regime has one universal
class despite the presence of arbitrarily deep finite dead ends. The chain regime splits into the
classes $(\mathrm{C}_\Lambda)$, since long unary stretches isolate finite clusters of branch points
and make their total excesses visible at large scales. This is precisely what the semigroups
$\Lambda_\theta$ encode.

\paragraph*{\bf GCDs and tails} A finitely generated additive submonoid of $\bbN_0$ with a positive
generator contains every sufficiently large multiple of the greatest common divisor of its positive
generators. Thus the greatest common divisor determines the eventual tail, but finitely many
exceptional gaps can still distinguish two branching semigroups.

\paragraph*{\bf Embeddability}
If we drop the coarse-surjectivity requirement and ask only for a quasi-isometric embedding,
the distinctions between the supercritical classes disappear. In particular, the branching
semigroup no longer obstructs comparison in either direction. We write $\bbB$ for the rooted
binary tree.

\begin{theorem}[Mutual embeddability]\label{thm:embedding-hierarchy}
  Let $\theta$ be a finitely supported supercritical offspring distribution, and let $\cT$ be a
  Galton--Watson tree with offspring distribution $\theta$, conditioned on having infinite
  diameter. Almost surely, there are quasi-isometric embeddings
  \[
    \cT\hookrightarrow\bbB
    \quad\text{and}\quad
    \bbB\hookrightarrow\cT.
  \]
\end{theorem}

\paragraph*{\bf Pruning} We prove \cref{thm:embedding-hierarchy} in \cref{sec:embedding-hierarchy}, using the
classification to choose a convenient representative in each chain class and pruning an infinite
regular subtree from that representative.

\paragraph*{\bf A strict order} By composition, any two independent survival-conditioned supercritical realisations with finitely
supported offspring laws almost surely admit quasi-isometric embeddings in both directions, even
when their laws belong to different classes in \cref{thm:trichotomy}. Together with the elementary
finite and ray comparisons, this gives the strict order
\[
  \mathrm{Fin}\;\hookrightarrow\;\mathrm{Ray}\;\hookrightarrow\;\mathrm{Supercritical}.
\]
Here each arrow denotes quasi-isometric embeddability, with no embedding in the reverse direction.
The finite and ray categories each consist of one quasi-isometry class. All full tree, chain and
bushy classes belong to the last category, but remain distinct under quasi-isometry. In particular,
\textit{mutual quasi-isometric embeddability does not imply quasi-isometry}.

\paragraph*{\bf Random branching-time}
While we only deal with finitely supported offspring distributions in our main results,
the random branching-time model in \cref{sec:branching-times} goes beyond this
setting. Its unit-time coarse graining can have unbounded arity while retaining the quasi-isometry
type of a weighted binary tree. Thus \cref{thm:branching-times} identifies classes represented
by some unbounded-arity random trees. We do not attempt a classification for arbitrary offspring
laws with infinite support.

\subsection{Quantitative and geometric consequences}

\paragraph*{\bf Universality and effectiveness}
The classification is qualitative, but the basic chain model admits an explicit estimate for the
quasi-isometry constant. It is also the simplest setting in which the matching mechanism of this
paper appears.

\begin{theorem}[Universality in the two-value family]\label{thm:twovalue}
  Let $\theta$ be an offspring distribution with $\theta_1+\theta_2=1$, and let
  $\cT(\omega),\cT(\omega')$ be two independent Galton--Watson trees with offspring distribution
  $\theta$. Then, almost surely, there is a root-preserving quasi-isometry
  $\cT(\omega)\to\cT(\omega')$.
  Moreover, if $0<\theta_1<1$, there are constants $D_0=D_0(\theta_1)$ and
  $C=C(\theta_1)$ such that for every integer $D\geq D_0$,
  \begin{equation}\label{eq:rate}
    \bbP\bigl(\text{there is no root-preserving }(D^2+3)\text{-quasi-isometry }
    \cT(\omega)\to\cT(\omega')\bigr)
    \leq C\,\theta_1^{D(D-5/2)}.
  \end{equation}
\end{theorem}

\noindent That is, the failure probability has an exponentially small upper bound in the
quasi-isometry constant $D$. It may be interesting to see if this upper bound is sharp (that is,
if there is a corresponding lower bound for the failure probability). The exponential estimate
makes a zero--one law unnecessary, although \cref{rem:no-zero-one} records one that applies to
these realisations. Analogous, but less explicit estimates drive the cross-law and bushy-regime
arguments in \cref{part:binary-classification,part:full-classification}.

\paragraph*{\bf From QI to QS} Quasi-isometries between proper hyperbolic spaces induce quasisymmetries of their visual
boundaries. Our classification therefore has the following random-fractal consequence.

\begin{corollary}[Quasisymmetry of the boundaries]\label{thm:boundary}
  Let $\theta$ and $\theta'$ satisfy the hypotheses of \cref{thm:trichotomy}, suppose that neither
  belongs to the ray class, and suppose that they belong to the same class. Let $\cT(\omega)$ and
  $\cT'(\omega')$ be independent Galton--Watson trees with these offspring distributions,
  conditioned on having infinite diameter. Almost surely, their visual boundaries are compact,
  perfect, totally disconnected metric spaces, and there is a quasisymmetric homeomorphism
  \[
    (\partial\cT(\omega),d_{\mathrm{vis}})
    \longrightarrow(\partial\cT'(\omega'),d_{\mathrm{vis}}).
  \]
\end{corollary}

\paragraph*{\bf Supercritical fractal percolation} In particular, consider two independent strongly separated fractal percolation constructions in
the nontrivial supercritical phase. Their underlying iterated function systems (IFSs) and retention probabilities may both
differ. Conditioned on nonextinction, their limit sets are almost surely quasisymmetrically
equivalent. The realisations in \cref{fig:fractal-percolation} illustrate this conclusion when
the IFS is fixed and only the retention probability changes. We prove the more general consequence
in \cref{thm:fractal-percolation}.

\paragraph*{\bf Quasisymmetry and conformal dimension} Quasisymmetry is a central equivalence relation in analysis on metric spaces and in the geometry of
self-similar sets. See David--Semmes~\cite{DavidSemmes1997} and Heinonen~\cite{Heinonen2001} for
more background.
Conformal dimension in the sense of Pansu~\cite{Pansu1989} is the infimum of the Hausdorff
dimensions in a quasisymmetric gauge, see also Mackay--Tyson~\cite{MackayTyson2010} for a systematic
account.
Bonk--Meyer~\cite{BonkMeyer2022} characterised the metric trees quasisymmetrically equivalent to
the continuum self-similar tree as precisely the trivalent, uniformly branching quasiconformal
trees. Chrontsios-Garitsis--Ioannidis--Vellis~\cite[Theorem~1.9]
{ChrontsiosGaritsisIoannidisVellis2026} extended this classification to every finite valence: for
each $n\geq 3$, the uniformly $n$-branching quasiconformal trees form a single quasisymmetry class
represented by a self-similar tree. Leaves and unary chains are exactly what break the corresponding
uniformity in our setting. Beyond trees,
Bonk's uniformisation theorem~\cite{Bonk2011} maps suitable Sierpi\'nski carpets to round carpets,
whereas Bonk--Merenkov~\cite{BonkMerenkov2013} proved that standard square carpets are
quasisymmetrically rigid.

\paragraph*{\bf Mandelbrot percolation} For random fractal geometry, the classical grid-based model is \textit{Mandelbrot percolation}. Conditioned
on nonextinction, Rossi--Suomala~\cite{RossiSuomala2021} showed that almost every classical
fractal percolation realisation has conformal dimension strictly below its Hausdorff dimension. For
the fat variant, Anttila--Eriksson-Bique--Py\"or\"al\"a~\cite{AnttilaErikssonBiquePyorala2025}
showed that, under the same conditioning, almost every realisation is minimal for conformal
dimension. For the dense variant they obtained the analogous conclusion for power quasisymmetries.
Fraser--Tyson~\cite{FraserTyson2026} studied the distortion of intermediate dimensions under
quasiconformal and quasisymmetric maps, with
consequences for Mandelbrot percolation realisations.
We refer the reader to the recent survey by Kolossv\'ary--Troscheit~\cite{KolossvaryTroscheit2025}
for more background.
We highlight here that these results compute or bound quasisymmetric invariants, whereas our
conclusion instead produces the stronger fact of quasisymmetric equivalence between two independent
realisations in strongly separated models.

\paragraph*{\bf Random branching times} Our matching theorems \cref{thm:matching,thm:markov-matching}
also apply when the binary genealogy is deterministic and the times between
successive branch points are random. In \cref{sec:branching-times}, we introduce a family which
permits branch points to be arbitrarily close and arbitrarily far apart. We prove that two
independent realisations, even with different parameters, are almost surely quasi-isometric. This
family has non-doubling boundary and is separated from the shifted-exponential model, whose
branching times are bounded below and whose boundary is doubling. Finite truncations of both
families are shown in \cref{fig:branching-times}.

\paragraph*{\bf Brownian trees and spheres} Finally, we highlight another connection of branching processes with mathematical physics.
The Brownian continuum random tree of
Aldous~\cite{Aldous1991,Aldous1993} is the universal scaling limit of critical Galton--Watson
trees with finite offspring variance conditioned to have $n$ vertices, as $n\to\infty$. A related
object from random geometry is the Brownian sphere. Le~Gall~\cite{LeGall2013} and
Miermont~\cite{Miermont2013} established convergence to it from uniform quadrangulations and
related families of planar maps.
While it is known, see Troscheit~\cite{Troscheit2021}, that neither
the Brownian tree nor the Brownian sphere admits a quasisymmetric embedding into a
doubling metric space, a finer question asks whether two independent realisations are
quasisymmetrically equivalent almost surely. This was recently answered in the negative for the
Brownian sphere by Miller--Tian~\cite{MillerTian2026c}, who also computed the conformal dimensions
of the continuum random tree and Brownian sphere as $1$ and $2$, respectively. See
\cite{MillerTian2026a,MillerTian2026b} for details.
The natural question that is left open between our regime of supercritical branching processes and
that of the Brownian sphere is that for the continuum random tree, which was also asked by
Bonk--Meyer in \cite[Problem~9.1]{BonkMeyer2022}.
A finer analysis of our matching results may be able to answer this question.

\subsection{Ideas of the proof}
The main new tool we introduce in this paper is an automorphism-matching principle for random
graph-indexed labellings of trees, which are of independent interest. We first contract or decompose
a Galton--Watson realisation until its coarse geometry is recorded by a random choice of
graph-indexed labels on a \emph{branching skeleton}. A tree automorphism of the skeleton then
matches two independent label fields while allowing bounded local errors, and local quasi-isometries
between compatible labels glue to a global quasi-isometry.

\paragraph*{\bf Graph labellings}
In many of our settings, the skeleton will be an infinite binary tree, and the labels at each vertex
will be independent and take values in an arbitrary countable graph. The graph records which local
pieces are compatible, and a single scalar potential controls the matching failures. This gives
\cref{thm:matching}. More generally, the skeleton has several possible arities. To compare such
skeletons on the binary tree, we replace a split with $k$ children by a finite binary tree with $k$
leaves. Several successive binary splits may now represent a single original branch point, so the
resulting labels are dependent. We describe this dependence by a tree-indexed Markov field, in the
sense of Benjamini--Peres~\cite{BenjaminiPeres1994}, and extend the matching argument to this
setting in \cref{thm:markov-matching}.

The binary replacements must also allow us to compare trees with different branching laws.
In the chain regime, equality of the branching semigroups lets us choose compatible
replacements. In the bushy regime, the finite decorations provide enough flexibility to compare
any two offspring laws. Both matching theorems are independent of Galton--Watson trees and may
be useful in other tree-indexed matching problems.

\paragraph*{\bf Embeddings}
For \cref{thm:embedding-hierarchy}, the classification itself becomes a tool. Within each chain
class, we choose an offspring law concentrated at a large arity whose excess belongs to the given
branching semigroup. We repeatedly discard vertices that do not have this arity or have too few
children remaining from the previous round. A binomial estimate gives positive probability that
the root survives every round, and independent trials along a ray locate a surviving vertex at
finite depth almost surely. The retained descendants contain an isometric binary tree, whose
embedding transfers to the original law by quasi-isometry. For a bushy realisation, we apply this
argument to its survival skeleton and leave the finite bushes unused. The reverse embedding
follows from the bounded offspring support.

\subsection{Lean formalisation}

\paragraph*{\bf Overview}
The Galton--Watson classification, mutual embeddability and the i.i.d.\ and Markov matching
theorems have been formalised in Lean~4 over the mathematical library Mathlib. The Lean files are
freely available at
\url{https://github.com/sascha2718/qi-trees-and-matching-lean} and are archived and machine-verified on
\textit{Palomar}~\cite{AthreyaTroscheit2026}.
The development consists of four libraries. The
first, \texttt{BranchingProcess}, constructs Galton--Watson trees from i.i.d.\ offspring fields and
proves the required measurability, extinction, conditioning, Harris decomposition, progeny tails,
pruning estimates and coarse geometric results. \texttt{GraphMatching} proves the finite and
infinite i.i.d.\ automorphism-matching theorems, for both leaf and full labellings. \texttt{GraphMarkovMatching}
proves the general Markov matching theorem at finite and infinite height, together with its
quantitative bounds. \texttt{ChainClasses} formalises the chain and shape encodings, their laws, the
coupling and geometric transfer arguments, and the universality, separation and classification of
conditioned infinite realisations, both within and across arbitrary bounded-support offspring laws.
In particular, all four classes $(\mathrm{R})$, $(\mathrm{F})$, $(\mathrm{C}_{\Lambda})$
and $(\mathrm{B})$, as well as the qualitative and quantitative two-value results, are covered.
The development also proves \cref{thm:embedding-hierarchy}, mutual embeddability for two
survival-conditioned laws and the strict hierarchy with finite trees and the ray. The pruning
estimates, conditional law of the retained descendant tree and regular-subtree construction are
proved in \texttt{BranchingProcess}. The library \texttt{ChainClasses} locates a retained vertex
almost surely, proves the explicit first-child-ray tail bound, and transfers the embedding
through the classification.

\paragraph*{\bf Challenge files} 
To verify the machine-checked version of our results we use the \textit{comparator} audit tool
(\url{https://github.com/leanprover/comparator}).
It requires a single challenge file, \texttt{Challenge.lean}, which imports Mathlib
and nothing from our four
libraries, and defines only the vocabulary needed to state thirteen headline results of this paper:
the four finite- and infinite-height conclusions of \cref{thm:markov-matching}; the leaf and full
finite-height bounds
\cref{eq:leaf-bound,eq:full-bound} and the infinite-tree conclusion of
\cref{thm:matching}\labelcref{it:matching-full}; the complete classification of
\cref{thm:trichotomy} in three statements, that bounded connected graphs are quasi-isometric to a
point, that a law with mean at most one and $\theta_1\neq1$ dies out almost surely, and the
classification itself over the unconditioned laws, finite class included; the mutual-embeddability
statement of \cref{thm:embedding-hierarchy}; and the qualitative
conclusion and quantitative bound of
\cref{thm:twovalue}, the latter being \cref{eq:rate}. Its theorems have the placeholder proof
\texttt{sorry} by design. The proofs live in the separate file
\texttt{Solution.lean}, and the comparator audit checks that they prove exactly the challenge statements.
Our libraries and the solution are free of \texttt{sorry} and assume no external mathematical
result as an axiom. Only the standard axioms of
propositional extensionality, choice, and quotient soundness are assumed.
This comparator check was also performed by the Palomar registration. 

\paragraph*{\bf Lacunae} The formalisation does not include the random
branching-time theorem of \cref{sec:branching-times} or the external passage from quasi-isometries
to quasisymmetries of visual boundaries used in \cref{thm:boundary} and its fractal applications.
Formalising them would require significant work on an interface between the different categories
rather than checking their otherwise simple proofs.
For the declaration correspondences, see \cref{app:lean}. Detailed notes on differences in
formulation are contained in \texttt{paper-correspondence.yaml}, as part of the Lean development and
also recorded in the Palomar record \cite{AthreyaTroscheit2026}.

\subsection{Organisation.}
In the remainder of \cref{part:introduction} we fix the common notation, prove the boundary
applications, outline the proof of the classification, and state the i.i.d.\ and Markov matching
theorems. We then apply the i.i.d.\ theorem to random branching times in
\cref{sec:branching-times} and end with the further
directions in \cref{sec:beyond}. \Cref{part:binary-classification} develops the binary-support model,
including the explicit estimate in \cref{thm:twovalue} and the treatment of finite bushes.
\Cref{part:full-classification} restates the Markov matching theorem in full detail, proves it,
and assembles the full classification. It concludes with the embedding theorem and strict hierarchy
in \cref{sec:embedding-hierarchy}. A more detailed breakdown of the proof ideas can be found in
\cref{sec:strategy} after we introduce more notation. We strongly urge the reader to read the
remainder of \cref{part:introduction} before reading
\cref{part:binary-classification,part:full-classification}. This gives a more concise introduction
to our proofs. We further recommend studying \cref{part:binary-classification} before
\cref{part:full-classification}. While the results in \cref{part:full-classification} strictly
subsume those in \cref{part:binary-classification}, extending the techniques to offspring
distributions with arity greater than $2$ requires more technical work. A glossary and symbol list
can be found in \cref{app:glossary}, which collects the notation used throughout the paper. It is
followed by the Lean appendix in \cref{app:lean}, which lists the declaration correspondences.

\section*{Acknowledgements}

This project started in 2017 and over its almost decade-long history many funders were
involved and many colleagues contributed through discussions. We record here a necessarily
incomplete account of the agencies and colleagues to whom the authors are grateful. Without them,
this article would not have been completed, at least not in its current form.

JA was partially supported by National Science Foundation (NSF) grants DMS-2003528 and DMS-2404705, the
Pacific Institute for the Mathematical Sciences, the Royalty Research Fund and the Victor Klee
Fund at the University of Washington, and the Chaire Jean Morlet programme at the Centre
International de Rencontres Math\'ematiques (CIRM) Luminy.

ST was initially supported by the Natural Sciences and Engineering Research Council of Canada
(NSERC) grants 2014-03154 and 2016-03719, and the University of Waterloo; subsequently by the
Austrian Science Fund (FWF) grant M-2813 and by the European Union/European Research Council
(EU/ERC)Marie Sk\l{}odowska-Curie Actions Postdoctoral Fellowship no.\ 101064701; and from 2025 by
the Royal Swedish Academy of Sciences grant MA2025-0070.

The authors thank CIRM Luminy for hospitality and ideal working conditions in Autumn 2024, where
many of the `grand ideas' for this paper came together, even if ST broke his literal back.
In particular, we thank Alice Contat for
an inspiring talk at the CIRM workshop \textit{Probability and Geometry in, on and of non-Euclidean spaces}, on the car-parking problem on the binary tree studied by
Aldous--Contat--Curien--H\'enard~\cite{AldousContatCurienHenard2023} and its generalisations.
This eventually led to the matching theorems of this article, even if the connection is no longer
recognisable.

We would also like to thank Krishna B. Athreya, Benny Avelin, Svante Janson, Pekka Pankka, Steffen Rohde, Eero Saksman,
Richard Schwartz, Jeff Steif, and Ville Suomala for interesting discussions.

For the formalisation, ST would like to thank Tobias Osborne for making the `informal
formalisation' tools \textit{vibefeld} and \textit{alethfeld} freely available. These tools helped
us structure several earlier versions of the proofs and allowed us to overcome technical hurdles.

ST is also very grateful to Kalle Kyt\"ol\"a for encouraging him to learn Lean and thanks Heather
Macbeth for an excellent introductory talk on the topic. The recent AI-assisted and machine-checked
account of interior De Giorgi--Nash--Moser theory by Armstrong and Kempe~\cite{ArmstrongKempe2026}
as well as the Lean formalisation by Armstrong and Kuusi of their coarse-graining theory
(\url{https://github.com/scottnarmstrong/CoarseGraining}), which covers a uniformly elliptic case
of their homogenisation estimates~\cite{ArmstrongKuusi2025}, provided further motivation for the
full formalisation of this article.
Terence Tao's article
\href{https://terrytao.wordpress.com/2023/11/18/formalizing-the-proof-of-pfr-in-lean4-using-blueprint-a-short-tour/}{\emph{Formalizing the proof of PFR in Lean4 using Blueprint: a short tour}}
was particularly helpful for us in thinking how one can organise a large formalisation.

\section*{Statement on the use of artificial intelligence}

This disclosure follows the principles of the Leiden Declaration on the responsible use of
artificial intelligence in research and writing (\url{https://leidendeclaration.ai}). The AI models
Anthropic's Claude and OpenAI's ChatGPT/Codex were used to formalise arguments in Lean, optimise
parts of the original proofs, and assist with copy-editing.

Two AI-assisted formalisation workflows were explored: the \textit{vibefeld} system developed by
Tobias Osborne, which produced a readable intermediate formalisation, and a custom derived workflow
developed by ST for the final Lean formalisation. The authors selected, checked, and curated every
retained mathematical argument and verified any non-human parts of the manuscript and the
formalisation. They accept full responsibility for their contents.

\section{Notation and preliminaries}

We collect the conventions used throughout the paper. We first describe rooted trees and their
Galton--Watson laws, and then recall the two metric equivalence notions used for the trees and their boundaries.

\subsection{Words and rooted trees}\label{sec:words}

For $N\in\bbN$, let $\cN=\cN(N)$ be the tree of finite words over
$\set{1,\dots,N}$. Its root $\troot$ is the empty word, and $v$ is joined to $vj$ for every
letter $j$. We equip $\cN$ with the graph metric
\[
  d(v,w)=\abs v+\abs w-2\abs{v\wedge w},
\]
where $v\wedge w$ is the longest common prefix. The binary tree is $\bbB=\cN(2)$. We write
$B_Y(y,r)$ for the closed ball in a metric space $Y$ and $\Aut(\bbB)$ for the rooted
automorphism group of $\bbB$.

\subsection{Galton--Watson trees}\label{sec:gw-trees}

Let $\theta=(\theta_0,\dots,\theta_N)$ be a finitely supported probability distribution on
$\bbN_0$, and let $(\chi_v)_{v\in\cN}$, the child count of vertex $v$, be independent random variables with this law. The
associated Galton--Watson tree $\cT\subseteq\cN$ contains the root, and whenever $v\in\cT$, it
contains exactly the children $v1,\dots,v\chi_v$. The law
is \emph{supercritical} when
\[
  \sum_{k=0}^N k\theta_k>1.
\]
Unless stated otherwise, realisations in the infinite classification are conditioned on having
infinite diameter and distinct realisations appearing in the same statement are assumed to come from
independent sampling.
We use $\bbP$ and $\bbE$ for the probability and expectation under the offspring law
specified in the current statement or section, including its conditioning on infinite diameter
and any independent auxiliary variables. We write $\bbP^*$ when we need to distinguish the
survival-conditioned law explicitly from the original Galton--Watson law.

\subsection{Quasi-isometries}\label{sec:qi-qs}

\begin{definition}[Quasi-isometry]\label{def:qi}
  Let $(Y_1, d_1)$ and $(Y_2, d_2)$ be metric spaces. We say that a map $\psi\colon(Y_1,d_1)\to(Y_2,d_2)$ is an \emph{$(A,B,C)$-quasi-isometry} if
  \begin{enumerate}
    \item\label{it:embed} $A^{-1}d_1(x,y)-B\leq d_2(\psi(x),\psi(y))\leq A d_1(x,y)+B$ for all
      $x,y\in Y_1$;
    \item every $y\in Y_2$ lies at distance at most $C$ from $\psi(Y_1)$,
  \end{enumerate}
  where $A\geq1$ and $B,C\geq0$.
\end{definition}
For rooted spaces, a quasi-isometry is \emph{root-preserving} if it sends the source root to the
target root. Changing only the image of the root is a bounded change of the map and therefore
preserves the quasi-isometry property, although its constants may change.

If a map satisfies only the metric inequality in \cref{def:qi}\labelcref{it:embed}, we call it a
\emph{quasi-isometric embedding}. We write $X\preccurlyeq Y$ when such an embedding
$X\hookrightarrow Y$ exists, and $X\prec Y$ when $X\preccurlyeq Y$ but $Y\not\preccurlyeq X$.
Quasi-isometric embeddings compose, so $\preccurlyeq$ is transitive.
Note that a quasi-isometry from $Y_1$ to $Y_2$ induces a quasi-isometry from $Y_2$ to $Y_1$ and
hence being quasi-isometric is an equivalence relation.
We write $Y_1\simeq Y_2$ when a quasi-isometry exists.
Throughout the paper, a \emph{$D$-quasi-isometry} means a $(D,D,D)$-quasi-isometry. We do not
aim to optimise constants.

\subsection{Visual boundaries and quasisymmetries}\label{sec:visual}

For an infinite locally finite rooted tree $\cT$, let $\partial\cT$ be the set of infinite rays
from the root. For $\xi,\zeta\in\partial\cT$, set
\begin{equation}\label{eq:visual-metric}
  d_{\mathrm{vis}}(\xi,\zeta)=
  \begin{cases}
    2^{-\abs{\xi\wedge\zeta}},&\xi\ne\zeta,\\
    0,&\xi=\zeta.
  \end{cases}
\end{equation}
This defines a compact ultrametric boundary.

\begin{definition}[Quasisymmetry]\label{def:qs}
  Let $\eta\colon[0,\infty)\to[0,\infty)$ be a homeomorphism. A homeomorphism
  $\psi\colon(Y_1,d_1)\to(Y_2,d_2)$ is \emph{$\eta$-quasisymmetric} if, for all distinct
  $x,y,z\in Y_1$ and all $t>0$,
  \[
    d_1(x,z)\leq t d_1(y,z)
    \quad\Longrightarrow\quad
    d_2(\psi(x),\psi(z))\leq\eta(t)d_2(\psi(y),\psi(z)).
  \]
\end{definition}

We refer to Buyalo--Schroeder~\cite{BuyaloSchroeder2007} for more background on visual boundaries and to
Heinonen~\cite{Heinonen2001} for an extensive treatment of quasisymmetric maps.

\section{Boundaries and quasisymmetry}

\paragraph*{\bf Boundaries} We now prove the boundary consequence stated in \cref{thm:boundary}. Recall that the ray regime
is excluded there because its boundary is a singleton rather than a Cantor space.

\begin{proof}[Proof of \cref{thm:boundary}]
  By \cref{thm:trichotomy}, the two trees are almost surely quasi-isometric. The unit-edge geometric
  realisation of each tree is a proper geodesic $0$-hyperbolic space, and the inclusion of the
  vertex set is a quasi-isometry. The two geometric realisations are therefore quasi-isometric.
  Buyalo--Schroeder~\cite[Theorem~5.2.17]{BuyaloSchroeder2007} gives a quasisymmetric homeomorphism
  between their Gromov boundaries for any choices of visual metrics. The end boundary is canonically
  the same as that of the vertex tree, and the metric in \cref{eq:visual-metric} is the visual
  metric based at the root with parameter $\log 2$.

  The end boundary of an infinite locally finite rooted tree is compact and totally disconnected.
  Outside the ray and finite regimes, almost surely every vertex of the survival skeleton has a
  descendant with at least two infinite child subtrees. This is immediate in the full tree regime.
  In the chain and bushy regimes, the branching property gives probability zero to a surviving
  subtree whose skeleton remains a single ray forever. Applying this fact at the countably many
  vertices shows that no boundary point is isolated. Both boundaries are therefore perfect.
\end{proof}

\subsection{Random Cantor boundaries} 
The visual boundary of a tree gives a metric that records the successive branching choices along
infinite rays. Long unary chains produce gaps between branching scales and decrease diameters of
balls, while finite bushes disappear from the boundary but still affect the geometry used to
construct the quasi-isometry. Quasi-isometries of the trees correlate to quasisymmetries on this
boundary and our \Cref{thm:boundary} shows that two
independent realisations are still equivalent up to the `right' notion of mapping. As mentioned
above this also implies that quasisymmetric invariants, like the conformal dimension, are preserved
between independent generic realisations of the same class. The reverse implication is not
necessarily true.

\paragraph*{\bf Separated systems} 
This quasisymmetric conclusion persists even when the geometry of the boundary is replaced by
strongly separated self-similar systems. The coding maps are bi-Lipschitz when we equip the boundaries
with the visual metrics determined by the contraction ratios. The corresponding tree edge lengths
are uniformly comparable to unit lengths. For the percolation models below, the retained trees
belong to class (B), so their weighted versions are almost surely quasi-isometric.
The two same-IFS examples in
\cref{fig:fractal-percolation} are therefore a special case of the following corollary. We use
\emph{fractal percolation} to mean a general percolated structure of self-similar iterated function
constructions. Mandelbrot percolation is the special
case in which the IFS is generated by subdividing an $n$-dimensional cube into congruent smaller
$n$-dimensional cubes.

\begin{figure}[t]
  \centering
  \includegraphics[width=.42\textwidth]{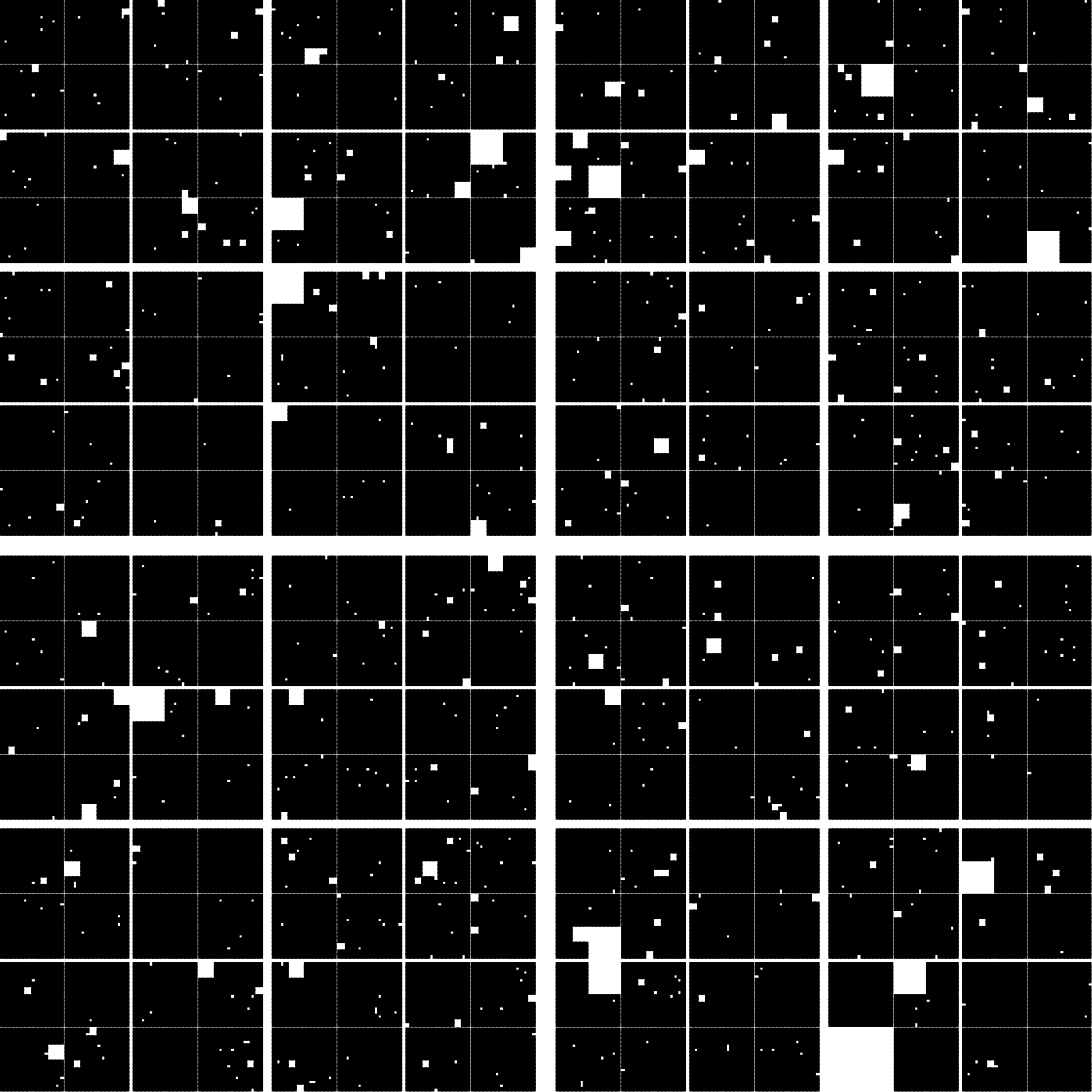}\hfill
  \includegraphics[width=.42\textwidth]{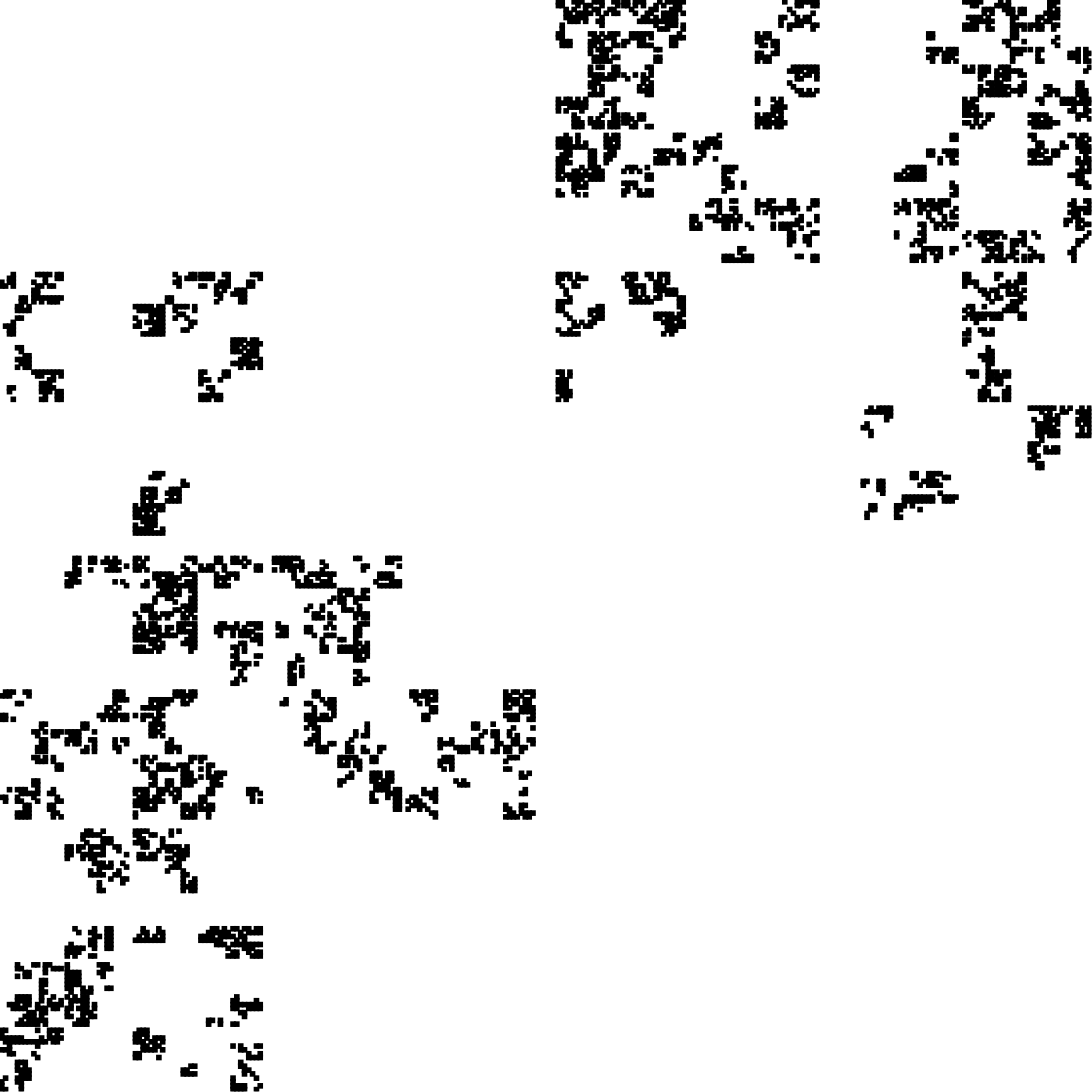}
  \caption{Finite approximations to two fractal percolation realisations for the IFS formed by
    the four corner similarities of ratio $0.49$. The retention probabilities are $0.99$ on the
    left and $0.7$ on the right.}
  \label{fig:fractal-percolation}
\end{figure}

\begin{corollary}[Fractal percolation across different IFSs]
  \label{thm:fractal-percolation}
  Let $\Phi,\Phi'$ be finite strongly separated iterated function systems consisting of $M,M'\geq2$
  contracting similarities, and retain each child independently with probability $p,p'$,
  respectively. Suppose that
  \[
    0<p,p'<1
    \quad\text{and}\quad
    pM,p'M'>1.
  \]
  Let $E$ and $E'$ be independent fractal percolation limit sets for these constructions.
  Conditioned on both sets being nonempty, they are almost surely quasisymmetrically equivalent.
\end{corollary}

\begin{proof}[Proof of \cref{thm:fractal-percolation}]
  The retained words for $\Phi$ and $\Phi'$ form Galton--Watson trees $\cT$ and $\cT'$ with
  offspring laws $\operatorname{Bin}(M,p)$ and $\operatorname{Bin}(M',p')$, respectively. These
  laws are supercritical and have respective supports $\set{0,\ldots,M}$ and
  $\set{0,\ldots,M'}$, so both belong to regime (B). By \cref{thm:trichotomy}, conditioned on
  nonextinction, the unit-edge trees $\cT$ and $\cT'$ are almost surely quasi-isometric.

  Write $r_i$, $1\leq i\leq M$, for the contraction ratios (so $r_i \in (0, 1)$) of the maps in
  $\Phi$, and write $r'_i$, $1\leq i\leq M'$, for those of the maps in $\Phi'$. We give an edge of
  $\cT$ ending in the letter $i$ length $\log (1/r_i)$, and define the edge lengths of $\cT'$
  analogously using $r'_i$. Since both systems are finite, the identity map from each unit-edge tree
  to its weighted version is bi-Lipschitz. The two weighted trees are therefore quasi-isometric, and
  the boundary theorem used in the proof of \cref{thm:boundary} gives a quasisymmetric homeomorphism
  between their visual boundaries.

  It remains to identify these boundaries with the random limit sets. If two coding rays have
  longest common prefix $u$, their visual distance in the weighted tree is the product of the
  contraction ratios along $u$. Strong separation bounds the Euclidean distance of their images
  above and below by fixed multiples of the same product. The coding maps from $\partial\cT$ and
  $\partial\cT'$ to $E$ and $E'$, respectively, are consequently bi-Lipschitz. Composing these
  coding maps with the boundary homeomorphism proves the claim.
\end{proof}

\paragraph*{\bf Percolation structures}
We note that choosing the same percolation structure was purely for convenience and all we required
here was that their associated Galton--Watson trees lie in the same quasi-isometry class of
\cref{thm:trichotomy}. Analogous results hold for much wider classes of stochastically self-similar
fractals. Further, the limit sets of those constructions are totally disconnected and have conformal
dimension $0$. Our result goes beyond that by saying that these different realisations of fractal
percolation are also quasisymmetrically equivalent when their associated trees lie in the same
quasi-isometry class. This is possibly counter-intuitive as these sets are not uniformly perfect
(they have lower dimension $0$), yet are maximal with respect to the Assouad
dimension~\cite{FraserMiaoTroscheit2018,KolossvaryTroscheit2025}.

\section{Proof strategy}\label{sec:strategy}

\paragraph*{\bf Scheme}
The regimes illustrated in \cref{fig:classes} require different local approaches, but each case of
quasi-isometric equivalence is proved by the same strategy:
\[
  \begin{gathered}
    \text{encoding realisations}
    \longrightarrow\text{branching skeleton with local pieces}
    \longrightarrow\text{quantised random labels}\\
    \longrightarrow\text{automorphism matching}
    \longrightarrow\text{(root preserving) quasi-isometry}.
  \end{gathered}
\]
\paragraph*{\bf Labellings and matchings}
The construction of the labels and the matching theorem change from one regime to another, but
the transfer back to the original realisations always glues uniformly controlled local maps.

\subsection{The two-value chain model}
When $0<\theta_1<1$ and $\theta_1+\theta_2=1$, we contract maximal unary chains and label vertices
in the binary tree $\bbB$ by that contraction length. The contracted lengths form an i.i.d.~field
of geometrically distributed random variables on $\bbB$. To try and build a quasi-isometry
depending on a parameter $D>1$, we quantise a length
$\lambda$ by its logarithmic (base $D$) scale, that is, $\floor{\log_D\lambda}$. Matchable
`adjacent' quantised
labels correspond to lengths comparable within a factor of order $D^2$, and \cref{thm:matching}
aligns all labels by one automorphism. The transfer theorem then produces a
$(D^2+3)$-quasi-isometry. The potential estimate \cref{eq:eta-bound} gives the probability bound \cref{eq:rate}, and
taking the union over $D$ proves almost-sure equivalence.

\subsection{Leaves and finite shapes}
When $\theta_0>0$, a realisation conditioned on having infinite diameter splits into an infinite
skeleton and finite bushes attached to it, the \emph{Harris decomposition}.
This decomposition makes the decorations independent after conditioning on
the skeleton. A local label now records a finite decorated shape rather than one chain length. We
replace the shape space by a countable net at scale $D$, use entropy dilution to bound its graph
potential, apply the matching theorem, and glue the matched local quasi-isometries.
\Cref{part:binary-classification} carries this out for offspring supported on $\set{0,1,2}$.

\subsection{General arity and two offspring laws} 
For the general, arbitrary bounded support case, the skeleton has several possible arities.
We replace each split
with $k$ children by a finite binary tree with $k$ leaves. Its root carries the label of the
original piece, and its other internal vertices carry the label of a one-vertex piece. These
replacements have bounded height for each fixed pair of offspring laws, so the resulting
assemblies are quasi-isometric to the original trees.

The choice of a replacement determines where the next independent labels occur. We record
the position within each replacement as a type, obtaining a tree-indexed Markov field to which
\cref{thm:markov-matching} applies. To compare two branching laws, we choose the replacements
from a common family of smaller binary trees. In the bushy regime, the reduced skeletons both
allow binary splits. In the chain regime, a common family exists whenever the branching
semigroups agree. We can then match the labels and glue the local comparisons as in the binary
case.

\subsection{Coarse obstructions}
The negative results use deterministic features that occur at arbitrarily large scales, namely:
\begin{itemize}
  \item Three diverging rays rule out the ray regime.
  \item Long unary chains create thin balls and separate the chain regime from the full tree regime.
  \item Arbitrarily deep finite dead ends separate the bushy regime from the leafless regimes.
\end{itemize}
Finally, if $m\in\Lambda_\theta\setminus\Lambda_{\theta'}$, the first chain law almost surely
produces isolated clusters with total excess $m$ and arbitrarily long outgoing arms. No realisation
from the second law can reproduce the corresponding coarse star. These arguments are proved
alongside the classification results in
\cref{part:binary-classification,part:full-classification}.

\section{Matching random labellings}\label{sec:matching}

\paragraph*{\bf Automorphism matching}
The common probabilistic ingredient is the following automorphism-matching theorem. Let $G=(V,E)$ be a
countable graph and let $\mu$ be a probability measure on $V$. We view the vertex set $V$ as a set
of possible \emph{labels}, and will call maps from fixed binary trees to $G$ \emph{labellings}. We
say labels $v,w\in V$ are \emph{compatible} when $d_G(v,w)\leq1$. Write
\[
  b(v)=\mu\bigl(B_G(v,1)\bigr)
\]
and, for $\alpha\geq1$, we define the graph potential (to penalise labels that are likely to occur
but difficult to match) by:
\begin{equation}\label{eq:potential}
  \eta_{G,\alpha}(\mu)
  =\sum_{v:\,\mu(v)>0}\mu(v)\left(\frac{1-b(v)}{b(v)^\alpha}\right).
\end{equation}

\paragraph*{\bf Leaf and full labellings} Let $\bbB_h$ be the complete binary tree of height $h$.
A \emph{$G$-leaf labelling} is a map $x\colon\{1,2\}^h\to V$ assigning a label to each leaf.
A \emph{full $G$-labelling} is a map $x\colon\bbB_h\to V$ assigning a label to every vertex.
An automorphism matches two labellings, either leaf or full, if it sends every relevant vertex
to one carrying a compatible label.

\begin{theorem}[Matching theorem]\label{thm:matching}
  Let two independent labellings of $\bbB_h$ be given, with labels sampled independently
  from $\mu$ within each labelling.
  \begin{enumerate}[(1),ref=(\arabic*)]
    \item\label{it:matching-leaf} If $\eta_{G,5/2}(\mu)\leq\tfrac1{256}$, then for leaf
      labellings
      \begin{equation}\label{eq:leaf-bound}
        \bbP(\text{no automorphism of }\bbB_h\text{ matches every leaf})
        \leq\eta_{G,5/2}(\mu)\left(\frac{253}{256}\right)^{\!h},
      \end{equation}
      which tends to zero exponentially as $h\to\infty$.
    \item\label{it:matching-full} If $\eta_{G,5/2}(\mu)\leq10^{-4}$, then for full
      labellings, uniformly in the height,
      \begin{equation}\label{eq:full-bound}
        \bbP(\text{some automorphism of }\bbB_h\text{ matches every vertex})
        \geq1-16\eta_{G,5/2}(\mu)>0.
      \end{equation}
      The same lower bound holds for the existence of a single automorphism matching every
      vertex of two independent labellings of the infinite rooted binary tree, each with
      i.i.d.\ labels of law $\mu$.
  \end{enumerate}
\end{theorem}

\paragraph*{\bf Decay and finite-height problems} The decay in \cref{eq:leaf-bound} solves the finite-height leaf problem. For full labellings,
K\H{o}nig's infinity lemma~\cite{Konig1927} turns compatible finite-height automorphisms into one
automorphism of the infinite tree, so the uniform estimate \cref{eq:full-bound} gives the stated
infinite conclusion. The numerical hypotheses are convenient rather than sharp, and will be
discussed in more detail in \cref{rem:matching-constants}.

\subsection{Scalar recursion and a failure criterion}\label{sec:iid-scalar-recursion}

Write
\[
  p_c=\bbP\bigl(d_G(X,Y)\leq1\bigr),
  \qquad X,Y\sim\mu\text{ independently},
\]
for the one-site compatibility probability. Let $\mu_h$ be the product law on full height-$h$
labellings, write $x\approx_hy$ when some automorphism of $\bbB_h$ matches $x$ and $y$, and put
\[
  u_h=\bbP(X\approx_hY),
  \qquad X,Y\sim\mu_h\text{ independently}.
\]
For two fixed height-$h$ labellings define their mutual acceptance probability by
\[
  a_h(x_1,x_2)=\bbP(x_1\approx_h Z,\ x_2\approx_h Z),
  \qquad Z\sim\mu_h.
\]
Here and below, $X_1,X_2$ are independent with law $\mu_h$.

\paragraph*{\bf A recursion for labelling} The mutual acceptance moment measures the overlap between the straight and crossed child
pairings. It gives the correction term in the following exact recursion.

\begin{proposition}[Exact full-labelling recursion]\label{thm:full-probability-recursion}
  We have $u_0=p_c$, and, for every $h\geq0$,
  \[
    u_{h+1}=p_c\left(2u_h^2-\bbE\bigl[a_h(X_1,X_2)^2\bigr]\right).
  \]
\end{proposition}

\begin{proof}
  The identity $u_0=p_c$ follows from the definitions. At a binary branch, the choice between the
  straight and crossed child pairings is the $2\times2$ case of Hall's marriage
  theorem~\cite{Hall1935}. Each pairing succeeds with probability $u_h^2$. Conditional on the two
  source subtrees $X_1,X_2$, both pairings succeed precisely when each of the two independent target
  subtrees matches both sources, an event of probability $a_h(X_1,X_2)^2$. Inclusion--exclusion
  gives the expression in parentheses, and independence of the root labels from the principal
  subtrees supplies the factor $p_c$.
\end{proof}

\paragraph*{\bf Non-closed form and moment bounds}
The recursion does not close in terms of $u_h$, since the mutual acceptance moment contains
information not determined by the mean matching probability. Universal moment bounds nevertheless
give a useful obstruction.

\begin{corollary}[Scalar bounds and failure criterion]\label{thm:iid-scalar-failure}
  For every $h\geq0$,
  \[
    \bbE a_h(X_1,X_2)\geq u_h^2,
    \qquad
    \bbE\bigl[a_h(X_1,X_2)^2\bigr]\geq u_h^4,
  \]
  and consequently
  \[
    u_{h+1}\leq p_cu_h^2(2-u_h^2)
    \leq\frac{4\sqrt6}{9}p_cu_h.
  \]
  In particular, if
  \[
    p_c<\frac{9}{4\sqrt6}=\frac{3\sqrt6}{8},
  \]
  then $u_h\to0$ geometrically, and two independent infinite full labellings almost surely admit
  no matching automorphism.
\end{corollary}

\begin{proof}
  Fubini's theorem and Jensen's inequality give
  \[
    \bbE \left[a_h(X_1,X_2)\right]
    =\bbE\left[\bbP(X\approx_hY\mid Y)^2\right]
    \geq u_h^2.
  \]
  A second application of Jensen's inequality gives
  \[
    \bbE\bigl[a_h(X_1,X_2)^2\bigr]
    \geq\bigl(\bbE a_h(X_1,X_2)\bigr)^2
    \geq u_h^4.
  \]
  Substitution in \cref{thm:full-probability-recursion} gives the first scalar bound. The second
  follows from
  \[
    \max_{0\leq t\leq1}t(2-t^2)=\frac{4\sqrt6}{9}.
  \]
  Under the stated hypothesis, the resulting linear contraction shows that $u_h\to0$
  geometrically. For each $h$, the infinite matching event is contained in the height-$h$ matching
  event, which has probability $u_h$, so the infinite matching event has probability zero.
\end{proof}

\paragraph*{\bf Thresholds} The threshold in \cref{thm:iid-scalar-failure} is
$3\sqrt6/8=0.918558653\ldots$. For comparison, a direct first-moment count for full labellings
gives
\[
  \bbE\bigl[\text{number of matching automorphisms of }\bbB_h\bigr]
  =2^{2^h-1}p_c^{2^{h+1}-1},
\]
which tends to zero only under the more restrictive condition $p_c<1/\sqrt{2}$.

\section{Matching Markov labellings}

\paragraph*{\bf Markov labellings} The independent-label theorem does not directly apply when a split with several children is
represented by successive binary splits. As described in \cref{sec:strategy}, the original
vertices retain their random labels and the inserted vertices receive the label of a
one-vertex piece. The branching choices determine where new random labels occur, creating
dependence between the labels. We describe this dependence by assigning a \emph{type} to each
vertex. The next theorem applies to general transitions between types, including the
binary encodings used in the classification.

\subsection{The process and its one-site potential}\label{sec:process}

We use the graph $G=(V,E)$, label law $\mu$, compatible mass $b(v)$ and potential
$\eta_{G,\alpha}(\mu)$ of \cref{sec:matching}. Fix a distinguished label $0\in V$.
In addition to comparing two independent labels with law $\mu$, we must compare such a
label with $0$. The latter comparison fails with probability $1-b(0)$. For $\alpha\geq1$,
we set
\begin{equation}\label{eq:root-defect}
 \zeta_{G,\alpha}(\mu)\coloneqq
 \max\left\{\eta_{G,\alpha}(\mu),\frac{1-b(0)}{b(0)^\alpha}\right\}.
\end{equation}
For fixed $G$, we abbreviate the two potentials by $\eta_\alpha(\mu)$ and
$\zeta_\alpha(\mu)$, omitting $\mu$ when it is fixed as well. We make the convention that
$\zeta_\alpha=\infty$ when $b(0)=0$.

The second term in \cref{eq:root-defect} can be necessary even when $\mu(0)=0$. For
example, take $G$ to be the path $0$--$1$--$2$ and let
$\mu=(1-t)\delta_2+t\delta_1$, where $0<t<1$. The supported labels are mutually compatible,
so $\eta_\alpha=0$, whereas $b(0)=t$. A root with prescribed label $0$ is incompatible with
an independent root with law $\mu$ with probability $1-t$, and the corresponding term is
$(1-t)/t^\alpha$. Thus $\eta_\alpha$ alone cannot control comparisons involving inserted
vertices, whose label is $0$. If $\mu(0)\geq p>0$, the summand of $\eta_\alpha$ at $0$ gives
\[
 \zeta_\alpha\leq\eta_\alpha/p.
\]

\paragraph*{\bf Sets of types} Let $\cI$ be a nonempty countable set of types and choose a subset $\cI_\mu\subseteq\cI$.
At a vertex whose type belongs to $\cI_\mu$, we draw a label with law $\mu$; at every
other vertex we put the label $0$. For each $t\in\cI$, let $P_t$ be a probability law
on $\cI\times\cI$. At a vertex of type $t$, independently of its label, we draw the
types of its two children with law $P_t$. Conditioned on these types, the two descendant
processes are independent and follow the same rule. We use independent
draws at distinct vertices.

\paragraph*{\bf Matching}
Matching has the same meaning as in \cref{sec:matching}: an automorphism of the rooted binary tree
must pair compatible labels at every vertex. We emphasise here: \textit{It need not preserve types.}
For independent processes started at $s,t\in\cI$, write $\cM_h(s,t)$ for the event of a match
through height $h$, and $\cM_\infty(s,t)$ for a match of the infinite labellings. Two different
models are included by taking the disjoint union of their type sets and letting each transition law
stay within its own model.

\subsection{The finite-type hypotheses}\label{sec:finite-hypotheses}

We now state the additional assumptions for a finite type set. Partition $\cI$ into classes
$\cI_0,\ldots,\cI_{g-1}$, where $g\geq1$. All types receiving random labels belong to
$\cI_0$, and a vertex with type in $\cI_j$ has both child vertex types in $\cI_{j+1}$, with indices read
modulo $g$. Note that our model allows random types for children with deterministic labels. Thus we require
\[
 \cI_\mu\subseteq\cI_0,\qquad
 \supp P_t\subseteq\cI_{j+1}\times\cI_{j+1}
 \quad(t\in\cI_j).
\]
We allow $g=1$, when there is a single class $\cI_0=\cI$.
\Cref{fig:markov-type-transitions} illustrates this case and a model with two classes.
A possible type path follows either child in transitions of positive probability.
We require the following two conditions:
\begin{enumerate}[label=(M\arabic*),ref=(M\arabic*),leftmargin=*]
\item\label{it:markov-positivity} \emph{Positive matching.} For every $h\geq0$ and
$s,t\in\cI_\mu$, each labelling that occurs with positive probability from $s$ through height $h$
has a positive probability of matching an independent labelling started at $t$.
\item\label{it:markov-returns} \emph{Bounded returns.} There is an integer $H\geq0$ such that,
whenever $s,t$ belong to the same class $\cI_i$, every possible type path from $s$ visits
$\cI_\mu$ at some depth $n\leq H$ for which a possible type path from $t$ also visits $\cI_\mu$.
\end{enumerate}
The path from $t$ in condition~\labelcref{it:markov-returns} may depend on the path from $s$.

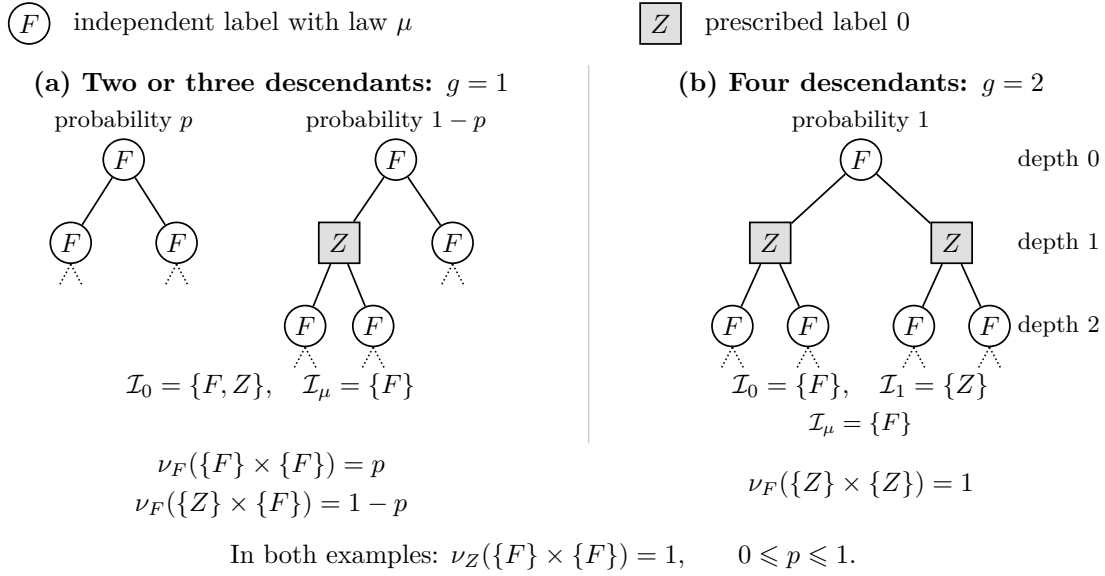
\begin{figure}[tbp]
  \centering
\begin{tikzpicture}[
  x=1cm,y=1cm,font=\small,
  random/.style={circle,draw=black,line width=.65pt,fill=white,
    minimum size=5.4mm,inner sep=0pt},
  prescribed/.style={rectangle,draw=black,line width=.65pt,fill=black!12,
    minimum size=5.4mm,inner sep=0pt},
  edge/.style={line width=.65pt},
  continuation/.style={densely dotted,line width=.6pt},
  annotation/.style={font=\footnotesize,inner sep=1pt}
]
  \node[random] at (.3,1.45) {$F$};
  \node[anchor=west] at (.75,1.45) {independent label with law $\mu$};
  \node[prescribed] at (8.65,1.45) {$Z$};
  \node[anchor=west] at (9.1,1.45) {prescribed label $0$};

  \node at (3.5,.65) {\textbf{(a) Two or three descendants: $g=1$}};
  \node at (11.3,.65) {\textbf{(b) Four descendants: $g=2$}};
  \draw[black!25] (7.7,.9) -- (7.7,-4.1);

  \node[annotation] at (1.55,.16) {probability $p$};
  \node[random] (a0) at (1.55,-.35) {$F$};
  \node[random] (a1) at (.85,-1.45) {$F$};
  \node[random] (a2) at (2.25,-1.45) {$F$};
  \draw[edge] (a0) -- (a1) (a0) -- (a2);

  \node[annotation] at (5.15,.16) {probability $1-p$};
  \node[random] (b0) at (5.15,-.35) {$F$};
  \node[prescribed] (b1) at (4.4,-1.45) {$Z$};
  \node[random] (b2) at (5.9,-1.45) {$F$};
  \node[random] (b11) at (3.95,-2.55) {$F$};
  \node[random] (b12) at (4.85,-2.55) {$F$};
  \draw[edge] (b0) -- (b1) (b0) -- (b2) (b1) -- (b11) (b1) -- (b12);

  \node[annotation] at (11.3,.16) {probability $1$};
  \node[random] (c0) at (11.3,-.35) {$F$};
  \node[prescribed] (c1) at (10.1,-1.45) {$Z$};
  \node[prescribed] (c2) at (12.5,-1.45) {$Z$};
  \node[random] (c11) at (9.6,-2.55) {$F$};
  \node[random] (c12) at (10.6,-2.55) {$F$};
  \node[random] (c21) at (12,-2.55) {$F$};
  \node[random] (c22) at (13,-2.55) {$F$};
  \draw[edge] (c0) -- (c1) (c0) -- (c2)
    (c1) -- (c11) (c1) -- (c12) (c2) -- (c21) (c2) -- (c22);

  \foreach \v in {a1,a2,b2,b11,b12,c11,c12,c21,c22}{
    \draw[continuation] (\v.south) -- ++(-.15,-.28);
    \draw[continuation] (\v.south) -- ++(.15,-.28);
  }

  \node[annotation,anchor=east] at (14.5,-.35) {depth $0$};
  \node[annotation,anchor=east] at (14.5,-1.45) {depth $1$};
  \node[annotation,anchor=east] at (14.5,-2.55) {depth $2$};

  \node at (3.5,-3.35) {$\cI_0=\{F,Z\},\quad\cI_\mu=\{F\}$};
  \node at (11.3,-3.35) {$\cI_0=\{F\},\quad\cI_1=\{Z\}$};
  \node[annotation] at (11.3,-3.85) {$\cI_\mu=\{F\}$};

  \node[align=center] at (3.5,-4.65) {
    $\nu_F(\{F\}\times\{F\})=p$\\[3pt]
    $\nu_F(\{Z\}\times\{F\})=1-p$};
  \node at (11.3,-4.65) {$\nu_F(\{Z\}\times\{Z\})=1$};

  \node at (7.1,-5.6) {
    In both examples: $\nu_Z(\{F\}\times\{F\})=1$,\qquad $0\leq p\leq1$.};
\end{tikzpicture}
  \caption{Two examples of type transitions. Circles receive independent labels with law
    $\mu$, and squares receive label $0$.
    In (a), the next $F$ vertices occur at depths $1,1$ or $1,2,2$, according to the chosen
    transition. Both types belong to the same class, and $p=0$ gives the encoding of the
    ternary tree.
    In (b), the next four $F$ vertices all have depth two and the classes
    alternate. The return condition \labelcref{it:markov-returns} holds with $H=2$ in (a)
    and $H=1$ in (b).}
  \label{fig:markov-type-transitions}
\end{figure}

\subsection{The admissible exponent}\label{sec:admissible-exponent}

For $\alpha\geq1$, we set
\begin{equation}\label{eq:markov-exponent}
 \lambda_\alpha\coloneqq2\min_{0\leq\beta\leq1}\left(
 \max_{0\leq q\leq1}
 \left\{\frac{q}{(1+q)^\alpha}+\beta(1-q)^\alpha\right\}
 +\frac{\alpha^\alpha(1-\beta)^{\alpha+1}}{(\alpha+1)^{\alpha+1}}
 \right).
\end{equation}
The condition $\lambda_\alpha<1$ holds, in particular, for $\alpha\geq4/3$.
We evaluate this coefficient in \cref{sec:exponent-values}.

\begin{theorem}[Markov matching]\label{thm:markov-matching}
Fix $\alpha\geq1$ with $\lambda_\alpha<1$. Consider the Markov models above, and
assume either that $b(0)=1$, or that the type set is finite and satisfies the assumptions
of \cref{sec:finite-hypotheses}. There are $K<\infty$ and $\eps>0$ such that, if $\zeta_\alpha\leq\eps$, then
\[
 \bbP\bigl(\cM_h(s,t)^c\bigr)\leq K\zeta_\alpha\quad(h\geq0),\qquad
 \bbP\bigl(\cM_\infty(s,t)^c\bigr)\leq K\zeta_\alpha.
\]
When $b(0)=1$, this holds for every pair of initial types, and the constants depend
only on $\alpha$. Otherwise it holds whenever $s,t$ belong to the same class $\cI_j$, and the
constants depend only on $\alpha,H$ and the family of transition laws $(P_t)_{t\in\cI}$.
\end{theorem}

\paragraph*{\bf Constants and type transitions} We give a more detailed restatement in \cref{sec:markov-proof}, specifying the uniform
dependence of the constants on the type transitions and the cost of using
$\eta_\alpha$. When every label in the support of $\mu$ is compatible with $0$, we have $b(0)=1$. The
theorem then imposes no restriction on the type transitions. It allows arbitrarily many
vertices labelled $0$ between successive draws from $\mu$, or branches with no further
draws. Otherwise, the finite-type assumptions ensure that comparisons between random and
prescribed labels can be controlled uniformly in the height. For two finitely supported arity laws with the same branching semigroup, we construct finite
binary replacements satisfying the theorem's hypotheses. This supplies the matching needed
for the general classification. We prove the theorem in \cref{sec:markov-proof} and construct the replacements
in \cref{sec:common-presentations}.

\section{Random branching times}\label{sec:branching-times}

\paragraph*{\bf Random edge lengths} The chain encoding separates the binary genealogy from the
distances between successive branch points. We now use this observation when the genealogy is the
deterministic binary tree and assign independent positive random lengths to its vertices. The two
edges leaving each vertex share its length. The metric geometry is governed by the same matching
theorem. Given a function $l\colon\bbB\to(0,\infty)$, let $X_l$ be the geometric realisation of
$\bbB$ in which both edges from $w$ to its children have length $l(w)$. Thus $X_l$ is a metric tree
whose branch vertices are indexed by $\bbB$.

\paragraph*{\bf Branching-time densities} We consider two families of branching-time densities. For $\lambda>0$, put
\[
  f^+_\lambda(t)=
  \begin{cases}
    \lambda e^{-\lambda(t-1)}&t\geq1,\\
    0&0<t<1,
  \end{cases}
  \qquad
  f^\pm_\lambda(t)=
  \begin{cases}
    (\lambda/2)e^\lambda t^{-2}e^{-\lambda/t}&0<t<1,\\
    (\lambda/2)e^\lambda e^{-\lambda t}&t\geq1.
  \end{cases}
\]
The first is a shifted exponential law. For the second, the restrictions to $[1,\infty)$ and
$(0,1)$ each have mass $1/2$, and conditional on either half, respectively $l$ and $l^{-1}$ have
the shifted exponential law. Let $X^+_\lambda$ and $X^\pm_\lambda$ denote the random spaces
$X_l$ when $(l(w))_{w\in\bbB}$ are independent with the corresponding density. The top and
bottom rows of \cref{fig:branching-times} illustrate the one-sided and two-sided families,
respectively.

\begin{figure}[t]
  \begin{center}
  \includegraphics[width=.4235\textwidth]{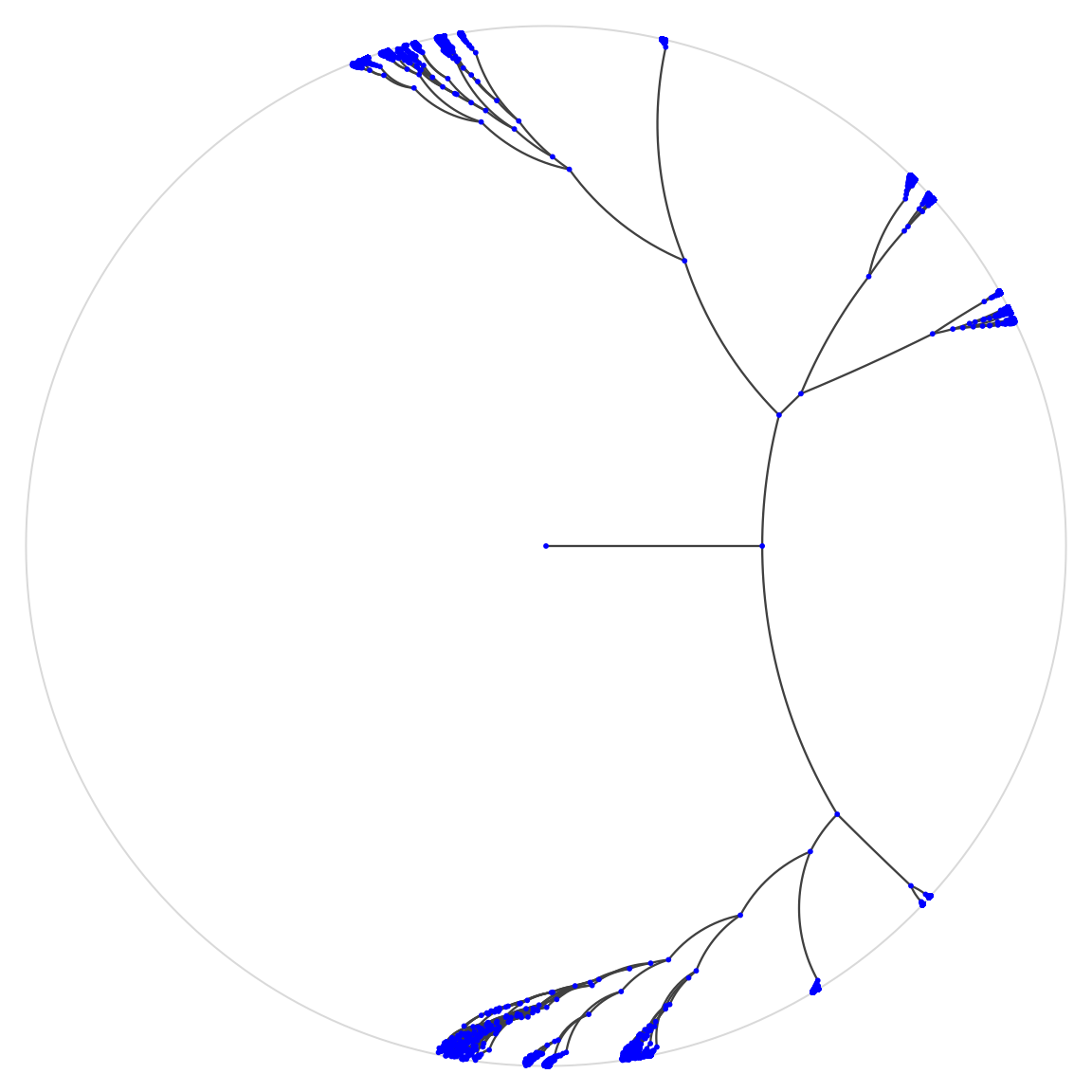}\hfill
  \includegraphics[width=.4235\textwidth]{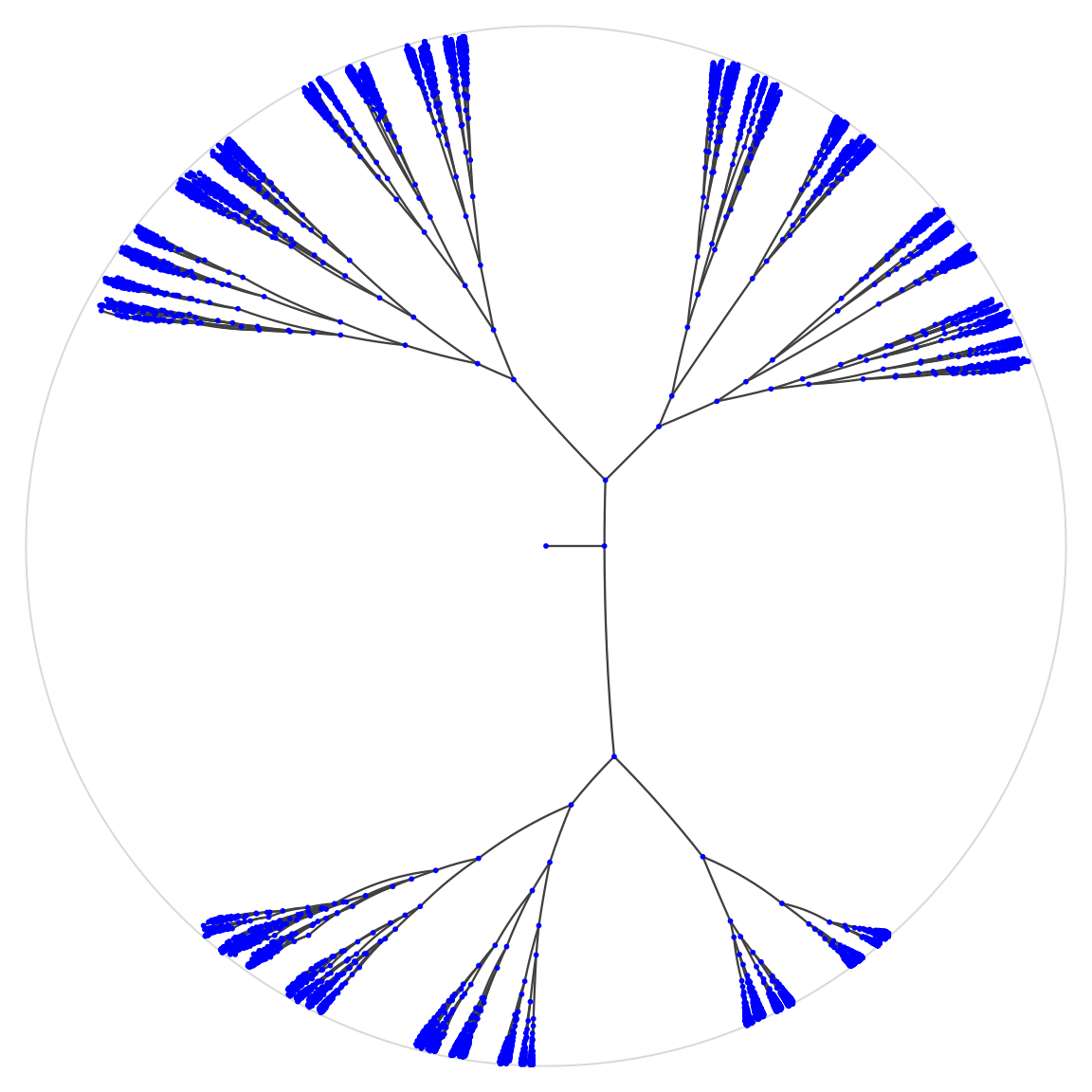}
\end{center}

\begin{center}
  \includegraphics[width=.4235\textwidth]{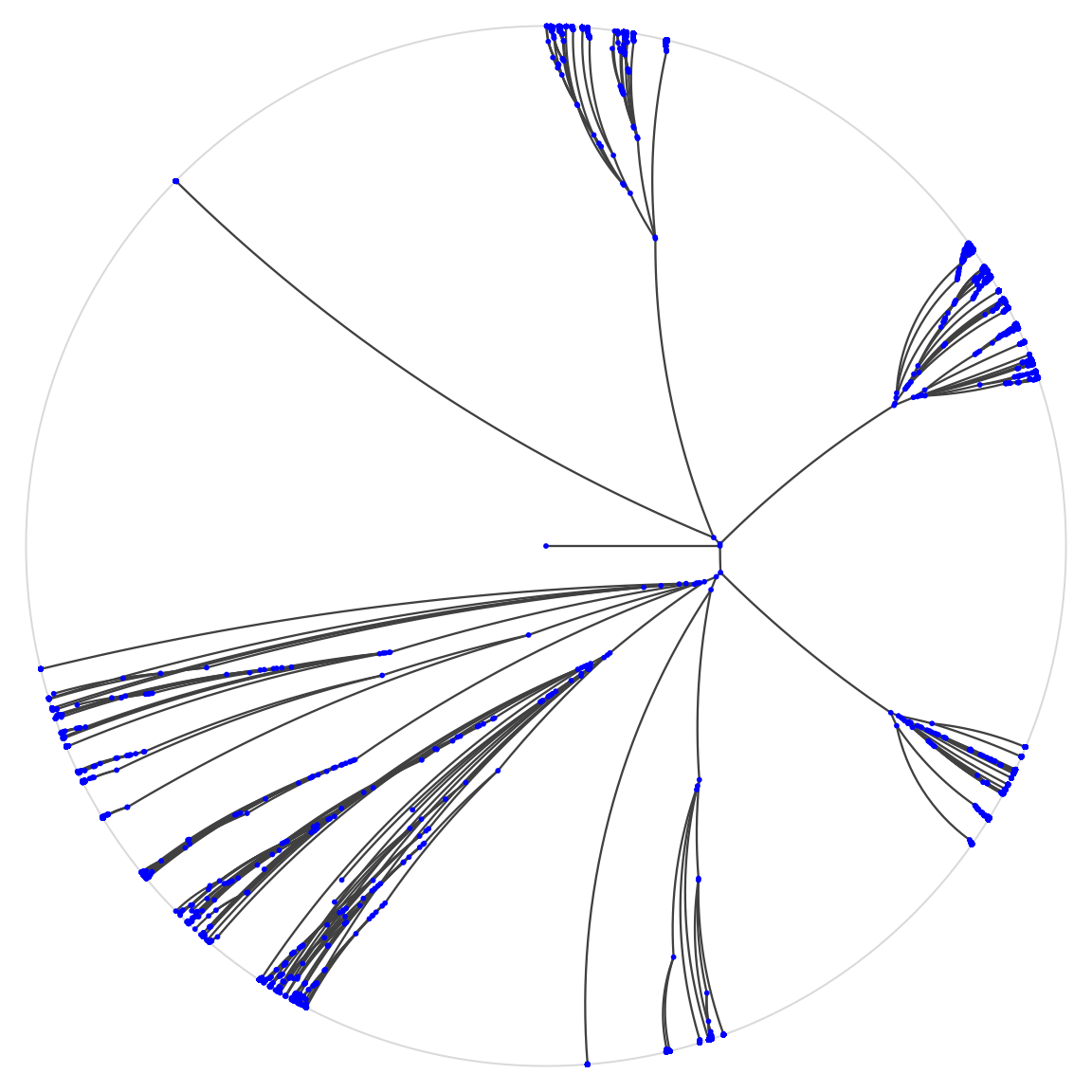}\hfill
  \includegraphics[width=.4235\textwidth]{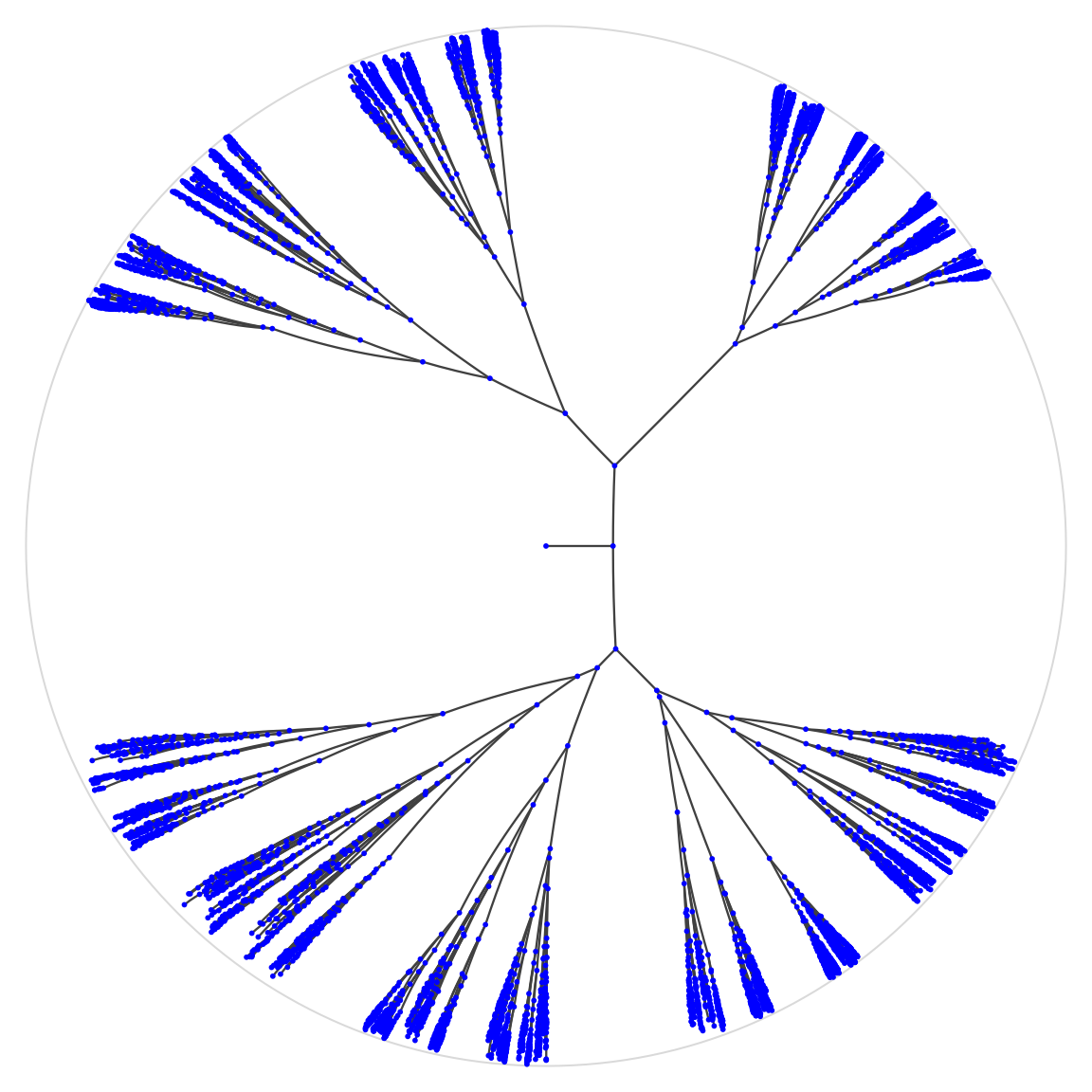}
\end{center}
  \caption{Finite truncations of the random branching-time models embedded in the Poincar\'e disc.
    The top row shows $X^+_{0.2}$ and $X^+_{0.5}$, which are almost surely quasi-isometric by
    \cref{thm:branching-times}\labelcref{it:branching-times-one-sided}. The bottom row shows
    $X^\pm_{0.1}$ and $X^\pm_1$, which are almost surely quasi-isometric by
    \cref{thm:branching-times}\labelcref{it:branching-times-two-sided}. Independent realisations
    represented in different rows are almost surely not quasi-isometric by
    \cref{thm:branching-times}\labelcref{it:branching-times-separation}. The arbitrarily short edges
    in the bottom row illustrate this separation.}
  \label{fig:branching-times}
\end{figure}

\paragraph*{\bf Collapsing} The branching-time interpretation also gives a discrete model with unbounded arity. Collapse each
connected component of
\[
  \set{x\in X_l:n\leq d_{X_l}(\troot,x)<n+1},
  \qquad n\in\bbN_0,
\]
to a vertex, and join two such vertices when the closures of the corresponding components meet.
Each collapsed component has diameter at most two, so the resulting rooted tree is quasi-isometric
to $X_l$. In the two-sided model, arbitrarily many binary splits can occur during one unit of time.
Indeed, for every $h$, there is positive probability that all edges in the first $h$ generations
have length less than $1/h$. The set of possible arities in the coarse tree is therefore unbounded.
Consequently, \cref{thm:branching-times} also compares a structured family of unbounded-arity
random trees but does not give a classification for general offspring laws with infinite support.

\paragraph*{\bf Quantising} We quantise the time until splits in the one-sided family by the
half-line $\bbN_0$, as in the chain model. The two-sided
family will be quantised by the path on $\bbZ$ instead. The following estimate gives the only additional
required probabilistic bounds.

\begin{lemma}[Two-sided branching-time potential]\label{thm:two-sided-potential}
  Let $a\in(0,1)$ and, for an integer $D\geq2$, define
  \[
    p_k=a^{D^k-1}-a^{D^{k+1}-1}\qquad(k\geq0).
  \]
  Define a probability measure $q^{(D)}$ on $\bbZ$ by
  \[
    q^{(D)}_k=q^{(D)}_{-k-1}=\frac{p_k}{2}\qquad(k\geq0).
  \]
  If $\mathsf P_{\bbZ}$ is the path on $\bbZ$, then
  \[
    \eta_{\mathsf P_{\bbZ},5/2}\bigl(q^{(D)}\bigr)\longrightarrow0
    \qquad\text{as }D\to\infty.
  \]
\end{lemma}

\begin{proof}
  The numbers $p_k$ are successive differences of the tail $a^{D^k-1}$, so they sum to one and
  the displayed formula defines a probability measure. Write $q_k=q^{(D)}_k$ and
  $b_k=q_{k-1}+q_k+q_{k+1}$. The symmetry $q_k=q_{-k-1}$ lets us estimate the two central terms
  and then twice the terms on the positive half-line.

  For all sufficiently large $D$, we have $p_0\geq1/2$. Since
  \[
    b_0=b_{-1}=p_0+\frac{p_1}{2}
    \quad\text{and}\quad
    1-b_0\leq1-p_0=a^{D-1},
  \]
  the two central terms in the potential tend to zero. The pair of terms at $1$ and $-2$ is at
  most a constant multiple of $p_1\leq a^{D-1}$.

  For $k\geq2$, we have $b_k\geq p_{k-1}/2$. Increasing the threshold for $D$ if necessary,
  \[
    p_{k-1}
    =a^{D^{k-1}-1}\bigl(1-a^{D^k-D^{k-1}}\bigr)
    \geq\frac12a^{D^{k-1}-1}.
  \]
  Dropping the factor $1-b_k\leq1$, the sum of the terms at $k$ and $-k-1$ is therefore at most
  \[
    2^{5/2}p_kp_{k-1}^{-5/2}
    \leq32a^{D^{k-1}(D-5/2)}.
  \]
  For large $D$, consecutive terms of the series on the right have ratio at most $1/2$. Their sum
  is at most $64a^{D(D-5/2)}$, which also tends to zero. Combining the three estimates proves the
  claim.
\end{proof}

\paragraph*{\bf Exponential branching times}
For an ordinary exponential branching time,
\[
  \bbP(l<t)=1-e^{-\lambda t}\sim\lambda t
\]
as $t\downarrow0$. This decay is too slow for our matching estimate: under the same geometric
quantisation, adjacent bins near zero have comparable masses, and the graph potential is infinite.
The two-sided density $f^\pm_\lambda$ instead gives
$\bbP(l<t)=\tfrac12e^{\lambda-\lambda/t}$ for $0<t<1$, whose faster decay makes the
preceding potential estimate possible. It would be very interesting to try and study the
quasi-isometric properties of the standard exponential branching time model.

\begin{theorem}[Universality and separation for random branching times]
  \label{thm:branching-times}
  Let $\lambda,\lambda'>0$, and take all realisations below independently.
  \begin{enumerate}[(1),ref=(\arabic*),itemsep=5pt]
    \item\label{it:branching-times-one-sided} Almost surely,
      $X^+_\lambda\simeq X^+_{\lambda'}$.
    \item\label{it:branching-times-two-sided} Almost surely,
      $X^\pm_\lambda\simeq X^\pm_{\lambda'}$.
    \item\label{it:branching-times-separation} Almost surely,
      $X^+_\lambda\not\simeq X^\pm_{\lambda'}$.
  \end{enumerate}
  The quasi-isometries in \labelcref{it:branching-times-one-sided,it:branching-times-two-sided}
  can be chosen to be root-preserving.
\end{theorem}

\begin{proof}[Proof of \cref{thm:branching-times}]
  We first prove the two positive statements. Fix an integer $D\geq2$, put $a=e^{-\lambda}$ and
  $c=\lambda/\lambda'$, and define
  \[
    h_k=1+c(D^k-1)\qquad(k\geq0).
  \]
  Quantise a length from the first realisation by the intervals $[D^k,D^{k+1})$, and a length
  from the second by $[h_k,h_{k+1})$. For the one-sided laws, both quantised labels have masses
  \[
    p_k=a^{D^k-1}-a^{D^{k+1}-1}\qquad(k\geq0).
  \]
  This is the law $p^{(D)}$ in \cref{thm:quantised-law}, with $\theta_1=a$. Hence
  \cref{thm:eta-bound} shows that its potential on the path $\bbN_0$ tends to zero.

  For the two-sided laws, use the same intervals above $1$, give their labels the indices
  $k\geq0$, and use the reciprocal intervals below $1$, with label $-k-1$ assigned to
  \[
    (D^{-k-1},D^{-k}]
    \quad\text{and}\quad
    (h_{k+1}^{-1},h_k^{-1}]
  \]
  in the first and second realisations, respectively. Both label fields then have the law
  $q^{(D)}$ of \cref{thm:two-sided-potential}. Its potential on the path $\bbZ$ tends to zero.

  In either case the two quantised fields are independent, and within each field the labels are
  independent and identically distributed. For all sufficiently large $D$, \cref{thm:matching}
  gives a matching automorphism with probability at least $1-16\eta_D$, where $\eta_D\to0$ is the
  relevant potential.

  Put $m=\min(1,c)$, $M=\max(1,c)$, and $\rho=M/m$. The cut points satisfy
  \[
    mD^k\leq h_k\leq MD^k\qquad(k\geq0).
  \]
  Thus labels at distance at most one correspond to lengths comparable within the factor
  $\rho D^2$, both for the direct intervals and their reciprocals. On the matching event,
  extending the matching automorphism linearly along edges gives a root-preserving
  $\rho D^2$-bi-Lipschitz homeomorphism, by summing the edgewise bounds along geodesics.
  The event that the two spaces are quasi-isometric consequently has probability
  at least $1-16\eta_D$ for every sufficiently large $D$. Letting $D\to\infty$ proves
  \cref{it:branching-times-one-sided,it:branching-times-two-sided}.

  It remains to separate the two families. For a proper weighted binary tree $X_l$, equip its end
  boundary with the visual metric
  \[
    d_l(\xi,\zeta)=\exp\bigl(-d_{X_l}(\troot,\xi\wedge\zeta)\bigr).
  \]
  Every edge of $X^+_\lambda$ has length at least one. Each boundary ball is therefore a cylinder,
  and its two child cylinders have diameter at most $e^{-1}$ times its diameter. They cover the
  original ball by balls of half its radius, so the boundary of every realisation of
  $X^+_\lambda$ is doubling.

  We next show that the boundary of $X^\pm_{\lambda'}$ is almost surely not doubling. For each
  $n\geq1$, put $\eps_n=(\log2)/(2n)$. At any fixed vertex, the event that all edge lengths in the
  complete binary subtree of height $n$ are below $\eps_n$ has positive probability. Taking these
  events over an infinite antichain gives independent trials, so such a subtree occurs almost
  surely for every $n$.

  Let $v$ be the root of one of these subtrees and put
  $\rho=\exp(-d_{X_l}(\troot,v))$. The cylinder at $v$ is a boundary ball of radius $\rho$.
  Choose one ray through each of the $2^n$ vertices at depth $n$ below $v$. Any two chosen rays
  separate at distance at most $n\eps_n=(\log2)/2$ from $v$, so their boundary distance is at least
  $\rho/\sqrt2>\rho/2$. Hence the cylinder cannot be covered by fewer than $2^n$ balls of radius
  $\rho/2$.
  Since this holds for every $n$, the boundary is not doubling.

  Finally, the two-sided weighted tree is almost surely proper. Indeed, choose $r>0$ so small that
  $2\bbP(l<r)<1$. An infinite ray of finite length would eventually use only edges of length less
  than $r$, whereas the clusters of such edges are dominated by a subcritical binary
  Galton--Watson process. Almost surely no such ray exists. K\H{o}nig's infinity lemma then shows
  that every bounded ball meets only finitely many branch vertices, and hence is compact.

  A quasi-isometry between the two proper metric trees would induce a quasisymmetry between their
  visual boundaries, as in the proof of \cref{thm:boundary}. Quasisymmetries preserve doubling,
  which contradicts the boundary properties above and proves \cref{it:branching-times-separation}.
\end{proof}

\section{Further directions}\label{sec:beyond}

\paragraph*{\bf Possible extensions} The classification settles the bounded-support problem: extinct realisations form the finite class,
while \cref{thm:trichotomy} classifies every infinite case. It leaves several natural extensions.

\begin{enumerate}
  \item \emph{Unbounded offspring support.}
    The unit-time coarse graining in \cref{sec:branching-times} shows that the matching theorem
    already applies to some structured random trees with unbounded arity. For arbitrary offspring
    laws, whether an exponential offspring tail suffices remains open. A complementary problem is
    to identify a heavy-tailed law whose realisations are not quasi-isometric to any bounded-support
    realisation.

  \item \emph{Critical and size-biased trees.}
    For a critical law, the local weak limit under conditioning on survival to generation $n$ is
    Kesten's size-biased tree. Its distinguished spine and critical bushes produce a geometry
    unlike the supercritical realisations studied here and suggest different coarse invariants.
    One may ask whether two independent realisations admit a natural equivalence after a
    scale-dependent rescaling.

  \item \emph{Fractal percolation without strong separation.}
    The proof of \cref{thm:fractal-percolation} uses strong separation to identify the random limit
    set bi-Lipschitzly with the visual boundary of its Galton--Watson coding tree. In Mandelbrot
    percolation, retained cubes may share boundary points, so the coding map need not be injective
    and Euclidean distances are not controlled by common-prefix length. The retained cubes still
    have a Galton--Watson genealogy, but this tree does not record their horizontal adjacencies. A
    ringed tree records these adjacencies by joining neighbouring retained cubes at each level. It
    is open whether two independent nonempty realisations of Mandelbrot percolation with the same
    parameters are almost surely quasisymmetrically equivalent. More generally, one may ask the
    same question for general fractal percolation without strong separation.

  \item \emph{Continuum random trees.}
    The Brownian-tree problem discussed above asks whether two independent Brownian continuum
    random trees are almost surely quasisymmetrically equivalent. Their non-doubling
    geometry~\cite{Troscheit2021} prevents the discrete bounded-degree reduction used here, but the
    comparison between independent realisations suggests looking for a continuum analogue of the
    matching principle.

  \item \emph{Random higher rank.}
    Li--Yu--Zheng~\cite{LiYuZheng2026} find rigidity for random subsets of products of regular
    trees. It remains to understand which products or random subspaces retain the rank-one
    flexibility of \cref{thm:trichotomy}, and which inherit higher-rank rigidity.
\end{enumerate}

\ifSubfilesClassLoaded{%

}{}

\end{document}

\clearpage
\part{Quasi-isometry classification for offspring distributions supported on \texorpdfstring{$\{0,1,2\}$}{\{0,1,2\}}}\label{part:binary-classification}

\section{Quasi-isometries through chains}
\paragraph*{\bf Binary branching}
We first consider only offspring distributions $\theta$ that are fully supported on $\set{1,2}$,
i.e.\ those with $\theta_1+\theta_2=1$ and $0<\theta_1<1$. The case $\theta_1=1$ gives the
deterministic ray, and the case $\theta_1 = 0$ gives the deterministic binary tree, making the
conclusions of \cref{thm:twovalue} trivial in those settings.
\subsection{The chain encoding}\label{sec:encoding}
Under the assumption $\theta_1+\theta_2=1$, $0<\theta_1<1$, the random tree
$\cT(\omega)\subseteq\cN$ consists of binary branch points connected by chains of single-descendant
vertices. Collapsing the chains encodes the random tree $\cT$ by a binary tree with a random
labelling, where the values at each node are independent, identically
distributed geometric $\theta_2$-random variables. Formally, given $v\in\cN$, define
$\kappa(v)=\kappa(v,\omega)$ to be the least $k\geq1$ such that
$\chi_{v1^{k-1}}=2$, where $1^k=11\cdots1$ ($k$ times). If no such $k$ exists set
$\kappa(v)=\infty$. Since $\theta_2>0$, the first-child ray $v,v1,v11,\dots$ from each
$v\in\cN$ almost surely contains a vertex with offspring count $2$, with
\begin{equation*}
  \bbP(\kappa(v) = k) = \theta_1^{k-1} \theta_2.
\end{equation*}
To make this more formal in a way
that will be useful for us when we generalise this construction, note that since $\cN$ is countable
and $\theta_1+\theta_2=1$, the event
\[
  \Omega_0=\set{\omega\in\Omega\colon
    \chi_v(\omega)\in\set{1,2}\text{ and }\kappa(v,\omega)<\infty\text{ for every }v\in\cN}
\]
has full probability. Conditioning on this full probability event $\Omega_0$, we define the
\emph{initial vertex map} $\iota\colon\bbB\times\Omega_0\to\cN$, which locates the initial vertex
in $\cT$ of the chain associated to the vertex $w \in \bbB$ and a realisation $\omega \in \Omega_0$.
Suppressing the argument $\omega$, $\iota(\troot)=\troot$ and, for $w\in\bbB$ and $j\in\set{1,2}$,
\begin{equation}\label{eq:phi-def}
  \iota(wj)=\iota(w)1^{\kappa(\iota(w))-1}j.
\end{equation}
The labelling $\lambda\colon\bbB\times\Omega_0\to\bbN$ is formally defined by
\begin{equation}\label{eq:lambda-def}
  \lambda(w)=\kappa(\iota(w)).
\end{equation}
To reach $\iota(wj)$ from $\iota(w)$, we follow the chain to its terminating branch point and
then take its $j$th child. Thus $d(\iota(w),\iota(wj))=\lambda(w)$: the number of vertices in
the chain is the distance to the start of the child chains. The labelled binary tree
$(\bbB,\lambda)$ determines $\cT(\omega)$, see \cref{fig:binary}. Unfolding \cref{eq:phi-def} gives
\begin{equation}\label{eq:normal-form}
  \iota(w_1\cdots w_n)=1^{k_0}w_1\,1^{k_1}w_2\cdots1^{k_{n-1}}w_n,
  \qquad
  k_q=\lambda(w_1\cdots w_q)-1,
\end{equation}
where the empty prefix $w_1\cdots w_0$ is read as the root $\troot$. In particular, $\iota(w)$ is a
function of the labels at the strict prefixes of $w$. The chain associated to $w \in \bbB$ is given by
\[
  \set{\iota(w)1^{l}\colon0\leq l\leq\lambda(w)-1}.
\]
It begins at $\iota(w)$, ends at the branch point
$\iota(w)1^{\lambda(w)-1}$, and carries $\lambda(w)$ vertices. By construction, it is clear that
these chains partition the tree and that their first-child rays are separated. We record these
observations about the decomposition of $\cT$ and the distribution of the labels as a pair of
lemmas for later reference.

\begin{lemma}\label{thm:chains}
  Let $\omega\in\Omega_0$.
  \begin{enumerate}[(i),ref=(\roman*)]
    \item Every $v\in\cT(\omega)$ has a unique representation
      $v=\iota(w)1^{l}$ with $w\in\bbB$ and $0\leq l\leq\lambda(w)-1$: the tree is the
      disjoint union of the chains.
    \item Let $w,w'\in\bbB$ and suppose that $w$ is not a prefix of
      $w'$. Then the ray $\set{\iota(w)1^{j}\colon j\geq0}$ is disjoint from the chain of
      $w'$.
  \end{enumerate}
\end{lemma}

\begin{figure}[tb]
  \begin{center}
\begin{tikzpicture}[]
  \node[circle,draw=black] (a0) at (0,0) {$2$};

  \node[circle,draw=black] (a00) at (-1,-1) {$2$};
  \node[circle,draw=black] (a01) at (1,-1) {$3$};
  \draw (a0)--(a00);
  \draw (a0)--(a01);

  \node[circle,draw=black] (a000) at (-1-.5,-2) {$1$};
  \node[circle,draw=black] (a001) at (-1+0.5,-2) {$5$};
  \node[circle,draw=black] (a010) at (1-0.5,-2) {$1$};
  \node[circle,draw=black] (a011) at (1+0.5,-2) {$2$};
  \draw (a00)--(a000);
  \draw (a00)--(a001);
  \draw (a01)--(a010);
  \draw (a01)--(a011);

  \node () at (-1-.5,-2.5) {$\vdots$};
  \node () at (-1+.5,-2.5) {$\vdots$};
  \node () at (1-.5,-2.5) {$\vdots$};
  \node () at (1+.5,-2.5) {$\vdots$};

  \node (embed) at (3.1,-1) {$\longmapsto$};
  \node (embed2) at (3.1,-1.5) {$\iota(\,\cdot\,,\omega)$};

  \newcommand{\xoff}{6}
  \newcommand{\yoff}{1.5}
  \newcommand{\hcord}{0+\yoff}
  \node[circle,fill=black,minimum size=3pt,inner sep=0pt, label=right:{\tiny$\troot$}] (b0) at (0+\xoff,\hcord) {};

  \renewcommand{\hcord}{-0.5+\yoff}
  \node[circle,fill=black,minimum size=3pt,inner sep=0pt,label=right:{\tiny$(1)$}] (b00) at (0+\xoff,\hcord) {};
  \draw (b0)--(b00);

  \renewcommand{\hcord}{-1.0+\yoff}
  \node[circle,fill=black,minimum size=3pt,inner sep=0pt, label=left:{\tiny$(11)$}] (b000) at (-.6+\xoff,\hcord) {};
  \node[circle,fill=black,minimum size=3pt,inner sep=0pt, label=right:{\tiny$(12)$}] (b001) at (.6+\xoff,\hcord) {};
  \draw (b00)--(b000);
  \draw (b00)--(b001);

  \renewcommand{\hcord}{-1.5+\yoff}
  \node[circle,fill=black,minimum size=3pt,inner sep=0pt, label=left:{\tiny$(111)$}] (b0000) at (-.6+\xoff,\hcord) {};
  \draw (b000)--(b0000);

  \renewcommand{\hcord}{-2.0+\yoff}
  \node[circle,fill=black,minimum size=3pt,inner sep=0pt, label=left:{\tiny$(1111)$}] (b00000) at (-.6-0.3+\xoff,\hcord) {};
  \draw (b0000)--(b00000);
  \node[circle,fill=black,minimum size=3pt,inner sep=0pt] (b00001) at (-.6+0.3+\xoff,\hcord) {};
  \draw (b0000)--(b00001);

  \renewcommand{\hcord}{-2.5+\yoff}
  \node[circle,fill=black,minimum size=3pt,inner sep=0pt, label=left:{\tiny$(11111)$}] (b000000) at (-.6-0.3-0.15+\xoff,\hcord) {};
  \draw (b00000)--(b000000);
  \node[circle,fill=black,minimum size=3pt,inner sep=0pt] (b000001) at (-.6-0.3+0.15+\xoff,\hcord) {};
  \draw (b00000)--(b000001);
  \node () at (-0.6-0.3+\xoff,\hcord-0.3) {$\vdots$};

  \renewcommand{\hcord}{-2.5+\yoff}
  \node[circle,fill=black,minimum size=3pt,inner sep=0pt] (b000010) at (-.6+0.3+\xoff,\hcord) {};
  \draw (b00001)--(b000010);

  \renewcommand{\hcord}{-3.0+\yoff}
  \node[circle,fill=black,minimum size=3pt,inner sep=0pt] (b0000100) at (-.6+0.3+\xoff,\hcord) {};
  \draw (b000010)--(b0000100);

  \renewcommand{\hcord}{-3.5+\yoff}
  \node[circle,fill=black,minimum size=3pt,inner sep=0pt, label=left:{\tiny$(1112111)$}] (b00001000) at (-.6+0.3+\xoff,\hcord) {};
  \draw (b0000100)--(b00001000);

  \renewcommand{\hcord}{-4.0+\yoff}
  \node[circle,fill=black,minimum size=3pt,inner sep=0pt] (b000010000) at (-.6+0.3+\xoff,\hcord) {};
  \draw (b00001000)--(b000010000);

  \renewcommand{\hcord}{-4.5+\yoff}
  \node[circle,fill=black,minimum size=3pt,inner sep=0pt, label=left:{\tiny$(111211111)$}] (b0000100000) at (-.6+0.3-0.15+\xoff,\hcord) {};
  \draw (b000010000)--(b0000100000);
  \node[circle,fill=black,minimum size=3pt,inner sep=0pt, label=right:{\tiny$(111211112)$}] (b0000100001) at (-.6+0.3+0.15+\xoff,\hcord) {};
  \draw (b000010000)--(b0000100001);
  \node () at (-0.6+0.3+\xoff,\hcord-0.3) {$\vdots$};

  \renewcommand{\hcord}{-1.5+\yoff}
  \node[circle,fill=black,minimum size=3pt,inner sep=0pt, label=right:{\tiny$(121)$}] (b0010) at (.6+\xoff,\hcord) {};
  \draw (b001)--(b0010);

  \renewcommand{\hcord}{-2.0+\yoff}
  \node[circle,fill=black,minimum size=3pt,inner sep=0pt] (b00100) at (.6+\xoff,\hcord) {};
  \draw (b0010)--(b00100);

  \renewcommand{\hcord}{-2.5+\yoff}
  \node[circle,fill=black,minimum size=3pt,inner sep=0pt] (b001000) at (.6-0.3+\xoff,\hcord) {};
  \draw (b00100)--(b001000);
  \node[circle,fill=black,minimum size=3pt,inner sep=0pt, label=right:{\tiny$(12112)$}] (b001001) at (.6+0.3+\xoff,\hcord) {};
  \draw (b00100)--(b001001);

  \renewcommand{\hcord}{-3.0+\yoff}
  \node[circle,fill=black,minimum size=3pt,inner sep=0pt] (b0010000) at (.6-0.3-.15+\xoff,\hcord) {};
  \draw (b001000)--(b0010000);
  \node[circle,fill=black,minimum size=3pt,inner sep=0pt] (b0010001) at (.6-0.3+.15+\xoff,\hcord) {};
  \draw (b001000)--(b0010001);
  \node () at (0.6-0.3+\xoff,\hcord-0.3) {$\vdots$};

  \renewcommand{\hcord}{-3.0+\yoff}
  \node[circle,fill=black,minimum size=3pt,inner sep=0pt] (b0010010) at (.6+0.3+\xoff,\hcord) {};
  \draw (b001001)--(b0010010);

  \renewcommand{\hcord}{-3.5+\yoff}
  \node[circle,fill=black,minimum size=3pt,inner sep=0pt] (b00100100) at (.6+0.3-.15+\xoff,\hcord) {};
  \draw (b0010010)--(b00100100);
  \node[circle,fill=black,minimum size=3pt,inner sep=0pt] (b00100101) at (.6+0.3+.15+\xoff,\hcord) {};
  \draw (b0010010)--(b00100101);
  \node () at (0.6+0.3+\xoff,\hcord-0.3) {$\vdots$};

\end{tikzpicture}
\end{center}
  \caption{The labelled binary tree and its corresponding Galton--Watson tree.}
  \label{fig:binary}
\end{figure}
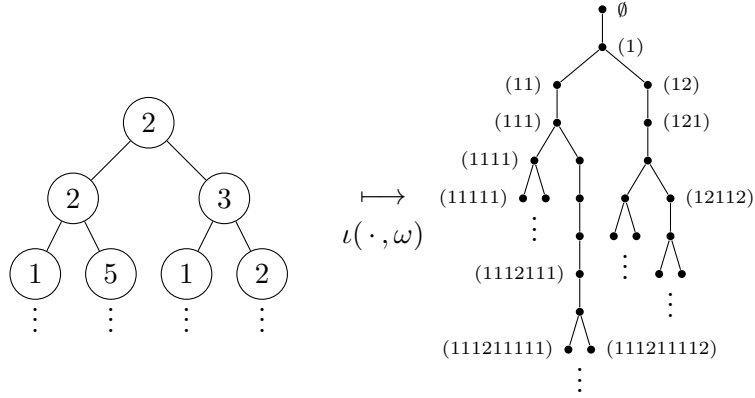

\begin{lemma}\label{thm:geometric}
  The random variables $\lambda(w)$, $w\in\bbB$, are independent and geometrically distributed:
  \[
    \bbP\bigl(\lambda(w)=n\bigr)=\theta_1^{n-1}\theta_2\qquad(n\geq1).
  \]
\end{lemma}

\subsection{From label comparability to quasi-isometry}\label{sec:transfer}

We now compare trees associated with arbitrary positive integer labellings. A \emph{labelling}
of $\bbB$ is a function $\lambda\colon\bbB\to\bbN$. Applying the recursion
\cref{eq:phi-def} to a labelling, $\iota(\troot)=\troot$ and
$\iota(wj)=\iota(w)1^{\lambda(w)-1}j$, defines the \emph{associated tree}
\[
  \cT=\bigcup_{w\in\bbB}\set{\iota(w)1^{l}\colon0\leq l\leq\lambda(w)-1}\subseteq\cN,
\]
with the metric inherited from $\cN$. Conditioned on $\Omega_0$, the Galton--Watson tree is
the tree associated with its label field by \cref{thm:chains}. The chain-separation argument
uses only the recursion and positivity of the labels, so every vertex of an associated tree
has unique coordinates $x=\iota(w)1^l$ with $0\leq l\leq\lambda(w)-1$. Expanding $\iota(w)$
gives the normal form of \cref{eq:normal-form}. An automorphism of $\bbB$ acts on a labelling
and preserves chain lengths and branching, inducing the following isometry.

\begin{lemma}\label{thm:isometry}
  Let $\pi\in\Aut(\bbB)$ and let $\lambda$ be a labelling with associated tree $\cT$. The tree
  $\cT'$ associated with the labelling $\lambda'(w)=\lambda(\pi^{-1}w)$ admits a
  root-preserving isometry from $\cT$.
\end{lemma}

\paragraph*{\bf Label fields and quasi-isometries}
The key deterministic argument converts multiplicative comparability of the two label fields into a
quasi-isometry of the encoded trees.

\begin{lemma}[Transfer lemma]\label{thm:transfer}
  Let $\lambda,\lambda'$ be labellings of $\bbB$ with values in $\bbN$ and let $D\geq1$ satisfy
  \[
    D^{-1}\lambda'(w)\leq\lambda(w)\leq D\,\lambda'(w)
    \qquad\text{for all }w\in\bbB.
  \]
  Then there exists a root-preserving $(D+3)$-quasi-isometry $\psi\colon\cT\to\cT'$
  between the associated trees.
\end{lemma}

\begin{proof}
  We write $\iota'$ for the chain-start map of $\lambda'$ and define
  \begin{align*}
    \phi_w(l)&=\floor{l\,\frac{\lambda'(w)-1}{\lambda(w)}}
      &&\bigl(0\leq l\leq\lambda(w)\bigr),\\
    \psi\bigl(\iota(w)1^l\bigr)&=\iota'(w)1^{\phi_w(l)}
      &&\bigl(0\leq l<\lambda(w)\bigr).
  \end{align*}
  Unique chain coordinates make $\psi$ well defined, and $\phi_\troot(0)=0$ gives root
  preservation. The floor inequalities give
  $l\lambda'(w)/\lambda(w)-2<\phi_w(l)\leq l\lambda'(w)/\lambda(w)$, so the label comparison
  yields
  \begin{equation}\label{eq:chain-ratio}
    \phi_w(l)\leq D\,l
    \qquad\text{and}\qquad
    l\leq D\,\phi_w(l)+2D.
  \end{equation}
  Subtracting from $\lambda'(w)$ gives the corresponding bounds for the distance to the
  start of either child chain:
  \begin{equation}\label{eq:chain-tail}
    \lambda'(w)-\phi_w(l)\leq D\bigl(\lambda(w)-l\bigr)+2
    \qquad\text{and}\qquad
    \lambda(w)-l\leq D\bigl(\lambda'(w)-\phi_w(l)\bigr).
  \end{equation}

  Segments between successive chain starts have lengths $\lambda(v)$, so their lengths
  compare term by term within factor $D$, without additive error. For two points in one chain,
  $\phi_w(l+a)-\phi_w(l)\in[\phi_w(a),\phi_w(a)+1]$, and \cref{eq:chain-ratio} applies to their
  separation $a$. For points in different chains, the corresponding geodesics consist of
  complete segments and either a tail and a head, or two heads and the two unit edges leaving
  their common branch point. Applying \cref{eq:chain-ratio,eq:chain-tail} to the endpoint
  pieces and summing gives
  \begin{equation}\label{eq:transfer-bounds}
    d(\psi(x),\psi(y))\leq D\,d(x,y)+2
    \qquad\text{and}\qquad
    d(x,y)\leq D\,d(\psi(x),\psi(y))+4D.
  \end{equation}
  The two unit edges contribute no additional error because $D\geq1$.

  Finally, $\phi_w(0)=0$, $\phi_w(\lambda(w))=\lambda'(w)-1$, and successive values of
  $\phi_w$ differ by at most $D+1$. This also bounds the gap from the last image vertex to
  the target chain endpoint, so $\psi(\cT)$ is $(D+1)$-dense in $\cT'$. Dividing the reverse
  inequality in \cref{eq:transfer-bounds} by $D$ gives an additive error bounded by $4$. Since
  $D$, $4$ and $D+1$ are all at most $D+3$, the map $\psi$ is a $(D+3)$-quasi-isometry.
\end{proof}

\subsection{Quantisation into level classes}\label{sec:quantisation}

We group chain lengths into classes so that lengths in neighbouring classes are multiplicatively
comparable. For an integer $D\geq2$, define the \emph{level map}
$\ell_D\colon\bbN\to\bbN_0$ and its classes by
\[
  \ell_D(n)=\floor{\log_D n},
  \qquad
  \cC_k=\set{n\in\bbN\colon D^k\leq n<D^{k+1}}\qquad(k\in\bbN_0).
\]
The classes are the fibres of $\ell_D$ and partition $\bbN$. We represent $\cC_k$ by the chain
of $D^k$ vertices. For $n\in\cC_k$ we have $1\leq n/D^k<D$, and two lengths $n,m\in\cC_k$
satisfy $D^{-1}<n/m<D$. If $\abs{\ell_D(m)-\ell_D(m')}\leq1$, then
\[
  m<D^{\ell_D(m)+1}\leq D^{\ell_D(m')+2}\leq D^2m'.
\]
Interchanging $m$ and $m'$ now gives the following lemma.

\begin{lemma}\label{thm:level}
  Let $D\geq2$ be an integer and $m,m'\in\bbN$. If $\abs{\ell_D(m)-\ell_D(m')}\leq1$, then
  $D^{-2}m'<m<D^2m'$.
\end{lemma}

\begin{lemma}[Quantised chain law]\label{thm:quantised-law}
  Let $\theta_1+\theta_2=1$ with $0<\theta_1<1$, and let $D\geq2$ be an integer. The labels
  $\ell_D(\lambda(w))$, $w\in\bbB$, are independent with common law $p^{(D)}$, the image under
  $\ell_D$ of the geometric law $\theta_1^{n-1}\theta_2$, $n\geq1$. For every $k\in\bbN_0$,
  \begin{equation}\label{eq:qk}
    p^{(D)}_k=\bbP\bigl(\lambda(w)\in\cC_k\bigr)
      =\theta_1^{D^k-1}-\theta_1^{D^{k+1}-1}.
  \end{equation}
  In particular,
  \[
    p^{(D)}_0=1-\theta_1^{D-1}
    \quad\text{and}\quad
    p^{(D)}_k\leq\theta_1^{D^k-1}\qquad(k\geq1).
  \]
\end{lemma}

\begin{proof}
  By \cref{thm:geometric}, the chain lengths are independent with the stated geometric law,
  so their images under $\ell_D$ are independent with law $p^{(D)}$. The tail identity
  $\bbP(\lambda(w)\geq n)=\theta_1^{n-1}$ gives \cref{eq:qk} by taking the difference at
  $D^k$ and $D^{k+1}$. The mass at zero and the tail bound follow from this formula.
\end{proof}

\subsection{Chain classes across two laws}

By \cref{thm:level}, lengths in adjacent classes are comparable up to the factor $D^2$.
This comparison depends only on the class endpoints, while the offspring distribution 
affects the class probabilities \cref{eq:qk}.
Applying the same classes to a second offspring distribution with support $\set{1,2}$ generally
gives different class probabilities. For sufficiently large $D$, we instead adjust the subdivision
of classes for the second distribution and randomise the assignment of lengths at their boundaries
so that the resulting indices have the same law $p^{(D)}$ as the first distribution. The new divisions
remain within multiplicative factors of the original ones, depending only on the two offspring laws.

\begin{lemma}[Chain coupling]\label{thm:chain-coupling}
  Let $\theta_1,\theta_1'\in(0,1)$, put
  \[
    \gamma=\frac{\log\theta_1}{\log\theta_1'}>0,
  \]
  and let $D\geq2$ be an integer with $\gamma(D-1)\geq1$. Let $p^{(D)}$ be the law of
  \cref{thm:quantised-law} for $\theta_1$ and $D$. Let $\lambda'$ be geometric with
  $\bbP(\lambda'=n)=(\theta_1')^{n-1}(1-\theta_1')$ for $n\geq1$, and let $U$ be uniform on
  $[0,1)$ and independent of $\lambda'$. Then there are integers $1=D_0'<D_1'<D_2'<\cdots$, the
  \emph{cut points} at which consecutive classes of the second law meet, and a map
  $\ell'\colon\bbN\times[0,1)\to\bbN_0$, nondecreasing in its first argument, such that the
  following hold for every $k\in\bbN_0$.
  \begin{enumerate}[(i),ref=(\roman*)]
    \item\label{it:chain-coupling-classes} For every $u\in[0,1)$ the class
      $\cC_k'(u)=\set{n\in\bbN\colon\ell'(n,u)=k}$ consists of consecutive integers and
      satisfies
      \[
        \set{n\in\bbN\colon D_k'<n<D_{k+1}'}\subseteq\cC_k'(u)
        \subseteq\set{n\in\bbN\colon D_k'\leq n\leq D_{k+1}'},
      \]
      so that only the two endpoints $D_k'$ and $D_{k+1}'$ have a class depending on $u$.
    \item\label{it:chain-coupling-law} $\bbP\bigl(\ell'(\lambda',U)=k\bigr)=p^{(D)}_k$.
    \item\label{it:chain-coupling-ratio} $\min(1,\gamma/2)\leq D_k'/D^k\leq\max(1,\gamma)$.
  \end{enumerate}
\end{lemma}

\begin{proof}
  We first equate the exponential formulas for the two tails. Since
  $\theta_1=(\theta_1')^{\gamma}$, the real thresholds $t_k=1+\gamma(D^k-1)$ satisfy
  \begin{equation}\label{eq:threshold}
    \theta_1^{D^k-1}=(\theta_1')^{t_k-1}.
  \end{equation}
  These thresholds need not be integers, so we put $D_k'=\floor{t_k}$. Then $D_0'=1$ because
  $t_0=1$, and
  $t_{k+1}-t_k=\gamma\,D^k(D-1)\geq\gamma(D-1)\geq1$, so the $D_k'$ are strictly increasing and
  tend to infinity with the $t_k$.

  From $D_k'\leq t_k<D_k'+1$, $0<\theta_1'<1$ and \cref{eq:threshold} we get
  \begin{equation}\label{eq:atom}
    (\theta_1')^{D_k'}<\theta_1^{D^k-1}\leq(\theta_1')^{D_k'-1}.
  \end{equation}
  The boundary length $D_k'$ has mass
  $\bbP(\lambda'=D_k')=(\theta_1')^{D_k'-1}-(\theta_1')^{D_k'}$. We assign it to the preceding
  class with probability
  \[
    \beta_k=\frac{(\theta_1')^{D_k'-1}-\theta_1^{D^k-1}}
      {(\theta_1')^{D_k'-1}-(\theta_1')^{D_k'}}\in[0,1).
  \]
  Define
  \[
    \ell'(n,u)=k\quad\text{for }D_k'<n<D_{k+1}',
    \qquad
    \ell'(D_k',u)=
    \begin{cases}
      k-1,&u<\beta_k,\\
      k,&u\geq\beta_k.
    \end{cases}
  \]
  The two clauses cover $\bbN$ because $D_0'=1$ and the $D_k'$ increase to infinity, and
  $\beta_0=0$ since $t_0=D_0'=1$, so $\ell'$ indeed takes values in $\bbN_0$ and is
  nondecreasing in its first argument. \Cref{it:chain-coupling-classes} holds by construction.

  For \cref{it:chain-coupling-law}, we distinguish lengths strictly between the two
  cut points from the two boundary lengths. Independence of $\lambda'$ and $U$ gives
  \begin{align*}
    \bbP\bigl(\ell'(\lambda',U)=k\bigr)
    &=\bbP\bigl(D_k'<\lambda'<D_{k+1}'\bigr)
      +(1-\beta_k)\,\bbP(\lambda'=D_k')
      +\beta_{k+1}\,\bbP(\lambda'=D_{k+1}')\\
    &=\Bigl((\theta_1')^{D_k'}-(\theta_1')^{D_{k+1}'-1}\Bigr)
      +\Bigl(\theta_1^{D^k-1}-(\theta_1')^{D_k'}\Bigr)
      +\Bigl((\theta_1')^{D_{k+1}'-1}-\theta_1^{D^{k+1}-1}\Bigr)\\
    &=\theta_1^{D^k-1}-\theta_1^{D^{k+1}-1}=p^{(D)}_k,
  \end{align*}
  where the second equality uses the definitions of $\beta_k$ and $\beta_{k+1}$,
  and the final equality is \cref{eq:qk}.

  For \cref{it:chain-coupling-ratio}, $t_k/D^k=\gamma\bigl(1-D^{-k}\bigr)+D^{-k}$ is a convex
  combination of $\gamma$ and $1$, so it lies between $\min(1,\gamma)$ and $\max(1,\gamma)$. In
  particular $D_k'\leq t_k\leq\max(1,\gamma)D^k$. For the lower bound, $D_0'=1=D^0$, while for
  $k\geq1$ we have $D^k\geq2$ and therefore
  \[
    D_k'>t_k-1=\gamma\bigl(D^k-1\bigr)\geq\tfrac{\gamma}{2}D^k,
  \]
  as required.
\end{proof}

The chain lengths of the second law are independent by \cref{thm:geometric}. Applying this
construction with independent uniform variables, also independent of those lengths, gives
independent class labels with the same law $p^{(D)}$. 

\section{Allowing extinction: matching shapes}\label{sec:shapes}

The chain encoding rested on two features of the two-value family: every realisation is infinite,
and every vertex has an infinite line of descent. Both fail for $\theta_0>0$. A realisation is now
finite with positive probability, and any two finite trees are quasi-isometric, being bounded
metric spaces. The classification therefore conditions on the family tree having infinite
diameter. The vertices with infinite progeny then form the \emph{surviving skeleton}, and the
remaining vertices form finite subtrees hanging off it, the bushes. A branch point of the
realisation need not be a branch point of the skeleton, so collapsing the chains of the
realisation encodes the wrong tree.
Instead, we collapse the chains of the skeleton and replace their integer lengths by
\emph{shapes}. Each shape records one skeleton chain, including its terminating branch point,
together with the bushes attached along it, see \cref{fig:shapes}.

We first describe the surviving skeleton, its bushes and the resulting shape law. We then
quantise the shapes in \cref{sec:shape-net} and show how comparisons of corresponding shapes
give quasi-isometries of their assemblies in \cref{sec:shape-transfer}.

Throughout this section $\theta$ is supported on $\set{0,1,2}$ and supercritical with
$\theta_0>0$, so that $\theta_2>\theta_0>0$. We write
$\bbP^*=\bbP(\,\cdot\mid\diam\cT(\omega)=\infty)$ for the law conditioned on having infinite
diameter.

\subsection{Shapes and the Harris decomposition}\label{sec:shape-harris}

A vertex of $\cT(\omega)$ is a \emph{skeleton} vertex if its subtree is infinite, a
\emph{split} if it has two children and both are skeleton vertices, and a \emph{neck} vertex
if it is a skeleton vertex with exactly one skeleton child. The finite subtrees rooted at the
remaining children are the \emph{bushes}. Conditioned on the tree having infinite diameter, the
skeleton is again Galton--Watson. Elementary calculations give its offspring law and, conditioned
on the skeleton, the law of its bushes, which are independent subcritical trees. For offspring
supported on $\set{0,1,2}$ both laws are explicit, and the skeleton falls in the chain regime, so
the machinery of \cref{sec:encoding} applies to it verbatim.

The decomposition is classical, due to Harris~\cite{Harris1948} and
Sevastyanov~\cite{Sevastyanov1951}, and we refer the reader to Athreya--Ney~\cite{AthreyaNey1972}
and Lyons--Peres~\cite{LyonsPeres2016} for standard accounts.
It holds for any supercritical offspring law. We state and prove the case of offspring supported
on $\set{0,1,2}$ because the later sections use its constants explicitly.
A tilde marks a quantity of the skeleton, the skeleton law $\widetilde\theta$ and its chain labels
$\widetilde\lambda$, and a dagger marks the conjugate law of the bushes.

\begin{proposition}[Harris decomposition]\label{thm:harris}
  Let $\theta$ be supported on $\set{0,1,2}$ and supercritical with $\theta_0>0$, and let
  $\cT(\omega)$ be a Galton--Watson tree with offspring distribution $\theta$. Then
  $\cT(\omega)$ is finite with probability $\theta_0/\theta_2$. Conditioned on
  $\cT(\omega)$ having infinite diameter, the following hold:
  \begin{enumerate}[(i),ref=(\roman*)]
    \item\label{it:harris-skeleton} the skeleton is a Galton--Watson tree with offspring
      distribution
      \[
        \widetilde\theta_1=\theta_1+2\theta_0,
        \qquad
        \widetilde\theta_2=\theta_2-\theta_0,
      \]
      which is a two-value law, $\widetilde\theta_1+\widetilde\theta_2=1$, lying in the chain
      regime: supercriticality gives $\theta_2>\theta_0$, and hence
      $0<\widetilde\theta_1<1$;
    \item\label{it:harris-bushes} conditionally on the skeleton, the skeleton vertices are
      decorated independently. A split has no dying child, and so carries no bush. A neck
      vertex has a second, dying child with probability $2\theta_0/(\theta_1+2\theta_0)$ and
      no second child with the complementary probability $\theta_1/(\theta_1+2\theta_0)$; a
      dying child is the root of a bush, which is an independent Galton--Watson tree with the
      conjugate law
      \[
        \theta^\dagger=(\theta^\dagger_0,\theta^\dagger_1,\theta^\dagger_2)
        =(\theta_2,\theta_1,\theta_0),
      \]
      whose mean is $\widetilde\theta_1$ and which is therefore subcritical.
  \end{enumerate}
\end{proposition}

\begin{proof}
  Write $q$ for the extinction probability, which is the root in $[0,1)$ of
  $s=\theta_0+\theta_1s+\theta_2s^2$. Rearranging and factoring gives
  $(s-1)(\theta_2s-\theta_0)=0$, so $q=\theta_0/\theta_2$. A fixed child founds a surviving
  subtree with probability $1-q$, independently of its siblings. A vertex whose subtree is
  infinite has either one child which survives, two children of which exactly one survives,
  or two children which both survive. Their respective weights are $\theta_1(1-q)$,
  $2\theta_2q(1-q)$ and $\theta_2(1-q)^2$. These sum to $1-q$, the probability of survival.
  Hence,
  \[
    \bbP(\text{two survive}\mid\text{some survive})
    =\frac{\theta_2(1-q)^2}{1-q}
    =\theta_2(1-q)=\theta_2-\theta_0.
  \]
  The complementary probability is $\theta_1+2\theta_2q=\theta_1+2\theta_0$. These give the
  offspring probabilities in \cref{it:harris-skeleton}, and they sum to one. Since
  $\theta_0>0$, we have $\widetilde\theta_1>0$, and supercriticality gives
  $\widetilde\theta_1<1$.

  A child conditioned to die founds a Galton--Watson tree with the law tilted by the
  extinction probability, $\theta^\dagger_j=\theta_jq^{j-1}$, which is
  $(\theta_2,\theta_1,\theta_0)$. Its mean is $\theta_1+2\theta_0=\widetilde\theta_1$, so
  subcriticality of the bushes is the same inequality. For the two neck cases, cancelling
  the common factor $1-q$ leaves weights $\theta_1$ and $2\theta_2q=2\theta_0$. Normalising
  by their sum $\theta_1+2\theta_0$ gives the two decoration probabilities of
  \cref{it:harris-bushes}. A split has both children surviving, so it carries no bush.

  Conditioned on the realisation having infinite diameter, the skeleton is a
  Galton--Watson tree with the offspring law computed
  above~\cite[Proposition~5.28]{LyonsPeres2016}. The conditional
  independence of the decorations follows from the embedding into a two-type branching
  process, one type the vertices with infinite line of descent and the other the vertices
  with finitely many descendants~\cite[Chapter~I.12]{AthreyaNey1972}.
\end{proof}

We encode the skeleton as in \cref{sec:encoding} but with $\widetilde\theta$ in place of
$\theta$. Its chains are indexed by $\bbB$, and under $\bbP^*$ their lengths
$\widetilde\lambda(w)$, $w\in\bbB$, are independent geometric variables with
$\bbP^*(\widetilde\lambda(w)=m)=\widetilde\theta_1^{\,m-1}\widetilde\theta_2$, by
\cref{thm:geometric} applied to the skeleton law.

\begin{definition}[Shapes]\label{def:shape}
  A \emph{shape} is a pair $\sigma=(m,(b_1,\dots,b_{m-1}))$, where $m\geq1$ and each $b_i$ is
  either empty or a finite tree with offspring numbers in $\set{0,1,2}$. Its
  \emph{realisation} is the finite rooted tree obtained from a path $v_1\cdots v_m$ by
  attaching, for each nonempty $b_i$, the root of $b_i$ by an edge at $v_i$. The \emph{entry}
  is $v_1$ and the \emph{exit} is $v_m$. The \emph{size} $\abs\sigma$ is the number of
  vertices of the realisation, and we write $\cS$ to denote the (countable) set of all shapes.
\end{definition}

\paragraph*{\bf Assembling shapes} We write $S(w)$ for the shape formed by the skeleton chain indexed by $w$, including its
terminating split and the bushes at its preceding neck vertices. A family of shapes indexed by $\bbB$ is \emph{assembled} into a tree by joining, for every
$w\in\bbB$, the exit of the shape at $w$ to the entries of the shapes at $w1$ and $w2$ by an
edge, and rooting the result at the entry of the shape at $\troot$.

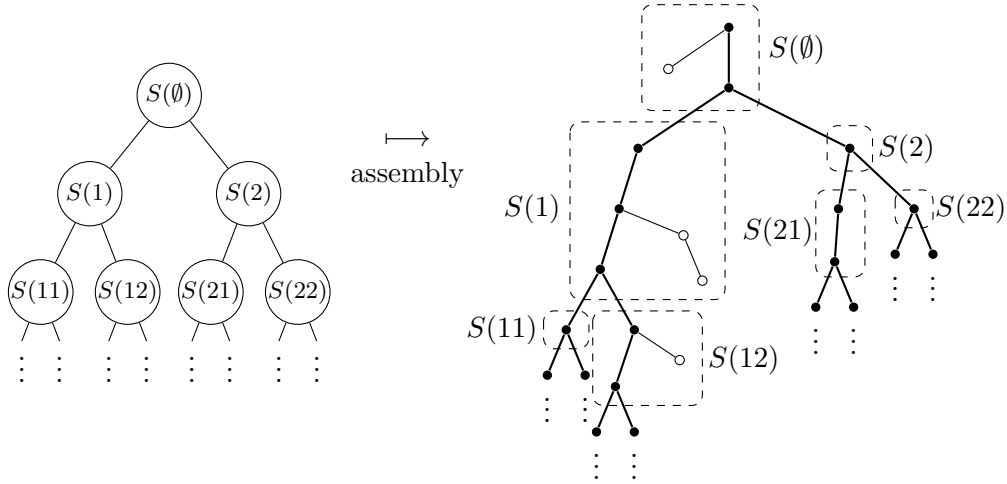
\begin{figure}[tb]
  \begin{center}
\begin{tikzpicture}[shapenode/.style={circle,draw=black,minimum size=0.85cm,inner sep=0pt,font=\footnotesize}]
  \node[shapenode] (s0) at (-4.9,-0.9) {$S(\troot)$};
  \node[shapenode] (s1) at (-5.95,-2.2) {$S(1)$};
  \node[shapenode] (s2) at (-3.85,-2.2) {$S(2)$};
  \draw (s0)--(s1);
  \draw (s0)--(s2);
  \node[shapenode] (s11) at (-6.6,-3.5) {$S(11)$};
  \node[shapenode] (s12) at (-5.45,-3.5) {$S(12)$};
  \node[shapenode] (s21) at (-4.35,-3.5) {$S(21)$};
  \node[shapenode] (s22) at (-3.2,-3.5) {$S(22)$};
  \draw (s1)--(s11);
  \draw (s1)--(s12);
  \draw (s2)--(s21);
  \draw (s2)--(s22);
  \draw (s11)--(-6.85,-4.15);
  \draw (s11)--(-6.35,-4.15);
  \draw (s12)--(-5.7,-4.15);
  \draw (s12)--(-5.2,-4.15);
  \draw (s21)--(-4.6,-4.15);
  \draw (s21)--(-4.1,-4.15);
  \draw (s22)--(-3.45,-4.15);
  \draw (s22)--(-2.95,-4.15);
  \node () at (-6.85,-4.42) {$\vdots$};
  \node () at (-6.35,-4.42) {$\vdots$};
  \node () at (-5.7,-4.42) {$\vdots$};
  \node () at (-5.2,-4.42) {$\vdots$};
  \node () at (-4.6,-4.42) {$\vdots$};
  \node () at (-4.1,-4.42) {$\vdots$};
  \node () at (-3.45,-4.42) {$\vdots$};
  \node () at (-2.95,-4.42) {$\vdots$};

  \node () at (-1.75,-1.5) {$\longmapsto$};
  \node () at (-1.75,-1.95) {assembly};

  \begin{scope}[xshift=1.5cm]
  \node[circle,fill=black,minimum size=3.5pt,inner sep=0pt] (r) at (1,0) {};
  \node[circle,draw=black,minimum size=3.5pt,inner sep=0pt] (x) at (0.2,-0.55) {};
  \node[circle,fill=black,minimum size=3.5pt,inner sep=0pt] (a) at (1,-0.8) {};
  \draw (r)--(x);
  \draw[thick] (r)--(a);

  \node[circle,fill=black,minimum size=3.5pt,inner sep=0pt] (b) at (-0.2,-1.6) {};
  \node[circle,fill=black,minimum size=3.5pt,inner sep=0pt] (c) at (2.6,-1.6) {};
  \draw[thick] (a)--(b);
  \draw[thick] (a)--(c);

  \node[circle,fill=black,minimum size=3.5pt,inner sep=0pt] (b1) at (-0.45,-2.4) {};
  \draw[thick] (b)--(b1);
  \node[circle,draw=black,minimum size=3.5pt,inner sep=0pt] (y1) at (0.4,-2.75) {};
  \node[circle,draw=black,minimum size=3.5pt,inner sep=0pt] (y2) at (0.65,-3.35) {};
  \draw (b1)--(y1);
  \draw (y1)--(y2);
  \node[circle,fill=black,minimum size=3.5pt,inner sep=0pt] (b2) at (-0.7,-3.2) {};
  \draw[thick] (b1)--(b2);

  \node[circle,fill=black,minimum size=3.5pt,inner sep=0pt] (d1) at (-1.15,-4.0) {};
  \node[circle,fill=black,minimum size=3.5pt,inner sep=0pt] (d2) at (-0.25,-4.0) {};
  \draw[thick] (b2)--(d1);
  \draw[thick] (b2)--(d2);
  \node[circle,fill=black,minimum size=3.5pt,inner sep=0pt] (d11) at (-1.4,-4.6) {};
  \node[circle,fill=black,minimum size=3.5pt,inner sep=0pt] (d12) at (-0.9,-4.6) {};
  \draw[thick] (d1)--(d11);
  \draw[thick] (d1)--(d12);
  \node () at (-1.4,-4.95) {$\vdots$};
  \node () at (-0.9,-4.95) {$\vdots$};
  \node[circle,draw=black,minimum size=3.5pt,inner sep=0pt] (z) at (0.35,-4.4) {};
  \node[circle,fill=black,minimum size=3.5pt,inner sep=0pt] (e) at (-0.5,-4.75) {};
  \draw (d2)--(z);
  \draw[thick] (d2)--(e);
  \node[circle,fill=black,minimum size=3.5pt,inner sep=0pt] (e1) at (-0.75,-5.35) {};
  \node[circle,fill=black,minimum size=3.5pt,inner sep=0pt] (e2) at (-0.25,-5.35) {};
  \draw[thick] (e)--(e1);
  \draw[thick] (e)--(e2);
  \node () at (-0.75,-5.7) {$\vdots$};
  \node () at (-0.25,-5.7) {$\vdots$};

  \node[circle,fill=black,minimum size=3.5pt,inner sep=0pt] (c1) at (2.45,-2.4) {};
  \node[circle,fill=black,minimum size=3.5pt,inner sep=0pt] (c2) at (3.45,-2.4) {};
  \draw[thick] (c)--(c1);
  \draw[thick] (c)--(c2);
  \node[circle,fill=black,minimum size=3.5pt,inner sep=0pt] (f1) at (2.4,-3.1) {};
  \draw[thick] (c1)--(f1);
  \node[circle,fill=black,minimum size=3.5pt,inner sep=0pt] (f11) at (2.15,-3.7) {};
  \node[circle,fill=black,minimum size=3.5pt,inner sep=0pt] (f12) at (2.65,-3.7) {};
  \draw[thick] (f1)--(f11);
  \draw[thick] (f1)--(f12);
  \node () at (2.15,-4.05) {$\vdots$};
  \node () at (2.65,-4.05) {$\vdots$};
  \node[circle,fill=black,minimum size=3.5pt,inner sep=0pt] (g1) at (3.2,-3.0) {};
  \node[circle,fill=black,minimum size=3.5pt,inner sep=0pt] (g2) at (3.7,-3.0) {};
  \draw[thick] (c2)--(g1);
  \draw[thick] (c2)--(g2);
  \node () at (3.2,-3.35) {$\vdots$};
  \node () at (3.7,-3.35) {$\vdots$};

  \draw[dashed,rounded corners=4pt] (-0.15,0.3) rectangle (1.4,-1.1);
  \node () at (1.9,-0.3) {$S(\troot)$};
  \draw[dashed,rounded corners=4pt] (-1.1,-1.25) rectangle (0.95,-3.6);
  \node () at (-1.6,-2.4) {$S(1)$};
  \draw[dashed,rounded corners=4pt] (2.3,-1.3) rectangle (2.9,-1.9);
  \node () at (3.35,-1.6) {$S(2)$};
  \draw[dashed,rounded corners=4pt] (-1.45,-3.75) rectangle (-0.85,-4.25);
  \node () at (-2.0,-4.0) {$S(11)$};
  \draw[dashed,rounded corners=4pt] (-0.8,-3.75) rectangle (0.65,-5.0);
  \node () at (1.2,-4.4) {$S(12)$};
  \draw[dashed,rounded corners=4pt] (2.15,-2.15) rectangle (2.8,-3.3);
  \node () at (1.65,-2.7) {$S(21)$};
  \draw[dashed,rounded corners=4pt] (3.2,-2.15) rectangle (3.7,-2.65);
  \node () at (4.20,-2.4) {$S(22)$};
  \end{scope}

\end{tikzpicture}
\end{center}
  \caption{A bushy realisation conditioned on having infinite diameter, and its shape labelling.
  Filled vertices denote the surviving skeleton and thick edges denote its chains. Hollow vertices form the
  bushes. The dashed regions are the shapes: each runs from a child of a split, or from the root,
  to the next split, and collects the bushes of its neck vertices. Assembling the shape-labelled
  binary tree recovers the realisation up to isometry.}
  \label{fig:shapes}
\end{figure}

\begin{proposition}[Shape decomposition]\label{thm:shape-iid}
  Conditioned on the tree having infinite diameter, the shapes $S(w)$, $w\in\bbB$, are i.i.d.\
  with the law $\mu$ of $(m,(B_1,\dots,B_{m-1}))$, where
  $\bbP(m=k)=\widetilde\theta_1^{\,k-1}\widetilde\theta_2$ and, given $m$, the decorations
  $B_1,\dots,B_{m-1}$ are independent, each empty with probability
  $\theta_1/\widetilde\theta_1$ and otherwise an independent Galton--Watson tree with law
  $\theta^\dagger$. Moreover $\cT(\omega)$ is isometric to the assembly of
  $(S(w))_{w\in\bbB}$.
\end{proposition}

\begin{proof}
  We use the independent geometric chain labels described above, obtained from
  \cref{thm:harris}\labelcref{it:harris-skeleton} and \cref{thm:geometric}. The chains
  partition the skeleton by \cref{thm:chains}. Conditioned on the skeleton, the decorations
  at the neck vertices are independent, each nonempty with probability
  $2\theta_0/\widetilde\theta_1$ by \cref{thm:harris}\labelcref{it:harris-bushes}.
  The shape $S(w)$ is the pair of
  $\widetilde\lambda(w)$ and the decorations at the $\widetilde\lambda(w)-1$ neck vertices of
  the chain of $w$. Distinct chains are disjoint, so, conditioned on the skeleton, the
  decoration families of distinct shapes are independent, and the conditional law of the
  decorations of $S(w)$ depends on the skeleton only through $\widetilde\lambda(w)$.
  Together with the independence of the labels, this shows that the shapes are i.i.d.\ with
  the stated law.
  Finally, the skeleton chains together with their bushes partition the vertices of
  $\cT(\omega)$, and the chain of $w$ with its bushes realises $S(w)$. The edges between
  consecutive shapes are exactly the skeleton edges from a split to its skeleton children,
  so identifying each shape with its realisation gives an isometry from the assembly to
  $\cT(\omega)$.
\end{proof}

\paragraph*{\bf Shape mass bounds} Let $S$ be a shape with law $\mu$. We give an exponential upper bound for the probability that
$S$ is large and an exponential lower bound for the mass of each shape with positive probability.

\begin{lemma}[Shape mass bounds]\label{thm:shape-mass}
  There are $c,C,n_0>0$, depending on $\theta$, such that
  \begin{enumerate}[(i),ref=(\roman*)]
    \item\label{it:shape-mass-tail} $\mu(\abs S\geq n)\leq e^{-cn}$ for all $n\geq n_0$, and
    \item\label{it:shape-mass-point} $\mu(\sigma)\geq e^{-C\abs\sigma}$ for every
      $\sigma\in\cS$ with $\mu(\sigma)>0$.
  \end{enumerate}
\end{lemma}

\begin{proof}
  For \cref{it:shape-mass-point}, fix $\sigma=(m,(b_1,\dots,b_{m-1}))$ with
  $\mu(\sigma)>0$. The geometric chain probability contributes $m$ factors, the decoration
  choices contribute $m-1$, and each bush vertex contributes one offspring factor. Thus
  $\mu(\sigma)$ is a product of at most $2\abs\sigma$ factors, each of them one of
  $\widetilde\theta_1$, $\widetilde\theta_2$,
  $\theta_1/\widetilde\theta_1$, $2\theta_0/\widetilde\theta_1$, or an offspring weight
  $\theta^\dagger_j$ of a bush vertex. Let $p$ be the least positive number among these
  seven constants. Since $\mu(\sigma)>0$, every factor is at least $p\in(0,1]$, so
  $\mu(\sigma)\geq p^{2\abs\sigma}$. With $C=2\log p^{-1}$ this is $e^{-C\abs\sigma}$.

  For \cref{it:shape-mass-tail}, a Chernoff bound for the offspring sums in a tree
  exploration gives a finite exponential moment for the total progeny $Z$ under the
  subcritical law $\theta^\dagger$, whose support is bounded. By \cref{thm:shape-iid},
  $\abs S$ is stochastically dominated by $m+Z_1+\dots+Z_{m-1}$, with $m$ geometric and the
  $Z_i$ independent copies of $Z$ independent of $m$, so $\bbE e^{t\abs S}<\infty$ for some
  $t>0$. Another Chernoff bound gives
  $\mu(\abs S\geq n)\leq(\bbE e^{t\abs S})e^{-tn}\leq e^{-tn/2}$ for all sufficiently
  large $n$.
\end{proof}

\subsection{Net quantisation of shape classes}\label{sec:shape-net}

We compare shapes through maps that control their entries and exits, the positions where the shapes
join in the assembled trees. A \emph{$K$-marked quasi-isometry} $\sigma\to\tau$ is a $K$-quasi-isometry
of the realisations mapping the entry of $\sigma$ within distance $K$ of the entry of $\tau$, and likewise
for the exits. Two shapes are \emph{$K$-comparable} if a $K$-marked quasi-isometry exists in at
least one of the two directions. Comparability at a fixed scale is reflexive and symmetric but
need not be transitive, so we build the classes as a net around representatives.

\begin{definition}[The net and the label graph]\label{def:shape-net}
  Fix an integer $D\geq2$ and enumerate $\cS$ by size. If two shapes have the same size, enumerate
  them arbitrarily. We process this enumeration in order, retaining a shape precisely when it
  admits no $D$-marked quasi-isometry to an earlier retained shape. We write $\cR_D\subseteq\cS$ for the
  set of retained shapes. For $\sigma\in\cS$, let $\rep_D(\sigma)$ be the earliest member of
  $\cR_D$ to which $\sigma$ admits a $D$-marked quasi-isometry. The \emph{label graph} $G_D$ has
  vertex set $\cR_D$, with an edge between distinct
  $\sigma^*,\tau^*\in\cR_D$ whenever they are $27D^4$-comparable.
\end{definition}

\paragraph*{\bf Measurability} Since $\cS$ is countable, $\rep_D$ and every event built from it below are measurable.

\begin{lemma}[Basic net properties]\label{thm:shape-net}
  The composition of a $K$-marked and a $K'$-marked quasi-isometry is $3KK'$-marked, and
  every $K$-marked quasi-isometry has a $3K^2$-marked quasi-inverse.

  Every shape $\sigma\in\cS$ has a representative $\rep_D(\sigma)$. For each
  $\sigma^*\in\cR_D$, we have $\rep_D(\sigma^*)=\sigma^*$, so every class
  $\rep_D^{-1}(\sigma^*)$ contains its representative. The size bound
  $\abs{\rep_D(\sigma)}\leq\abs\sigma$ holds for every $\sigma\in\cS$.

  Any two members of one class admit a $9D^3$-marked quasi-isometry. Members of equal or adjacent
  classes are $729D^7$-comparable.
\end{lemma}

\begin{proof}
  The bounds $3KK'$ and $3K^2$ follow from \cref{def:qi} and the triangle inequality,
  choosing a quasi-inverse $f^-$ of a $K$-marked map $f$ with $d(f(f^-(y)),y)\leq K$.
  The greedy construction gives every shape a representative no later in the size enumeration
  and fixes retained shapes, proving all assertions about $\rep_D$.

  Composing through the common representative gives a marked quasi-isometry between members
  of one class with constant $3D(3D^2)=9D^3$. For adjacent classes, exchange the shapes if
  necessary to orient the $27D^4$-marked map from the first representative to the second.
  Composing the $D$-marked map to the first representative, this map, and a $3D^2$-marked
  quasi-inverse from the second representative to the second shape gives the constant
  $9D(27D^4)(3D^2)=729D^7$.
\end{proof}

\paragraph*{\bf Connectivity} Although we do not use connectivity below, we can see it by deleting
leaves. Given a representative $\sigma^*$ of size at least two, delete a leaf from a nonempty bush,
or the exit if the shape is a bare path, to obtain a shape $\tau$ of size $\abs{\sigma^*}-1$.
Define a map $\sigma^*\to\tau$ by mapping the deleted leaf to its neighbour and
fixing every other vertex. This is a $1$-marked quasi-isometry.
Composing with $\tau\to\rep_D(\tau)$ gives a $3D$-marked map by \cref{thm:shape-net}, so
$\sigma^*$ is adjacent to the strictly smaller representative $\rep_D(\tau)$. Iterating
connects every vertex to the singleton representative, proving the following lemma.

\begin{lemma}[Connectivity of the label graph]\label{thm:shape-connected}
  The graph $G_D$ is connected.
\end{lemma}

\paragraph*{\bf Counting representatives} We next bound the number of representatives below a given size. The factor $D^{-1/3}$ in the bound records the gain from allowing comparisons at scale $D$.

\begin{lemma}[Entropy dilution]\label{thm:dilution}
  There is an absolute constant $c_0$ such that for all integers $D\geq2$ and $n\geq1$,
  every family of shapes of size at most $n$, no two of which admit a $D$-marked
  quasi-isometry, has at most $\exp\bigl(c_0(nD^{-1/3}+1)\bigr)$ members. In particular, there
  are at most $\exp\bigl(c_0(nD^{-1/3}+1)\bigr)$ representatives of size at most $n$.
\end{lemma}

\begin{proof}
  The Catalan bound of $4^{k-1}$ for rooted plane trees with $k$
  vertices~\cite[Section~I.5.1]{FlajoletSedgewick2009} gives at most
  $\sum_{k=1}^K k4^{k-1}\leq e^{3K}$ rooted trees with at most $K$ vertices and one additional
  marked vertex, up to isomorphism. For $D\leq729$, both counts are at most
  $e^{3n}\leq e^{27nD^{-1/3}}$, since the marked realisations have at most $n$ vertices and a
  marked isometry gives a $D$-marked map in either direction.

  For $D>729$, set $s=\floor{(D/72)^{1/3}}\geq2$, so that $72s^3\leq D$ and
  $s^{-1}\leq9D^{-1/3}$. We cut the realisation of a shape $\sigma$ into connected parts.
  We process the vertices from the leaves towards the root. At each vertex, we collect
  the vertex itself and all its descendants that have not yet been assigned to a part.
  If this set has at least $s$ vertices, we declare it a part. Otherwise, we leave its
  vertices unassigned. Since each child subtree contains fewer than $s$ unassigned vertices
  and each vertex has at most two children, every cut part has between $s$ and $2s-1$ vertices.
  For $\abs\sigma\leq n$, the disjoint cut parts and the possible smaller remainder at the root
  therefore form at most $n/s+1$ parts. Contract them to obtain $T(\sigma)$, rooted at the part
  of the entry and marked at the part of the exit, and let $P$ be the projection.
  The path from $x$ to $y$ contains $d(P(x),P(y))$ edges between parts and at most
  $2s-2$ edges within each part it meets. Hence
  \[
    d(P(x),P(y))\leq d(x,y)\leq(2s-1)d(P(x),P(y))+2s-2.
  \]
  Thus $P$ is a surjective $2s$-marked quasi-isometry preserving both marks exactly.

  If $T(\sigma)$ and $T(\tau)$ are isomorphic as rooted marked trees, composing the projections
  and their quasi-inverses through the isomorphism gives marked quasi-isometries in both
  directions with constant $3\cdot2s\cdot3(2s)^2=72s^3\leq D$, by \cref{thm:shape-net}.
  Hence distinct family members have nonisomorphic marked contractions, as do distinct
  representatives, since a map from a later representative to an earlier one is excluded by
  the greedy construction. Counting the contractions gives both bounds with $c_0=27$, since
  \[
    e^{3(n/s+1)}\leq e^{27nD^{-1/3}+3}\leq e^{27(nD^{-1/3}+1)}.\qedhere
  \]
\end{proof}

\subsection{Glued transfer}\label{sec:shape-transfer}

We combine marked quasi-isometries between corresponding shapes by defining the map separately
on each copy. A path between two copies traverses each intermediate chain together with a
joining edge. These traversals have positive length, so their additive errors can be absorbed
into a multiplicative bound. Only the two endpoint terms leave a fixed additive error.

\begin{lemma}[Glued transfer]\label{thm:glued-transfer}
  Let $(\sigma_w)_{w\in\bbB}$ and $(\sigma'_w)_{w\in\bbB}$ be families of shapes with
  assemblies $\cT$ and $\cT'$, let $K\geq1$, and suppose that for every $w\in\bbB$ there is
  a $K$-marked quasi-isometry $\sigma_w\to\sigma'_w$. Then there is a root-preserving
  $8K^2$-quasi-isometry $\cT\to\cT'$.
\end{lemma}

\begin{proof}
  Choose a $K$-marked quasi-isometry $\varphi_w\colon\sigma_w\to\sigma'_w$ for each $w$ and
  define $\psi$ by these maps on the disjoint copies. Write $d,d'$ for the assembly metrics,
  $a_w,b_w$ for the entry and exit of the copy at $w$, and $m_w=d(a_w,b_w)+1$. We denote the
  corresponding target quantities with primes.
  Each copy is an induced subtree, so its internal
  distances agree with assembly distances. For $x$ in the copy at $u$ and $y$ in the copy at
  $v$, each intermediate copy contributes its entry-to-exit distance and one joining edge,
  giving
  \begin{align}
    d(x,y)&=d(x,b_u)+1+\sum_{u<t<v}m_t+d(y,a_v)&&(u\text{ a strict prefix of }v),
      \label{eq:assembly-anc}\\
    d(x,y)&=d(x,a_u)+d(y,a_v)+2+\sum_{z<t<u}m_t+\sum_{z<t<v}m_t
      &&(z=u\wedge v,\ u\neq z\neq v),\label{eq:assembly-div}
  \end{align}
  where sums run over strict prefix intervals and the additional $1$ or $2$ counts the
  joining edges at the ancestor exit or the common ancestor exit $b_z$.

  The marked quasi-isometry bounds give
  \[
    d'\bigl(\psi(x),a'_u\bigr)\leq K\,d(x,a_u)+2K,
    \qquad
    d(x,a_u)\leq K\,d'\bigl(\psi(x),a'_u\bigr)+2K^2,
  \]
  and likewise for the exits. Comparing both marks gives additive errors $3K$ and $3K^2$
  for the entry-to-exit distances, which we absorb using $m_w,m'_w\geq1$ to obtain
  \[
    m'_w\leq4K\,m_w,
    \qquad
    m_w\leq8K^2m'_w.
  \]

  Comparing \cref{eq:assembly-anc,eq:assembly-div} with their primed versions term by term,
  after exchanging $x,y$ if necessary, gives
  \[
    d'\bigl(\psi(x),\psi(y)\bigr)\leq4K\,d(x,y)+4K,
    \qquad
    d(x,y)\leq8K^2d'\bigl(\psi(x),\psi(y)\bigr)+4K^2,
  \]
  which also hold within one copy by the bounds for $\varphi_w$.

  The image is $K$-dense, since this holds copy by copy. Finally, modify $\psi$ only at the root,
  setting $\psi(a_{\troot})=a'_{\troot}$. This changes its value there by at most $K$.
  The additive errors increase to $5K$ and $8K^3+4K^2$, and the image is $2K$-dense.
  Since $K\geq1$, these bounds satisfy
  \cref{def:qi} with constant $8K^2$.
\end{proof}

\section{Matching independent graph labels on the binary tree}\label{sec:setup}

We now prove \cref{thm:matching} for independent, identically distributed labellings. At each
split, an automorphism can either preserve or exchange the two child subtrees. We control this
recursive pairing by a potential that contracts in successive steps when it is of small value. 
For full labellings, the
matching must also respect the labels at the root, so we need a second estimate for combining
independent conditions.

\subsection{Setup and the compatibility rule}
We recall the notation of \cref{sec:words,sec:matching}. For
$h\geq0$, the complete binary tree $\bbB_h$ consists of the words over $\set{1,2}$ of length at most
$h$, with leaves $\set{1,2}^h$. An automorphism $\pi\in\Aut(\bbB_h)$ fixes the root, preserves
levels and recursively fixes or interchanges the two child subtrees at each internal vertex.

\smallskip
\paragraph*{\bf Compatibility} The label graph $G=(V,E)$ is countable and undirected and $\mu$ is a probability measure on $V$.
While the definitions allow disconnected $G$, taking distances between components to be infinite,
the hypotheses in \cref{thm:matching} imply that $\mu$ is supported on a single connected
component. See also \cref{rem:connectedness} below.
We say that two labels $v,w\in V$ are compatible when
\begin{equation}\label{eq:compat}
 d_G(v,w)\leq1 ,
\end{equation}
and write $B_G(v,1)=\set{w\in V\colon d_G(v,w)\leq1}$ for the closed unit ball. The relation
\cref{eq:compat} is symmetric and reflexive, the only two properties used by the estimates.

\smallskip
\paragraph*{\bf Leaf and full labellings} A leaf labelling of $\bbB_h$ is a map $x\colon\set{1,2}^h\to V$ and a full labelling is a map
$x\colon\bbB_h\to V$. For two labellings of the same kind, we write
$x\approx^{\mathrm{leaf}}_h y$, respectively $x\approx_h y$, if there is an automorphism
$\pi\in\Aut(\bbB_h)$ such that
\[
 d_G\bigl(x(\pi v),y(v)\bigr)\leq1
\]
simultaneously at every leaf, respectively at every vertex.
Throughout, labels within each labelling are independent with common law $\mu$, and the two
labellings are independent of each other.

\smallskip
\paragraph*{\bf Masses and potentials} Recall that the compatible mass at $v$ is the probability that an independent label is
compatible with $v$,
\begin{equation}\label{eq:mv}
 b(v)=\mu\bigl(B_G(v,1)\bigr).
\end{equation}
For $\alpha\geq1$, the graph potential of \cref{eq:potential} is
\begin{equation}\label{eq:etaG}
 \eta_{G,\alpha}(\mu)=\sum_{v\colon\mu(v)>0}\mu(v)\,\frac{1-b(v)}{b(v)^{\alpha}} ,
\end{equation}
with $\alpha=5/2$ in the matching theorem and a variable exponent in the estimates, see also
\cref{rem:exponent} below. Reflexivity gives $b(v)\geq\mu(v)>0$ on the support of $\mu$, so every
denominator in \cref{eq:etaG} is positive and the potential is well-defined in $[0,\infty]$.

\begin{restate}{thm:matching}
Let two independent labellings of $\bbB_h$ be given, with labels sampled independently
from $\mu$ within each labelling.
\begin{enumerate}[(1)]
\item If $\eta_{G,5/2}(\mu)\leq\tfrac1{256}$, then for leaf labellings
\[
 \bbP(\text{no automorphism of }\bbB_h\text{ matches every leaf})
 \leq\eta_{G,5/2}(\mu)\left(\frac{253}{256}\right)^{\!h},
\]
which tends to zero as $h\to\infty$.
\item If $\eta_{G,5/2}(\mu)\leq10^{-4}$, then for full labellings, uniformly in the height,
\[
 \bbP(\text{some automorphism of }\bbB_h\text{ matches every vertex})
 \geq1-16\,\eta_{G,5/2}(\mu)>0 .
\]
The same lower bound holds for the existence of a single automorphism matching every
vertex of two independent labellings of the infinite rooted binary tree, each with
i.i.d.\ labels of law $\mu$.
\end{enumerate}
\end{restate}

\begin{remark}\label{rem:connectedness}
Either hypothesis of \cref{thm:matching} confines $\mu$ to a single connected component of $G$.
Let $C$ be a component with $m:=\mu(C)>0$. Every $v\in C$ has $B_G(v,1)\subseteq C$ and hence
$b(v)\leq m$, so the terms of \cref{eq:etaG} indexed by $C$ sum to at least
$m(1-m)m^{-5/2}=(1-m)m^{-3/2}$, which is at least $\sqrt2$ once $m\leq\tfrac12$. As
$\sqrt2>\tfrac1{256}$, every component of positive mass carries more than half of it, so exactly
one component carries mass and it carries all of it.
\end{remark}

\paragraph*{\bf Potential estimates} For $\alpha\geq1$, write
\begin{equation}\label{eq:alpha}
 \phi_\alpha(t)\coloneqq\frac{t}{(1-t)^\alpha}\qquad(0\leq t<1),
\end{equation}
so the summand in \cref{eq:etaG} is
$\mu(v)\phi_\alpha(1-b(v))$. This weight is nonnegative and, since $(1-t)^\alpha\leq1$,
\begin{equation}\label{eq:phi-dominates}
 \phi_\alpha(t)\geq t\qquad(0\leq t<1).
\end{equation}

\noindent The following elementary estimates control the child-pair recursion and the combination of independent
conditions.

\begin{lemma}[Pointwise potential estimates]\label{thm:toolkit}
Let $\alpha\geq1$, let $L,K>0$ satisfy
\begin{equation}\label{eq:pointwise-constants}
 \frac{t}{(1+t)^\alpha}\leq L\quad(0\leq t\leq1),\qquad
 t(1-t)^\alpha\leq K\quad(0\leq t\leq1),
\end{equation}
and put $c_\alpha:=2(2^\alpha-1)$. For $x,y,q,e\in[0,1)$ and $z\in[0,1/2]$,
\begin{align}
 \phi_\alpha(xy)&\leq\frac{L}2
 \bigl(\phi_\alpha(x)+\phi_\alpha(y)\bigr),
 \label{eq:product-estimate}\\
 (1-q)^{-\alpha}&\leq1+\alpha\,\phi_\alpha(q),
 \label{eq:tangent}\\
 q^2&\leq K\,\phi_\alpha(q),
 \label{eq:Ksq}\\
 (1-z)^{-\alpha}&\leq1+c_\alpha z,
 \label{eq:chord}\\
 \phi_\alpha\bigl(e+(1-e)q\bigr)
 &\leq\phi_\alpha(e)+\phi_\alpha(q)
 +2\alpha\,\phi_\alpha(e)\phi_\alpha(q).
 \label{eq:split}
\end{align}
\end{lemma}

\begin{proof}
We prove the five inequalities in turn.

\emph{Proof of \cref{eq:product-estimate}.}
Expanding both squares,
\[
 (1-xy)^2-(1-x^2)(1-y^2)
 =(1-2xy+x^2y^2)-(1-x^2-y^2+x^2y^2)
 =(x-y)^2\geq0,
\]
so $1-xy\geq\sqrt{(1-x^2)(1-y^2)}$ and therefore
\[
 \phi_\alpha(xy)=\frac{xy}{(1-xy)^\alpha}
 \leq\frac{xy}{(1-x^2)^{\alpha/2}(1-y^2)^{\alpha/2}}
 =\sqrt{\phi_\alpha(x^2)\,\phi_\alpha(y^2)}
 \leq\frac{\phi_\alpha(x^2)+\phi_\alpha(y^2)}2,
\]
the last step by the AM--GM inequality. Finally
\[
 \phi_\alpha(x^2)=\frac{x^2}{(1-x)^\alpha(1+x)^\alpha}
 =\phi_\alpha(x)\cdot\frac{x}{(1+x)^\alpha}
 \leq L\,\phi_\alpha(x)
\]
by \cref{eq:pointwise-constants}, and likewise for $y$.

\emph{Proof of \cref{eq:tangent}.}
Multiplying by $(1-q)^\alpha>0$, the claim is equivalent to
$1\leq(1-q)^\alpha+\alpha q$. This is Bernoulli's inequality in the form
$(1-q)^\alpha\geq1-\alpha q$.

\emph{Proof of \cref{eq:Ksq}.}
By \cref{eq:pointwise-constants},
\[
 q^2=\phi_\alpha(q)\,q(1-q)^\alpha\leq K\,\phi_\alpha(q).
\]

\emph{Proof of \cref{eq:chord}.}
The map $z\mapsto(1-z)^{-\alpha}$ is convex on $[0,1/2]$, hence lies below the chord through
its endpoints $(0,1)$ and $(1/2,2^\alpha)$. The chord is
$z\mapsto1+2(2^\alpha-1)z=1+c_\alpha z$.

\emph{Proof of \cref{eq:split}.}
Write $r:=1-q$ and $s:=1-e$, so that $1-(e+(1-e)q)=sr$ and
\[
 \phi_\alpha\bigl(e+(1-e)q\bigr)
 =\frac{e+sq}{s^\alpha r^\alpha}
 =\frac{e}{s^\alpha}\cdot r^{-\alpha}
 +\frac{q}{r^\alpha}\cdot s^{1-\alpha}
 =\phi_\alpha(e)\,r^{-\alpha}+\phi_\alpha(q)\,s^{1-\alpha}.
\]
By \cref{eq:tangent}, $r^{-\alpha}\leq1+\alpha\phi_\alpha(q)$ and
$s^{1-\alpha}\leq s^{-\alpha}\leq1+\alpha\phi_\alpha(e)$. Expanding the two products gives
$\phi_\alpha(e)+\phi_\alpha(q)+2\alpha\phi_\alpha(e)\phi_\alpha(q)$, as claimed.
\end{proof}

\subsection{Overlap recursion and contraction}\label{sec:contraction}

We prove one step of the matching recursion for an arbitrary symmetric reflexive relation.
Let $(\cX,\mu)$ be a countable probability space and let $R$ be such a relation on $\cX$.
For $x\in\cX$ define the bad and good degrees
\[
 q(x):=\mu\{y\colon x\not\mathrel{R}y\},\qquad r(x):=1-q(x),
\]
and write $\supp\mu:=\{x\in\cX\colon\mu\{x\}>0\}$ for the support. Reflexivity gives
$x\mathrel{R}x$, so $r(x)\geq\mu\{x\}>0$ and hence $q(x)\leq1-\mu\{x\}<1$
for every $x\in \supp\mu$. Thus $\phi_\alpha(q(x))$ is defined on $\supp\mu$. Put
\begin{equation}\label{eq:iid-potential}
 \Phi(R,\mu):=\bbE\,\phi_\alpha(q(X))=\sum_{x\in \supp\mu}\mu\{x\}\,\phi_\alpha(q(x)),\qquad X\sim\mu,
\end{equation}
taking values in $[0,\infty]$. On $\cX^2$, with
the product measure $\mu^2$, define the \emph{symmetrised square relation} $R^\square$ by
\begin{equation}\label{eq:square}
 (x_1,x_2)\mathrel{R^\square}(y_1,y_2)
 \iff
 \bigl(x_1\mathrel{R}y_1\wedge x_2\mathrel{R}y_2\bigr)
 \vee
 \bigl(x_1\mathrel{R}y_2\wedge x_2\mathrel{R}y_1\bigr).
\end{equation}
It is again symmetric and reflexive, and the support of $\mu^2$ is
$\supp\mu\times \supp\mu$.

\paragraph*{\bf Pairings} The symmetrised square describes one step in matching the two child subtrees. Its two possible
pairings use the same target pair, so their success events overlap. These overlaps prevent the
mean bad degree from closing the recursion by itself. The weighted potential $\Phi$ supplies
the additional moment control needed after averaging. We choose its exponent so that the linear
coefficient in the resulting estimate is strictly less than one.

\paragraph*{\bf Constants} For $\alpha\geq1$, define the constants for the pointwise bounds by
\begin{equation}\label{eq:maxima}
 \begin{split}
 L_\alpha&\coloneqq\max_{0\leq t\leq1}\frac{t}{(1+t)^\alpha}
 =\begin{cases}
   2^{-\alpha},&1\leq\alpha\leq2,\\
   \alpha^{-\alpha}(\alpha-1)^{\alpha-1},&\alpha>2,
 \end{cases}\\
 K_\alpha&\coloneqq\max_{0\leq t\leq1}t(1-t)^\alpha
 =\alpha^\alpha(1+\alpha)^{-(1+\alpha)}.
 \end{split}
\end{equation}
With the chord constant $c_\alpha$ from \cref{thm:toolkit}, set
\begin{equation}\label{eq:contraction-constants}
 M_\alpha:=\max\Bigl(2L_\alpha c_\alpha+\tfrac12,\ \tfrac52\Bigr),
 \qquad
 A_\alpha:=2L_\alpha+4K_\alpha,
 \qquad
 B_\alpha:=M_\alpha+4\alpha^2 .
\end{equation}
In the contraction below, $A_\alpha$ is the coefficient of the linear term in $\Phi(R,\mu)$,
and $B_\alpha$ controls the quadratic correction. For sufficiently small potential, the recursion
in \cref{sec:reduction} closes when $A_\alpha<1$. We discuss the choice of exponent in
\cref{rem:exponent}.

\begin{lemma}[Uniform contraction]\label{thm:contraction}
Let $\alpha\geq1$. For every countable $(\cX,\mu,R)$ as above,
\begin{equation}\label{eq:one-step}
 \Phi(R^\square,\mu^2)\leq A_\alpha\,\Phi(R,\mu)+B_\alpha\,\Phi(R,\mu)^2 .
\end{equation}
In particular, if $A_\alpha+B_\alpha\Phi(R,\mu)\leq\rho$ for some $\rho<1$, then
\begin{equation}\label{eq:strict}
 \Phi(R^\square,\mu^2)\leq\rho\,\Phi(R,\mu).
\end{equation}
\end{lemma}

At $\alpha=5/2$ the constants are
\begin{equation}\label{eq:constants-at-five-halves}
 A_\alpha=\frac{12\sqrt{15}}{125}+\frac87\left(\frac57\right)^{5/2}<\frac78,
 \qquad
 M_\alpha<4,
 \qquad
 B_\alpha<29 ,
\end{equation}
so $\Phi(R,\mu)\leq\tfrac1{256}$ gives $\rho=\tfrac{253}{256}$ in \cref{eq:strict}.

\paragraph*{\bf Good transversals} We observe the following description of failure for the two possible pairings, which we record as a lemma for future reference:

\begin{lemma}[$2\times2$ transversal]\label{thm:transversal}
Mark each cell of a $2\times2$ array good or bad. If no row and no column
is entirely bad, then the array has a good transversal, i.e.\ a good main
diagonal or a good antidiagonal.
\end{lemma}

\begin{proof}[Proof of \cref{thm:contraction}]
Fix $x_1,x_2\in \supp\mu$, write $q_i=q(x_i)$, $r_i=1-q_i$, and let
$Y,Y_1,Y_2\sim\mu$ be independent. Put
\[
 c:=\bbP(x_1\not\mathrel{R}Y,\;x_2\not\mathrel{R}Y).
\]
Let $r^\square=r^\square(x_1,x_2)$ be the probability that
$(x_1,x_2)\mathrel{R^\square}(Y_1,Y_2)$, and $q^\square=1-r^\square$; thus $q^\square$ is the
bad degree of $(x_1,x_2)$ under $R^\square$.

\emph{Good degree.}
The two diagonals in \cref{eq:square} each succeed with probability
$r_1r_2$, by independence of $Y_1,Y_2$. Both succeed exactly when each of
$Y_1,Y_2$ is compatible with both $x_1$ and $x_2$. A single label is compatible with both
with probability at most $\min(r_1,r_2)$, so independence and inclusion--exclusion give
\begin{equation}\label{eq:R-exact}
 \begin{split}
 r^\square&=2r_1r_2-\bbP(x_1\mathrel{R}Y,\;x_2\mathrel{R}Y)^2\\
 &\geq2r_1r_2-\min(r_1,r_2)^2.
 \end{split}
\end{equation}

\emph{Bad degree.}
Form the $2\times2$ array with cell $(i,j)$ good if and only if
$x_i\mathrel{R}Y_j$. A good transversal is precisely the event
$(x_1,x_2)\mathrel{R^\square}(Y_1,Y_2)$. By
\cref{thm:transversal}, absence of a good transversal forces a bad
row or a bad column. Row $i$ is bad with probability $q_i^2$ (independence
of $Y_1,Y_2$), and each column is bad with probability $c$. The union
bound gives
\begin{equation}\label{eq:Q-union}
 q^\square\leq q_1^2+q_2^2+2c.
\end{equation}

\emph{Ordering the bad degrees.}
Put
\[
 q_{\max}:=\max(q_1,q_2),\qquad q_{\min}:=\min(q_1,q_2),
\]
whence $0\leq q_{\min}\leq q_{\max}<1$ and
$(1-q_{\max})(1-q_{\min})=r_1r_2$. Since $\min(r_1,r_2)=1-q_{\max}$,
\cref{eq:R-exact} gives
\begin{equation}\label{eq:iid-R-lower}
 \begin{split}
 r^\square&\geq2(1-q_{\max})(1-q_{\min})-(1-q_{\max})^2\\
 &=(1-q_{\max})s,\qquad s:=1+q_{\max}-2q_{\min} .
 \end{split}
\end{equation}
Also
\begin{equation}\label{eq:s-v}
 s-(1-q_{\min})=q_{\max}-q_{\min}\geq0 ,
\end{equation}
so $s\geq1-q_{\min}$.
Since $\phi_\alpha(q^\square)=q^\square/(r^\square)^\alpha$,
the inequalities \cref{eq:Q-union,eq:iid-R-lower} yield the pointwise bound
\begin{equation}\label{eq:iid-split}
 \frac{q^\square}{(r^\square)^\alpha}
 \leq\frac{q_{\max}^2+q_{\min}^2}{(1-q_{\max})^\alpha s^\alpha}
 +\frac{2c}{(1-q_{\max})^\alpha s^\alpha}.
\end{equation}

\emph{Row contribution.}
We claim the pointwise estimate
\begin{equation}\label{eq:row-bound}
 \begin{split}
 \frac{q_{\max}^2+q_{\min}^2}{(1-q_{\max})^\alpha s^\alpha}
 &\leq L_\alpha\bigl(\phi_\alpha(q_{\max})+\phi_\alpha(q_{\min})\bigr)\\
 &\quad+M_\alpha\,\phi_\alpha(q_{\max})\phi_\alpha(q_{\min}).
 \end{split}
\end{equation}
If $q_{\max}=0$ then $q_1=q_2=0$, both sides vanish, and there is nothing to prove.
Assume instead $q_{\max}>0$.

Suppose first $q_{\min}\leq q_{\max}/2$. Then
$0\leq2q_{\min}/(1+q_{\max})\leq\tfrac12$, so \cref{eq:chord} gives
\[
 \begin{aligned}
 s^{-\alpha}
 &=(1+q_{\max})^{-\alpha}
 \left(1-\frac{2q_{\min}}{1+q_{\max}}\right)^{-\alpha}\\
 &\leq\frac{1+2c_\alpha q_{\min}}{(1+q_{\max})^\alpha}.
 \end{aligned}
\]
Using this bound, \cref{eq:maxima} and \cref{eq:phi-dominates}, we obtain
\begin{equation}\label{eq:d-term}
 \begin{split}
 \frac{q_{\max}^2}{(1-q_{\max})^\alpha s^\alpha}
 &=\phi_\alpha(q_{\max})\,\frac{q_{\max}}{s^\alpha}\\
 &\leq L_\alpha\phi_\alpha(q_{\max})
 +2L_\alpha c_\alpha\,\phi_\alpha(q_{\max})\phi_\alpha(q_{\min}).
 \end{split}
\end{equation}
Moreover, by \cref{eq:s-v} and $q_{\min}\leq q_{\max}/2$,
\begin{equation}\label{eq:e-term}
 \begin{split}
 \frac{q_{\min}^2}{(1-q_{\max})^\alpha s^\alpha}
 &=\phi_\alpha(q_{\max})\phi_\alpha(q_{\min})\,
 \frac{q_{\min}}{q_{\max}}\left(\frac{1-q_{\min}}s\right)^{\alpha}\\
 &\leq\tfrac12\,\phi_\alpha(q_{\max})\phi_\alpha(q_{\min}).
 \end{split}
\end{equation}
The coefficient of the product term in \cref{eq:d-term,eq:e-term} is $2L_\alpha c_\alpha+\tfrac12$,
which is at most $M_\alpha$ by \cref{eq:contraction-constants}, so adding
\cref{eq:d-term,eq:e-term} proves \cref{eq:row-bound} in this case.

If instead $q_{\min}>q_{\max}/2>0$, then
\[
 \begin{aligned}
 \frac{q_{\max}^2+q_{\min}^2}{(1-q_{\max})^\alpha s^\alpha}
 &\leq\frac{q_{\max}^2+q_{\min}^2}{(1-q_{\max})^\alpha(1-q_{\min})^\alpha}\\
 &=\phi_\alpha(q_{\max})\phi_\alpha(q_{\min})
 \left(\frac{q_{\max}}{q_{\min}}+\frac{q_{\min}}{q_{\max}}\right)\\
 &<\tfrac52\,\phi_\alpha(q_{\max})\phi_\alpha(q_{\min}),
 \end{aligned}
\]
because $1\leq q_{\max}/q_{\min}<2$ and $\tfrac52\leq M_\alpha$, which again implies
\cref{eq:row-bound}.

\emph{Averaging the row contribution.}
Now take independent $X_1,X_2\sim\mu$, independent also of $Y$, and apply the preceding
estimates with $x_i=X_i$. Write $\Phi:=\Phi(R,\mu)$. The degrees $q_{\max},q_{\min}$
are $q_1=q(X_1)$ and $q_2=q(X_2)$ in some order, so the symmetric combinations are unchanged:
$\phi_\alpha(q_{\max})+\phi_\alpha(q_{\min})=\phi_\alpha(q_1)+\phi_\alpha(q_2)$ and
$\phi_\alpha(q_{\max})\phi_\alpha(q_{\min})=\phi_\alpha(q_1)\phi_\alpha(q_2)$.
Independence of $X_1,X_2$ gives $\bbE[\phi_\alpha(q_i)]=\Phi$ and
$\bbE[\phi_\alpha(q_1)\phi_\alpha(q_2)]=\Phi^2$.
Taking expectations in \cref{eq:row-bound},
\begin{equation}\label{eq:first-final}
 \bbE\!\left[\frac{q_{\max}^2+q_{\min}^2}{(1-q_{\max})^\alpha s^\alpha}\right]
 \leq2L_\alpha\Phi+M_\alpha\Phi^2 .
\end{equation}

\emph{Overlap contribution.}
For $y\in \supp\mu$, write
$H(y):=\bbE_X\bigl[\mathbf 1_{\{X\not\mathrel{R}y\}}r(X)^{-\alpha}\bigr]$, with $X\sim\mu$, for the
weighted mass incompatible with $y$. For a fixed pair $x_1,x_2$, we have $s\geq1-q_{\min}$
and $(1-q_{\max})(1-q_{\min})=r_1r_2$, so
\[
 \frac{2c}{(1-q_{\max})^\alpha s^\alpha}
 \leq\frac{2c}{(1-q_{\max})^\alpha(1-q_{\min})^\alpha}
 =\frac{2c}{r_1^\alpha r_2^\alpha}.
\]
Apply this bound with $x_i=X_i$ and write
$c=\bbE_Y\bigl[\mathbf 1_{\{X_1\not\mathrel{R}Y\}}
\mathbf 1_{\{X_2\not\mathrel{R}Y\}}\bigr]$. Fubini's theorem for nonnegative integrands lets
us first average over $X_1,X_2$. Conditioned on $Y$, these variables are independent with law
$\mu$, so their weighted incompatibility indicators each have conditional expectation $H(Y)$.
The conditional expectation of their product is therefore $H(Y)^2$, giving
\begin{equation}\label{eq:overlap-id}
 \bbE\!\left[\frac{2c}{(1-q_{\max})^\alpha s^\alpha}\right]
 \leq2\,\bbE_{X_1,X_2}\!\left[\frac{c}{r(X_1)^\alpha r(X_2)^\alpha}\right]
 =2\,\bbE_Y\bigl[H(Y)^2\bigr].
\end{equation}
To bound $H$, we separate the ordinary bad degree from the weighted correction. Applying
\cref{eq:tangent} with $q=q(x)$ gives
\begin{equation}\label{eq:W-bound}
 r(x)^{-\alpha}-1=\frac{1-r(x)^\alpha}{r(x)^\alpha}\leq\alpha\,\phi_\alpha(q(x)).
\end{equation}
By \cref{eq:W-bound} and symmetry of $R$ (which gives
$\bbE_X\mathbf 1_{\{X\not\mathrel{R}y\}}
=q(y)$),
\[
 \begin{aligned}
 H(y)&\leq q(y)+\alpha\,\bbE_X\bigl[\mathbf 1_{\{X\not\mathrel{R}y\}}
 \phi_\alpha(q(X))\bigr]\\
 &\leq q(y)+\alpha\Phi,
 \end{aligned}
\]
where we bounded the indicator by $1$ and used $\phi_\alpha\geq0$.
Thus $H(y)^2\leq2q(y)^2+2\alpha^2\Phi^2$, and averaging gives
\begin{equation}\label{eq:iid-H-bound}
 2\,\bbE\bigl[H(Y)^2\bigr]\leq4\,\bbE\bigl[q(Y)^2\bigr]+4\alpha^2\Phi^2 .
\end{equation}
By \cref{eq:Ksq} with $K=K_\alpha$, we have $\bbE[q(Y)^2]\leq K_\alpha\Phi$. Consequently,
\cref{eq:overlap-id,eq:iid-H-bound} bound the expectation of the overlap term of
\cref{eq:iid-split} by
\begin{equation}\label{eq:overlap-final}
 4K_\alpha\Phi+4\alpha^2\Phi^2 .
\end{equation}

\emph{Combination.}
Averaging \cref{eq:iid-split} and adding \cref{eq:first-final,eq:overlap-final},
\[
 \Phi(R^\square,\mu^2)
 \leq(2L_\alpha\Phi+M_\alpha\Phi^2)+(4K_\alpha\Phi+4\alpha^2\Phi^2)
 =A_\alpha\Phi+B_\alpha\Phi^2 ,
\]
which is \cref{eq:one-step}. If moreover $A_\alpha+B_\alpha\Phi\leq\rho$, then
$\Phi(R^\square,\mu^2)\leq(A_\alpha+B_\alpha\Phi)\Phi\leq\rho\Phi$, which is \cref{eq:strict}.
\end{proof}

\begin{proof}[Proof of \cref{eq:constants-at-five-halves}]
An elementary differentiation evaluates the first maximum of \cref{eq:maxima} at $\alpha=5/2$
\begin{equation}\label{eq:lambda-max}
 L_{5/2}=\frac{2/3}{(5/3)^{5/2}}=\frac{6\sqrt{15}}{125}
 \quad(\text{at }t=\tfrac23),
\end{equation}
and the second maximum
\begin{equation}\label{eq:kappa}
 K_{5/2}=\frac27\left(\frac57\right)^{5/2}
 \quad(\text{at }t=\tfrac27).
\end{equation}
Hence $2L_{5/2}=\tfrac{12\sqrt{15}}{125}$ and $4K_{5/2}=\tfrac87(5/7)^{5/2}$.
Substitution gives $A_{5/2}<0.865<\tfrac78$ and $M_{5/2}<3.963<4$,
hence $B_{5/2}=M_{5/2}+25<29$. Finally, writing $\Phi=\Phi(R,\mu)$,
$\Phi\leq\tfrac1{256}$ gives
$A_{5/2}+B_{5/2}\Phi<\tfrac78+\tfrac{29}{256}=\tfrac{253}{256}$.
\end{proof}

\begin{remark}\label{rem:matching-constants}
The numerical constants in \cref{thm:matching} are convenient uniform choices rather than
optimised values. For the leaf estimate, we replace $A_{5/2}$ and $B_{5/2}$ by the rational
bounds $7/8$ and $29$, which gives the contraction factor $253/256$ at potential $1/256$.
The choices for the full matching estimate are checked in \cref{sec:reduction}. Optimising
the intermediate pointwise bounds would improve both conclusions.
\end{remark}

\begin{remark}\label{rem:exponent}
The exponent $5/2$ is not structural. The pointwise estimates and the product step require only
$\alpha\geq1$, while the present contraction estimate requires $A_\alpha<1$. This fails at
$\alpha=2$, where $A_2=59/54$, and holds at $\alpha=5/2$ by
\cref{eq:constants-at-five-halves}. We use $5/2$
because it gives short explicit bounds and because squaring inequalities involving a
half-integer power reduces them to polynomial inequalities with rational coefficients.
For full labellings, the stronger estimate in \cref{thm:four-law-contraction}, used to prove
\cref{thm:markov-matching}, allows every $\alpha\geq4/3$, including exponents below $2$,
with an $\alpha$-dependent constant and smallness threshold. Smaller exponents reduce the
potential at the cost of the constant in the matching bound as well as the admissible threshold.
\end{remark}

\subsection{A product lemma}

For full labellings, matching the two child subtrees must be combined with compatibility at the
root. These conditions involve independent labels, so we need a product estimate for their
potentials.

\begin{lemma}[Product of relations]\label{thm:product}
Let $(\cX_i,\mu_i,R_i)$, $i=1,2$, be countable probability spaces with
symmetric reflexive relations, and let $R_1\otimes R_2$ be the
relation on $\cX_1\times\cX_2$ defined by
\[
 (x_1,x_2)\mathrel{R_1\otimes R_2}(y_1,y_2)
 \iff x_1\mathrel{R_1}y_1\ \text{and}\ x_2\mathrel{R_2}y_2 .
\]
With $\Phi_i:=\Phi(R_i,\mu_i)$ and any $\alpha\geq1$,
\begin{equation}\label{eq:product-bound}
 \Phi(R_1\otimes R_2,\mu_1\otimes\mu_2)
 \leq\Phi_1+\Phi_2+2\alpha\,\Phi_1\Phi_2 .
\end{equation}
\end{lemma}

\begin{proof}
Fix $x=(x_1,x_2)\in \supp\mu_1\times \supp\mu_2$, and let $q_i$ be the bad degree of $x_i$ in $R_i$ and
put $r_i=1-q_i$. Under $R_1\otimes R_2$, the good degree of $x$ is $r_1r_2$, so its bad degree is
$q=q_1+(1-q_1)q_2$. Thus \cref{eq:split} gives
\[
 \phi_\alpha(q)
 \leq\phi_\alpha(q_1)+\phi_\alpha(q_2)
 +2\alpha\phi_\alpha(q_1)\phi_\alpha(q_2).
\]
Taking expectations over the independent coordinates $X_1\sim\mu_1$ and $X_2\sim\mu_2$ gives
\cref{eq:product-bound}.
\end{proof}

\subsection{Reduction to the one-site potential}\label{sec:reduction}

The reduction uses only symmetry and reflexivity of the label relation. We therefore fix an
arbitrary symmetric reflexive relation $R_0$ on a countable label space $(V,\mu)$ and let
\begin{equation}\label{eq:Phi0}
 \Phi_0:=\Phi(R_0,\mu)
\end{equation}
be its one-site potential. We use $x\approx^{\mathrm{leaf}}_h y$ and $x\approx_h y$ for the
existence of an automorphism matching the relevant labels by $R_0$, at every leaf and every
vertex respectively. The compatibility rule \cref{eq:compat} is the special case
$v\mathrel{R_0}w\iff d_G(v,w)\leq1$, treated in \cref{sec:graph}. Leaf matching iterates the
symmetrised square. Full matching also imposes the independent root condition, so it uses the
product estimate as well as contraction. The recursion identities do not depend on the
exponent, and the potential estimates of \cref{thm:contraction,thm:product} hold for every
$\alpha\geq1$.

We split height-$(h+1)$ labellings according to the two principal subtrees. For leaf labellings,
write $x=(x_1,x_2)$ and $y=(y_1,y_2)$, where $x_i,y_i\colon\set{1,2}^h\to V$. For full
labellings, write $x=(a,x_1,x_2)$ and $y=(a',y_1,y_2)$, where $a,a'\in V$ are the root labels
and $x_i,y_i$ are the full height-$h$ labellings of the corresponding subtrees. An automorphism
of $\bbB_{h+1}$ fixes the root, either preserves or swaps the two principal subtrees, and then
acts by independent automorphisms inside them. Hence
\begin{align}
 x\approx^{\mathrm{leaf}}_{h+1}y &\iff
 (x_1\approx^{\mathrm{leaf}}_h y_1\wedge x_2\approx^{\mathrm{leaf}}_h y_2)\vee
 (x_1\approx^{\mathrm{leaf}}_h y_2\wedge x_2\approx^{\mathrm{leaf}}_h y_1),
 \label{eq:leaf-rec}\\[2pt]
 x\approx_{h+1}y &\iff
 a\mathrel{R_0}a'\ \wedge\
 \bigl((x_1\approx_h y_1\wedge x_2\approx_h y_2)\vee
 (x_1\approx_h y_2\wedge x_2\approx_h y_1)\bigr).
 \label{eq:full-rec}
\end{align}
At height $0$ the tree is a single vertex, so both $\approx^{\mathrm{leaf}}_0$ and $\approx_0$
coincide with $R_0$, and
\begin{equation}\label{eq:base}
 \Phi(\approx^{\mathrm{leaf}}_0,\mu)=\Phi(\approx_0,\mu)=\Phi_0 .
\end{equation}

In each of the leaf and full cases, write $\cX_h$ for the corresponding countable labelling
space and $\mu_h$ for its product law. Each matching relation is symmetric and reflexive:
this holds at height zero, and \cref{eq:leaf-rec,eq:full-rec} preserve both properties because
the symmetrised square and the product of relations do so. Under the subtree decomposition,
the leaf law at height $h+1$ is $\mu_h^2$ on $\cX_h^2$, with matching relation
$\bigl(\approx^{\mathrm{leaf}}_h\bigr)^\square$. The full law is $\mu\otimes\mu_h^2$ on
$V\times\cX_h^2$, with matching relation $R_0\otimes\approx_h^\square$. The root label is
independent of the two subtree labellings because all labels are i.i.d.

When $A_\alpha<1$, a sufficiently small one-site potential makes the leaf potential decrease
with height. In the full case, the root condition contributes at each step, and we instead keep the
potential below a fixed multiple of $\Phi_0$. These two bounds give the following conclusions.

\begin{theorem}[Reduction]\label{thm:reduction}
Let $\alpha\geq1$ and let $R_0$ be symmetric and reflexive with one-site potential $\Phi_0$.
\begin{enumerate}
\item If $\rho:=A_\alpha+B_\alpha\Phi_0<1$, then for leaf labellings
\begin{equation}\label{eq:leaf-red}
 \bbP(\text{no leaf matching automorphism of }\bbB_h)
 \leq\Phi_0\,\rho^{\,h}.
\end{equation}
\item Let $\mathfrak K>0$ and suppose
\begin{equation}\label{eq:full-hyp}
 \bigl(1+2\alpha\Phi_0\bigr)\bigl(A_\alpha\mathfrak K+B_\alpha\mathfrak K^2\Phi_0\bigr)
 \leq\mathfrak K-1 .
\end{equation}
Then for full labellings, uniformly in $h$,
\begin{equation}\label{eq:full-red}
 \bbP(\text{some automorphism of }\bbB_h\text{ matches every vertex})
 \geq1-\mathfrak K\Phi_0 ,
\end{equation}
and the same lower bound holds for two independent infinite labelled
trees.
\end{enumerate}
\end{theorem}

The condition \cref{eq:full-hyp} can hold only if $A_\alpha<1$, since it implies
$A_\alpha\mathfrak K\leq\mathfrak K-1$. Conversely, when $A_\alpha<1$, every
$\mathfrak K>1/(1-A_\alpha)$ satisfies it for sufficiently small $\Phi_0$.
At $\alpha=5/2$, \cref{eq:constants-at-five-halves} gives
$\rho\leq\tfrac{253}{256}$ for $\Phi_0\leq\tfrac1{256}$, and $\mathfrak K=16$ with
$\Phi_0\leq10^{-4}$ satisfies \cref{eq:full-hyp}, since then
$(1+5\Phi_0)(14+7424\Phi_0)\leq\tfrac{2001}{2000}\cdot\tfrac{9214}{625}<15$.

\begin{proof}
\emph{Leaf case.}
Put $\Phi_h:=\Phi(\approx^{\mathrm{leaf}}_h,\mu_h)$, so $\Phi_0$ agrees with
\cref{eq:Phi0} by \cref{eq:base}. Because
$\approx^{\mathrm{leaf}}_{h+1}=\bigl(\approx^{\mathrm{leaf}}_h\bigr)^\square$,
\cref{thm:contraction} applies at each level. We show
$\Phi_h\leq\rho^h\Phi_0$ by induction. It holds at $h=0$. If
$\Phi_h\leq\Phi_0$ then $A_\alpha+B_\alpha\Phi_h\leq\rho$, so \cref{eq:strict} gives
$\Phi_{h+1}\leq\rho\,\Phi_h\leq\rho^{h+1}\Phi_0$, and in
particular $\Phi_{h+1}\leq\Phi_0$, so the induction continues. Write $q_h(x)$ for the bad
degree of a leaf labelling $x$, the probability that an independent leaf labelling does not
match it. For independent leaf labellings $x,y$, the failure probability is the mean bad
degree, and $q\leq\phi_\alpha(q)$ by
\cref{eq:phi-dominates}:
\[
 \bbP(x\not\approx^{\mathrm{leaf}}_h y)=\bbE\,q_h(x)
 \leq\bbE\,\phi_\alpha(q_h(x))=\Phi_h\leq\Phi_0\,\rho^{\,h}.
\]

\emph{Full case.}
Put $\Theta_h:=\Phi(\approx_h,\mu_h)$, so $\Theta_0=\Phi_0$. Let
$\Phi^\square_h:=\Phi(\approx_h^\square,\mu_h^2)$ be the
potential of the child-pair relation. \Cref{thm:contraction} gives
\begin{equation}\label{eq:sigma}
 \Phi^\square_h\leq A_\alpha\Theta_h+B_\alpha\Theta_h^2 .
\end{equation}
By \cref{eq:full-rec}, $\approx_{h+1}=R_0\otimes\approx_h^\square$ on the
product space $V\times\cX_h^2$ with the product measure
$\mu\otimes\mu_h^2$. The two factors have potentials $\Phi_0$ and $\Phi^\square_h$.
Applying \cref{thm:product} and then \cref{eq:sigma} gives
\begin{equation}\label{eq:theta-rec}
 \begin{aligned}
 \Theta_{h+1}&\leq\Phi_0+\Phi^\square_h+2\alpha\Phi_0\Phi^\square_h\\
 &=\Phi_0+(1+2\alpha\Phi_0)\Phi^\square_h\\
 &\leq\Phi_0+(1+2\alpha\Phi_0)\bigl(A_\alpha\Theta_h+B_\alpha\Theta_h^2\bigr).
 \end{aligned}
\end{equation}
Write $\mathfrak f(t)$ for the last expression with $\Theta_h$ replaced by $t\geq0$.
This map is increasing. The hypothesis \cref{eq:full-hyp} makes the interval
$[0,\mathfrak K\Phi_0]$ invariant under $\mathfrak f$: for $0\leq t\leq\mathfrak K\Phi_0$,
\[
 \begin{aligned}
 \mathfrak f(t)&\leq\mathfrak f(\mathfrak K\Phi_0)\\
 &=\Phi_0+(1+2\alpha\Phi_0)
   \bigl(A_\alpha\mathfrak K\Phi_0+B_\alpha\mathfrak K^2\Phi_0^2\bigr)\\
 &=\Phi_0+\Phi_0\,(1+2\alpha\Phi_0)
   \bigl(A_\alpha\mathfrak K+B_\alpha\mathfrak K^2\Phi_0\bigr)\\
 &\leq\Phi_0+(\mathfrak K-1)\Phi_0=\mathfrak K\Phi_0.
 \end{aligned}
\]
The same hypothesis implies $\mathfrak K\geq1$, so
$\Theta_0=\Phi_0\leq\mathfrak K\Phi_0$. Induction therefore yields
\begin{equation}\label{eq:theta-uniform}
 \Theta_h\leq\mathfrak K\Phi_0\qquad(h\geq0).
\end{equation}
Now let $q_h(x)$ be the bad degree for the full matching relation. Exactly as in the leaf
case, for independent full labellings $x,y$,
\[
 \bbP(x\not\approx_h y)=\bbE\,q_h(x)\leq\Theta_h\leq\mathfrak K\Phi_0,
\]
which is \cref{eq:full-red}.

\emph{Infinite tree.}
Matches chosen separately at each finite height need not agree under restriction. We use
compactness to obtain a compatible sequence of them.
Realise the labels on two infinite rooted binary trees, let $M$ be the
event that some automorphism of the infinite tree matches every vertex, and
let $\cM_h$ be the event that some $\pi\in\Aut(\bbB_h)$ matches all vertices
through level $h$. Then $\bbP(\cM_h)\geq1-\mathfrak K\Phi_0$ by \cref{eq:full-red}, as
$\cM_h$ depends only on levels $0,\dots,h$. An automorphism of $\bbB_{h+1}$
restricts, by level preservation, to one of $\bbB_h$ matching through level
$h$. Hence $\cM_{h+1}\subseteq\cM_h$, and continuity from above gives
\[
 \bbP\Bigl(\bigcap_{h\geq0}\cM_h\Bigr)=\lim_{h\to\infty}\bbP(\cM_h)\geq1-\mathfrak K\Phi_0 .
\]
We claim $M=\bigcap_{h\geq0}\cM_h$. An automorphism of the infinite tree
preserves levels, so its restriction to $\bbB_h$ matches through level $h$, giving
$M\subseteq\bigcap_h\cM_h$. For the converse, take two labellings for which every finite
level admits a matching. The sets $\cF_h\subseteq\Aut(\bbB_h)$ of matching automorphisms are
nonempty and finite, and each $\pi_{h+1}\in\cF_{h+1}$ restricts to a unique parent
in $\cF_h$. This is an infinite, finitely-branching, rooted tree, so
K\H{o}nig's infinity lemma provides a compatible branch $(\pi_h)_{h\geq0}$,
whose common extension is an automorphism of the infinite tree matching
every vertex. The claim gives both the measurability of $M$ and
$\bbP(M)\geq1-\mathfrak K\Phi_0$.
\end{proof}

\section{Graph matching and special cases}\label{sec:graph}

We specialise to the one-site relation \cref{eq:compat},
$v\mathrel{R_0}w\iff d_G(v,w)\leq1$, which is symmetric and reflexive.
For a fixed label $v$ the probability that an independent label is
compatible with it is $b(v)=\mu(B_G(v,1))$, so the good degree is $b(v)$
and the bad degree $1-b(v)$. Hence, for every $\alpha\geq1$, the one-site potential
\cref{eq:Phi0} is exactly
\[
 \Phi_0=\Phi(R_0,\mu)
 =\sum_{v\colon\mu(v)>0}\mu(v)\,\frac{1-b(v)}{b(v)^{\alpha}}
 =\eta_{G,\alpha}(\mu).
\]

\begin{proof}[Proof of \cref{thm:matching}]
Apply \cref{thm:reduction} at $\alpha=5/2$ with $\Phi_0=\eta_{G,5/2}(\mu)$. By
\cref{eq:constants-at-five-halves}, $\Phi_0\leq\tfrac1{256}$ gives $\rho\leq\tfrac{253}{256}$ in
\cref{eq:leaf-red}, which is part \labelcref{it:matching-leaf}, and $\Phi_0\leq10^{-4}$ makes
\cref{eq:full-hyp} hold at $\mathfrak K=16$, so \cref{eq:full-red} and the infinite-tree statement
give part \labelcref{it:matching-full}. Positivity holds since $16\cdot10^{-4}<1$.
\end{proof}

\paragraph*{\bf Local dominance} To estimate $\eta_{G,\alpha}(\mu)$, we separate the contribution of a distinguished label from
the remaining weighted sum. The first term is controlled by the mass outside its closed unit
ball, while each remaining term is bounded using its compatible mass.

\begin{corollary}[Local-dominance condition]\label{thm:dominance}
Fix $\alpha\geq1$ and $0\in V$ with $\mu(0)>0$, and put
\[
 \eps_0:=1-\mu(0),
 \qquad
 \tau:=\mu\bigl(V\setminus B_G(0,1)\bigr),
 \qquad
 \cD_{G,\alpha}(\mu,0):=
 \frac{\tau}{(1-\eps_0)^{\alpha}}
 +\sum_{\substack{v\neq0\\ \mu(v)>0}}\frac{\mu(v)}{\mu(B_G(v,1))^{\alpha}} .
\]
Then $\eta_{G,\alpha}(\mu)\leq\cD_{G,\alpha}(\mu,0)$. If $\cD_{G,5/2}(\mu,0)\leq\tfrac1{256}$ the
leaf conclusion of \cref{thm:matching} holds, and if $\cD_{G,5/2}(\mu,0)\leq10^{-4}$ the full
conclusion holds with $16\eta_{G,5/2}(\mu)$ replaced by $16\,\cD_{G,5/2}(\mu,0)$.
\end{corollary}

\begin{proof}[Proof of \cref{thm:dominance}]
Since $0\in B_G(0,1)$ we have $b(0)\geq\mu(0)=1-\eps_0$, while
$1-b(0)=\mu(V\setminus B_G(0,1))=\tau$. Using $\mu(0)\leq1$ as well,
the $v=0$ term of $\eta_{G,\alpha}(\mu)$ satisfies
\[
 \mu(0)\,\frac{1-b(0)}{b(0)^{\alpha}}\leq\frac{\tau}{(1-\eps_0)^{\alpha}} .
\]
For every $v\neq0$ with $\mu(v)>0$, dropping $1-b(v)\leq1$ gives
\[
 \mu(v)\,\frac{1-b(v)}{b(v)^{\alpha}}\leq\frac{\mu(v)}{\mu(B_G(v,1))^{\alpha}} .
\]
Summing proves $\eta_{G,\alpha}(\mu)\leq\cD_{G,\alpha}(\mu,0)$. Under the respective
hypotheses, \cref{thm:matching} gives the leaf conclusion and a full matching probability at least
$1-16\eta_{G,5/2}(\mu)\geq1-16\cD_{G,5/2}(\mu,0)$.
\end{proof}

\paragraph*{\bf Stars and paths} We now specialise the potential \cref{eq:etaG} and the local-dominance bound of
\cref{thm:dominance} to two label graphs, a star and the path on $\bbN_0$. We use the notation of \cref{thm:dominance}, with $\mu(0)>0$ and
$\eps_0=1-\mu(0)$. Suppose every other vertex in the support of $\mu$ is adjacent to
$0$. Then $B_G(0,1)$ contains the whole support, so $\tau=0$, and
$b(v)\geq\mu(0)=1-\eps_0$ for every such $v$. Hence
\begin{equation}\label{eq:star}
 \eta_{G,\alpha}(\mu)\leq\sum_{v\neq0,\;\mu(v)>0}\frac{\mu(v)}{(1-\eps_0)^{\alpha}}
 =\frac{\eps_0}{(1-\eps_0)^{\alpha}} ,
\end{equation}
and whenever $\eps_0/(1-\eps_0)^{5/2}$ is at most $10^{-4}$ the full matching
probability is uniformly at least $1-16\eps_0/(1-\eps_0)^{5/2}$. For the graph consisting of a single edge, every pair of labels is compatible, so the matching
probability is one and the exact potential is zero. Thus the coarse bound in \cref{eq:star}
need not be sharp.

\subsection{The integer alphabet as a path graph}\label{sec:integer}

Let $\mathsf P$ be the path on $\bbN_0$ with edges $\set{j,j+1}$,
$j\geq0$. Then $d_{\mathsf P}(j,k)=\abs{j-k}$, so the compatibility rule
\cref{eq:compat} becomes the nearest-neighbour rule $\abs{j-k}\leq1$. Let
$p=(p_j)_{j\geq0}$ be a probability law on $\bbN_0$ and set $p_{-1}=0$. Then
\[
 B_{\mathsf P}(j,1)=\set{j-1,j,j+1}\cap\bbN_0,
 \qquad
 b(j)=s_j:=p_{j-1}+p_j+p_{j+1}.
\]
For every $j$ with $p_j>0$, we have $s_j\geq p_j>0$, so \cref{eq:etaG} becomes
\[
 \eta_{\mathsf P,\alpha}(p)=\sum_{\substack{j\geq0\\p_j>0}}p_j\,\frac{1-s_j}{s_j^{\alpha}} ,
\]
and \cref{thm:matching} specialises verbatim to i.i.d.\ labels on
$\bbN_0$ with the rule $\abs{j-k}\leq1$: if $\eta_{\mathsf P,5/2}(p)\leq\tfrac1{256}$
the leaf matching probability tends to one, and if $\eta_{\mathsf P,5/2}(p)\leq10^{-4}$
the full matching probability is at least $1-16\eta_{\mathsf P,5/2}(p)$, uniformly in the
height and for the infinite tree.
The following rapidly decaying law lets us control the terms with $j\geq2$ using the heavier
preceding label. We estimate the terms at $j=0$ and $j=1$ separately.

\begin{corollary}[Double-exponential tail]\label{thm:double-exp}
For $D\geq5$ put $p_j=e^{-D^j}$ $(j\geq1)$ and
$p_0=1-\sum_{j\geq1}e^{-D^j}$. Then $p_0>\tfrac12$ and
\begin{equation}\label{eq:eta-double}
 \eta_{\mathsf P,5/2}(p)\leq4\,e^{-D(D-5/2)}<10^{-4} .
\end{equation}
Consequently the full matching probability is at least
$1-16\eta_{\mathsf P,5/2}(p)\geq1-64\,e^{-D(D-5/2)}>0.9996$, uniformly in the height and
for the infinite tree.
\end{corollary}

\begin{proof}
For $j\geq1$,
\[
 \frac{p_{j+1}}{p_j}=e^{-(D-1)D^j}\leq e^{-(D-1)D}\leq e^{-20}<\tfrac12 ,
\]
so $\sum_{j\geq\ell}p_j\leq2p_\ell$ for every $\ell\geq1$.
In particular
$\sum_{j\geq1}p_j\leq2e^{-D}<\tfrac1{10}$, whence $p_0>\tfrac12$.
For the first two terms, $s_0,s_1>p_0>\tfrac12$, so
$s_0^{-5/2},s_1^{-5/2}<2^{5/2}<6$. Since
$1-s_0=\sum_{k\geq2}p_k$ and $1-s_1=\sum_{k\geq3}p_k$, and $p_0\leq1$,
\[
 p_0\frac{1-s_0}{s_0^{5/2}}\leq6\sum_{k\geq2}p_k\leq12e^{-D^2},
 \qquad
 p_1\frac{1-s_1}{s_1^{5/2}}\leq6e^{-D}\sum_{k\geq3}p_k\leq12e^{-(D+D^3)} .
\]
For $j\geq2$, $s_j\geq p_{j-1}$ and $1-s_j\leq1$, so
\[
 p_j\frac{1-s_j}{s_j^{5/2}}\leq\frac{p_j}{p_{j-1}^{5/2}}
 =\exp\bigl(-(D-\tfrac52)D^{j-1}\bigr),
\]
and the ratio of consecutive terms of the right side is at most
$\exp(-(D-\tfrac52)(D-1)D)\leq e^{-50}<\tfrac12$. Put $\beta:=e^{-D(D-5/2)}$.
The geometric tail bound gives
$\sum_{j\geq2}p_j\frac{1-s_j}{s_j^{5/2}}\leq2e^{-(D-5/2)D}=2\beta$.

To compare the first two bounds with $\beta$, note that
$e^3>\sum_{k=0}^{8}3^k/k!=\tfrac{89641}{4480}>20$.
It follows that $e^{(5/2)D}\geq e^{12}>20^4>12$. Since
$D^3-D^2+\tfrac72D>\tfrac52D$, we have
\begin{align*}
 \frac{12e^{-D^2}}{\beta}&=12e^{-(5/2)D}\leq1,\\
 \frac{12e^{-(D+D^3)}}{\beta}&=12e^{-(D^3-D^2+(7/2)D)}\leq1.
\end{align*}
Each of the first two contributions is therefore at most $\beta$. Adding the
three contributions, $\eta_{\mathsf P,5/2}(p)\leq\beta+\beta+2\beta=4\beta$. Finally
$D(D-\tfrac52)\geq\tfrac{25}{2}>12$, so
\[
 4\beta<4e^{-12}<\frac{4}{20^4}=\frac1{40000}<10^{-4},
\]
and $64\beta<64e^{-12}<64/20^4=0.0004$, giving the stated bounds.
\end{proof}

\section{Matching and quasi-isometry}

We now combine the matching theorem with the chain and shape constructions to prove the
quasi-isometry classification for offspring distributions supported on $\set{0,1,2}$.
Matching supplies compatible labels, and the local comparison and transfer results turn the local
quasi-isometries into a global quasi-isometry.

\subsection{The potential of the quantised law}

The quantised labels of \cref{sec:quantisation} are independent and identically distributed on
$\bbN_0$ with the law $p^{(D)}$. We compare them on the path $\mathsf P$ of \cref{sec:integer},
so two labels are compatible when they differ by at most one. \Cref{thm:matching} then applies
once the potential $\eta_{\mathsf P,5/2}(p^{(D)})$ falls below $10^{-4}$. By
\cref{thm:quantised-law}, the class $0$ carries all but $\theta_1^{D-1}$ of the mass, and class $j$
carries at most $\theta_1^{D^j-1}$ for $j\geq1$. Thus the masses decay doubly exponentially in
the class index. For large $D$, each class away from the origin has a heavier preceding class.
The estimate below uses this compatible mass to control the denominator in the potential.

\begin{lemma}[Chain potential bound]\label{thm:eta-bound}
  For every $\theta_1\in(0,1)$ there exists $D_0=D_0(\theta_1)\geq5$ such that for every integer
  $D\geq D_0$ the law $p^{(D)}$ of \cref{thm:quantised-law} satisfies
  \begin{equation}\label{eq:eta-bound}
    \eta_{\mathsf P,5/2}\bigl(p^{(D)}\bigr)\leq16\,\theta_1^{D(D-5/2)}\leq10^{-4}.
  \end{equation}
\end{lemma}

\begin{proof}
  Write $a=\theta_1$. Since $0<a<1$, we may choose $D_0=D_0(a)\geq5$ so that for all
  $D\geq D_0$,
  \begin{equation}\label{eq:d0-conditions}
    a^{D-1}\leq\tfrac1{10},\qquad
    a^{D^2-D}\leq\tfrac12,\qquad
    6\,a^{(5/2)D-1}\leq1,\qquad
    16\,a^{D(D-5/2)}\leq10^{-4}.
  \end{equation}
  Fix an integer $D\geq D_0$ and write $p_j=p^{(D)}_j$ and
  $s_j=p_{j-1}+p_j+p_{j+1}$, with $p_{-1}=0$. By \cref{thm:quantised-law}, $p_j>0$ and hence
  $s_j>0$ for every $j\geq0$. We write $t_j=p_j(1-s_j)/s_j^{5/2}$ for the terms of
  \cref{eq:potential} with $G=\mathsf P$, and estimate the first two terms separately from the
  remaining tail.

  The first condition gives $p_0=1-a^{D-1}\geq\tfrac9{10}$, so $s_0,s_1\geq p_0>\tfrac12$ and
  $s_0^{-5/2},s_1^{-5/2}<2^{5/2}<6$. The tails $1-s_0$ and $1-s_1$ are the tails of the geometric
  law at $D^2$ and $D^3$, namely $a^{D^2-1}$ and $a^{D^3-1}$. Both are at most $a^{D^2-1}$, so
  $t_0+t_1\leq12\,a^{D^2-1}$. Writing
  $a^{D^2-1}=a^{(5/2)D-1}a^{D(D-5/2)}$ and applying the third condition gives
  $t_0+t_1\leq2\,a^{D(D-5/2)}$.

  For $j\geq2$, \cref{thm:quantised-law} and the second condition give
  $s_j\geq p_{j-1}=a^{D^{j-1}-1}\bigl(1-a^{D^j-D^{j-1}}\bigr)\geq\tfrac12\,a^{D^{j-1}-1}$, so,
  dropping $1-s_j\leq1$,
  \[
    t_j\leq a^{D^j-1}\bigl(\tfrac12\,a^{D^{j-1}-1}\bigr)^{-5/2}
    =2^{5/2}a^{3/2}\,a^{D^{j-1}(D-5/2)}
    \leq6\,a^{D^{j-1}(D-5/2)}.
  \]
  Consecutive bounds have ratio $a^{(D-5/2)(D^j-D^{j-1})}$, whose exponent is at least $D^2-D$
  because $D\geq5$ gives $D-5/2\geq1$ and $D^j-D^{j-1}\geq D^2-D$ for $j\geq2$. The second
  condition therefore bounds the ratio by $\tfrac12$, so summing the geometric series gives
  $\sum_{j\geq2}t_j\leq12\,a^{D(D-5/2)}$. Adding the two contributions and then using the
  fourth condition, we obtain
  \[
    \eta_{\mathsf P,5/2}(p^{(D)})\leq14\,a^{D(D-5/2)}
    \leq16\,a^{D(D-5/2)}\leq10^{-4}.\qedhere
  \]
\end{proof}

\subsection{Proof of \texorpdfstring{\cref{thm:twovalue}}{the two-value theorem}}
\label{sec:proof-main}

The encoding of \cref{sec:encoding} turns each of the two realisations into a label field on $\bbB$,
and by \cref{thm:quantised-law} the quantised fields are independent of each other and, within
each realisation, independent and identically distributed with the law $p^{(D)}$.
\Cref{thm:eta-bound} shows that this law satisfies the hypothesis of \cref{thm:matching}.
Thus, with probability close to one, a single automorphism of $\bbB$ aligns the two fields to
within one class. By \cref{thm:level}, this alignment makes the chain lengths comparable up to
the factor $D^2$, and \cref{thm:transfer} turns this comparison into a quasi-isometry of the two
realisations. Taking a union over integer scales $D$ then gives the almost sure statement.

\begin{restate}{thm:twovalue}
  Let $\theta$ be an offspring distribution with $\theta_1+\theta_2=1$, and let
  $\cT(\omega),\cT(\omega')$ be two independent Galton--Watson trees with offspring distribution
  $\theta$. Then, almost surely, there is a root-preserving quasi-isometry
  $\cT(\omega)\to\cT(\omega')$.
  Moreover, if $0<\theta_1<1$, there are constants $D_0=D_0(\theta_1)$ and $C=C(\theta_1)$ such
  that for every integer $D\geq D_0$,
  \[
    \bbP\bigl(\text{there is no root-preserving }(D^2+3)\text{-quasi-isometry }
    \cT(\omega)\to\cT(\omega')\bigr)
    \leq C\,\theta_1^{D(D-5/2)}.
  \]
\end{restate}

\begin{proof}[Proof of \cref{thm:twovalue}]
  If $\theta_1\in\set{0,1}$ the tree is deterministic and the statement is trivial, so assume
  $0<\theta_1<1$. Conditioned on the almost sure event $\omega,\omega'\in\Omega_0$, both
  encodings $(\bbB,\lambda(\cdot,\omega))$ and $(\bbB,\lambda(\cdot,\omega'))$ of
  \cref{sec:encoding} are defined, with independent label fields, each i.i.d.\ with the
  geometric law of \cref{thm:geometric}.

  Fix an integer $D\geq D_0(\theta_1)$, with $D_0$ as in \cref{thm:eta-bound}, and consider the
  quantised label fields $\ell_D(\lambda(w,\omega))$ and $\ell_D(\lambda(w,\omega'))$, which are
  i.i.d.\ with law $p^{(D)}$ on $\bbN_0$. Define the event
  \[
    E_D=\set*{\exists\,\pi\in\Aut(\bbB)\ \forall w\in\bbB\colon
      \abs[\big]{\ell_D\bigl(\lambda(\pi w,\omega)\bigr)-\ell_D\bigl(\lambda(w,\omega')\bigr)}
      \leq1}.
  \]
  By \cref{thm:eta-bound} the potential condition of \cref{thm:matching} holds for the path graph
  on $\bbN_0$, so
  \begin{equation}\label{eq:ed-bound}
    \bbP(E_D)\geq1-16\,\eta_{\mathsf P,5/2}\bigl(p^{(D)}\bigr)\geq1-256\,\theta_1^{D(D-5/2)}.
  \end{equation}

  Conditioned on $E_D$, fix a matching automorphism $\pi$. By \cref{thm:level},
  \[
    D^{-2}\lambda(w,\omega')\leq\lambda(\pi w,\omega)\leq D^{2}\lambda(w,\omega')
    \qquad\text{for all }w\in\bbB.
  \]
  By \cref{thm:isometry} the tree encoded by $w\mapsto\lambda(\pi w,\omega)$ is isometric to
  $\cT(\omega)$, and \cref{thm:transfer}, applied with constant $D^2$, yields a
  root-preserving $(D^2+3)$-quasi-isometry $\cT(\omega)\to\cT(\omega')$. This proves
  \cref{eq:rate} with $C=256$.

  Finally, every $E_D$ is contained in the event that a root-preserving quasi-isometry
  $\cT(\omega)\to\cT(\omega')$ exists. Since
  \[
    \bbP\Bigl(\bigcup_{\substack{D\in\bbN\\D\geq D_0}}E_D\Bigr)
    \geq\sup_{\substack{D\in\bbN\\D\geq D_0}}
    \left(1-256\,\theta_1^{D(D-5/2)}\right)=1,
  \]
  such a map exists almost surely.
\end{proof}

\begin{remark}[The zero--one law is not needed]\label{rem:no-zero-one}
  The almost sure statement does not use a zero--one law. Every $E_D$ is an explicit event
  contained in the event that the trees admit a root-preserving quasi-isometry, and
  $\bbP(E_D)\to1$ by
  \cref{eq:ed-bound}. The argument requires neither monotonicity of $E_D$ in $D$ nor the reverse
  containment. Indeed, a quasi-isometry of the two trees need not arise from a matching
  automorphism.

  A zero--one law is nevertheless available. Represent the tree by one of the i.i.d.\ fields used
  in this part: the chain lengths $\lambda(w)$, $w\in\bbB$, of \cref{thm:geometric} when
  $\theta_0=0$, or the shapes $S(w)$, $w\in\bbB$, of \cref{thm:shape-iid} when $\theta_0>0$ and
  the tree is conditioned on having infinite diameter. Exchanging the values at finitely many
  vertices affects the constants of any quasi-isometry but not whether a quasi-isometry exists.
  Applying the Hewitt--Savage zero--one law~\cite[Theorem~2.5.4]{Durrett2019} then gives the
  desired zero--one law for $\set{\cT(\omega)\simeq\cT(\omega')}$.
\end{remark}

\subsection{Universality in the chain regime}

Two realisations drawn from different two-value laws carry label fields with different geometric
parameters, so their quantised classes have different probabilities.
\Cref{thm:chain-coupling} re-quantises the second law, using boundary randomisation with one
auxiliary uniform variable per vertex, so that its classes carry the probabilities $p^{(D)}_k$
of the first. The two quantised fields then have the same law, and we can apply the matching
argument of \cref{sec:proof-main}. The classes of the second law meet at the cut points $D_k'$
rather than at the powers $D^k$. These cut points are comparable to $D^k$ up to a factor depending
only on the two laws, which introduces an extra bounded factor in the distortion.

\begin{proposition}[Universality in the chain regime]\label{thm:cross-law}
  Let $\theta$ and $\theta'$ be offspring distributions with $\theta_1+\theta_2=1$ and
  $\theta_1'+\theta_2'=1$, where $\theta_1,\theta_1'\in(0,1)$, and let $\cT(\omega)$ and
  $\cT'(\omega')$ be independent Galton--Watson trees with these offspring distributions. Then,
  almost surely, there is a root-preserving quasi-isometry $\cT(\omega)\to\cT'(\omega')$.
  Moreover, with $\gamma$ as in \cref{thm:chain-coupling} and
  \[
    \gamma_*=\frac{\max(1,\gamma)}{\min(1,\gamma/2)},
  \]
  there is $D_1=D_1(\theta_1,\theta_1')$ such that for every integer $D\geq D_1$,
  \begin{equation}\label{eq:rate-cross}
    \bbP\bigl(\text{there is no root-preserving }(\gamma_*D^2+3)\text{-quasi-isometry }
    \cT(\omega)\to\cT'(\omega')\bigr)
    \leq256\,\theta_1^{D(D-5/2)}.
  \end{equation}
\end{proposition}

\begin{proof}
  Let $D_1$ be the least integer $D\geq D_0(\theta_1)$ with $\gamma(D-1)\geq1$, where $D_0$ is
  the threshold of \cref{thm:eta-bound}, and fix an integer $D\geq D_1$.

  We condition on the almost sure event $\omega,\omega'\in\Omega_0$, so that both encodings of
  \cref{sec:encoding} are defined. Write
  $\lambda(\cdot,\omega)$ and $\lambda'(\cdot,\omega')$ for the two label fields, which by
  \cref{thm:geometric} are independent of each other and i.i.d.\ within themselves, with the
  geometric laws of parameters $\theta_2$ and $\theta_2'$. Enlarge the probability space by a
  family $(U(w))_{w\in\bbB}$ of independent uniform variables on $[0,1)$, independent of both
  fields, and quantise the two fields by
  \[
    x(w)=\ell_D\bigl(\lambda(w,\omega)\bigr),
    \qquad
    y(w)=\ell'\bigl(\lambda'(w,\omega'),U(w)\bigr),
  \]
  with $\ell'$ as in \cref{thm:chain-coupling}, which applies since $\gamma(D-1)\geq1$. By
  \cref{thm:quantised-law} the field $x$ is i.i.d.\ with law $p^{(D)}$ on $\bbN_0$. The pairs
  $\bigl(\lambda'(w,\omega'),U(w)\bigr)$ are i.i.d.\ in $w$, so by
  \cref{thm:chain-coupling}\labelcref{it:chain-coupling-law} the field $y$ is i.i.d.\ with the same
  law $p^{(D)}$, and the two fields are independent. Consider the event
  \[
    E_D=\set*{\exists\,\pi\in\Aut(\bbB)\ \forall w\in\bbB\colon
      \abs[\big]{x(\pi w)-y(w)}\leq1}.
  \]
  Since $D\geq D_0(\theta_1)$, \cref{thm:eta-bound} bounds the potential of $p^{(D)}$, so
  \cref{thm:matching} applies to the path $\mathsf P$ on $\bbN_0$ and gives
  \begin{equation}\label{eq:ed-cross}
    \bbP(E_D)\geq1-16\,\eta_{\mathsf P,5/2}(p^{(D)})\geq1-256\,\theta_1^{D(D-5/2)}.
  \end{equation}

  Conditioned on $E_D$, fix a matching automorphism $\pi$. Let $w\in\bbB$ and put $k=x(\pi w)$
  and $k'=y(w)$, so that $\abs{k-k'}\leq1$. By the definition in \cref{sec:quantisation}, we have
  $D^k\leq\lambda(\pi w,\omega)<D^{k+1}$, and by
  \cref{thm:chain-coupling}\labelcref{it:chain-coupling-classes} we have
  $D_{k'}'\leq\lambda'(w,\omega')\leq D_{k'+1}'$. Hence
  \cref{thm:chain-coupling}\labelcref{it:chain-coupling-ratio} gives
  \begin{align*}
    \frac{\lambda'(w,\omega')}{\lambda(\pi w,\omega)}
      &\leq\frac{D_{k'+1}'}{D^{k}}\leq\max(1,\gamma)\,D^{k'+1-k}\leq\gamma_*D^2,\\
    \frac{\lambda'(w,\omega')}{\lambda(\pi w,\omega)}
      &\geq\frac{D_{k'}'}{D^{k+1}}\geq\min(1,\gamma/2)\,D^{k'-k-1}\geq\gamma_*^{-1}D^{-2},
  \end{align*}
  that is,
  \[
    \bigl(\gamma_*D^2\bigr)^{-1}\lambda'(w,\omega')
    \leq\lambda(\pi w,\omega)
    \leq\gamma_*D^2\,\lambda'(w,\omega')
    \qquad\text{for all }w\in\bbB.
  \]
  By \cref{thm:isometry} the tree encoded by $w\mapsto\lambda(\pi w,\omega)$ is isometric to
  $\cT(\omega)$, and \cref{thm:transfer}, applied with the constant $\gamma_*D^2$, yields a
  root-preserving $(\gamma_*D^2+3)$-quasi-isometry $\cT(\omega)\to\cT'(\omega')$. Thus, in the
  enlarged probability space, $E_D$ is contained in the event that the original trees admit such a
  quasi-isometry. This latter event depends only on the two trees, so its probability under
  their original product law equals its probability in the enlarged space. The bound
  \cref{eq:ed-cross} therefore proves \cref{eq:rate-cross}.

  As in \cref{sec:proof-main}, letting $D\to\infty$ through integers gives the almost sure
  conclusion, since the failure bound in \cref{eq:rate-cross} tends to zero.
\end{proof}

\subsection{The bushy regime}
We now consider $\theta$ supported on $\set{0,1,2}$ and supercritical with
$\theta_0>0$. Recall the shapes $S(w)$ of \cref{def:shape}, the net $\cR_D$ with the map
$\rep_D$, the label graph $G_D$ of \cref{def:shape-net}, and the law $\bbP^*$ conditioned on
survival from \cref{sec:shapes}. We write $\mu$ for the shape law under $\bbP^*$. We bound the
potential of the shape labels and construct a common label law for two offspring distributions.
Matching these labels and applying glued transfer will give universality within the regime.

\subsubsection{The shape label law and its potential}\label{sec:shape-eta}

The label of $w$ is the class $\rep_D(S(w))$ of its shape. By \cref{thm:shape-iid} the labels
are independent and identically distributed with the pushforward law $\mu_D=\mu\circ\rep_D^{-1}$ on
the vertex set $\cR_D$ of $G_D$, and two of them are compatible when they are equal or adjacent in
$G_D$. Thus \cref{thm:matching} applies once the potential of $\mu_D$ is small.

To bound the potential, we shrink large shapes to smaller members of the same support. Their
representatives are compatible, and each smaller shape supplies a lower bound for the compatible
mass. We then group the potential sum by the sizes of the original shapes. For large $D$, the
exponential size tail of \cref{thm:shape-mass} outweighs the growth of the reciprocal compatible
mass, giving a geometric sum. This is the same comparison as in \cref{thm:eta-bound}, where
the geometric chain law gives double-exponential decay in the class index because the cut
points grow as powers of $D$.

We record the shrinking first, in the form that both this bound and the coupling of
\cref{sec:shape-coupling} use, in which the target support may belong to a second offspring
distribution.

\begin{lemma}[Shrinking a shape into a support]\label{thm:shape-shrink}
  Let $\theta'$ be supported on $\set{0,1,2}$ and supercritical with $\theta_0'>0$, and let $\mu'$
  be its shape law.
  \begin{enumerate}[(i),ref=(\roman*)]
    \item\label{it:shape-fix} Every shape $\sigma$ admits a shape $\sigma^\circ\in\supp\mu'$, its
      \emph{support fix}, with $\abs{\sigma^\circ}\leq2\abs\sigma$ and a $2$-marked quasi-isometry
      $\sigma\to\sigma^\circ$, and $\sigma^\circ=\sigma$ whenever $\sigma\in\supp\mu'$.
    \item\label{it:shape-shrink} Let $D\geq2$ be an integer and put $s=\floor{\sqrt{D/2592}}$. If
      $s\geq1$, then every shape $\sigma$ of size $n$ admits a shape $\sigma_0\in\supp\mu'$ with
      $\abs{\sigma_0}\leq16(n/s+1)$ and a $D$-marked quasi-isometry $\sigma\to\sigma_0$.
  \end{enumerate}
  We fix the choices in these constructions, so that $\sigma\mapsto\sigma^\circ$ and
  $\sigma\mapsto\sigma_0$ are well-defined maps $\cS\to\supp\mu'$.
\end{lemma}

\begin{proof}
  For \cref{it:shape-fix}, the shape law of \cref{thm:shape-iid} has full support when
  $\theta_1'>0$, so we take $\sigma^\circ=\sigma$. When $\theta_1'=0$, its support consists
  precisely of the shapes with a bush at every neck vertex other than the exit and with zero or
  two children at every bush vertex. Attach a one-vertex bush at each bare neck vertex other
  than the exit and a leaf at each bush vertex with one child. This adds at most one vertex per
  original vertex, giving $\sigma^\circ\in\supp\mu'$ with
  $\abs{\sigma^\circ}\leq2\abs\sigma$, and leaves supported shapes unchanged. The inclusion
  preserves distances and both marks and has $1$-dense image, so it is $2$-marked. Collapsing
  the added leaves to their attachment vertices preserves both marks and changes distances by
  at most $2$, giving a $2$-marked map in the reverse direction as well.

  For \cref{it:shape-shrink}, apply the cut from the proof of \cref{thm:dilution} at scale $s$.
  The projection onto the tree $T$ of parts is $2s$-marked, with the entry part as root and the
  exit part as mark. There are $k\leq n/s+1$ parts, each containing at most $2s-1$ vertices.
  A connected part with $m$ vertices in a tree with at most two children per vertex has at most
  $m+1$ outgoing child edges, so every vertex of $T$ has at most $2s$ children. We turn $T$ into
  a supported shape in three steps.

  First, replace each vertex with $d>2$ children by a spine of $d-1$ vertices, attaching its
  children in their order so that every spine vertex has two children. Send each old vertex
  to the first vertex of its spine, or to itself if it is unchanged, and take the images of the
  root and mark as the marks of the resulting binary tree $B$. Collapsing the spines back to
  their original vertices does not increase distances, while each old edge becomes a path of
  length at most $2s$. Every added vertex is within $2s$ of the image, so this map is
  $4s$-marked. At most $k-1$ vertices are added, giving $\abs B\leq2k$.

  Second, append a new exit as a child of the marked vertex and take the path from the root to
  this exit as the neck. If the mark had two children, attach them through one new vertex to
  form a single bush. Every other neck vertex already has at most one subtree off the neck,
  so the result is a shape with at most $\abs B+2$ vertices. The map keeping the original
  vertices changes distances by at most one and has $1$-dense image. It preserves the root
  and sends the old mark within one step of the new exit, so it is $2$-marked and hence
  $6$-marked.

  Third, let $\sigma_0$ be the support fix of this shape. By \cref{it:shape-fix}, this step is
  $2$-marked and at most doubles the size, so
  \[
    \abs{\sigma_0}\leq2(\abs B+2)\leq4k+4\leq16k\leq16(n/s+1).
  \]
  By \cref{thm:shape-net}, composing the projection and the three maps gives a marked
  quasi-isometry with constant
  \[
    3^3\cdot2s\cdot4s\cdot6\cdot2=2592s^2\leq D.
  \]
  The deterministic cut and fixed child order for the spines and leaf attachments make these
  assignments well-defined.
\end{proof}

\begin{lemma}[Shape potential bound]\label{thm:shape-eta}
  There are $c=c(\theta)>0$ and $D_0=D_0(\theta)$ such that for every integer
  $D\geq D_0$,
  \begin{equation}\label{eq:shape-eta}
    \eta_{G_D,5/2}(\mu_D)\leq e^{-cD^{2}}\leq10^{-4}.
  \end{equation}
  Moreover the class of the one-vertex shape carries $\mu_D$-mass at least
  $1-e^{-cD^{2}}\geq\tfrac12$.
\end{lemma}

\begin{proof}
  Write $c$, $C$ and $n_0$ for the constants of \cref{thm:shape-mass} and $v_0$ for the class of
  the one-vertex shape, which is the first shape in the enumeration and therefore lies in $\cR_D$.
  A shape of size at most $D^2$ has diameter at most $D^2$, and the map sending every one of its
  vertices to the one vertex is then a $D$-marked quasi-isometry, its only nontrivial requirement
  being $D^{-1}d(x,y)-D\leq0$. The one-vertex shape is the earliest representative, so
  $\rep_D(\sigma)=v_0$ for every such $\sigma$, and every class other than $v_0$ has all its
  $\mu$-positive members of size exceeding $D^2$. For $D^2\geq n_0$,
  \[
    \mu_D(v_0)\geq\mu\bigl(\abs S\leq D^2\bigr)\geq1-e^{-cD^2}\geq\tfrac12,
  \]
  which is the final claim.

  Write $b(v)=\mu_D(B_{G_D}(v,1))$ for the compatible mass at $v$. We bound it using a smaller
  shape in a compatible class. Let
  $\sigma\in\supp\mu$ have size $n$, put $s=\floor{\sqrt{D/2592}}$, which is at least $1$ for
  $D\geq D_0$, and let $\sigma_0\in\supp\mu$ be the shape that
  \cref{thm:shape-shrink}\labelcref{it:shape-shrink} produces from $\sigma$ with $\theta$ itself as
  the target law, so that $\abs{\sigma_0}\leq16(n/s+1)$ and $\sigma\to\sigma_0$ is $D$-marked.
  Passing to the representatives composes a $3D^2$-marked quasi-inverse of
  $\sigma\to\rep_D(\sigma)$ with the $D$-marked map $\sigma\to\sigma_0$ and the $D$-marked map
  $\sigma_0\to\rep_D(\sigma_0)$, so the composition and quasi-inverse bounds of
  \cref{thm:shape-net} make $\rep_D(\sigma)\to\rep_D(\sigma_0)$ a
  $3\cdot(3\cdot3D^2\cdot D)\cdot D=27D^4$-marked quasi-isometry. The edge condition of
  \cref{def:shape-net} is set at exactly this value, so the two representatives are equal or
  adjacent in $G_D$, and consequently
  \begin{equation}\label{eq:shape-domination}
    b\bigl(\rep_D(\sigma)\bigr)
    \geq\mu_D\bigl(\rep_D(\sigma_0)\bigr)
    \geq\mu(\sigma_0)
    \geq e^{-16C(n/s+1)}
    \geq e^{-C_1(nD^{-1/2}+1)},
  \end{equation}
  where $C_1:=16C\bigl(2\sqrt{2592}+1\bigr)$,
  using \cref{thm:shape-mass}\labelcref{it:shape-mass-point} and $s^{-1}\leq2\sqrt{2592/D}$. The
  lower bound depends on the size $n$ of the original shape through $nD^{-1/2}$.

  We now sum the terms of \cref{eq:potential}. Since $b(v_0)\geq\mu_D(v_0)\geq\tfrac12$ we have
  $b(v_0)^{-5/2}\leq2^{5/2}=\sqrt{32}\leq6$, so the $v_0$-term is at most
  $(1-b(v_0))\,b(v_0)^{-5/2}\leq6\,e^{-cD^2}$. For the remaining classes, every
  $\mu$-positive member has size exceeding $D^2$, so grouping the sum by members and applying
  \cref{eq:shape-domination} to each,
  \begin{align*}
    \sum_{\substack{v\neq v_0\\\mu_D(v)>0}}\mu_D(v)\,\frac{1-b(v)}{b(v)^{5/2}}
    &\leq\sum_{\substack{\sigma\in\supp\mu\\\abs\sigma>D^2}}
      \mu(\sigma)\,e^{\frac52C_1(\abs\sigma D^{-1/2}+1)}\\
    &\leq e^{3C_1}\sum_{n>D^2}e^{-cn}\,e^{\frac52C_1nD^{-1/2}}
    \leq e^{3C_1}\sum_{n>D^2}e^{-cn/2},
  \end{align*}
  using \cref{thm:shape-mass}\labelcref{it:shape-mass-tail} and, for $D\geq D_0$,
  $\tfrac52C_1D^{-1/2}\leq\tfrac c2$. The final sum is at most
  $(1-e^{-c/2})^{-1}e^{-cD^2/2}$, so, with $C_2=e^{3C_1}(1-e^{-c/2})^{-1}$,
  \[
    \eta_{G_D,5/2}(\mu_D)\leq6\,e^{-cD^2}+C_2\,e^{-cD^2/2}.
  \]
  Dividing by $e^{-cD^2/4}$ and using $e^{-y}\leq y^{-1}$ for $y>0$, the right-hand side is at
  most $e^{-cD^2/4}$ as soon as
  \[
    6e^{-3cD^2/4}+C_2e^{-cD^2/4}\leq\frac{8}{cD^2}+\frac{4C_2}{cD^2}\leq1,
  \]
  that is, as soon as $cD^2\geq8+4C_2$. Renaming $c/4$ to $c$ and increasing $D_0$ so that
  $e^{-cD^2}\leq10^{-4}$ for every $D\geq D_0$ gives \cref{eq:shape-eta}.
\end{proof}

\subsubsection{Shape classes across two laws}\label{sec:shape-coupling}

Two bushy laws generally have different shape laws. To apply \cref{thm:matching}, we give their
label fields a common law. \Cref{thm:chain-coupling} achieves this for two chain laws by
re-cutting the second law at thresholds chosen to reproduce the class probabilities of the
first. For shapes, we use a coupling of the two laws. Each shape draws a partner from the other
law, and we record the pair with the same coordinate order in both labellings. The resulting
labels have a common law on pairs. We require comparisons between compatible pair labels to
give marked maps between the original shapes at a scale polynomial in $D$, while the label law
has a small potential.

\begin{lemma}[Shape coupling]\label{thm:shape-coupling}
  Let $\theta$ and $\theta'$ be supercritical, supported on $\set{0,1,2}$, with
  $\theta_0,\theta_0'>0$, and let $\mu$ and $\mu'$ be their shape laws. There are
  $D_1=D_1(\theta,\theta')$ and $c_3=c_3(\theta,\theta')>0$ such that for every integer
  $D\geq D_1$ there exist a countable graph $\mathsf Q$, a law $\mu_{\mathsf Q}$ on its
  vertices, and maps $\ell,\ell'\colon\cS\times[0,1)\to V(\mathsf Q)$ with the following properties.
  \begin{enumerate}[(i),ref=(\roman*)]
    \item\label{it:shape-coupling-law} With $S\sim\mu$, $S'\sim\mu'$, and $U,U'$ uniform on
      $[0,1)$, all independent, the labels $\ell(S,U)$ and $\ell'(S',U')$ both have law
      $\mu_{\mathsf Q}$.
    \item\label{it:shape-coupling-qi} If
      $d_{\mathsf Q}\bigl(\ell(\sigma,u),\ell'(\sigma',u')\bigr)\leq1$, then there is a
      $972D^{4}$-marked quasi-isometry $\sigma\to\sigma'$, and the same holds for two labels
      of one law.
    \item\label{it:shape-coupling-eta}
      $\eta_{\mathsf Q,5/2}(\mu_{\mathsf Q})\leq e^{-c_3D^2}\leq10^{-4}$.
  \end{enumerate}
\end{lemma}

\begin{proof}
  We construct $\mu_{\mathsf Q}$ as a coupling of $\mu$ and $\mu'$. Shapes of size at most
  $D^2$ can all be compared to one another, while larger shapes will be paired with smaller
  shapes in the other support. We use an alternating series to assign masses to these pairs so
  that the resulting measure has the prescribed marginals on all larger shapes. We then couple
  the remaining mass arbitrarily among shapes of size at most $D^2$.

  We put
  \[
    A_D=\set{\sigma\in\cS\colon\abs\sigma\leq D^2},
    \qquad T=\supp\mu\setminus A_D,
    \qquad T'=\supp\mu'\setminus A_D.
  \]
  We call the shapes in $A_D$ \emph{small} and those in $T$ and $T'$ \emph{tail shapes}
  for the respective laws. By \cref{thm:shape-mass}, we may choose $c_4,C_\star>0$,
  depending only on the two laws, such that, for sufficiently large $D$,
  \[
    \mu(T)\leq e^{-c_4D^2},
    \qquad \mu'(T')\leq e^{-c_4D^2},
    \qquad \mu(\sigma)\geq e^{-C_\star\abs\sigma}
    \quad(\sigma\in\supp\mu),
  \]
  with the same point mass bound for $\mu'$ on its support.
  Every shape in $A_D$ maps to the one-vertex shape by a $D$-marked quasi-isometry, as in the
  proof of \cref{thm:shape-eta}. Composing one such map with a $3D^2$-marked quasi-inverse
  shows that any two shapes in $A_D$ admit $9D^3$-marked maps in both directions.

  To choose partners for the tail shapes, set $E=\floor{\sqrt{D/3}}$ and
  $s=\floor{\sqrt{E/2592}}$. For $\sigma\in T$, let $\shrink(\sigma)\in\supp\mu'$ be
  the shape supplied by \cref{thm:shape-shrink}\labelcref{it:shape-shrink} at scale $E$,
  with $\mu'$ as the target shape law. Interchanging the laws gives a second map
  $\shrink'\colon T'\to\supp\mu$. We have
  $\abs{\shrink(\sigma)}\leq16(\abs\sigma/s+1)$, and the same size bound holds for
  $\shrink'$. The map $\sigma\to\shrink(\sigma)$ is $E$-marked and has a
  $3E^2$-marked quasi-inverse by \cref{thm:shape-net}. Our choice of $E$ therefore gives
  $D$-marked maps in both directions, and likewise for $\shrink'$. Either shrinking
  may send a tail shape to a small shape.

  The total $\mu$-mass of shapes assigned to any one partner is at most half of its
  $\mu'$-mass. More precisely, after increasing the lower threshold for $D$,
  \begin{equation}\label{eq:capacity}
    \mu\bigl(\shrink^{-1}(\tau)\bigr)\leq\tfrac12\mu'(\tau)
    \qquad(\tau\in\supp\mu'),
  \end{equation}
  and the analogous bound holds for $\shrink'$ with the laws interchanged.
  Indeed, if the fibre is nonempty and $n>D^2$ is its smallest source size, then
  $\abs\tau\leq16(n/s+1)$. The tail and point mass bounds of \cref{thm:shape-mass} give
  \[
    \frac{\mu\bigl(\shrink^{-1}(\tau)\bigr)}{\mu'(\tau)}
    \leq e^{-c_4n+C_\star\abs\tau}
    \leq e^{-(c_4-16C_\star/s)n+16C_\star}
    \leq\tfrac12
  \]
  uniformly in $\tau$ for large $D$, since $s\to\infty$ and $n>D^2$. The bound is immediate
  for an empty fibre.

  We now choose how much mass to give each pair. A tail shape $\sigma\in T$ occurs both
  in its own pair with $\shrink(\sigma)$ and as the partner of shapes in $T'$. The mass
  assigned to its own pair must therefore be $\mu(\sigma)$ minus the total mass assigned
  through those other pairs, and the corresponding condition holds on $T'$. To solve these
  equations, for $\sigma\in\supp\mu$ and $\tau\in\supp\mu'$ define
  \begin{align*}
    M_0(\sigma)&=\mu(\sigma),
    & M'_0(\tau)&=\mu'(\tau),\\
    M_{k+1}(\sigma)&=\sum_{\rho\in(\shrink')^{-1}(\sigma)}M'_k(\rho),
    & M'_{k+1}(\tau)&=\sum_{\rho\in\shrink^{-1}(\tau)}M_k(\rho).
  \end{align*}
  Thus $M_1(\sigma)$ is the $\mu'$-mass of shapes whose chosen partner is $\sigma$,
  and each further step repeats this operation with the laws interchanged.
  The capacity bound \cref{eq:capacity} and induction give
  $M_{k+1}(\sigma)\leq\tfrac12M_k(\sigma)$ and the same inequality for $M'_k$.
  For $\sigma\in T$ and $\tau\in T'$, we set
  \[
    \mathrm{out}(\sigma)=\sum_{k\geq0}(-1)^kM_k(\sigma),
    \qquad
    \mathrm{out}'(\tau)=\sum_{k\geq0}(-1)^kM'_k(\tau).
  \]
  The halving bounds imply absolute convergence. Since $M_k(\sigma)$ decreases with $k$,
  the alternating-series bounds give
  \[
    \tfrac12\mu(\sigma)
    \leq\mu(\sigma)-M_1(\sigma)
    \leq\mathrm{out}(\sigma)\leq\mu(\sigma),
  \]
  and the corresponding bounds hold for $\mathrm{out}'$ and $\mu'$.
  The double series over $k$ and $\rho\in(\shrink')^{-1}(\sigma)$ has sum of absolute values
  at most $2M_1(\sigma)$, so we may sum term by term. For $\sigma\in T$, this gives
  \[
    \sum_{\rho\in(\shrink')^{-1}(\sigma)}\mathrm{out}'(\rho)
    =\sum_{k\geq0}(-1)^kM_{k+1}(\sigma)
    =\mu(\sigma)-\mathrm{out}(\sigma).
  \]
  For $\sigma\in A_D\cap\supp\mu$, the sum on the left is at most
  $\mu'((\shrink')^{-1}(\sigma))\leq\tfrac12\mu(\sigma)$, using
  $\mathrm{out}'\leq\mu'$ and \cref{eq:capacity}. Interchanging the laws gives the
  corresponding conclusions on $\supp\mu'$.

  Let $\mu_{\mathsf Q}^{\mathrm{cas}}$ be the measure obtained by assigning weight
  $\mathrm{out}(\sigma)$ to $(\sigma,\shrink(\sigma))$ for each $\sigma\in T$, and
  weight $\mathrm{out}'(\tau)$ to $(\shrink'(\tau),\tau)$ for each $\tau\in T'$.
  We call these the \emph{cascade pairs}. The first marginal of
  $\mu_{\mathsf Q}^{\mathrm{cas}}$ equals $\mu$ on $T$ and is at most $\mu$ on $A_D$, while
  its second marginal satisfies the same conditions for $\mu'$ and $T'$. Since
  $\mathrm{out}\leq\mu$ and $\mathrm{out}'\leq\mu'$, its total mass satisfies
  \[
    \eps\coloneqq \mu_{\mathsf Q}^{\mathrm{cas}}(\cS\times\cS)
    \leq\mu(T)+\mu'(T')\leq2e^{-c_4D^2}.
  \]
  Subtracting its marginals from $\mu$ and $\mu'$ leaves two nonnegative measures
  supported on $A_D$, each of mass $1-\eps$. Let $\mu_{\mathsf Q}^{\mathrm{cl}}$ be any
  coupling of these remainders, and put
  $\mu_{\mathsf Q}=\mu_{\mathsf Q}^{\mathrm{cas}}+\mu_{\mathsf Q}^{\mathrm{cl}}$.
  We call the pairs in $\supp\mu_{\mathsf Q}^{\mathrm{cl}}\subseteq A_D\times A_D$ the
  \emph{closure pairs}. This constructs a coupling of $\mu$ and $\mu'$ in which the
  components of every pair in $\supp\mu_{\mathsf Q}$ admit $9D^3$-marked quasi-isometries
  in both directions.

  We take $\supp\mu_{\mathsf Q}$ as the vertex set of $\mathsf Q$, joining distinct pairs when each
  component of one admits a $27D^4$-marked quasi-isometry to each component of the other
  in both directions. For $\sigma\in\supp\mu$, partition $[0,1)$ into half-open intervals
  of lengths $\mu_{\mathsf Q}(\sigma,\tau)/\mu(\sigma)$, indexed by the $\tau$ with
  $\mu_{\mathsf Q}(\sigma,\tau)>0$, and set $\ell(\sigma,u)=(\sigma,\tau)$ on the
  corresponding interval. Define $\ell'$ similarly using the conditional masses
  $\mu_{\mathsf Q}(\tau,\sigma')/\mu'(\sigma')$ and the labels $(\tau,\sigma')$. By
  construction, $\ell(S,U)$ and $\ell'(S',U')$ both have law $\mu_{\mathsf Q}$, proving
  \cref{it:shape-coupling-law}. For inputs outside the respective supports, extend each
  labelling by first taking the support fix from
  \cref{thm:shape-shrink}\labelcref{it:shape-fix}.

  Components selected from equal labels admit $9D^3$-marked maps, and those selected from
  adjacent labels admit $27D^4$-marked maps. To compare the original input shapes, compose
  such a map with the inclusion into the source support fix and the map collapsing the leaves added
  by the target support fix. Both are $2$-marked, as shown in the proof of
  \cref{thm:shape-shrink}, so the composition is
  $3^2\cdot2\cdot27D^4\cdot2=972D^4$-marked. The graph condition allows either component
  of each label to be selected, so this also applies to two labels of the same law.
  This proves \cref{it:shape-coupling-qi}.

  It remains to bound the potential. For $v\in\supp\mu_{\mathsf Q}$, write
  $b(v)=\mu_{\mathsf Q}(B_{\mathsf Q}(v,1))$. Every closure pair is equal or adjacent to every other
  closure pair. The same holds between a cascade pair whose chosen partner lies in $A_D$
  and any closure pair: each component comparison uses a $9D^3$-marked map between two
  shapes in $A_D$, possibly composed with one $D$-marked shrinking map, giving a constant
  at most $27D^4$. For all these vertices,
  \[
    b(v)\geq \mu_{\mathsf Q}^{\mathrm{cl}}(\cS\times\cS)=1-\eps,
    \qquad
    \frac{1-b(v)}{b(v)^{5/2}}\leq\frac{\eps}{(1-\eps)^{5/2}}\leq6\eps
  \]
  once $\eps\leq\tfrac12$. Their total contribution to the potential is therefore at most
  $6\eps\leq12e^{-c_4D^2}$.

  Now let $v=(\sigma,\tau)$ with $\sigma\in T$ and $\tau=\shrink(\sigma)\in T'$.
  This pair is equal or adjacent to $(\shrink'(\tau),\tau)$, since the component
  comparisons use at most two $D$-marked shrinking maps and hence have constant at most
  $3D^2\leq27D^4$. The latter pair has $\mu_{\mathsf Q}$-mass at least
  $\mathrm{out}'(\tau)\geq\tfrac12\mu'(\tau)\geq\tfrac12e^{-C_\star\abs\tau}$.
  Consequently,
  \[
    \frac{1-b(v)}{b(v)^{5/2}}
    \leq2^{5/2}e^{\frac52C_\star\abs\tau}
    \leq2^{5/2}e^{40C_\star(\abs\sigma/s+1)}.
  \]
  The same argument applies to pairs generated from $T'$ whose partners lie in $T$.
  We multiply by the assigned weights, use $\mathrm{out}\leq\mu$ and
  $\mathrm{out}'\leq\mu'$, and sum over both families of pairs. Grouping their sources by
  size gives
  \begin{align*}
    \eta_{\mathsf Q,5/2}(\mu_{\mathsf Q})
    &\leq12e^{-c_4D^2}+C_1\sum_{n>D^2}e^{-c_4n+40C_\star n/s}\\
    &\leq12e^{-c_4D^2}+C_2e^{-c_4D^2/2},
  \end{align*}
  where $C_1,C_2$ depend only on the two laws. Here we used
  $\mu(\set{\sigma\in\cS\colon\abs\sigma=n})\leq e^{-c_4n}$ and its $\mu'$ counterpart
  from \cref{thm:shape-mass}, and increased the threshold for $D$ so that
  $40C_\star/s\leq c_4/2$.
  Taking $c_3=c_4/4$ and choosing $D_1$ large enough
  to meet all the preceding conditions and absorb the fixed prefactors gives
  $\eta_{\mathsf Q,5/2}(\mu_{\mathsf Q})\leq e^{-c_3D^2}\leq10^{-4}$.
\end{proof}

\subsubsection{Universality in the bushy regime}

\Cref{thm:shape-coupling} labels the two realisations by fields that are independent and identically
distributed with one law on a graph of small potential. \Cref{thm:matching} matches the two fields
along a single automorphism of $\bbB$, and a matched pair of labels supplies a marked
quasi-isometry between the shapes sitting at the two vertices it identifies.
\Cref{thm:glued-transfer} glues these into a quasi-isometry of the two assemblies.

\begin{proposition}[Universality in the bushy regime]\label{thm:hairy}
  Let $\theta$ and $\theta'$ be supercritical, supported on $\set{0,1,2}$, with
  $\theta_0,\theta_0'>0$, and let $\cT,\cT'$ be independent Galton--Watson trees with these
  offspring distributions, conditioned on having infinite diameter. Then almost surely there
  is a root-preserving quasi-isometry $\cT\to\cT'$. Moreover, there is $D_2=D_2(\theta,\theta')$
  such that for every integer $D\geq D_2$, with
  $\mathsf Q$, $\mu_{\mathsf Q}$ and $c_3$ the label graph, the label law and the constant of
  \cref{thm:shape-coupling} at scale $D$,
  \begin{equation}\label{eq:hairy-rate}
    \bbP\bigl(\text{there is no root-preserving }8\cdot972^2D^{8}\text{-quasi-isometry }\cT\to\cT'\bigr)
    \leq16\,\eta_{\mathsf Q,5/2}(\mu_{\mathsf Q})\leq16\,e^{-c_3D^2}.
  \end{equation}
\end{proposition}

\begin{proof}
  We work with the product of the two laws conditioned on having infinite diameter and enlarge
  it by families $(U(w))_{w\in\bbB}$ and $(U'(w))_{w\in\bbB}$ of uniform variables on $[0,1)$,
  all independent of one another and of the two trees. Let $D_2$ be at least the $D_1$ of
  \cref{thm:shape-coupling} and fix
  an integer $D\geq D_2$, with $\mathsf Q$, $\mu_{\mathsf Q}$, $\ell$, $\ell'$ as there. By
  \cref{thm:shape-iid}, $\cT$ is isometric to the assembly of an i.i.d.\ shape sequence
  $(S(w))_{w\in\bbB}$ with law $\mu$ and $\cT'$ to that of $(S'(w))_{w\in\bbB}$ with law
  $\mu'$, the two sequences independent. The label fields
  \[
    x(w)=\ell\bigl(S(w),U(w)\bigr),
    \qquad
    y(w)=\ell'\bigl(S'(w),U'(w)\bigr)
  \]
  are therefore independent of each other and i.i.d.\ with the common law $\mu_{\mathsf Q}$, by
  \cref{thm:shape-coupling}\labelcref{it:shape-coupling-law}. By
  \cref{thm:shape-coupling}\labelcref{it:shape-coupling-eta} the potential condition of
  \cref{thm:matching} holds for $\mathsf Q$, so with probability at least
  $1-16\,\eta_{\mathsf Q,5/2}(\mu_{\mathsf Q})$ there is
  $\pi\in\Aut(\bbB)$ with $d_{\mathsf Q}\bigl(x(\pi w),y(w)\bigr)\leq1$ for all $w\in\bbB$.
  When this occurs, \cref{thm:shape-coupling}\labelcref{it:shape-coupling-qi} provides a
  $972D^4$-marked quasi-isometry $S(\pi w)\to S'(w)$ for every $w$. The assembly of
  $(S(\pi w))_{w\in\bbB}$ is isometric to that of $(S(w))_{w\in\bbB}$, since carrying the copy of
  $S(\pi w)$ at $w$ identically to its copy at $\pi w$ respects the gluing, $\pi$ mapping
  the children of a word to the children of its image. Hence \cref{thm:glued-transfer},
  applied with $K=972D^4$, yields a root-preserving $8\cdot972^2D^8$-quasi-isometry
  $\cT\to\cT'$. The identifications with the assemblies send each root to the entry of the
  root copy, and the relabelling isometry preserves this entry because $\pi$ fixes the root.
  The event of successful label matching is therefore contained in the event that such a
  quasi-isometry exists. This latter event depends only on the original trees, so its probability
  is unchanged on passing to their conditioned product law. This proves \cref{eq:hairy-rate}, its
  second inequality being \cref{thm:shape-coupling}\labelcref{it:shape-coupling-eta}. As at the
  end of the proof of \cref{thm:twovalue}, letting $D\to\infty$ through integers in
  \cref{eq:hairy-rate} gives the almost sure statement.
\end{proof}

\subsection{Quasi-isometry classification for offspring supported on
\texorpdfstring{$\set{0,1,2}$}{\{0,1,2\}}}\label{sec:trichotomy-simple}

The pieces of this part assemble into the classification restricted to offspring distributions
supported on $\set{0,1,2}$.

\begin{theorem}[Quasi-isometry classification for offspring supported on $\set{0,1,2}$]
  \label{thm:trichotomy-simple}
  Let $\theta=(\theta_0,\theta_1,\theta_2)$ and
  $\theta'=(\theta_0',\theta_1',\theta_2')$ be offspring distributions, each either
  supercritical or concentrated at $1$, and let $\cT,\cT'$ be independent Galton--Watson trees
  with these offspring distributions, conditioned on having infinite diameter. The classes of
  \cref{thm:trichotomy} take
  the form
  \begin{itemize}
    \item[(R)] $\theta_1=1$, in which $\cT$ is the ray.
    \item[(F)] $\theta_0=\theta_1=0$, so that $\theta_2=1$ and $\cT$ is the binary tree.
    \item[$(\mathrm{C}_{\langle1\rangle})$] $\theta_0=0<\theta_1<1$, the two-value family
      $\theta_1+\theta_2=1$ of \cref{thm:twovalue}.
    \item[(B)] $\theta_0>0$, where supercriticality implies $\theta_2>\theta_0$.
  \end{itemize}
  If $\theta$ and $\theta'$ belong to the same class, then almost surely there is a
  root-preserving quasi-isometry $\cT\to\cT'$. If they belong to different classes, then almost
  surely $\cT\not\simeq\cT'$. In particular, almost every realisation from a law outside class (R) is
  not quasi-isometric to the ray.
\end{theorem}

This is the case $J=2$ of \cref{thm:trichotomy}, with all four infinite regimes present. The
branching semigroup does not further split the chain regime. Indeed, a chain-regime law supported on
$\set{0,1,2}$ gives weight to the arity $2$ and no other arity above one, so
$\Lambda_\theta=\langle1\rangle=\bbN_0$. The chain regime is therefore the single class
$(\mathrm{C}_{\langle1\rangle})$ here. The positive conclusions follow from the deterministic ray and binary tree cases and from
\cref{thm:cross-law,thm:hairy}. We now separate the different classes by their geometry.

\paragraph*{\bf Rays and bushes} Almost surely, a realisation in class (F) or $(\mathrm{C}_{\langle1\rangle})$ has every vertex
below its first branch point on a geodesic line, and the remaining vertices lie in a bounded set
near the root. A realisation in class (B) almost surely carries bushes of unbounded depth, each
cut off from the rest of the tree by a single vertex. A vertex at the bottom of a deep bush is
far from every line. The first proposition below separates these cases. The second uses three
rays from a common vertex to separate all three non-ray classes from the ray. We then use long
chains to separate the chain class from the binary tree, giving \cref{thm:converse}.

\paragraph*{\bf Lines} A \emph{line} (or \emph{bi-infinite geodesic}) in a tree $T$ is a subset isometric to $\bbZ$, equivalently the union of two rays
from a common vertex leaving it through distinct neighbours. A \emph{bush of depth $h$} attached
at a vertex $c$ is a finite component $P$ of $T\setminus\set c$ with
$P\subseteq B_T(c,h)$ containing a vertex at distance $h$ from $c$. Removing a vertex of $P$
leaves components that lie inside $P$, hence are bounded, together with one component containing
$c$. Thus no vertex of a bush lies on a line, and bushes of unbounded depth contain vertices
arbitrarily far from every line. 

\paragraph*{\bf Separation} Both separation proofs below use the following elementary form of geodesic stability in a tree, whose proof is standard but which we include for completeness. A more general statement, for quasi-geodesics in a Gromov hyperbolic space, is~\cite[Theorem~III.H.1.7]{BridsonHaefliger1999}.

\begin{lemma}[Stability in a tree]\label{thm:tree-stability}
  Let $T$ and $T'$ be trees, let $\psi\colon T\to T'$ be a $D$-quasi-isometry, and let
  $x,y\in T$. Every vertex of the geodesic $[\psi(x),\psi(y)]$ lies within distance $D$ of the
  image under $\psi$ of the geodesic $[x,y]$.
\end{lemma}

\begin{proof}
  Write $x=z_0,z_1,\dots,z_n=y$ for the vertices of $[x,y]$ in order. Consecutive ones are
  adjacent, so $d'(\psi(z_i),\psi(z_{i+1}))\leq D+D=2D$. Let $m$ lie on $[\psi(x),\psi(y)]$.
  If $m=\psi(z_i)$ for some $i$, the conclusion is immediate. Otherwise every image avoids $m$.
  Removing $m$ separates $\psi(x)$ from $\psi(y)$, and $\psi(z_0)$ and $\psi(z_n)$ therefore lie
  in different components of $T'\setminus\set m$, so some consecutive pair
  $\psi(z_i),\psi(z_{i+1})$ does too. The geodesic joining that pair passes through $m$, whence
  $d'(\psi(z_i),m)+d'(m,\psi(z_{i+1}))\leq2D$ and one of the two summands is at most $D$, as
  required.
\end{proof}

\begin{proposition}[Bushes against lines]\label{thm:bush-separation}
  Let $T$ and $T'$ be trees. Suppose that $T$ has bushes of unbounded depth
  and that every vertex of $T'$ lies within distance $r$ of a line of $T'$. Then $T\not\simeq T'$.
\end{proposition}

\begin{proof}
  Suppose $\psi\colon T\to T'$ is a $D$-quasi-isometry. Let $P$ be a bush of depth
  $h>2D^2+Dr$ at a vertex $c$, with $x\in P$ at distance $h$ from $c$. Write $w=\psi(c)$ and
  $y=\psi(x)$.

  Choose a line of $T'$ within distance $r$ of $y$ and a vertex $y'$ on it with
  $d'(y,y')\leq r$. Of the two rays of the line from $y'$, choose one for which
  $y'\in[w,z]$ for every vertex $z$ on the ray. Take such a $z$ with
  $d'(y',z)>Dh+2D$, and choose $p\in T$ with $d'(\psi(p),z)\leq D$. Then $p\notin P$, since
  otherwise $d'(w,\psi(p))\leq Dh+D$, whereas
  $d'(w,\psi(p))\geq d'(w,z)-D\geq d'(y',z)-D>Dh+D$.

  Moreover, $y'\in[w,\psi(p)]$: the common initial segment of $[w,z]$ and
  $[w,\psi(p)]$ ends within distance $D$ of $z$, beyond $y'$. Hence
  \cref{thm:tree-stability} gives $q\in[c,p]$ with $d'(\psi(q),y')\leq D$, so
  $d'(\psi(q),\psi(x))\leq D+r$ and $d(q,x)\leq D(D+r)+D^2=2D^2+Dr$. On the other hand $c$
  and $p$ lie outside the component $P$, so $[c,p]$ misses $P$ and
  $d(q,x)=d(q,c)+h\geq h>2D^2+Dr$, a contradiction.
\end{proof}

\begin{proposition}[Three rays against the ray]\label{thm:three-rays}
  Let $T$ be a tree containing three rays that pairwise meet only in their common initial vertex.
  Then $T$ is not quasi-isometric to the ray $\bbN_0$.
\end{proposition}

\begin{proof}
  Suppose $\psi\colon T\to\bbN_0$ is a $D$-quasi-isometry and let $v$ be the common initial vertex
  of three such rays. Fix an integer $r>2D^2$ and let $a_1,a_2,a_3$ be the vertices at distance
  $r$ from $v$ on the three rays.
  Each of the three images is at least $\psi(v)$ or at most $\psi(v)$, so two of them lie weakly on
  one side of $\psi(v)$, and of those two the nearer to $\psi(v)$ lies on the geodesic joining
  $\psi(v)$ to the further. Relabelling, $\psi(a_1)$ lies on $[\psi(v),\psi(a_2)]$.

  By \cref{thm:tree-stability} there is $q$ on the geodesic $[v,a_2]$ with
  $\abs{\psi(q)-\psi(a_1)}\leq D$, so
  $d(q,a_1)\leq D^2+D^2$. But $[v,a_2]$ is an initial segment of the second ray, which meets the
  first ray only at $v$, so every path from $a_1$ to a vertex of $[v,a_2]$ passes through $v$ and
  $d(a_1,q)\geq d(a_1,v)=r>2D^2$, a contradiction.
\end{proof}

\paragraph*{\bf The supercritical regimes} We now verify these geometric hypotheses almost surely in the three supercritical regimes.
Class (B) supplies the bushes of unbounded depth, classes (F) and
$(\mathrm{C}_{\langle1\rangle})$ supply the nearby lines, and all three contain the required rays.

\begin{lemma}[The regimes against the obstructions]\label{thm:regime-obstructions}
  Let $\theta$ be supercritical and supported on $\set{0,1,2}$, and let $\cT$ be a Galton--Watson
  tree with offspring distribution $\theta$, conditioned on having infinite diameter.
  \begin{enumerate}[(i),ref=(\roman*)]
    \item\label{it:obstr-rays} Almost surely $\cT$ contains three rays that pairwise meet only in
      their common initial vertex.
    \item\label{it:obstr-lines} If $\theta_0=0$, then almost surely there is a finite $r$ such that
      every vertex of $\cT$ lies within distance $r$ of a line.
    \item\label{it:obstr-bushes} If $\theta_0>0$, then almost surely $\cT$ has bushes of unbounded
      depth.
  \end{enumerate}
\end{lemma}

\begin{proof}
  Supercriticality forces $\theta_1+2\theta_2>1$ and hence $\theta_2>0$. Write $\cT^*$ for $\cT$
  when $\theta_0=0$ and for its skeleton when $\theta_0>0$. In both cases $\cT^*$ is a
  Galton--Watson tree whose offspring distribution is supported on $\set{1,2}$, by
  \cref{thm:harris}\labelcref{it:harris-skeleton} in the second case, so every vertex of $\cT^*$
  has an infinite line of descent, and $\cT^*$ is a subtree of $\cT$ carrying the induced metric.

  For \cref{it:obstr-rays}, we exhibit the three rays in $\cT^*$. If $\theta_0=\theta_1=0$, then
  $\cT^*=\cT=\bbB$ and a child $v$ of the root carries them, one descending from each of the two
  children of $v$ and one climbing to the root and descending into its other child. Otherwise the
  offspring distribution of $\cT^*$ gives positive weight to $1$, so \cref{thm:chains} applies to
  it: the chain of $\troot$ ends at a branch point $u$, and the chain starting at the second
  child of $u$ ends at a branch point $v$ strictly below $u$. Descending from each of the two
  children of $v$ gives two rays leaving $v$ through distinct neighbours, and climbing from $v$
  to $u$ and descending into the other child of $u$ gives a third ray, leaving $v$ through its
  parent. The three pairwise meet only at $v$.

  For \cref{it:obstr-lines}, let $\theta_0=0$, so that $\cT=\cT^*$ and every subtree of $\cT$ is
  infinite. If $\theta_1=0$, then $\cT=\bbB$ and every vertex lies on a line, so $r=0$ works.
  Otherwise $0<\theta_1<1$. Let $b$ be the branch point ending the chain of $\troot$, at distance
  $\lambda(\troot)-1$ from the root. The two subtrees below the children of $b$ are infinite, so
  $b$ lies on a line. Let $v$ lie strictly below $b$, say in the subtree below one child of $b$.
  The subtree below $v$ is infinite, and the component of $\cT\setminus\set v$ containing the
  parent of $v$ contains the subtree below the other child of $b$ and is infinite as well, so $v$
  lies on a line too. Every remaining vertex lies on the chain of $\troot$ strictly above $b$, at
  distance at most $\lambda(\troot)-1$ from $b$. So $r=\lambda(\troot)-1$ works, and it is finite
  almost surely by \cref{thm:geometric}.

  For \cref{it:obstr-bushes}, let $\theta_0>0$. A bush attached at a neck vertex $c$ is joined
  to the rest of $\cT$ by the single edge at $c$, so it is a component of
  $\cT\setminus\set c$, and it is finite. Its depth is the largest distance from $c$ to one of
  its vertices. By
  \cref{thm:harris}\labelcref{it:harris-bushes}, conditioned on the skeleton, the
  neck decorations are independent. Each is nonempty with probability
  $2\theta_0/\widetilde\theta_1>0$ and, when nonempty, is a Galton--Watson tree with conjugate
  law $\theta^\dagger=(\theta_2,\theta_1,\theta_0)$. Fix $h\in\bbN$. Since
  $\theta^\dagger_0=\theta_2>0$ and $\theta^\dagger_2=\theta_0>0$, the
  complete binary tree of height $h$ has positive $\theta^\dagger$-probability, so a given neck
  vertex carries a bush of depth exceeding $h$ with a probability $p_h>0$ that does not depend on
  the vertex. \Cref{thm:geometric} applied to the skeleton makes the chain lengths
  $\widetilde\lambda(w)$, $w\in\bbB$, independent and geometric, so almost surely infinitely many
  of them exceed $1$ and $\cT$ has infinitely many neck vertices. Conditioned on the skeleton, the
  decorations are independent, so the second Borel--Cantelli lemma gives, almost surely, infinitely
  many bushes of depth exceeding $h$. Intersecting over $h\in\bbN$ proves
  \cref{it:obstr-bushes}, as required.
\end{proof}

\paragraph*{\bf The binary tree and the chain regime} The remaining separation within the two-value family is between the chain regime and the binary
tree. Long chains create arbitrarily large balls with only linear growth, whereas balls in the
binary tree grow exponentially.

\begin{proposition}[Bottleneck obstruction]\label{thm:bottleneck}
  Let $\cT$ be a tree such that for every $\ell\in\bbN$ there is a vertex $v_\ell$ whose ball
  $B_{\cT}(v_\ell,\ell)$ is isometric to a segment of $\bbZ$. Then $\cT$ is not quasi-isometric
  to $\bbB$.
\end{proposition}

\begin{proof}
  Assume for contradiction that $\psi\colon\cT\to\bbB$ is a $D$-quasi-isometry. Fix
  $\ell\in\bbN$, put $v=v_\ell$, and let $r\in\bbN$. We first show that
  \begin{equation}\label{eq:ball-cover}
    B_{\bbB}\bigl(\psi(v),r\bigr)
    \subseteq\bigcup_{x\in B_{\cT}(v,\,D(r+2D))}B_{\bbB}\bigl(\psi(x),D\bigr).
  \end{equation}
  Indeed, let $s\in B_{\bbB}(\psi(v),r)$. Coarse density provides $x\in\cT$ with
  $d(\psi(x),s)\leq D$, so $d(\psi(x),\psi(v))\leq r+D$. The lower quasi-isometry bound then gives
  $d(x,v)\leq D(r+2D)$, which proves \cref{eq:ball-cover}.

  Now choose $r=\floor{\ell/D-2D}$, which is positive for all sufficiently large $\ell$. Then
  $D(r+2D)\leq\ell$. Since $B_{\cT}(v,\ell)$ is a segment, the index set on the right-hand side of
  \cref{eq:ball-cover} has at most $2\ell+1$ elements, and every ball of radius $D$ in $\bbB$ has
  at most $3\cdot2^D$ vertices. On the other hand, $B_{\bbB}(\psi(v),r)$ contains the complete
  subtree of depth $r$ below $\psi(v)$ and hence at least $2^r$ vertices. Therefore
  \[
    2^{\floor{\ell/D-2D}}\leq(2\ell+1)\cdot3\cdot2^D,
  \]
  which fails for all sufficiently large $\ell$.
\end{proof}

\begin{theorem}[Separation from the binary tree]\label{thm:converse}
  Let $\theta$ satisfy $\theta_1+\theta_2=1$ with $0<\theta_1<1$, and let $\cT$ be a
  Galton--Watson tree with offspring distribution $\theta$. Then, almost surely, $\cT$ is
  not quasi-isometric to the binary tree.
\end{theorem}

\begin{proof}
  Since $0<\theta_1<1$, the labels $\lambda(w)$, $w\in\bbB$, are independent with unbounded
  support by \cref{thm:geometric}. Thus almost surely $\sup_w\lambda(w)=\infty$, and $\cT$ contains
  chains of single-descendant vertices of arbitrary length. Choose a middle vertex $v_\ell$ of
  a chain of length at least $2\ell+2$. Its radius-$\ell$ ball stays inside the chain and is
  isometric to a segment of $\bbZ$. Thus \cref{thm:bottleneck} shows that almost surely
  $\cT$ is not quasi-isometric to $\bbB$.
\end{proof}

\begin{proof}[Proof of \cref{thm:trichotomy-simple}]
  Suppose first that $\theta$ and $\theta'$ belong to the same class. In class (R) both
  realisations are the ray. In class (F) the support forces $\theta_2=\theta_2'=1$ and both
  realisations are the binary tree. In class $(\mathrm{C}_{\langle1\rangle})$ both are
  two-value laws with $\theta_1,\theta_1'\in(0,1)$, every realisation has infinite diameter, and
  \cref{thm:cross-law} gives $\cT\simeq\cT'$ almost surely. In class (B) it is \cref{thm:hairy}.

  Suppose next that they belong to different classes. If one law belongs to class (R), its
  realisation is the ray, the other law is supercritical, and by
  \cref{thm:regime-obstructions}\labelcref{it:obstr-rays} its realisation almost surely contains
  three rays meeting only in their common initial vertex, so \cref{thm:three-rays} gives
  $\cT\not\simeq\cT'$ almost surely. If one belongs to class (F) and the other to
  $(\mathrm{C}_{\langle1\rangle})$, then the realisation in the chain class is almost surely not
  quasi-isometric to the binary tree, by \cref{thm:converse}. In the two remaining cases one law
  belongs to class (B) and the other to class (F) or $(\mathrm{C}_{\langle1\rangle})$.
  Exchanging the names if necessary, we may take $\theta_0>0$ and $\theta_0'=0$. By
  \cref{thm:regime-obstructions}\labelcref{it:obstr-bushes,it:obstr-lines}, almost surely $\cT$
  has bushes of unbounded depth and there is a finite $r$ such that every vertex of $\cT'$ lies
  within distance $r$ of a line. Hence \cref{thm:bush-separation} gives
  $\cT\not\simeq\cT'$ almost surely.

  The first different-class case also gives the stated separation from the ray.
\end{proof}

\clearpage
\part{Full classification of Galton--Watson trees with finitely supported offspring distributions}\label{part:full-classification}

We now extend the classification of \cref{part:binary-classification} to all finitely supported
offspring laws. We again compare trees by matching labelled pieces and assembling local comparisons
into quasi-isometries. The new difficulty is that
branching can vary both within each tree and between the two offspring laws. Encoding this variable
branching in a binary tree introduces dependence between labels, which requires a matching theorem
beyond the independent setting of \cref{part:binary-classification}. We first prepare the geometric
descriptions needed for this encoding, then prove the Markov matching theorem. We combine these
ingredients with geometric separation arguments to complete the classification. We then use the
classification to prove mutual quasi-isometric embeddability across all supercritical regimes.

\section{Matching classes for arbitrary offspring laws}\label{sec:general-classes}
We first extend the shape construction of \cref{sec:shapes} to finitely supported offspring laws.
For offspring supported on $\set{0,1,2}$, every split has exactly two surviving children and
no dying children. The reduced skeleton is therefore the binary tree, and a split carries no
bush. Only the shape law had to be quantised and coupled.

With larger offspring numbers, splits may have different arities and may carry bushes of their
own. The shape ending at a split can therefore depend on its arity. We construct labels whose
law is independent of that arity, while retaining the geometric comparisons needed for
quasi-isometry. For two offspring laws, the labels must also have a common law, while the arity laws
may remain different. We will compare them by replacing each split by a finite binary tree
with one leaf for each child.

We first describe the general reduced skeleton and develop the shape relabelling in the bushy
regime. In the chain regime there are no bushes, and the neck length is already independent
of the arity. We recall this simpler label law in \cref{sec:general-chain} before constructing
binary encodings for both regimes.

Throughout, $\theta$ is supercritical with support in $\set{0,\dots,J}$ and $\theta_J>0$, so
that $J$ is its largest supported offspring number. We write
$f(s)=\sum_{j=0}^{J}\theta_js^j$ for its generating function and use the probability convention
of \cref{sec:gw-trees}: throughout the skeleton and shape constructions, $\bbP$ denotes the
survival-conditioned law, together with any independent relabelling variables.

\subsection{The reduced skeleton}

As in \cref{sec:shape-harris}, a vertex of $\cT(\omega)$ is a \emph{skeleton} vertex if its
subtree is infinite, a \emph{split} if at least two of its children are skeleton vertices, and a
\emph{neck} vertex if it is a skeleton vertex with exactly one skeleton child. The finite subtrees
rooted at the remaining children are the \emph{bushes}, and a split may carry bushes of its own.
Conditioned on $\cT(\omega)$ having infinite diameter, the \emph{reduced skeleton} is the tree
of splits. Its root is the first split at or below the original root. The children of a split in
the reduced skeleton are the first splits reached by descending from each of its skeleton children.

Starting at the original root or at a skeleton child of a split, we follow the unique skeleton
child until we reach a split. We define the \emph{neck length} $m$ to be the number of vertices
visited, including the terminating split. Thus there are $m-1$ preceding neck vertices, and
$m=1$ when the starting vertex is already a split. We call the descent from the original root the
\emph{initial neck}.

The decomposition is classical, due to Harris~\cite{Harris1948} and
Sevastyanov~\cite{Sevastyanov1951}. See
Athreya--Ney~\cite[Chapter~I.12]{AthreyaNey1972} and
Lyons--Peres~\cite[Proposition~5.28]{LyonsPeres2016} for standard accounts. As in Part 2, we rederive its
arithmetic to be fairly self-contained and because the chain and shape estimates use the
constants explicitly.

\begin{proposition}[Harris decomposition, general support]\label{thm:harris-general}
  Let $q\in[0,1)$ be the extinction probability, the root of $s=f(s)$. For a
  Galton--Watson tree with offspring distribution $\theta$, conditioned on having infinite
  diameter, the following hold:
  \begin{enumerate}[(i),ref=(\roman*)]
    \item\label{it:harris-general-skeleton} The skeleton is a Galton--Watson tree with offspring
      distribution $\widetilde\theta$ given by the Harris--Sevastyanov transform
      \begin{equation}\label{eq:hs-transform}
        \widetilde f(s)=\frac{f\bigl(q+(1-q)s\bigr)-q}{1-q},
      \end{equation}
      which satisfies $\widetilde\theta_0=\widetilde f(0)=0$ and
      $\widetilde\theta_1=\widetilde f'(0)=f'(q)$. The skeleton never dies out, is supercritical,
      and $0\leq\widetilde\theta_1<1$.
    \item\label{it:harris-general-reduced} The reduced skeleton is a Galton--Watson tree with
      offspring distribution $\widetilde\nu$ on $\set{2,\dots,J}$,
      \begin{equation}\label{eq:reduced-law}
        \widetilde\nu_k=\frac{\widetilde\theta_k}{1-\widetilde\theta_1}
        \qquad(2\leq k\leq J),
      \end{equation}
      and the neck lengths, including the initial neck, are independent of one another and of
      the reduced skeleton, with geometric law
      $\bbP(m=r)=\widetilde\theta_1^{\,r-1}(1-\widetilde\theta_1)$ for $r\geq1$. When
      $\widetilde\theta_1=0$, every neck has length one almost surely.
    \item\label{it:harris-general-bushes} If $q>0$, then, conditioned on the skeleton, the
      decorations at distinct vertices are independent. A skeleton vertex with $k$ surviving
      children and total offspring $j$ carries $j-k$ bushes. Conditioned on these counts, the
      bushes are independent Galton--Watson trees with the conjugate law
      \begin{equation}\label{eq:conjugate-general}
        \theta^\dagger_j=\theta_jq^{\,j-1},
      \end{equation}
      which is subcritical with mean $f'(q)$. When $q=0$, there are no bushes.
  \end{enumerate}
\end{proposition}

\begin{proof}
  Before conditioning, each child survives independently with probability $1-q$, so the number of
  surviving children has generating function
  \begin{equation}\label{eq:survive-coeff}
    f\bigl(q+(1-q)s\bigr)=\sum_{k=0}^{J}c_k s^k,
    \qquad
    c_k=\sum_{j=k}^{J}\theta_j\binom{j}{k}(1-q)^kq^{\,j-k},
  \end{equation}
  by the binomial theorem. Since $c_0=f(q)=q$, conditioning on survival gives
  $\widetilde\theta_0=0$ and $\widetilde\theta_k=c_k/(1-q)$ for $k\geq1$, which is
  \cref{eq:hs-transform}. Moreover $c_1=(1-q)f'(q)$, so
  $\widetilde\theta_1=f'(q)<1$, where the inequality follows from supercriticality and the
  convexity of $f$. Thus the skeleton never dies out and has positive probability of splitting,
  so it is supercritical.

  A neck of length $r$ consists of $r-1$ unary skeleton vertices followed by a split. Hence, for
  $r\geq1$ and $2\leq k\leq J$,
  \[
    \bbP(m=r,\ \text{terminating arity}=k)
    =\widetilde\theta_1^{\,r-1}\widetilde\theta_k
    =\widetilde\theta_1^{\,r-1}(1-\widetilde\theta_1)\widetilde\nu_k.
  \]
  Summation and factorisation give the geometric neck law and its independence from the terminating
  arity. Applying the branching property below every skeleton child, and at the original root,
  proves \cref{it:harris-general-reduced}. If $\widetilde\theta_1=0$, the geometric law is
  concentrated at one.

  Finally, suppose that $q>0$. A child conditioned on extinction gives rise to a Galton--Watson tree with
  the offspring law in \cref{eq:conjugate-general}, since
  $\sum_j\theta_jq^{j-1}=f(q)/q=1$. Its mean is $f'(q)<1$. A vertex with $j$ children, of which $k$
  survive, has $j-k$ such bushes. The branching property and the conditional independence in the
  Harris--Sevastyanov decomposition prove \cref{it:harris-general-bushes}. If $q=0$, every child
  survives almost surely, so there are no bushes.
\end{proof}

\Cref{thm:harris-general} specialises to \cref{thm:harris} at $J=2$: there
$\widetilde f(s)=\widetilde\theta_1s+\widetilde\theta_2s^2$ with
$\widetilde\theta_1=f'(q)=\theta_1+2\theta_0$ and $\widetilde\theta_2=\theta_2-\theta_0$, the
reduced law $\widetilde\nu$ is the point mass at $2$, and a split has $j=k=2$, and hence no bushes.

When extinction is possible, the transform gives positive weight to every arity from $2$ to
$J$, even if the original offspring law has gaps in its support.

\begin{corollary}[The reduced law has full support]\label{thm:full-support}
  If $\theta_0>0$, then $\widetilde\nu_k>0$ for every $2\leq k\leq J$.
\end{corollary}

\begin{proof}
  Fix $2\leq k\leq J$. Among the terms of the coefficient $c_k$ of \cref{eq:survive-coeff}, the one with
  $j=J$ is $\theta_J\binom{J}{k}(1-q)^kq^{\,J-k}$, and each of its four factors is positive:
  $\theta_J>0$ because $J$ is the largest offspring number in the support, $\binom{J}{k}>0$ because
  $k\leq J$, $1-q>0$ by supercriticality, and $q>0$ because $q=f(q)\geq f(0)=\theta_0>0$. The other
  terms are nonnegative, so $\widetilde\theta_k>0$ and hence $\widetilde\nu_k>0$, as claimed.
\end{proof}

A vertex with $J$ children can have exactly $k$ surviving children for any $2\leq k\leq J$.
This is why no arity in that interval is missing from the reduced law. In the chain regime,
$q=0$ and every child survives, so the surviving arity is the offspring number itself and gaps
may remain in the reduced law.

\subsection{Shapes at general arity}\label{sec:general-shapes}

We now discuss the shape construction and relabelling and assume $\theta_0>0$ throughout, that is
that we are in the bushy regime.
The shape of \cref{def:shape} widens in two ways: a neck vertex may now carry several bushes, and
the terminating split carries bushes of its own.

\begin{definition}[Shapes, general support]\label{def:shape-general}
  A \emph{shape} is a pair $\sigma=(m,(\beta_1,\dots,\beta_m))$, where $m\geq1$ and each $\beta_i$
  is a finite, possibly empty, list of finite trees with offspring numbers in $\set{0,\dots,J}$.
  Its \emph{realisation} is the finite rooted tree obtained from a path $v_1\cdots v_m$ by
  attaching the roots of the trees in $\beta_i$ by an edge at $v_i$. The \emph{entry (vertex)}
  is $v_1$, the \emph{exit (vertex)} is $v_m$, and $\abs\sigma$ is the number of vertices of the
  realisation.
  We write $\cS$ for the set of all shapes.
\end{definition}

\paragraph*{\bf Shapes, arities, and assembly} We write $S(w)$ for the shape at a vertex $w$ of the reduced skeleton, $k(w)$ for the arity of the
split at $w$, and $m(w)$ for the neck length of $S(w)$. This neck starts at a skeleton child of the
preceding split, or at the original root for the first shape, and ends at $w$. The lists
$\beta_1,\dots,\beta_{m-1}$ record the bushes of the preceding neck vertices, and the \emph{exit
bouquet} $\beta_m$ records the bushes of the terminating split. A family of shapes indexed by the
reduced skeleton is \emph{assembled} by joining the exit of the shape at $w$ to the entries of the
shapes at its children by an edge. The root of the assembly is the entry of the first shape. The
space $\cS$ also contains shapes that have zero probability under the offspring law. 
We first compute how the bushes at a split depend on its arity.

\begin{lemma}[The joint law at a split]\label{thm:split-joint}
  Conditioned on being a split, the arity $k$ of a skeleton vertex and the number $r$ of its
  bushes have the joint law
  \begin{equation}\label{eq:split-joint}
    \bbP(k,r)=\frac{\theta_{k+r}\binom{k+r}{k}(1-q)^{k-1}q^{\,r}}{1-\widetilde\theta_1}
    \qquad(k\geq2,\ r\geq0).
  \end{equation}
\end{lemma}

\begin{proof}
  Each child founds a surviving subtree independently with probability $1-q$, so a vertex has
  total offspring $k+r$ with exactly $k$ children surviving with probability
  $\theta_{k+r}\binom{k+r}{k}(1-q)^kq^{\,r}$, the $j=k+r$ term of the coefficient $c_k$ in
  \cref{eq:survive-coeff}. Being a split is the event $k\geq2$, whose probability is
  $(1-q)(1-\widetilde\theta_1)$: the vertex survives with probability $1-q$, and conditioned on
  survival its number of surviving children is at least two with probability
  $1-\widetilde\theta_1$, by \cref{thm:harris-general}\labelcref{it:harris-general-skeleton}.
  Conditioning divides the masses by this probability, and one factor $1-q$ cancels, which is
  \cref{eq:split-joint}.
\end{proof}

\begin{remark}[No product structure]
  For $J\geq3$, the joint law is not a product of a law in $k$ and a law in $r$. Indeed, it gives
  positive mass to $(J,0)$ and $(2,J-2)$, by $\theta_J>0$ and $q>0$, but none to $(J,J-2)$, since
  $2J-2>J$. The three points are distinct, whereas a product law giving positive mass to the first
  two would give positive mass to the third.
\end{remark}

\subsection{The conditional laws and the relabelling}\label{sec:general-relabel}

For the later matching argument, we need labels whose law is independent of the arity. The shape
$S(w)$ itself records the exit bouquet, so \cref{thm:split-joint} makes the shape and arity
dependent when $J\geq3$. We use a randomised relabelling to ensure independence from the arity.

Redistributing the exit bushes among the children creates other dependencies. Sending them to
all children makes siblings share a component, while sending them to one designated child makes
its law differ from those of its siblings. We instead leave each shape in place and couple its
law, conditioned on the terminating arity, to one reference law. We use independent randomness
to sample a partner from this coupling and read the label from the partner. The coupling may
depend on the arity, but the partner has the same law for every arity. We construct the couplings
so that each shape and its partner admit a controlled marked comparison.

\begin{definition}[Conditional shape laws]\label{def:conditional-laws}
  For $k\in\supp\widetilde\nu$ let $\mu_k$ be the law of the shape at the root of the reduced
  skeleton conditioned on its terminating split having arity $k$.
  We write $\mu$ for the shape law before conditioning on the terminating arity.
\end{definition}

\begin{lemma}[Explicit form of the conditional laws]\label{thm:conditional-explicit}
  Fix $k\in\supp\widetilde\nu$, and let
  $S=(m,(\beta_1,\dots,\beta_m))$ be a random shape with law $\mu_k$. Then
  \[
    \bbP_{\mu_k}(m=\ell)=\widetilde\theta_1^{\,\ell-1}(1-\widetilde\theta_1)
    \qquad(\ell\geq1).
  \]
  Conditioned on $m=\ell$, the bouquet sizes are independent. For $1\leq i<\ell$,
  \[
    \bbP_{\mu_k}(\#\beta_i=r\mid m=\ell)
    =\frac{(r+1)\theta_{r+1}q^{\,r}}{\widetilde\theta_1}
    \qquad(0\leq r\leq J-1),
  \]
  while the exit bouquet satisfies
  \[
    \bbP_{\mu_k}(\#\beta_\ell=r\mid m=\ell)
    =\frac{\theta_{k+r}\binom{k+r}{k}q^{\,r}}
    {\displaystyle\sum_{s=0}^{J-k}\theta_{k+s}\binom{k+s}{k}q^{\,s}}
    \qquad(0\leq r\leq J-k).
  \]
  Conditioned further on these sizes, all bushes in the bouquets are independent
  Galton--Watson trees with offspring law $\theta^\dagger$.
  The shape law is the mixture
  $\mu=\sum_{k\in\supp\widetilde\nu}\widetilde\nu_k\,\mu_k$.
\end{lemma}

\begin{proof}
  By \cref{thm:harris-general}\labelcref{it:harris-general-reduced}, the neck length has the
  stated geometric law and is independent of the arity of the terminating split. At a unary
  skeleton vertex, the probability of one surviving child and $r$ dying children before
  conditioning is
  \[
    \theta_{r+1}\binom{r+1}{1}(1-q)q^{\,r}.
  \]
  Summing over $r$ gives $(1-q)f'(q)=(1-q)\widetilde\theta_1$, so conditioning on one surviving
  child gives the stated law of $\#\beta_i$. At the terminating split, conditioning the law in
  \cref{eq:split-joint} on arity $k$ cancels the factors independent of $r$ and gives the stated
  law of $\#\beta_\ell$. The conditional independence of the bouquet sizes and the conjugate
  laws of all their bushes follow from
  \cref{thm:harris-general}\labelcref{it:harris-general-bushes} and the branching property.

  Finally the arity of the terminating split has law $\widetilde\nu$, since the reduced skeleton
  is Galton--Watson with that offspring law, so averaging the $\mu_k$ against $\widetilde\nu$
  gives the shape law.
\end{proof}

\paragraph*{\bf Conditional independence} Since distinct shapes use disjoint necks and decorations, the branching property gives the
following conditional independence, which we record for later use.

\begin{lemma}[Conditional independence]\label{thm:conditional-iid}
  Conditioned on the arity field $(k(w))_w$ of the reduced skeleton, the shapes $S(w)$ are
  independent with $S(w)\sim\mu_{k(w)}$.
\end{lemma}

\paragraph*{\bf Tail and point mass bounds} We now check that the conditional laws $\mu_k$ and their mixture $\mu$ satisfy, uniformly in $k$,
the tail and point-mass bounds of \cref{thm:shape-mass}.

\begin{lemma}[Uniform mass bounds]\label{thm:mass-uniform}
  There are $c,C,n_0>0$, depending only on $\theta$, such that every
  $\lambda\in\{\mu\}\cup\{\mu_k\colon k\in\supp\widetilde\nu\}$ satisfies
  \begin{enumerate}[(i),ref=(\roman*)]
    \item $\lambda(\{\sigma\in\cS\colon\abs\sigma\geq n\})\leq e^{-cn}$ for all $n\geq n_0$,
      and
    \item $\lambda(\sigma)\geq e^{-C\abs\sigma}$ for every $\sigma\in\cS$ with
      $\lambda(\sigma)>0$.
  \end{enumerate}
\end{lemma}

\begin{proof}
  Fix $k\in\supp\widetilde\nu$, let $S_k$ have law $\mu_k$ and write $m$ for its neck length.
  Let $Z$ be the total progeny of a Galton--Watson tree with law $\theta^\dagger$. The argument
  in the proof of \cref{thm:shape-mass} gives $\bbE e^{tZ}<\infty$ for all sufficiently small
  $t>0$. By \cref{thm:conditional-explicit}, $\abs{S_k}$ is stochastically dominated by
  \[
    m+\sum_{i=1}^{(J-1)m}Z_i,
  \]
  where the $Z_i$ are independent copies of $Z$, independent of $m$. Choose $t>0$ small enough
  that $\widetilde\theta_1e^t(\bbE e^{tZ})^{J-1}<1$. The geometric law of $m$ then gives
  \[
    \bbE e^{t\abs{S_k}}
    \leq\bbE\bigl[\bigl(e^t(\bbE e^{tZ})^{J-1}\bigr)^m\bigr]<\infty.
  \]
  The right-hand side is independent of $k$, so Markov's inequality gives the first bound
  uniformly in $k$.

  If $\sigma\in\supp\mu_k$, its mass given by \cref{thm:conditional-explicit} is a product of
  at most $(2+J)\abs\sigma$ positive factors drawn from a finite set determined by $\theta$.
  If $p$ is the least such factor, then, uniformly in $k$,
  \[
    \mu_k(\sigma)\geq p^{(2+J)\abs\sigma}=e^{-C_0\abs\sigma},
    \qquad C_0=(2+J)\log p^{-1}.
  \]

  Finally, the mixture identity in \cref{thm:conditional-explicit} preserves the tail bound. If
  $\mu(\sigma)>0$, then $\mu_k(\sigma)>0$ for some $k\in\supp\widetilde\nu$, and
  $\mu(\sigma)\geq\widetilde\nu_k\mu_k(\sigma)$. Taking the minimum of the finitely many
  positive weights $\widetilde\nu_k$ and using $\abs\sigma\geq1$ gives the second bound for
  $\mu$ after increasing $C_0$.
\end{proof}

\paragraph*{\bf Nets and cuts} The net construction of \cref{sec:shape-net} uses only countability, size and marked
comparability, so it applies to the general shape space $\cS$. We continue using its naming
conventions: 
$\cR_D$ is the set of representatives, $\rep_D$ the representative map and $G_D$ the label graph.
The cut used in \cref{thm:dilution} does not require bounded arity. We record its general form
and the additional degree bound.

\begin{lemma}[General-arity dilution]\label{thm:general-dilution}
  Let $\sigma\in\cS$ and let $s\geq1$ be an integer. There are a finite rooted marked tree $T$
  and a surjection $P$ from the realisation of $\sigma$ onto $T$, preserving the entry and exit,
  such that $P$ is a $2s$-marked quasi-isometry and
  \[
    \abs T\leq\frac{\abs\sigma}{s}+1,
    \qquad \diam P^{-1}(v)\leq2s-2\quad(v\in T).
  \]
  If the realisation has at most $J$ children per vertex, then
  \[
    \abs{P^{-1}(v)}\leq1+J(s-1)\qquad(v\in T),
  \]
  and every vertex of $T$ has at most $J\bigl(1+J(s-1)\bigr)$ children.

  For all integers $D\geq2$ and $n\geq1$, every family of shapes of size at most $n$ whose
  distinct members are not $D$-comparable has at most
  $\exp\bigl(27(nD^{-1/3}+1)\bigr)$ members. The same bound holds for the members of $\cR_D$
  of size at most $n$.
\end{lemma}

\begin{proof}
  Apply the greedy cut from the proof of \cref{thm:dilution}. Every part that is cut off has at least
  $s$ vertices, apart from a possible smaller remainder at the root, so there are at most
  $\abs\sigma/s+1$ parts. A remainder passed towards the root has fewer than $s$ vertices.
  Hence every part has diameter at most $2s-2$. Root the quotient at the part containing the
  entry and mark the part containing the exit. For the quotient projection $P$, the same path
  estimate gives
  \[
    d(Px,Py)\leq d(x,y)\leq(2s-1)d(Px,Py)+2s-2.
  \]
  Thus $P$ is a $2s$-marked quasi-isometry with the stated fibre bounds.

  Under the degree bound, a cut part consists of its highest vertex and at most $J$ remainders,
  each with at most $s-1$ vertices. Hence it has at most $1+J(s-1)$ vertices, from which at most
  $J\bigl(1+J(s-1)\bigr)$ child parts can issue.

  For the counting bound, the Catalan estimate in \cref{thm:dilution} applies to rooted plane
  trees of arbitrary finite arity. For $D>729$, set
  $s=\floor{(D/72)^{1/3}}$. Then $72s^3\leq D$ and $s^{-1}\leq9D^{-1/3}$. If two marked
  quotient trees are isomorphic, composing their projections and marked quasi-inverses gives
  a $72s^3$-marked quasi-isometry in both directions. The quotient tree has at most $n/s+1$
  vertices, so the Catalan estimate gives at most
  \[
    \exp\bigl(3(n/s+1)\bigr)
    \leq\exp\bigl(27(nD^{-1/3}+1)\bigr)
  \]
  possible marked quotients. For $2\leq D\leq729$, the direct bound $e^{3n}$ is no larger than
  the displayed expression. This proves the claim for families whose distinct members are
  not $D$-comparable and for representatives by the greedy definition of $\cR_D$.
\end{proof}

\paragraph*{\bf Glued transfer} The glued-transfer argument depends only on the rooted indexing tree and the entry and exit of
each copy, not on the number of children.

\begin{lemma}[Glued transfer at general arity]\label{thm:general-glued}
  Let $(\sigma_w)_w$ and $(\tau_w)_w$ be families of finite trees with an entry and an exit,
  indexed by the same rooted tree. Let $\cT$ and $\cT'$ be their assemblies, formed by joining
  each parent's exit to its children's entries. If $K\geq1$ and every $\sigma_w$ admits a
  $K$-marked quasi-isometry to $\tau_w$, then there is a root-preserving
  $8K^2$-quasi-isometry $\cT\to\cT'$.
\end{lemma}

\begin{proof}
  Define the assembly map by a chosen marked quasi-isometry on each copy. The distance formul\ae\ 
  \cref{eq:assembly-anc,eq:assembly-div} remain valid for an arbitrary rooted indexing tree, so
  the termwise estimates in the proof of \cref{thm:glued-transfer} give the $8K^2$ distance
  bounds. Coarse density holds copy by copy, and the same correction at the root makes the map
  root-preserving.
\end{proof}

\paragraph*{\bf Coupling} We now couple each conditional shape law to the mixture $\mu$ and use the net label of the
reference partner. The resulting labels have a common law, satisfy a geometric comparability, and
have bounded potentials for matching.

\begin{lemma}[Transport relabelling]\label{thm:relabel}
  There is $D_2=D_2(\theta)$ such that for every integer $D\geq D_2$ there are a probability law
  $\mu_D$ on $\cR_D$ and maps $\ell_k\colon\cS\times[0,1)\to\cR_D$, $k\in\supp\widetilde\nu$,
  into the net of the mixture $\mu$, with the following properties.
  \begin{enumerate}[(i),ref=(\roman*)]
    \item\label{it:relabel-law} If $S_k\sim\mu_k$ and $U$ is uniform on $[0,1)$, independent,
      then $\ell_k(S_k,U)\sim\mu_D$.
    \item\label{it:relabel-qi} For $\sigma\in\supp\mu_k$ and
      $\sigma'\in\supp\mu_{k'}$, if $\ell_k(\sigma,u)$ and $\ell_{k'}(\sigma',u')$ are equal
      or adjacent in $G_D$, then $\sigma$ and $\sigma'$ are $3^{15}D^{16}$-comparable.
    \item\label{it:relabel-eta} $\eta_{G_D,5/2}(\mu_D)\leq e^{-cD^2}\leq10^{-4}$ with
      $c=c(\theta)>0$, and the class $v_0$ of the one-vertex shape carries
      $\mu_D(v_0)\geq\tfrac12$.
  \end{enumerate}
\end{lemma}

\begin{proof}
  Fix $k\in\supp\widetilde\nu$. We use the cascade argument from the proof of
  \cref{thm:shape-coupling} and the uniform estimates of \cref{thm:mass-uniform} for the mass
  bounds. Put
  \[
    s=\floor{\sqrt{D/(216J^2)}}\geq1,
  \]
  where we let $D_2\geq 3$ be such that $s\geq1$ for $D\geq D_2$. Then $216J^2s^2\leq D$.

  We first construct the shrinking map from $\supp\mu_k$ to $\supp\mu$. Fix
  $\sigma\in\supp\mu_k$ of size $n$, and let $T$ be its contraction at scale $s$ from
  \cref{thm:general-dilution}. The projection $\sigma\to T$ is $2s$-marked. If $r=\abs T$, then
  $r\leq n/s+1$. Moreover, each vertex of $T$ has at most
  $J(1+J(s-1))\leq J^2s$ children. Let $B$ be the left-child/right-sibling binarisation of $T$,
  which has the same vertex set as $T$. A child edge of $T$ becomes a path through at most $J^2s$
  siblings in $B$, while adjacent siblings in $B$ are at distance two in $T$. Thus the identity
  on the vertex set is a $(J^2s+1)$-marked map $T\to B$.

  We extend the root-to-mark path of $B$ by a new exit vertex. We also add a second new child
  at the old mark and attach all its former child subtrees to this second child. The resulting
  shape $\rho$ has an empty exit bouquet. This changes only the edges at the mark and adds two
  vertices, so the natural map $B\to\rho$ is $2$-marked. We obtain $\widehat\sigma$ from $\rho$
  by adding leaves until each bush vertex has zero or $J$ children and each non-exit neck vertex
  has $J$ children.
  The inclusion $\rho\to\widehat\sigma$ is $1$-marked, and
  $\abs{\widehat\sigma}\leq J(r+2)$. By the composition bound in \cref{thm:shape-net}, the three
  maps from $T$ to $\widehat\sigma$ compose to an $18(J^2s+1)$-marked map. Since
  $J^2s+1\leq2J^2s$, its composition with the projection $\sigma\to T$ is $D$-marked, because
  \[
    3\cdot2s\cdot18(J^2s+1)
    \leq216J^2s^2
    \leq D.
  \]
  Moreover, $r\geq1$, and hence
  \[
    \abs{\widehat\sigma}\leq J(r+2)\leq3Jr\leq3J(n/s+1).
  \]
  Since $0,J\in\supp\theta$, these offspring numbers and the empty exit bouquet give
  $\widehat\sigma\in\supp\mu_J\subseteq\supp\mu$.

  For the shrinking map in the other direction, fix $\sigma\in\supp\mu$ of size $n$. Apply the
  contraction, binarisation, neck construction and padding to obtain $\widehat\sigma_J$. If $r$
  is the contracted size, then $\widehat\sigma_J\in\supp\mu_J$,
  $\abs{\widehat\sigma_J}\leq3Jr$, and the map $\sigma\to\widehat\sigma_J$ is $D$-marked. Add
  $J-k$ one-vertex bushes to the exit bouquet and denote the result by $\widehat\sigma_k$. The
  exit has $k+(J-k)=J$ children, so $\widehat\sigma_k\in\supp\mu_k$ and
  \[
    \abs{\widehat\sigma_k}
    \leq3Jr+(J-k)
    \leq4Jr
    \leq4J(n/s+1).
  \]
  The inclusion $\widehat\sigma_J\to\widehat\sigma_k$ is $1$-marked. Its composition with the
  $D$-marked map is therefore $3D$-marked. The greedy cut and the fixed child order make both
  shrinking maps deterministic. After increasing $D_2$, substituting these bounds and the
  uniform mass estimates in \cref{eq:capacity} gives the required capacity estimates.

  The cascade gives a coupling $\pi_k$ of $\mu_k$ and $\mu$, whose pairs are shrinking pairs
  or pairs of shapes of size at most $D^2$. Each shrinking pair admits a marked map in its
  construction direction at scale at most $3D$. By \cref{thm:shape-net}, its quasi-inverse is
  $27D^2$-marked, and $27D^2\leq9D^3$ since $D\geq3$. Between two shapes of size at most $D^2$,
  the constant map to the target entry is $D^2$-marked in either direction. Thus every pair in
  the support of $\pi_k$ admits $9D^3$-marked maps in both directions. There are finitely many
  arities, so we choose one threshold $D_2$ for all these couplings.

  Define $\mu_D:=\mu\circ\rep_D^{-1}$. For $\sigma\in\supp\mu_k$, set
  $\ell_k(\sigma,u):=\rep_D(\tau)$, where $u$ samples $\tau$ from
  $\pi_k(\,\cdot\mid\sigma)$ by partitioning $[0,1)$ into half-open intervals of the positive
  conditional masses. Outside $\supp\mu_k$, set $\ell_k(\sigma,u)=v_0$.
  The second marginal of $\pi_k$ is $\mu$, and hence $\ell_k(S_k,U)$ has law $\mu_D$. This proves
  \cref{it:relabel-law}.

  We next prove \cref{it:relabel-qi}. We write $\tau$ and $\tau'$ for the reference shapes
  behind two equal or adjacent labels. By \cref{thm:shape-net}, at least one direction between
  $\tau$ and $\tau'$ admits a $729D^7$-marked quasi-isometry. Since the desired conclusion is
  symmetric, we may exchange the two shapes and assume that this map runs from $\tau$ to $\tau'$.
  The coupling supports give $9D^3$-marked quasi-isometries $\sigma\to\tau$ and
  $\sigma'\to\tau'$. We compose the first map with $\tau\to\tau'$ and with a
  $3(9D^3)^2$-marked quasi-inverse of the second map. The resulting map from $\sigma$ to
  $\sigma'$ is marked at
  \[
    3\cdot\bigl(3\cdot9D^3\cdot729D^7\bigr)\cdot3(9D^3)^2
    =3^{15}D^{16}.
  \]
  Thus $\sigma$ and $\sigma'$ are $3^{15}D^{16}$-comparable.

  For \cref{it:relabel-eta}, every shape of size at most $D^2$ has a $D$-marked map to the
  one-vertex shape. Since this is the first representative, all such shapes have net label
  $v_0$. The tail bound of \cref{thm:mass-uniform} gives
  \[
    \mu_D(v_0)\geq\mu\bigl(\abs S\leq D^2\bigr)
    \geq1-e^{-cD^2}\geq\tfrac12,
  \]
  after increasing $D_2$. To estimate the potential, the shrinking construction above sends
  every $\sigma\in\supp\mu$ of size $n$ to a shape $\sigma_0\in\supp\mu_J\subseteq\supp\mu$
  of size at most $3J(n/s+1)$ by a $D$-marked map. The representative comparison in the proof of
  \cref{thm:shape-eta} makes their net labels equal or adjacent. The point-mass bound therefore
  gives, for some $C=C(\theta)>0$,
  \[
    \mu_D\bigl(B_{G_D}(\rep_D(\sigma),1)\bigr)
    \geq\mu(\sigma_0)\geq e^{-C(nD^{-1/2}+1)}.
  \]
  This and the uniform tail bound give
  $\eta_{G_D,5/2}(\mu_D)\leq e^{-cD^2}\leq10^{-4}$ after decreasing $c$ and increasing $D_2$
  sufficiently.
\end{proof}

\begin{proposition}[Product form for relabelled shapes]\label{thm:product-form}
  Let $\lambda$ be a probability law on a countable set $X$. For each
  $k\in\supp\widetilde\nu$, suppose that $L_k\colon\cS\times[0,1)\to X$ is a map such that
  $L_k(S_k,U)$ has law $\lambda$ whenever $S_k\sim\mu_k$ and $U$ is uniform on $[0,1)$,
  independent of $S_k$. Let $(U(w))_w$ be uniform on $[0,1)$, independent of one another and of
  the realisation, and set
  \[
    x(w)=L_{k(w)}\bigl(S(w),U(w)\bigr).
  \]
  Then the pairs $\bigl(x(w),k(w)\bigr)$ are i.i.d.\ with law
  $\lambda\times\widetilde\nu$.

  In particular, for $D\geq D_2$, the maps $L_k=\ell_k$ give the product law
  $\Pi_D=\mu_D\times\widetilde\nu$ on $\cR_D\times\set{2,\dots,J}$, and the class $v_0$ of the
  one-vertex shape carries $\mu_D(v_0)\geq\tfrac12$.
\end{proposition}

\begin{proof}
  Conditioned on the arity field, the shapes are independent with $S(w)\sim\mu_{k(w)}$ by
  \cref{thm:conditional-iid}. Hence the labels are conditionally independent with law $\lambda$,
  which does not depend on the arity field. The arities are i.i.d.\ with law $\widetilde\nu$ by
  \cref{thm:harris-general}\labelcref{it:harris-general-reduced}. This proves the product form.
  For $L_k=\ell_k$, the common label law and the bound $\mu_D(v_0)\geq\tfrac12$ follow from
  \cref{thm:relabel}\labelcref{it:relabel-law,it:relabel-eta}.
\end{proof}

\subsection{The cross-law step in the bushy regime}
We now focus on being able to match between two different distribution in the bushy regime.
Throughout this subsection we fix $\theta$ and $\theta'$ as two supercritical distributions
with bounded support and
$\theta_0,\theta_0'>0$, with offspring bounds $J\leq J'$. The objects of
\cref{sec:general-shapes,sec:general-relabel} are formed for each law separately: the
conditional laws $\mu_k$ and $\mu'_k$, the mixtures $\mu$ and $\mu'$, and the reduced laws
$\widetilde\nu$ and $\widetilde\nu'$, supported on the full intervals $\set{2,\dots,J}$ and
$\set{2,\dots,J'}$ by \cref{thm:full-support}. We denote the two shape spaces by $\cS_J$ and
$\cS_{J'}$, so that $\cS_J\subseteq\cS_{J'}$. For distribution $\theta$, we
form $\cR_D$, $\rep_D$ and $G_D$ on $\cS_J$ and
use $\mu$ as the reference mixture. We give the two realisations a common label law by coupling
both families of conditional laws to $\mu$ and labelling their partners in this one net. Thus the
net and its label graph depend only on $\theta$, while the arity laws may differ. We check their
compatibility in \cref{sec:trichotomy} using the matching theorem.

The binary encoding will insert vertices between a split and its children. Each inserted
vertex forms a one-vertex piece, to which we give the label $v_0$. We include these pieces in
the geometric comparison below, so that the labels will also compare the encoded trees.

\begin{lemma}[Cross-law relabelling]\label{thm:cross-relabel}
  There exists $D_3=D_3(\theta,\theta')$ such that for every integer $D\geq D_3$ there are a
  probability law $\mu_D$ on $\cR_D$ and maps
  \[
    \ell_k\colon\cS_J\times[0,1)\to\cR_D,
    \qquad
    \ell'_k\colon\cS_{J'}\times[0,1)\to\cR_D,
  \]
  indexed respectively by $k\in\supp\widetilde\nu$ and
  $k\in\supp\widetilde\nu'$, with the following properties.
  \begin{enumerate}[(i),ref=(\roman*)]
    \item\label{it:cross-relabel-law} Let $k\in\supp\widetilde\nu$ and
      $k'\in\supp\widetilde\nu'$. If $S_k\sim\mu_k$ and $S'_{k'}\sim\mu'_{k'}$, with $U$
      uniform on $[0,1)$ and independent of these shapes, then $\ell_k(S_k,U)\sim\mu_D$ and
      $\ell'_{k'}(S'_{k'},U)\sim\mu_D$.
    \item\label{it:cross-relabel-qi} Suppose that each of two pieces is either in the support of
      the conditional shape law used to label it or is the one-vertex piece. If their labels are
      equal or adjacent in $G_D$, then they admit $3^{31}D^{32}$-marked quasi-isometries in both
      directions.
    \item\label{it:cross-relabel-eta} $\eta_{G_D,5/2}(\mu_D)\leq e^{-cD^2}\leq10^{-4}$ with
      $c=c(\theta)>0$ as in \cref{thm:relabel}, and the class $v_0$ of the one-vertex shape
      has mass $\mu_D(v_0)\geq\tfrac12$.
  \end{enumerate}
\end{lemma}

\begin{proof}
  We take $\mu_D=\mu\circ\rep_D^{-1}$ and use the maps $\ell_k$ from \cref{thm:relabel} on
  $\cS_J$. For $\sigma\in\supp\mu_k$, we have $\ell_k(\sigma,u)=\rep_D(\tau)$, where $u$
  samples a partner $\tau$ from the coupling $\pi_k$ of $\mu_k$ and $\mu$.

  We fix $k\in\supp\widetilde\nu'$ and take
  \[
    s=\floor{\sqrt{D/(216(J')^2)}}\geq1,
  \]
  after increasing $D_3$ if necessary. The two instances of \cref{thm:mass-uniform}, for
  $\theta'$ and $\theta$, give the mass bounds required by the cascade. Given a shape
  $\sigma\in\supp\mu'_k$ of size $n$, we contract its realisation at scale $s$, binarise the
  contracted tree and pad the resulting shape at the support point $J$ of $\theta$. The image
  lies in $\supp\mu_J\subseteq\supp\mu$, has size at most $3J(n/s+1)$ and admits a $D$-marked
  quasi-isometry from $\sigma$. Given instead $\sigma\in\supp\mu$, we contract and binarise its
  realisation, pad the resulting shape at the support point $J'$ of $\theta'$, and add $J'-k$
  one-vertex bushes to its exit bouquet. The image lies in $\supp\mu'_k$, has size at most
  $4J'(n/s+1)$ and admits a $3D$-marked quasi-isometry from $\sigma$. For shapes of size at most
  $D^2$, the constant map to the entry of the target is $D^2$-marked and hence $9D^3$-marked.
  The cascade therefore gives a coupling $\pi'_k$ of $\mu'_k$ and $\mu$ supported on pairs
  admitting $9D^3$-marked maps in both directions, by the quasi-inverse bounds in the proof
  of \cref{thm:relabel}.

  For $\sigma\in\supp\mu'_k$, define $\ell'_k(\sigma,u):=\rep_D(\tau)$, where $u$ samples
  $\tau$ from $\pi'_k(\,\cdot\mid\sigma)$, and set $\ell'_k(\sigma,u)=v_0$ elsewhere. Each
  cascade has a threshold determined by its mass and shrinking bounds. Since
  $\supp\widetilde\nu'$ is finite, we choose $D_3$ to be the maximum of these
  thresholds and the constant $D_2$ from \cref{thm:relabel}.

  The conditional-sampling argument in the proof of \cref{thm:relabel} applies to each
  $\pi'_k$: its second marginal is the same mixture $\mu$, so applying $\rep_D$ gives law
  $\mu_D$. Together with the unprimed construction this proves \cref{it:cross-relabel-law}.

  For \cref{it:cross-relabel-qi}, each supported shape has a $\mu$-partner joined to it by
  $9D^3$-marked maps in both directions. Equal or adjacent labels give the same
  $729D^7$-marked comparison between the reference partners as in \cref{thm:relabel}.
  The one-vertex piece may be its own reference partner, with its identity map and label
  $v_0$. The same comparison therefore includes pairs with an inserted one-vertex piece.
  The composition argument in \cref{thm:relabel} gives a $3^{15}D^{16}$-marked map in one
  direction. Its marked quasi-inverse has constant
  $3(3^{15}D^{16})^2=3^{31}D^{32}$, which also bounds the original map.

  Finally, $\cR_D$, $\rep_D$ and $G_D$ were formed on $\cS_J$, while
  $\mu_D=\mu\circ\rep_D^{-1}$. Hence they all depend only on $\theta$, and
  \cref{thm:relabel}\labelcref{it:relabel-eta} gives \cref{it:cross-relabel-eta}.
\end{proof}

\section{Binary encodings of reduced skeletons}\label{sec:general-chain}
We now encode the reduced skeletons so that trees with different branching arities can be
compared by matching labels on the fixed binary tree $\bbB$. We first recall that, in the
chain regime, the quantised neck length and the terminating arity are independent. In both
regimes we then replace each split of arity $k$ by a finite binary tree with $k$ leaves,
keeping its original shape at the root and assigning one-vertex pieces to the other
internal vertices. We show that this construction preserves the quasi-isometry class and
identify the resulting label process. Finally, we show how uniformly controlled marked
comparisons of matched pieces give a root-preserving quasi-isometry of the original trees.

\subsection{Independent necks and arities}

We now restrict to chain-regime laws: $\theta_0=0$ and $\theta_1\in(0,1)$, with finite support
and largest supported offspring number $J$. Every vertex belongs to the skeleton, the
extinction probability is $q=0$, and the transform of \cref{thm:harris-general} is the identity.
There are no bushes, so a shape is determined by its neck length. These lengths lie in $\bbN$,
and their level labels from \cref{sec:quantisation} lie in $\bbN_0$, with the path graph
$\mathsf P$ of \cref{sec:integer} as their comparison graph.

\paragraph*{\bf Neck length} For a reduced-skeleton vertex $w$, let $m(w)$ be the length of the neck ending at $w$ and let
$k(w)$ be the arity of its terminating split. By
\cref{thm:harris-general}\labelcref{it:harris-general-reduced}, the pairs
$\bigl(m(w),k(w)\bigr)$ are i.i.d.\ and have the joint law
\begin{equation}\label{eq:chain-joint}
  \bbP\bigl(m(w)=r,\ k(w)=j\bigr)
  =\theta_1^{\,r-1}\theta_j
  =\theta_1^{\,r-1}(1-\theta_1)\widetilde\nu_j
  \qquad(r\geq1,\ j\geq2).
\end{equation}
Thus the neck length is geometric with success probability $1-\theta_1$ and is independent of
the terminating arity, which has law $\widetilde\nu$.

\paragraph*{\bf $D$-scales} For an integer $D\geq2$, set $x(w)=\ell_D\bigl(m(w)\bigr)$. By
\cref{eq:chain-joint,thm:quantised-law}, its marginal law is
\begin{equation}\label{eq:chain-class-law}
  p^{(D)}_i=\theta_1^{\,D^i-1}-\theta_1^{\,D^{i+1}-1}
  \qquad(i\geq0).
\end{equation}
Applying \cref{thm:product-form}, with a shape identified with its neck length and
$L_k(m,u)=\ell_D(m)$, shows that the pairs $\bigl(x(w),k(w)\bigr)$ are i.i.d.\ with product law
$\Pi_D=p^{(D)}\times\widetilde\nu$. Moreover
$p^{(D)}_0=1-\theta_1^{\,D-1}\geq\tfrac12$ once $D$ is large, and
$\eta_{\mathsf P,5/2}(p^{(D)})\to0$ as $D\to\infty$ by \cref{thm:eta-bound}. Thus the chain regime
needs no transport relabelling of the kind used in \cref{thm:product-form}.

\subsection{The binary encoding}\label{sec:binary-presentations}

We now return to both regimes. We replace each split by a finite binary tree, keeping its
original shape at the root and attaching its children at the leaves. The added vertices have
no geometric decoration. If the depths of the replacement trees are bounded, the resulting
assembly is quasi-isometric to the original one. This lets us seek a matching on the fixed
indexing tree $\bbB$, even when the original reduced skeletons have different arities.

\paragraph*{\bf Binary profiles} A \emph{binary profile} is a finite rooted tree in which every nonleaf vertex has two ordered
children. For each arity $k$ in a finite set $A\subseteq\set{2,3,\ldots}$, choose a profile $C_k$
with $k$ leaves, listed from left to right. To encode a vertex of arity $k$, place its original
piece at the root of $C_k$, a one-vertex piece at every other nonleaf vertex, and the encoded
child pieces at its leaves in their original order. Every joining edge runs from an exit to
an entry. For example, take $C_3$ to be the subtree of $\bbB$ with vertex set
$\set{\troot,1,2,21,22}$ and ordered leaves $1,21,22$. The original piece occupies $\troot$, the
inserted one-vertex piece occupies $2$, and the three child pieces attach at the leaves.

\paragraph*{\bf Assemblies} An \emph{assembly} over a rooted skeleton is the rooted tree obtained from finite tree pieces,
each with one entry and one exit, indexed by the vertices of the skeleton. The exit of each piece
is joined to the entries of its child pieces, and the entry of the root piece is the root of the
assembly. Repeating the construction above gives an assembly indexed by $\bbB$.

\paragraph*{\bf Profiles and distortions} We write
\[
  h_C\coloneqq\max_{k\in A}\operatorname{height}(C_k)\geq1.
\]

\noindent For an assembly $X$ and its binary encoding $X_C$, let $j_C\colon X\to X_C$ keep every vertex of
the original pieces, and let $p_C\colon X_C\to X$ fix these vertices and send each inserted vertex
to the exit of its original piece. Then $p_Cj_C=\operatorname{id}$, and $p_C$ is onto and
$1$-Lipschitz. Each original joining edge becomes a path of length at most $h_C$, while edges
inside the pieces are unchanged. Consequently,
\[
  d(x,y)\leq d_C\bigl(j_C(x),j_C(y)\bigr)\leq h_Cd(x,y),
  \qquad d_C\bigl(z,j_Cp_C(z)\bigr)\leq h_C.
\]
Moreover,
\[
  d_C(z,z')\leq2h_C+h_Cd\bigl(p_C(z),p_C(z')\bigr),
  \qquad d\bigl(p_C(z),p_C(z')\bigr)\leq d_C(z,z').
\]
Both maps fix the root, and these bounds make them $(h_C+1)$-quasi-isometries. We record this in
the following lemma.

\begin{lemma}[Bounded-profile encoding]\label{thm:profile-encoding}
  For every assembly $X$ over an infinite reduced skeleton whose arities belong to $A$, the maps
  $j_C\colon X\to X_C$ and $p_C\colon X_C\to X$ are root-preserving
  $(h_C+1)$-quasi-isometries. They satisfy $p_Cj_C=\operatorname{id}$ and
  \[
    d(x,y)\leq d_C\bigl(j_C(x),j_C(y)\bigr)\leq h_Cd(x,y),
    \qquad d_C\bigl(z,j_Cp_C(z)\bigr)\leq h_C.
  \]
\end{lemma}

\paragraph*{\bf The binary profile process} We next describe the label law separately. Let $\mu$ be a probability law on a countable label
set with distinguished element $0$, and let $\nu$ be a probability law on $A$. The corresponding
\emph{binary profile process} is the random label field on $\bbB$ obtained as follows. Draw a label
with law $\mu$ and, independently, an arity $k$ with law $\nu$. Give the root of $C_k$ this label
and its other internal vertices label $0$, then attach independent copies of the same process at
its $k$ leaves. For an i.i.d.\ $\mu\times\nu$ label--arity field, the descendant fields below each
vertex are independent copies of the original field, so its binary encoding has exactly this law.
We record this in the following lemma.

\begin{lemma}[Law of the binary encoding]\label{thm:profile-encoding-law}
  Suppose that the label--arity pairs at the vertices of an infinite reduced skeleton are i.i.d.\
  with law $\mu\times\nu$. If every inserted vertex has label $0$, then the encoded label field
  has the law of the binary profile process.
\end{lemma}

We will use the preceding metric comparison through a rooted automorphism of $\bbB$. Such an
automorphism may pair an original piece with an inserted one-vertex piece, so the local comparison
hypothesis must cover both kinds of pieces.

\begin{corollary}[Transfer through binary encodings]\label{thm:profile-transfer}
  Let $X_C$ and $Y_{C'}$ be binary encodings of assemblies $X$ and $Y$, with all profiles of
  height at most an integer $h\geq1$. Suppose that $K\geq1$ and that there is a rooted automorphism
  $\pi\in\Aut(\bbB)$ such that, for every $w\in\bbB$, the piece of $X_C$ indexed by $w$, whether
  original or inserted, admits a $K$-marked quasi-isometry to the piece of $Y_{C'}$ indexed by
  $\pi(w)$, whether original or inserted. Then there is a root-preserving
  $72(h+1)^2K^2$-quasi-isometry $X\to Y$.
\end{corollary}

\begin{proof}
  Reindexing the pieces of $Y_{C'}$ by $\pi$ preserves its metric and places the compared pieces
  over the same indexing vertex. Hence \cref{thm:general-glued} gives a root-preserving
  $8K^2$-quasi-isometry $F\colon X_C\to Y_{C'}$. By \cref{thm:profile-encoding}, the inclusion
  $j_C\colon X\to X_C$ and collapse $p_{C'}\colon Y_{C'}\to Y$ are
  $(h+1)$-quasi-isometries. The composition bound in \cref{thm:shape-net} therefore makes
  $p_{C'}Fj_C$ a root-preserving quasi-isometry with constant
  \[
    3\bigl(3(h+1)8K^2\bigr)(h+1)=72(h+1)^2K^2.\qedhere
  \]
\end{proof}

\paragraph*{\bf Dependence} The encoded labels are generally dependent: an arity choice determines which descendants
receive independent labels and which are prescribed to be $0$. The next section proves the
Markov matching theorem needed for this recursive law.

\section{The Markov matching theorem}\label{sec:markov-proof}

We now prove the Markov matching theorem, \cref{thm:markov-matching}. 
We first recall the construction in \cref{sec:process} to set up how our approach differs for the
general matching theorem.
We consider random labellings of the rooted binary tree $\bbB$ with values in the vertex
set $V$ of a countable graph $G=(V,E)$. Two labels $v,w\in V$ are compatible when
$d_G(v,w)\leq1$. We fix a probability law $\mu$ on $V$ and a distinguished label $0\in V$.

We describe the dependence between labels by assigning a type to each vertex of $\bbB$.
Let $\cI$ be a nonempty countable set of types. Choose a subset $\cI_\mu\subseteq\cI$ and, for
each $t\in\cI$, a probability law $P_t$ on $\cI\times\cI$. Starting from a fixed
root type, we generate the labels and child types recursively. At a vertex of type $t$,
we draw the label with law $\mu$ if $t\in\cI_\mu$, and choose label $0$ deterministically otherwise. 
Independently of this label, we draw the types of its two children with law $P_t$. Conditioned on
these types, the two descendant processes are independent and follow the same rule.
The choices at distinct vertices use independent draws. In the binary encoding of
\cref{sec:binary-presentations}, types in $\cI_\mu$ describe the original vertices,
and the remaining types describe the vertices inserted inside the profiles.

For two independent realisations started at types $s,t$, we write $\cM_h(s,t)$ for the
event that an automorphism of $\bbB$ pairs compatible labels through height $h$. We write
$\cM_\infty(s,t)$ for the event that one automorphism pairs compatible labels at every
vertex. The automorphism need not preserve types. We can compare models with different
transition laws by taking the disjoint union of their type sets and retaining the
transition laws within each model.

For a fixed label $v\in V$, the probability of matching a draw from $\mu$ is
$b(v)=\mu\bigl(B_G(v,1)\bigr)$. For $\alpha\geq1$, the one-site potentials are
\[
 \eta_\alpha=\sum_{v\colon\mu(v)>0}\mu(v)\frac{1-b(v)}{b(v)^\alpha},\qquad
 \zeta_\alpha=\max\left\{\eta_\alpha,\frac{1-b(0)}{b(0)^\alpha}\right\}.
\]
The sum averages over $v$ with law $\mu$, while the additional term in $\zeta_\alpha$
uses the fixed label $0$ in place of $v$.
We keep to the convention $\zeta_\alpha=\infty$ when $b(0)=0$.

When $b(0)=1$, no further hypotheses on the types or transitions are needed. When $b(0)<1$,
we assume that $\cI$ is finite and impose the conditions of \cref{sec:finite-hypotheses},
which we now recall. We require a partition of $\cI$ into classes $\cI_0,\ldots,\cI_{g-1}$,
where $g\geq1$, such that
\[
 \cI_\mu\subseteq\cI_0,\qquad
 \supp P_t\subseteq\cI_{i+1}\times\cI_{i+1}
 \quad(t\in\cI_i),
\]
with indices read modulo $g$. A possible type path follows either child in transitions
of positive probability. We require the following two conditions:
\begin{enumerate}[label=(M\arabic*),ref=(M\arabic*),leftmargin=*]
\item \emph{Positive matching.} For every $h\geq0$ and $s,t\in\cI_\mu$, each labelling that occurs
with positive probability from $s$ through height $h$ has a positive probability of matching an
independent labelling started at $t$.
\item \emph{Bounded returns.} There is an integer $H\geq0$ such that, whenever $s,t$ belong to the
same class $\cI_i$, every possible type path from $s$ visits $\cI_\mu$ at some depth $n\leq H$
for which a possible type path from $t$ also visits $\cI_\mu$.
\end{enumerate}
The path from $t$ in condition~\labelcref{it:markov-returns} may depend on the path from $s$.

The contraction estimate at a binary split uses the coefficient from
\cref{sec:admissible-exponent}:
\[
 \lambda_\alpha=2\min_{0\leq\beta\leq1}\left[
 \max_{0\leq q\leq1}
 \left\{\frac{q}{(1+q)^\alpha}+\beta(1-q)^\alpha\right\}
 +\frac{\alpha^\alpha(1-\beta)^{\alpha+1}}{(\alpha+1)^{\alpha+1}}
 \right].
\]
The condition $\lambda_\alpha<1$ holds, in particular, for $\alpha\geq4/3$.
We now restate the theorem with the quantitative dependence of its constants.

\begin{restate}{thm:markov-matching}
Fix $\alpha\geq1$ with $\lambda_\alpha<1$. Assume either that $b(0)=1$, or that the
finite-type hypotheses recalled above hold.

There are $K<\infty$ and $\eps>0$ such that, if $\zeta_\alpha\leq\eps$, then for every
pair of initial types $s,t$ when $b(0)=1$, and for every pair $s,t$ in the same class
$\cI_i$ otherwise,
\[
 \bbP\bigl(\cM_h(s,t)^c\bigr)\leq K\zeta_\alpha\quad(h\geq0),\qquad
 \bbP\bigl(\cM_\infty(s,t)^c\bigr)\leq K\zeta_\alpha.
\]
When $b(0)=1$, the constants depend only on $\alpha$. Otherwise, they depend only on
$\alpha,H$ and upper bounds for the largest class size and the sums of inverse powers
of the transition probabilities,
\[
 \max_{0\leq i<g}\#\cI_i,
 \qquad
 \max_{t\in\cI}\sum_{j\in\supp P_t}P_t(j)^{-\alpha}.
\]
This dependence may be sharpened by choosing a nonempty set $J_t\subseteq\supp P_t$ of
retained transitions for each type $t$. Suppose that, at every finite height, every pair of
source child labellings with values in $\supp\mu\cup\{0\}$ that has positive matching
probability against the $P_t$-mixture has positive matching probability conditional on some
transition $j\in J_t$. Then the second quantity in the preceding display may be replaced by
\[
 \max_{t\in\cI}\sum_{j\in J_t}P_t(j)^{-\alpha}.
\]
In both cases, $K$ and $\eps$ have no further dependence on the graph $G$, the label law
$\mu$, the type set or the transition laws.

Further, if $\mu(0)\geq p>0$, then $\eta_\alpha\leq p\eps$ implies
\[
 \bbP\bigl(\cM_h(s,t)^c\bigr)\leq\frac Kp\eta_\alpha\quad(h\geq0),\qquad
 \bbP\bigl(\cM_\infty(s,t)^c\bigr)\leq\frac Kp\eta_\alpha.
\]
\end{restate}

\paragraph*{\bf Positive matching probability} To adapt the i.i.d.\ proof, we fix a finite source labelling produced from one initial type
and consider its probability of matching an independent target labelling from another type.
This probability can be zero, in which case the potential used in the i.i.d.\ proof is
infinite. We therefore sum only over source labellings with positive matching
probability and estimate the probability of the excluded labellings separately. When
$b(0)=1$, every source labelling of positive probability has a possible match against every
target type, so no labellings of positive probability are excluded.

\paragraph*{\bf Zero matching probability} When $b(0)<1$, a source labelling of positive probability can have zero probability of
matching the target type. We show that one of three events then occurs. The root label is
incompatible with $0$, or each of the two child subtrees has zero probability of matching
some possible child type of the target, or one child subtree has zero probability of
matching every such type. If the third event occurs, we apply the same alternatives to that
child subtree and all possible child types of the target. Condition
\labelcref{it:markov-returns} ends this descent along a single source branch. Within $H$
steps, the source type and a type reached by a possible target path both belong to
$\cI_\mu$, where the positivity condition~\labelcref{it:markov-positivity} gives positive matching
probability. At each vertex, incompatibility with $0$ has probability at most $1-b(0)$. Conditioned on the
child types, the two child subtrees are independent, so the second event has probability
quadratic in the bounds for zero matching probability at smaller heights. When
$\zeta_\alpha$ is sufficiently small, these estimates and $\lambda_\alpha<1$ let us choose
a potential bound proportional to $\zeta_\alpha$ that holds at every finite height. Adding
the bound for the excluded labellings gives the stated failure estimate. K\H{o}nig's lemma
and continuity of probability then give the infinite-tree bound.

\section{The local matching estimate}

We extend the binary matching estimate of \cref{sec:contraction} to child subtrees with
different laws. We first define, for arbitrary source and target laws, the potential
restricted to source points with positive matching probability. The estimate for a binary
split will also involve the probability of the excluded points and a weighted sum over
them, which we estimate separately in \cref{sec:quantitative-stopping}.

We fix $\alpha>1$ and abbreviate
\(
 \eta=\eta_\alpha(\mu), \zeta=\zeta_\alpha(\mu).
\)
We assume $\zeta<\infty$, so $b(0)>0$.

\subsection{The restricted potential}\label{sec:restricted-potential}

Let $R$ be a symmetric relation on a countable set, and say that $x$ matches $y$ when
$x\mathrel R y$. For a source law $\rho$ and a target law $\tau$ on this set, the probability
that a fixed source point $x$ matches an independent target point is its \emph{matching
degree},
\[
 r_\tau(x)\coloneqq\tau\{y\colon x\mathrel R y\},\qquad
 q_\tau(x)\coloneqq1-r_\tau(x).
\]
We recall the weight from \cref{eq:alpha}, writing
\[
 \phi(q)=\phi_\alpha(q)=\frac{q}{(1-q)^\alpha}\qquad(0\leq q<1).
\]
The \emph{restricted potential} is
\[
 \cE(\rho,\tau)\coloneqq
 \sum_{x\colon r_\tau(x)>0}\rho\{x\}\,\phi\bigl(q_\tau(x)\bigr)
 =\bbE_\rho\bigl[\mathbf1_{\{r_\tau(X)>0\}}\phi\bigl(q_\tau(X)\bigr)\bigr],\qquad X\sim\rho.
\]
We sum over source points with positive matching degree, using their original $\rho$-mass.
Thus this is a restriction of the sum, without conditioning or renormalising the source
law. We use the convention that the subscript of an expectation denotes the law of the variable averaged
over, and for a pair of independent variables it is the product law. 
We write $\cE(\rho,\tau;R)$ when the relation needs to be explicit. If $R$ is reflexive and
$\rho=\tau$, every point with positive $\rho$-mass has positive matching degree, so
$\cE(\rho,\rho;R)$ equals the i.i.d.\ potential $\Phi(R,\rho)$ of \cref{eq:iid-potential}.

\paragraph*{\bf Weighting} The denominator in the potential leads us to estimate inverse powers of matching degrees.
We use the weight
\[
 W_\tau(x)\coloneqq
 \begin{cases}
  r_\tau(x)^{-\alpha},&r_\tau(x)>0,\\
  0,&r_\tau(x)=0.
 \end{cases}
\]
Its value at zero matching degree lets us sum over the whole source space while retaining
only the points with positive matching degree.

\subsection{Four child laws}

At a binary split, the source and target each have two child subtrees. Conditioned on their
types, the two source subtrees are independent, as are the two target subtrees. Their four
laws can all differ, and matching permits either pairing of the children. We now bound the
probabilities of this choice for arbitrary probability laws. For pairs of points, we use the
relation $R^\square$ of \cref{eq:square}: a source pair matches a target pair if either the
first points and the second points match, or the first source point matches the second
target point and the second source point matches the first target point.

The parameter $\beta\in[0,1]$ lets us optimise the linear coefficient in the potential
bound. We write this coefficient as the sum of two terms, with its minimum given by
$\lambda_\alpha$ from \cref{eq:markov-exponent}:
\begin{equation}\label{eq:mean-constants}
 \begin{split}
 L_\alpha(\beta)&\coloneqq
 \max_{0\leq q\leq1}
 \left\{\frac{q}{(1+q)^\alpha}+\beta(1-q)^\alpha\right\},\\
 K_\alpha(\beta)&\coloneqq
 \frac{\alpha^\alpha}{(\alpha+1)^{\alpha+1}}(1-\beta)^{\alpha+1},\\
 \lambda_\alpha(\beta)&\coloneqq2\bigl(L_\alpha(\beta)+K_\alpha(\beta)\bigr),\qquad
 \lambda_\alpha=
 \min_{0\leq\beta\leq1}\lambda_\alpha(\beta).
 \end{split}
\end{equation}
The extrema exist by compactness. At $\beta=0$, we recover the constants
$L_\alpha=L_\alpha(0)$ and $K_\alpha=K_\alpha(0)$ from \cref{eq:maxima}.
For the quadratic term, the proof splits the range of a sum of failure probabilities at
$u\in(0,1)$. The resulting coefficient is
\begin{equation}\label{eq:four-law-constants}
 C_\alpha(u)\coloneqq
 4L_\alpha\frac{(1-u)^{-\alpha}-1}{u}
 +\frac4u+8\alpha+2\alpha^2.
\end{equation}

We now prove a general contraction lemma.

\begin{lemma}\label{thm:four-law-contraction}
Let $\rho_1,\rho_2,\tau_1,\tau_2$ be probability laws on the countable set with symmetric
relation $R$, let $X_1\sim\rho_1$ and $X_2\sim\rho_2$ be independent, and let
$M,\delta_1,\delta_2\geq0$ be finite. Suppose, for every $i,j\in\{1,2\}$, that
\begin{align*}
 \cE(\rho_i,\tau_j),\ \cE(\tau_j,\rho_i)&\leq M,\\
 \rho_i\{r_{\tau_j}=0\},\ \tau_j\{r_{\rho_i}=0\}&\leq \delta_1,\\
 \bbE_{\rho_i}\bigl[\mathbf1_{\{r_{\tau_j}(X_i)=0\}}W_{\tau_{3-j}}(X_i)\bigr]&\leq\delta_2.
\end{align*}
For every $\beta\in[0,1]$ and $u\in(0,1)$,
\begin{equation}\label{eq:four-law-contraction}
 \begin{split}
 \cE(\rho_1\times\rho_2,\tau_1\times\tau_2;R^\square)
 &\leq\lambda_\alpha(\beta)M+C_\alpha(u)M^2\\
 &\quad+(2+2\beta+4\alpha M)\delta_1+4(1+\alpha M)\delta_2.
 \end{split}
\end{equation}
\end{lemma}

\begin{proof}
We begin with an outline of the proof. We will divide the source pairs that contribute to the
restricted potential into two cases. First, we will consider pairs for which
$r_{\tau_j}(x_i)>0$ for every $i,j\in\{1,2\}$, so each source point has positive matching
probability against either target law. If neither pairing succeeds, then some source point
matches neither target point, or some target point matches neither source point. A union bound
over these four events, together with the product and square estimates below, will bound this
part of the potential.

Finally, we will treat source pairs for which at least one of the four degrees
$r_{\tau_j}(x_i)$ is zero, but the pair still has positive matching probability. Only one
pairing can then succeed. We will bound this part of the potential using products of the two
corresponding weights, and use independence to apply the $\delta_2$ hypothesis to one factor
and a moment bound to the other.

To carry out this argument, we first refine the product and square estimates of
\cref{thm:toolkit}. The product bound
has a negative linear correction, which will cancel a corresponding term from the square
bound after taking expectations. For $0\leq x,y,q<1$,
\begin{align}
 \phi(xy)
 &\leq\frac{L_\alpha(\beta)}2\bigl(\phi(x)+\phi(y)\bigr)
 -\frac\beta2(x+y),\label{eq:four-law-product}\\
 \phi(xy)
 &\leq\frac{L_\alpha}2\bigl(\phi(x)+\phi(y)\bigr),
 \label{eq:four-law-quadratic-product}\\
 q^2&\leq K_\alpha(\beta)\phi(q)+\beta q,
 \label{eq:four-law-square}\\
 (1-q)^{-\alpha}&\leq1+\alpha\phi(q).
 \label{eq:four-law-weight}
\end{align}
As in the proof of \cref{eq:product-estimate}, we have
$\phi(xy)\leq(\phi(x^2)+\phi(y^2))/2$.
By the definition of $L_\alpha(\beta)$,
$\phi(q^2)\leq L_\alpha(\beta)\phi(q)-\beta q$, which proves
\cref{eq:four-law-product}. Its case $\beta=0$ is
\cref{eq:four-law-quadratic-product}. For \cref{eq:four-law-square}, the maximum of
$(q-\beta)(1-q)^\alpha$ on $[0,1]$ is $K_\alpha(\beta)$. When $\beta<1$, this follows
by putting $q=\beta+(1-\beta)v$ on $[\beta,1]$ and differentiating
$v(1-v)^\alpha$. The expression is nonpositive for $q<\beta$, and the case $\beta=1$
is immediate. The weight bound \cref{eq:four-law-weight} is \cref{eq:tangent}. Each weight
vanishes at zero matching degree, so applying the bound at positive degrees and using the
potential hypotheses gives, for every $i,j\in\{1,2\}$,
\begin{equation}\label{eq:four-law-moment}
 \bbE_{\rho_i}\bigl[W_{\tau_j}(X_i)\bigr]\leq1+\alpha M.
\end{equation}

We first consider source pairs for which all four child comparisons have positive matching
probability. These pairs form $F\times F$, where
\[
 F\coloneqq\{x\colon r_{\tau_1}(x)>0,\ r_{\tau_2}(x)>0\}.
\]
For the random source pair $(X_1,X_2)$, abbreviate
\[
 q_i^j\coloneqq q_{\tau_j}(X_i),\qquad
 r_i^j\coloneqq1-q_i^j,\qquad
 c_j\coloneqq\tau_j\{y\colon X_1\not\mathrel R y,\ X_2\not\mathrel R y\}.
\]
The matching and failure probabilities of the source pair against an independent target pair
are
\[
 \begin{split}
 r^\square&\coloneqq(\tau_1\times\tau_2)
 \bigl\{(y_1,y_2)\colon(X_1,X_2)\mathrel{R^\square}(y_1,y_2)\bigr\},\\
 q^\square&\coloneqq1-r^\square.
 \end{split}
\]
All of these are functions of $(X_1,X_2)$, and the bounds below between them hold pointwise.
The quantity $c_j$ is the probability that target point $j$ matches neither source point.
If both pairings
fail, either one source point matches neither target point, or one target point matches
neither source point. Taking a union bound over these four events gives
\begin{equation}\label{eq:four-law-numerator}
 q^\square\leq q_1^1q_1^2+q_2^1q_2^2+c_1+c_2.
\end{equation}
The two possible successful pairings separately give
\begin{equation}\label{eq:four-law-denominators}
 r^\square\geq r_1^1r_2^2,
 \qquad r^\square\geq r_1^2r_2^1.
\end{equation}
Dividing \cref{eq:four-law-numerator} by $(r^\square)^\alpha$ bounds the weight
$\phi(q^\square)$ of the source pair. We estimate the expectations of the four resulting
terms for realisations in $\{X_1,X_2\in F\}$.

For the term $q_1^1q_1^2/(r^\square)^\alpha$, put
\[
 s\coloneqq q_2^1+q_2^2,\qquad c\coloneqq\frac{(1-u)^{-\alpha}-1}{u}.
\]
If $s\leq u$, inclusion--exclusion for the two
successful pairings gives the next lower bound. Their intersection requires both target
points to match the first source point, so its probability is at most $r_1^1r_1^2$. Hence
\begin{align*}
 r^\square
 &\geq r_1^1r_2^2+r_1^2r_2^1-r_1^1r_1^2
 =1-q_1^1q_1^2-r_1^1q_2^2-r_1^2q_2^1
 \geq(1-s)(1-q_1^1q_1^2).
\end{align*}
Convexity bounds $(1-s)^{-\alpha}$ by its chord $1+cs$ on $[0,u]$. Thus
\[
 \frac{q_1^1q_1^2}{(r^\square)^\alpha}
 \leq\phi(q_1^1q_1^2)\bigl(1+c(q_2^1+q_2^2)\bigr).
\]
The right-hand side is a function of $X_1$ times a function of $X_2$. By
\cref{eq:four-law-quadratic-product} and the potential hypotheses,
\[
 \bbE_{\rho_i}\bigl[\mathbf1_{\{X_i\in F\}}\phi(q_i^1q_i^2)\bigr]\leq L_\alpha M
 \qquad(i\in\{1,2\}),
\]
and since $q\leq\phi(q)$ and $F\subseteq\{r_{\tau_j}>0\}$, the same hypotheses give
$\bbE_{\rho_2}[\mathbf1_{\{X_2\in F\}}s]\leq2M$. Independence therefore gives
\begin{align*}
 \bbE_{\rho_1\times\rho_2}\Bigl[\mathbf1_{\{X_1,X_2\in F\}}\mathbf1_{\{s\leq u\}}\frac{q_1^1q_1^2}{(r^\square)^\alpha}\Bigr]
 &\leq\bbE_{\rho_1}\bigl[\mathbf1_{\{X_1\in F\}}\phi(q_1^1q_1^2)\bigr]\,
 \bbE_{\rho_2}\bigl[\mathbf1_{\{X_2\in F\}}(1+cs)\bigr]\\
 &\leq(1+2cM)\,\bbE_{\rho_1}\bigl[\mathbf1_{\{X_1\in F\}}\phi(q_1^1q_1^2)\bigr]\\
 &\leq\bbE_{\rho_1}\bigl[\mathbf1_{\{X_1\in F\}}\phi(q_1^1q_1^2)\bigr]+2cL_\alpha M^2.
\end{align*}

If $s>u$, the first bound in \cref{eq:four-law-denominators} instead gives
\[
 \frac{q_1^1q_1^2}{(r^\square)^\alpha}
 \leq\phi(q_1^1)W_{\tau_2}(X_2).
\]
For realisations in $\{X_2\in F\}$, applying \cref{eq:four-law-weight} gives
\[
 \mathbf1_{\{s>u\}}W_{\tau_2}(X_2)
 \leq(u^{-1}+\alpha)\bigl(\phi(q_2^1)+\phi(q_2^2)\bigr).
\]
Indeed, the indicator is at most $s/u$, and the nonconstant part of the weight is at most
$\alpha\phi(q_2^2)$. Hence
\begin{align*}
 &\bbE_{\rho_1\times\rho_2}\Bigl[\mathbf1_{\{X_1,X_2\in F\}}\mathbf1_{\{s>u\}}\frac{q_1^1q_1^2}{(r^\square)^\alpha}\Bigr]\\
 &\quad\leq(u^{-1}+\alpha)\,\bbE_{\rho_1}\bigl[\mathbf1_{\{X_1\in F\}}\phi(q_1^1)\bigr]\,
 \bbE_{\rho_2}\bigl[\mathbf1_{\{X_2\in F\}}\bigl(\phi(q_2^1)+\phi(q_2^2)\bigr)\bigr]
 \leq2(u^{-1}+\alpha)M^2.
\end{align*}
Adding the two cases,
\begin{align*}
 &\bbE_{\rho_1\times\rho_2}\Bigl[\mathbf1_{\{X_1,X_2\in F\}}\frac{q_1^1q_1^2}{(r^\square)^\alpha}\Bigr]\\
 &\quad\leq\bbE_{\rho_1}\bigl[\mathbf1_{\{X_1\in F\}}\phi(q_1^1q_1^2)\bigr]
 +(2cL_\alpha+2u^{-1}+2\alpha)M^2.
\end{align*}
Exchanging the roles of $X_1$ and $X_2$ gives the same bound for
$q_2^1q_2^2/(r^\square)^\alpha$, with the expectation of $\phi(q_2^1q_2^2)$ under $\rho_2$ in
place of that of $\phi(q_1^1q_1^2)$ under $\rho_1$, since $r^\square$ and the hypotheses
are symmetric in the two source points. Together,
\begin{align*}
 &\bbE_{\rho_1\times\rho_2}\Bigl[\mathbf1_{\{X_1,X_2\in F\}}\frac{q_1^1q_1^2+q_2^1q_2^2}{(r^\square)^\alpha}\Bigr]\\
 &\quad\leq\sum_{i=1}^2\bbE_{\rho_i}\bigl[\mathbf1_{\{X_i\in F\}}\phi(q_i^1q_i^2)\bigr]
 +(4cL_\alpha+4u^{-1}+4\alpha)M^2.
\end{align*}

To estimate the two expectations of $\phi(q_i^1q_i^2)$, we retain the correction involving
$\beta$ in \cref{eq:four-law-product}. Its expectation involves the sum of the four failure
probabilities, which we denote by
\[
 Q_\Sigma\coloneqq\sum_{i,j=1}^2\bbE_{\rho_i}\bigl[q_{\tau_j}(X_i)\bigr]
 =\sum_{i,j=1}^2\bbE_{\tau_j}\bigl[q_{\rho_i}(Y_j)\bigr],
\]
where $Y_1\sim\tau_1$ and $Y_2\sim\tau_2$ are independent of each other and of $X_1,X_2$.
Both sides equal $\sum_{i,j}(\rho_i\times\tau_j)\{(x,y)\colon x\not\mathrel R y\}$, by the
symmetry of $R$. Since $F^c=\{r_{\tau_1}=0\}\cup\{r_{\tau_2}=0\}$, we have
$\rho_i(F^c)\leq2\delta_1$. Restricting each of the four expectations defining $Q_\Sigma$ to
$\{X_i\in F\}$ omits at most $2\delta_1$, since $q_{\tau_j}\leq1$. The total omitted contribution
is therefore at most $8\delta_1$. Applying \cref{eq:four-law-product} to $\phi(q_i^1q_i^2)$ in both
expectations, the combined contribution of the two product terms satisfies
\begin{equation}\label{eq:four-law-rows}
 \bbE_{\rho_1\times\rho_2}\Bigl[\mathbf1_{\{X_1,X_2\in F\}}\frac{q_1^1q_1^2+q_2^1q_2^2}{(r^\square)^\alpha}\Bigr]
 \leq2L_\alpha(\beta)M-\frac\beta2Q_\Sigma+4\beta \delta_1
 +(4cL_\alpha+4u^{-1}+4\alpha)M^2.
\end{equation}

We now estimate the terms containing $c_1,c_2$. For every $i,j\in\{1,2\}$ and every
target point $y$, \cref{eq:four-law-weight} and symmetry give
\[
 \bbE_{\rho_i}\bigl[\mathbf1_{\{X_i\not\mathrel R y\}}W_{\tau_j}(X_i)\bigr]
 \leq q_{\rho_i}(y)+\alpha M.
\]
For $c_1$, we use the first bound in \cref{eq:four-law-denominators} and expand the
definition of $c_1$ as a sum against $\tau_1$, so that for realisations in  $\{X_1,X_2\in F\}$,
\[
 \frac{c_1}{(r^\square)^\alpha}
 \leq c_1W_{\tau_1}(X_1)W_{\tau_2}(X_2)
 =\sum_y\tau_1\{y\}\,
 \mathbf1_{\{X_1\not\mathrel R y\}}W_{\tau_1}(X_1)\,
 \mathbf1_{\{X_2\not\mathrel R y\}}W_{\tau_2}(X_2).
\]
The right-hand side is nonnegative, so we may drop the restriction to $F$ and exchange the
sum with the expectation. By independence of $X_1$ and $X_2$, each term is then a product
of two expectations, each bounded by the preceding display, giving
\begin{align*}
 \bbE_{\rho_1\times\rho_2}\Bigl[\mathbf1_{\{X_1,X_2\in F\}}\frac{c_1}{(r^\square)^\alpha}\Bigr]
 &\leq\sum_y\tau_1\{y\}\bigl(q_{\rho_1}(y)+\alpha M\bigr)\bigl(q_{\rho_2}(y)+\alpha M\bigr)\\
 &=\bbE_{\tau_1}\bigl[\bigl(q_{\rho_1}(Y_1)+\alpha M\bigr)\bigl(q_{\rho_2}(Y_1)+\alpha M\bigr)\bigr].
\end{align*}
For $c_2$, we use the crossed pairing in \cref{eq:four-law-denominators} and obtain the
same bound with $\tau_2$ and $Y_2$ in place of $\tau_1$ and $Y_1$. The bounds for the
reversed comparisons, together with \cref{eq:four-law-square}, give
\begin{align*}
 \bbE_{\tau_j}\bigl[q_{\rho_i}(Y_j)\bigr]&\leq M+\delta_1,\\
 \bbE_{\tau_j}\bigl[q_{\rho_i}(Y_j)^2\bigr]
 &\leq K_\alpha(\beta)M
 +\beta\,\bbE_{\tau_j}\bigl[q_{\rho_i}(Y_j)\bigr]+(1-\beta)\delta_1.
\end{align*}
In the second bound the last term accounts for $q_{\rho_i}(Y_j)=1$, where the square is
$1$ and the term $\beta q_{\rho_i}(Y_j)$ contributes only $\beta$. Adding the bounds for
$c_1$ and $c_2$, we obtain
\begin{align}
 \bbE_{\rho_1\times\rho_2}\Bigl[\mathbf1_{\{X_1,X_2\in F\}}\frac{c_1+c_2}{(r^\square)^\alpha}\Bigr]
 &\leq\sum_{j=1}^2\bbE_{\tau_j}\bigl[q_{\rho_1}(Y_j)q_{\rho_2}(Y_j)\bigr]
 +\alpha M\sum_{i,j=1}^2\bbE_{\tau_j}\bigl[q_{\rho_i}(Y_j)\bigr]+2\alpha^2M^2\notag\\
 &\leq\frac12\sum_{i,j=1}^2\bbE_{\tau_j}\bigl[q_{\rho_i}(Y_j)^2\bigr]+\alpha MQ_\Sigma+2\alpha^2M^2\notag\\
 &\leq2K_\alpha(\beta)M+\frac\beta2Q_\Sigma+2(1-\beta)\delta_1+(4\alpha+2\alpha^2)M^2+4\alpha M\delta_1.
 \label{eq:four-law-overlaps}
\end{align}
The first step expands the products, the second applies
$2q_{\rho_1}q_{\rho_2}\leq q_{\rho_1}^2+q_{\rho_2}^2$ and the definition of $Q_\Sigma$, and
the third applies the two displays above and $Q_\Sigma\leq4(M+\delta_1)$. The terms containing
$Q_\Sigma$ cancel between \cref{eq:four-law-rows,eq:four-law-overlaps}, and adding the two
bounds gives
\[
 \bbE_{\rho_1\times\rho_2}\bigl[\mathbf1_{\{X_1,X_2\in F\}}\phi(q^\square)\bigr]
 \leq\lambda_\alpha(\beta)M
 +(4cL_\alpha+4u^{-1}+8\alpha+2\alpha^2)M^2+(2+2\beta+4\alpha M)\delta_1.
\]
The coefficient of the $M^2$ term is $C_\alpha(u)$ by \cref{eq:four-law-constants}, so this is the
asserted bound on $\{X_1,X_2\in F\}$, except for the term involving $\delta_2$.

It remains to consider the source pairs outside $F\times F$ that have positive matching
probability. By definition of the restricted potential, and since $r^\square\geq
r_1^1r_2^2>0$ on $\{X_1,X_2\in F\}$,
\begin{align*}
 \cE(\rho_1\times\rho_2,\tau_1\times\tau_2;R^\square)
 &=\bbE_{\rho_1\times\rho_2}\bigl[\mathbf1_{\{X_1,X_2\in F\}}\phi(q^\square)\bigr]\\
 &\quad+\bbE_{\rho_1\times\rho_2}\bigl[\mathbf1_{\{r^\square>0\}\setminus\{X_1,X_2\in F\}}\phi(q^\square)\bigr].
\end{align*}
The first expectation has just been bounded. For the second, let
\[
 A_1\coloneqq\{r_{\tau_1}(X_1)=0\},\qquad
 A_2\coloneqq\{r_{\tau_2}(X_1)=0\},
\]
and let $A_3$ and $A_4$ be the corresponding events for $X_2$. Since
$F^c=\{r_{\tau_1}=0\}\cup\{r_{\tau_2}=0\}$, the complement of $\{X_1,X_2\in F\}$ is
$A_1\cup A_2\cup A_3\cup A_4$.

For $A_1\cap\{r^\square>0\}$, the straight pairing, in which $X_1$ matches $Y_1$ and $X_2$
matches $Y_2$, has probability $r_{\tau_1}(X_1)r_{\tau_2}(X_2)=0$, so only the crossed
pairing can succeed and
\[
 r^\square=r_{\tau_2}(X_1)\,r_{\tau_1}(X_2)>0.
\]
Both factors are positive, so both weights are the corresponding powers, and since
$\phi(q)=q(1-q)^{-\alpha}\leq(1-q)^{-\alpha}$ for $q<1$,
\[
 \phi(q^\square)
 \leq(r^\square)^{-\alpha}
 =r_{\tau_2}(X_1)^{-\alpha}\,r_{\tau_1}(X_2)^{-\alpha}
 =W_{\tau_2}(X_1)\,W_{\tau_1}(X_2).
\]
The right-hand side is nonnegative and is a function of $X_1$ times a function of $X_2$.
Enlarging the event $A_1\cap\{r^\square>0\}$ to $A_1$, which depends on $X_1$ alone, and
using independence gives
\begin{align*}
 \bbE_{\rho_1\times\rho_2}\bigl[\mathbf1_{A_1\cap\{r^\square>0\}}\phi(q^\square)\bigr]
 &\leq\bbE_{\rho_1}\bigl[\mathbf1_{\{r_{\tau_1}(X_1)=0\}}W_{\tau_2}(X_1)\bigr]\,
 \bbE_{\rho_2}\bigl[W_{\tau_1}(X_2)\bigr]\\
 &\leq(1+\alpha M)\delta_2,
\end{align*}
by the weighted hypothesis with $i=j=1$ for the first factor and by
\cref{eq:four-law-moment} for the second.

For realisations in $A_2\cap\{r^\square>0\}$, the crossed pairing has probability
$r_{\tau_2}(X_1)r_{\tau_1}(X_2)=0$ instead, so $r^\square=r_{\tau_1}(X_1)r_{\tau_2}(X_2)$
and $\phi(q^\square)\leq W_{\tau_1}(X_1)W_{\tau_2}(X_2)$. The same argument gives the
bound $(1+\alpha M)\delta_2$, now by the weighted hypothesis with $i=1$, $j=2$. When $A_3$ and
$A_4$ hold, the roles of $X_1$ and $X_2$ are exchanged, and the weighted hypotheses with $i=2$
give the same bound. Since $\phi(q^\square)\geq0$, the expectation over
$\{r^\square>0\}\setminus\{X_1,X_2\in F\}$ is at most the sum of the four expectations over
$A_k\cap\{r^\square>0\}$, hence at most $4(1+\alpha M)\delta_2$. Adding this to the bound on
$\{X_1,X_2\in F\}$ proves \cref{eq:four-law-contraction}.
\end{proof}

\section{Labellings with zero matching probability}\label{sec:quantitative-stopping}

We now bound the two error terms in \cref{thm:four-law-contraction} for the laws of the
Markov model. We first show that both errors vanish when $b(0)=1$.

\subsection{Potentials and weights by type}\label{sec:type-potentials}\label{sec:positive-degrees}

For each integer $h\geq0$, put
\[
 \cL_h\coloneqq\{x\colon\bbB_h\to V\}.
\]
For $t\in\cI$, let $\rho_{t,h}$ be the law on the countable set $\cL_h$ of the labelling
generated from root type $t$, with the type information discarded. For $x,y\in\cL_h$, we write
$x\approx_h y$ when an automorphism of $\bbB_h$ matches their labels.
Using the notation of \cref{sec:restricted-potential}, for $x\in\cL_h$ we set
\[
 r_{t,h}(x)\coloneqq r_{\rho_{t,h}}(x)=\rho_{t,h}\{y\in\cL_h\colon x\approx_h y\},\qquad
 W_{t,h}(x)\coloneqq W_{\rho_{t,h}}(x),
\]
and, for $s,t\in\cI$,
\[
 \cE_h(s,t)\coloneqq\cE(\rho_{s,h},\rho_{t,h};\approx_h)
 =\sum_{\substack{x\in\cL_h\\r_{t,h}(x)>0}}\rho_{s,h}\{x\}\,\phi\bigl(1-r_{t,h}(x)\bigr).
\]

To bound the probability of zero matching degree uniformly over the types being compared,
we set
\[
 \delta_{1,h}\coloneqq
 \begin{cases}
  \displaystyle\sup_{s,t\in\cI}\rho_{s,h}\{x\in\cL_h\colon r_{t,h}(x)=0\},&b(0)=1,\\[4pt]
  \displaystyle\max_{\substack{0\leq i<g\\s,t\in\cI_i}}
  \rho_{s,h}\{x\in\cL_h\colon r_{t,h}(x)=0\},&b(0)<1.
 \end{cases}
\]
Thus we compare all initial types when $b(0)=1$, and types in the same class of the
partition $\cI_0,\ldots,\cI_{g-1}$ from the finite-type hypotheses otherwise.
In \cref{thm:four-law-contraction}, $\delta_{1,h}$ bounds the first error, the probability that a
source labelling has no possible match against a specified target type, and the second
error sums the weight against a type $u$ over the labellings that cannot match a
type $t$. For any types $s,u$, \cref{eq:tangent} gives
\begin{equation}\label{eq:moment}
 \sum_{x\in\cL_h}\rho_{s,h}\{x\}\,W_{u,h}(x)\leq1+\alpha \cE_h(s,u).
\end{equation}

When $b(0)=1$, we show by induction on height that every source labelling of positive
probability has positive matching degree against every target type. At height zero, this
follows because each label in $\supp\mu$ matches itself and $0$ is compatible with every
label in $\supp\mu$. For the induction step, choose a source transition that produces the
given child labellings with positive probability and any target transition of positive
probability. By induction, both child comparisons have positive matching degree. Conditioned
on the target transition, the target root and its child subtrees are independent, so the full
labelling also has positive matching degree. Thus, for all $s,t$ and $h\geq0$,
\begin{equation}\label{eq:positive-degrees}
 \rho_{s,h}\{x\in\cL_h\colon r_{t,h}(x)=0\}=0.
\end{equation}
Both error terms in \cref{thm:four-law-contraction} therefore vanish, including when the
set of types is infinite.

\subsection{The probability of an impossible match}\label{sec:unweighted}

Throughout this subsection and the following subsection, we assume the finite-type
hypotheses of \cref{thm:markov-matching}.
An impossible comparison can pass from a source labelling to one of its children. We show
that, within $H$ steps, either a source label is incompatible with $0$ or both source
children have an impossible comparison.

We write
\[
 T\coloneqq\max_{0\leq i<g}\#\cI_i
\]
for the largest number of types in one class. This bounds the number of target types that
can occur in a comparison at a given depth.
As we descend the tree, one source labelling may fail against several target types
simultaneously. For a source type $s\in\cI_i$ and a nonempty set $\cJ\subseteq\cI_i$, we
therefore consider the event
\[
 \cU_h(\cJ)\coloneqq
 \bigl\{x\in\cL_h\colon r_{t,h}(x)=0\text{ for every }t\in\cJ\bigr\}
\]
that a source labelling at height $h$ has degree zero against every type in $\cJ$. For a
singleton $\cJ=\{t\}$ with $s,t$ in one class, $\delta_{1,h}$ bounds $\rho_{s,h}(\cU_h(\{t\}))$.
The possible target child types form the set
\[
 \cJ'\coloneqq
 \bigcup_{t\in \cJ}\ \bigcup_{(t_1,t_2)\in\supp P_t}\{t_1,t_2\}
 \subseteq\cI_{i+1}.
\]
We include every transition of positive probability. Thus $\cJ'$ is nonempty, and it has
at most $T$ members. Both source child types also lie in $\cI_{i+1}$.

\begin{lemma}\label{thm:zero-degree-dichotomy}
Let $h\geq1$, and let $x\in\cL_h$ be a source labelling of positive probability whose root is
compatible with $0$, with child labellings $x_1,x_2\in\cL_{h-1}$. If $x\in\cU_h(\cJ)$, then
at least one of the following two conditions holds.
\begin{enumerate}[(i),ref=(\roman*)]
\item\label{it:dichotomy-one-child} One source child cannot match any type in $\cJ'$:
\[
 \exists k\in\{1,2\}\ \forall t\in \cJ',\qquad r_{t,h-1}(x_k)=0,
\]
that is, $x_k\in\cU_{h-1}(\cJ')$ for some $k$.
\item\label{it:dichotomy-both-children} Each source child has an impossible comparison,
with possibly different target types:
\[
 \forall k\in\{1,2\}\ \exists t\in \cJ',\qquad r_{t,h-1}(x_k)=0,
\]
that is, $x_1\in\cU_{h-1}(\{t_1\})$ and $x_2\in\cU_{h-1}(\{t_2\})$ for some
$t_1,t_2\in \cJ'$.
\end{enumerate}
\end{lemma}

\begin{proof}
The root comparison has positive probability against every type in $\cJ$. A target root
with prescribed label $0$ has compatible mass one, and a target root with law $\mu$ has
positive compatible mass, since the source label belongs to $\supp\mu\cup\{0\}$ and
$b(0)>0$. As $x\in\cU_h(\cJ)$, neither pairing of its children can therefore have positive
matching probability against any target child pair available from $\cJ$.
Suppose \labelcref{it:dichotomy-one-child} fails. Each source child then has positive
degree against some member of $\cJ'$. Choose an available target pair containing a type that
the first source child can match. The second source child must have degree zero against the
other type in that pair, or the pairing would succeed. Now choose a target pair containing
a type that the second source child can match. The first source child must have degree
zero against the other type. This gives \labelcref{it:dichotomy-both-children}.
\end{proof}

\paragraph*{\bf Source positions} There are $2^d$ source positions at depth $d$. The total number through depth $H$ is
\[
 S_H\coloneqq\sum_{d=0}^{H}2^d=2^{H+1}-1.
\]

\begin{lemma}\label{thm:explicit-zero-bound}
We have
\[
 \delta_{1,0}\leq1-b(0),\qquad
 \delta_{1,h}\leq S_H(1-b(0))+S_HT^2\left(\max_{j<h}\delta_{1,j}\right)^2\quad(h\geq1).
\]
If $4S_H^2T^2(1-b(0))\leq1$, then $\delta_{1,h}\leq Z$ for all $h$, where
\begin{equation}\label{eq:explicit-zero-bound}
 Z\coloneqq
 \frac{2S_H(1-b(0))}{1+\sqrt{1-4S_H^2T^2(1-b(0))}}\leq2S_H(1-b(0)).
\end{equation}
\end{lemma}

\begin{proof}
We apply \cref{thm:zero-degree-dichotomy} recursively. In \labelcref{it:dichotomy-one-child},
we repeat it for the event $x_k\in\cU_{h-1}(\cJ')$, including both choices of $k$ in our upper
bound. In \labelcref{it:dichotomy-both-children}, we bound the two events
$x_1\in\cU_{h-1}(\{t_1\})$ and $x_2\in\cU_{h-1}(\{t_2\})$ separately and stop descending.
We also stop at a source label incompatible with $0$ or at height zero. Whenever we use
independence of the source children, we condition on their type transition and then average
with its original law. This introduces no factor for the number of source transitions.

We next show why the recursion stops within $H$ steps. Each successive use of
\labelcref{it:dichotomy-one-child} follows a possible source type path. Iterating
$\cJ\mapsto \cJ'$ retains every type reached
at the corresponding depth by a possible target path from $\cJ$. Fix any initial target
type in $\cJ$. Condition \labelcref{it:markov-returns} gives a depth at most $H$ at which
the source type and an available target type both belong to $\cI_\mu$.
Condition \labelcref{it:markov-positivity} then gives positive matching probability for
the realised source subtree against that target type. This contradicts the condition of
degree zero against every current target type. Thus the recursion must already have
stopped, or reached height zero.

At any fixed source position, incompatibility with $0$ has probability at most $1-b(0)$.
In \labelcref{it:dichotomy-both-children}, each child has degree zero against at least one
of at most $T$ types in its own
class. At remaining height $j<h$, the union bound gives probability at most $T\delta_{1,j}$ for
each child. To bound their intersection, we discard the requirements imposed above this
position. Conditioned on the source transition, the two child labellings are independent
with their original laws. Their contribution is therefore at most
\[
 T^2\cdot\left(\max_{j<h}\delta_{1,j}\right)^2.
\]
At height zero, a root compatible with $0$ has positive degree against every target type,
so only root incompatibility can contribute.

Starting with $\cJ=\{t\}$, the event $\cU_h(\{t\})$ is contained in the union of these
terminal events over source positions through depth $H$. There are at most $S_H$ such
positions. Summing their bounds and maximising over the initial types gives the recurrence.

The value in \cref{eq:explicit-zero-bound} is the smaller nonnegative solution of
\[
 Z=S_H\cdot(1-b(0))+S_HT^2Z^2.
\]
It is at least $1-b(0)$, so the bound holds at height zero. If it holds at all smaller
heights, the recurrence gives $\delta_{1,h}\leq Z$.
\end{proof}

We now use this probability bound to control the weighted zero-degree sum in
\cref{thm:four-law-contraction}.

\subsection{The weighted mass of an impossible match}\label{sec:weighted-mass}

Continuing with the finite-type hypotheses, we now estimate the mass of a zero-degree event under
the weight $W_{u,h}$. The type $u$ in this weight may differ from every type defining the
event, as it does in the error term $\delta_2$ of \cref{thm:four-law-contraction}.
For a source root label $v\in\supp\mu\cup\{0\}$ and target type $u$, the
root factor in the weight is
\[
 \begin{cases}
  b(v)^{-\alpha},&u\in\cI_\mu,\\
  \mathbf1_{\{v\in B_G(0,1)\}},&u\notin\cI_\mu.
 \end{cases}
\]
We separate the contribution of source labels incompatible with $0$ from the bound on
compatible labels by putting
\begin{equation*}
 \begin{split}
 I_{\mathrm{bad}}&\coloneqq\sum_{\substack{v\notin B_G(0,1)\\\mu(v)>0}}\mu(v)b(v)^{-\alpha},\\
 R_\mu&\coloneqq
 \max\left\{1,\ b(0)^{-\alpha},\
             \sum_{\substack{v\in B_G(0,1)\\\mu(v)>0}}\mu(v)b(v)^{-\alpha}\right\}.
 \end{split}
\end{equation*}

By the weight inequality \cref{eq:tangent},
\begin{equation}\label{eq:root-parameter-bounds}
 1-b(0)\leq\zeta,\qquad
 I_{\mathrm{bad}}\leq1-b(0)+\alpha\eta\leq(1+\alpha)\zeta,\qquad
 R_\mu\leq1+\alpha\zeta.
\end{equation}
These bounds require no positive mass at $0$. If $\mu(0)\geq p>0$, we also have
\begin{equation*}
 \frac{1-b(0)}{b(0)^\alpha}\leq\frac\eta p,\qquad
 I_{\mathrm{bad}}\leq\frac\eta p,\qquad
 1-b(0)\leq\frac{(1-p)^\alpha}{p}\eta.
\end{equation*}
For the last two inequalities, a label incompatible with $0$ has
$1-b(v)\geq p$ and $b(v)\leq1-p$. The first follows from the summand of $\eta$ at $0$.

The source root label is independent of its transition and child labellings. We can
therefore use these averaged root factors at each step. To pass the weight to the children,
we first account for the target transition mixture.

Fix a height $h\geq1$, a source child pair $(x_1,x_2)\in\cL_{h-1}^2$, and a target type $u$.
A transition $j=(j_1,j_2)\in\supp P_u$ specifies the two target child
types. We write $r_j(x_1,x_2)$ for the matching probability against this fixed pair of
types, allowing either pairing, and $r(x_1,x_2)$ for its average over the target transition:
\[
 \begin{split}
 r_j(x_1,x_2)&\coloneqq(\rho_{j_1,h-1}\times\rho_{j_2,h-1})
 \bigl\{(y_1,y_2)\in\cL_{h-1}^2\colon
 (x_1,x_2)\mathrel{(\approx_{h-1})^\square}(y_1,y_2)\bigr\},\\
 r(x_1,x_2)&\coloneqq
 \sum_{j\in\supp P_u}P_u(j)r_j(x_1,x_2).
 \end{split}
\]
Suppressing the source pair in the notation, we have $r\geq P_u(j)r_j>0$ for at least one
$j$ whenever $r>0$. Hence
\begin{equation}\label{eq:selected-inverse-mixture}
 \mathbf1_{\{r>0\}}r^{-\alpha}
 \leq\sum_{j\in\supp P_u}P_u(j)^{-\alpha}
       \mathbf1_{\{r_j>0\}}r_j^{-\alpha}.
\end{equation}
Here and below, a restricted inverse power is understood to be zero when its degree is
zero. The sum of the coefficients in this expansion is at most
\begin{equation}\label{eq:transition-budget}
 B\coloneqq\max_{t\in\cI}
 \sum_{j\in\supp P_t}P_t(j)^{-\alpha}<\infty.
\end{equation}

For the selected-transition refinement in the theorem, its hypothesis ensures that at least
one component indexed by $J_u$ has positive degree whenever $r(x_1,x_2)>0$. The argument
giving \cref{eq:selected-inverse-mixture} therefore applies with the sum restricted to
$J_u$, and in all the following bounds we may replace $B$ by
\[
 \max_{t\in\cI}\sum_{j\in J_t}P_t(j)^{-\alpha}.
\]
The sets $\cJ'$ in the preceding argument still use every possible target transition.

We assume $4S_H^2T^2(1-b(0))\leq1$ and use $Z$ from \cref{eq:explicit-zero-bound}.
We write
\[
 G_H(x)\coloneqq\sum_{d=0}^{H}x^d\qquad(x\geq0).
\]
For a bound $M\geq0$ on the restricted potentials at smaller heights, define
\begin{equation}\label{eq:explicit-weighted-error}
 E(M)\coloneqq
 G_H\bigl(4BR_\mu(1+\alpha M)\bigr)
 \left[2B(1+\alpha M)^2I_{\mathrm{bad}}+2BR_\mu\bigl(TZ+\alpha M\bigr)^2\right].
\end{equation}
The two terms in brackets correspond to an incompatible root and to a restriction on both
source children. The finite sum accounts for the steps before either event occurs.

\begin{lemma}\label{thm:explicit-weighted-bound}
Assume $4S_H^2T^2(1-b(0))\leq1$. Let $h\geq1$ and $M\geq0$, and suppose
\[
 \cE_k(s,u)\leq M\qquad
 (0\leq k<h,\ 0\leq i<g,\ s,u\in\cI_i).
\]
For $s,u\in\cI_i$ and a nonempty set $\cJ\subseteq\cI_i$,
\[
 \sum_{x\in\cU_h(\cJ)}\rho_{s,h}\{x\}\,W_{u,h}(x)
 \leq E(M).
\]
At height zero the same bound holds with $M=0$.
\end{lemma}

\begin{proof}
We use the cases \labelcref{it:dichotomy-one-child,it:dichotomy-both-children} of
\cref{thm:zero-degree-dichotomy} and the recursion in the proof of
\cref{thm:explicit-zero-bound}, now weighting the event inclusions by $W_{u,h}$. First we
pass this weight to the two children. If $r_j>0$, at least one pairing has a positive product
of child degrees, and $r_j$ is at least this product. Hence
\[
 \mathbf1_{\{r_j>0\}}r_j^{-\alpha}
 \leq W_{j_1,h-1}(x_1)W_{j_2,h-1}(x_2)
 +W_{j_2,h-1}(x_1)W_{j_1,h-1}(x_2).
\]
We combine this bound with \cref{eq:selected-inverse-mixture}. The transition coefficients
sum to at most $B$, and there are two pairings. The root factor, summed against the source
root law over the labels $v$ compatible with $0$, is at most $R_\mu$. Thus passing the
weight to the children gives total coefficient at most $2BR_\mu$. The realised source child types, the target types
in $\cJ'$, and the types $j_1,j_2$ in the weights all belong to the next class. The moment
bounds at smaller heights therefore apply to every child comparison below.

In case \labelcref{it:dichotomy-one-child}, one child must still have degree zero against
every type in $\cJ'$. We continue the recursion on this child with its corresponding weight
and discard any further restrictions on the sibling. Conditioned on the source transition,
the resulting sum factors by independence. The sibling's weighted sum is at most $1+\alpha M$ by
\cref{eq:moment}. Including both choices of the child on which we continue gives a factor
at most
\[
 4BR_\mu(1+\alpha M)
\]
per step. Summing the conditional bounds over the source child types $(s_1,s_2)$ with
their probabilities $P_s(s_1,s_2)$ leaves this coefficient unchanged.

In case \labelcref{it:dichotomy-both-children}, each child must have degree zero against
at least one type in $\cJ'$. For any height $k\geq0$, types $s,u$ in one class, and a set
$\cJ$ in that class, the weight bound
\cref{eq:four-law-weight} gives
\begin{equation}\label{eq:zero-union-moment}
 \sum_{\substack{x\in\cL_k\\r_{t,k}(x)=0\text{ for some }t\in\cJ}}\rho_{s,k}\{x\}\,W_{u,k}(x)
 \leq\#\cJ\,\delta_{1,k}+\alpha \cE_k(s,u)
 \leq T\delta_{1,k}+\alpha \cE_k(s,u).
\end{equation}
The constant part of the weight contributes the mass of the union, which is at most
$\#\cJ\,\delta_{1,k}$ by the union bound, and its potential part is summed over the whole source
space. Applied at the smaller height, this bounds each child's weighted sum by
$TZ+\alpha M$.
We enlarge the source event to the product of these two child unions. Conditioned on the
source transition, the weighted sum then factors. Using the coefficient $2BR_\mu$ obtained
above, the contribution of case \labelcref{it:dichotomy-both-children} is at most
\[
 2BR_\mu(TZ+\alpha M)^2.
\]

The root factor, summed against the source root law over labels incompatible with $0$,
is at most $I_{\mathrm{bad}}$. The transition expansion and two pairings have total
coefficient at most $2B$, and the independent child weights each have expectation at most
$1+\alpha M$, after discarding the zero-degree restrictions on the children. The contribution
is therefore at most
\[
 2B(1+\alpha M)^2I_{\mathrm{bad}}.
\]
At height zero it is at most $I_{\mathrm{bad}}$, already bounded by this expression since $B\geq1$.

As in the unweighted proof, conditions \labelcref{it:markov-returns,it:markov-positivity}
prevent case \labelcref{it:dichotomy-one-child} from continuing beyond $H$ steps. After $d$
such steps, the coefficient
is at most $(4BR_\mu(1+\alpha M))^d$. Summing the two possible terminal contributions
over $0\leq d\leq H$ gives \cref{eq:explicit-weighted-error}. Every child or sibling
potential used in this calculation has height strictly smaller than $h$.
\end{proof}

For fixed $\alpha$ and bounded $H,T,B$, \cref{eq:root-parameter-bounds} bounds $Z$ by a
constant times $\zeta$ and $E(M)$ by a constant times $\zeta+M^2$ for $\zeta,M\leq1$, so the
two error terms add no linear term in $M$ to the local estimate and leave its contraction
intact. This will be made explicit in \cref{sec:completion}.

\section{The scalar induction}\label{sec:completion}

We now seek a bound $M$ on the restricted potentials that holds at every height. Throughout
this section, we compare all pairs of types when $b(0)=1$, and pairs in the same class
otherwise. Choose $\beta_\ast\in[0,1]$ attaining the minimum in
\cref{eq:mean-constants}. Thus
\[
 \lambda_\alpha(\beta_\ast)=\lambda_\alpha<1.
\]
We use $u=1/2$ in the four-law estimate. When $b(0)=1$, we set $Z=E(M)=0$.
Otherwise, $Z$ and $E(M)$ are the bounds from
\cref{eq:explicit-zero-bound,eq:explicit-weighted-error}, defined under the condition
$4S_H^2T^2(1-b(0))\leq1$. Put
\[
 C\coloneqq
 \begin{cases}
  4,&b(0)=1,\\
  4+8(\alpha+1)B,&b(0)<1,
 \end{cases}
\]
and define
\begin{equation}\label{eq:closure-map}
 Q(M)\coloneqq
 \lambda_\alpha M+C_\alpha(1/2)M^2
 +(2+2\beta_\ast+4\alpha M)Z+C(1+\alpha M)E(M).
\end{equation}
Finally, let
\[
 A_\mu\coloneqq\max\left\{1,\
       \alpha\eta+\sum_{v\colon\mu(v)>0}\mu(v)b(v)^{1-\alpha},\
       \alpha\phi(1-b(0))+b(0)^{1-\alpha}\right\}.
\]

The following lemma is the one-step estimate used in the height induction.

\begin{lemma}\label{thm:one-step-closure}
Let $h\geq0$ and $M\geq0$. Suppose that
\[
 \cE_k(s,t)\leq M\qquad(0\leq k\leq h)
\]
for all $s,t\in\cI$ when $b(0)=1$, and for all $s,t$ in the same class otherwise. When
$b(0)<1$, assume also that $4S_H^2T^2(1-b(0))\leq1$. Then,
\[
 \cE_{h+1}(s,t)\leq\zeta+A_\mu Q(M).
\]
Moreover,
\[
 A_\mu\leq1+(2\alpha-1)\zeta.
\]
\end{lemma}

\subsection{Averaging the transitions}\label{sec:averaging}

The target child law is a mixture over its type transition. Some components can have zero
matching degree even when the mixture has positive degree. A convexity argument bounds the
potential weight with a correction involving the total probability of those components.

\begin{lemma}\label{thm:zero-mixture-convexity}
Let $\alpha\geq1$, let $(p_1,\ldots,p_m)$ be a probability vector, and let
$r_1,\ldots,r_m\in[0,1]$. Put
\[
 r\coloneqq\sum_{j=1}^m p_jr_j,\qquad
 w\coloneqq\sum_{j\colon r_j=0}p_j.
\]
Assume that $r>0$. Then
\[
 \frac{1-r}{r^\alpha}
 \leq\sum_{j\colon r_j>0}p_j\frac{1-r_j}{r_j^\alpha}
       +(\alpha+1)wr^{-\alpha}.
\]
\end{lemma}

\begin{proof}
Since $r>0$ we have $w<1$. The positive-degree components have total weight $1-w$, so their
normalised mean degree is
\[
 \sum_{j\colon r_j>0}\frac{p_j}{1-w}r_j=\frac{r}{1-w}.
\]
The second derivative of $x\mapsto(1-x)x^{-\alpha}$ is
\[
 \alpha x^{-\alpha-2}\bigl((\alpha+1)-(\alpha-1)x\bigr)\geq0
 \qquad(0<x\leq1),
\]
so Jensen's inequality gives
\[
 \sum_{j\colon r_j>0}p_j\frac{1-r_j}{r_j^\alpha}
 \geq(1-w)^\alpha(1-w-r)r^{-\alpha}.
\]
The difference between $(1-r)r^{-\alpha}$ and the lower bound on the right is
\[
 \frac{1-(1-w)^{\alpha+1}-r\bigl(1-(1-w)^\alpha\bigr)}{r^\alpha}
 \leq\frac{1-(1-w)^{\alpha+1}}{r^\alpha}
 \leq(\alpha+1)wr^{-\alpha},
\]
where the first inequality uses $(1-w)^\alpha\leq1$, and the second is the tangent bound
for $(1-w)^{\alpha+1}$ at $w=0$.
\end{proof}

\begin{lemma}\label{thm:averaged-child-pair}
Under the hypotheses of \cref{thm:one-step-closure}, the potential of the two height-$h$
child labellings, after averaging the source and target transitions but before including
the root labels, is at most $Q(M)$.
\end{lemma}

\begin{proof}
The zero-degree probabilities at height $h$ are at most $Z$, independently of the assumed
potential bound. The weighted zero-degree sums are at most $E(M)$. When $b(0)<1$ and
$h=0$, we use the height-zero clause of \cref{thm:explicit-weighted-bound} and
$E(0)\leq E(M)$.

Conditioned on their type transitions, the source and target each have two independent
child subtrees of height $h$. Write $\rho_1,\rho_2$ for these source laws and
$\tau_1,\tau_2$ for the target laws. Their matching potential is
\[
 \cE^\square\coloneqq
 \cE(\rho_1\times\rho_2,\tau_1\times\tau_2;(\approx_h)^\square).
\]
When $b(0)<1$, all four child types lie in the same class at the next level. In either
case, the assumed bounds apply to all forward and reversed comparisons in
\cref{thm:four-law-contraction}. Hence
\begin{equation}\label{eq:closure-local-input}
 \cE^\square
 \leq\lambda_\alpha M+C_\alpha(1/2)M^2
 +(2+2\beta_\ast+4\alpha M)Z+4(1+\alpha M)E(M).
\end{equation}

First suppose $b(0)<1$, and keep the source child types $s_1,s_2$ fixed while averaging
the target transition. Write $p_j$ for the probability of a target transition
$j=(j_1,j_2)$, and let $r_j(x_1,x_2)$ be the matching probability of the source child pair
against types $j_1,j_2$, as in \cref{eq:selected-inverse-mixture}, now with child height
$h$. Put
\[
 r(x_1,x_2)\coloneqq\sum_jp_jr_j(x_1,x_2),\qquad
 w(x_1,x_2)\coloneqq\sum_{j\colon r_j(x_1,x_2)=0}p_j.
\]
For each source pair $(x_1,x_2)\in\cL_h^2$ with $r(x_1,x_2)>0$, apply
\cref{thm:zero-mixture-convexity}. If $r_j(x_1,x_2)=0$, neither pairing succeeds, so
\[
 \bigl\{(x_1,x_2)\in\cL_h^2\colon r_j(x_1,x_2)=0\bigr\}
 \subseteq
 \bigcup_{i,k=1}^2\bigl\{(x_1,x_2)\in\cL_h^2\colon r_{j_k,h}(x_i)=0\bigr\}.
\]
This gives four possible zero-degree events for the source children.

For $(x_1,x_2)\in\cL_h^2$ with $r(x_1,x_2)>0$, we expand the inverse degree in the correction
term using \cref{eq:selected-inverse-mixture}. Its coefficients sum to at most $B$.
Each retained target transition gives two products of child weights, one for each pairing.
The type defining a zero-degree event may differ from the target type in its weight.
For example, let $(\ell_1,\ell_2)$ be a retained target transition and $k\in\{1,2\}$.
Bounding by the zero-degree event and using independence, a zero-degree condition on the
first child gives
\[
 \begin{split}
 &\sum_{\substack{x_1\in\cL_h\\r_{j_k,h}(x_1)=0}}\ \sum_{x_2\in\cL_h}
 \rho_{s_1,h}\{x_1\}\rho_{s_2,h}\{x_2\}\,
 W_{\ell_1,h}(x_1)W_{\ell_2,h}(x_2)\\
 &\qquad=
 \left(\sum_{\substack{x_1\in\cL_h\\r_{j_k,h}(x_1)=0}}\rho_{s_1,h}\{x_1\}\,W_{\ell_1,h}(x_1)\right)
 \left(\sum_{x_2\in\cL_h}\rho_{s_2,h}\{x_2\}\,W_{\ell_2,h}(x_2)\right)
 \leq E(M)(1+\alpha M).
 \end{split}
\]
The bound is the same when the second child has the zero-degree condition or the weights
are paired in the other order.

Expanding $w$, applying \cref{eq:selected-inverse-mixture}, using
$\mathbf1_{\{r_j(x_1,x_2)=0\}}\leq\sum_{i,k=1}^2\mathbf1_{\{r_{j_k,h}(x_i)=0\}}$, and applying the
two-pairing child-weight bound from the proof of \cref{thm:explicit-weighted-bound} bounds
the expectation of the correction term. After extending the sum to all
$(x_1,x_2)\in\cL_h^2$, the bound without the factor $\alpha+1$ is
\[
 \begin{split}
 &\sum_jp_j\sum_{\ell\ \mathrm{retained}}p_\ell^{-\alpha}
 \sum_{i,k=1}^2\sum_{x_1,x_2\in\cL_h}
 \rho_{s_1,h}\{x_1\}\rho_{s_2,h}\{x_2\}\,
 \mathbf1_{\{r_{j_k,h}(x_i)=0\}}\\
 &\hspace{7cm}\cdot\bigl(
 W_{\ell_1,h}(x_1)W_{\ell_2,h}(x_2)
 +W_{\ell_2,h}(x_1)W_{\ell_1,h}(x_2)
 \bigr).
 \end{split}
\]
For each $j$ and $\ell$, each of the eight inner sums indexed by $i,k$ and the two pairings
is at most $E(M)(1+\alpha M)$ by the calculation above. Moreover,
$\sum_jp_j=1$ and $\sum_{\ell\ \mathrm{retained}}p_\ell^{-\alpha}\leq B$. Multiplying by
$\alpha+1$ therefore gives
\[
 (\alpha+1)\sum_{\substack{x_1,x_2\in\cL_h\\r(x_1,x_2)>0}}
 \rho_{s_1,h}\{x_1\}\rho_{s_2,h}\{x_2\}\,
 w(x_1,x_2)r(x_1,x_2)^{-\alpha}
 \leq8(\alpha+1)B(1+\alpha M)E(M).
\]
The first term in the convexity bound is the average of the fixed-pair potentials, each
bounded by \cref{eq:closure-local-input}. Averaging over the source transition preserves
this bound because its transition probabilities sum to one.

When $b(0)=1$, \cref{eq:positive-degrees} gives $w=0$ for every source pair of positive
probability, so the correction term is absent and the ordinary convexity bound applies, by
Jensen's inequality when the transition law is countable. In this case $E(M)=0$, so the
coefficient $C=4$ in \cref{eq:closure-map} has no effect. Thus, after averaging both
transitions, the child-pair potential is at most $Q(M)$.
\end{proof}

\subsection{The independent root}\label{sec:independent-root}

\begin{proof}[Proof of \cref{thm:one-step-closure}]
By \cref{thm:averaged-child-pair}, the child-pair potential is at most $Q(M)$. We now
include the root labels. First suppose the target root has law $\mu$, and fix a
source root label $v\in\supp\mu\cup\{0\}$. Its compatible target mass is
$b(v)>0$. For a child-pair matching probability $0<r\leq1$, independence of the target
root and children gives full matching degree $b(v)r$, and
\[
 \frac{1-b(v)r}{(b(v)r)^\alpha}
 =\phi(1-b(v))r^{-\alpha}+b(v)^{1-\alpha}\phi(1-r).
\]
We sum only over child pairs with positive degree. Their potential is at most
$Q(M)$, so their inverse moment is at most $1+\alpha Q(M)$ by the weight bound
\cref{eq:four-law-weight}. The source root label is independent of its child types and
subtrees, so the two terms factor when summed over the source law.
If the source root also has law $\mu$, we obtain the full potential bound
\[
 \eta+\left(\alpha\eta+
       \sum_{v\colon\mu(v)>0}\mu(v)b(v)^{1-\alpha}\right)Q(M).
\]
A source root with prescribed label $0$ instead gives the bound
\[
 \phi(1-b(0))+
 \bigl(\alpha\phi(1-b(0))+b(0)^{1-\alpha}\bigr)Q(M).
\]
If the target root has prescribed label $0$, fix a source root label
$v\in\supp\mu\cup\{0\}$. If $v\notin B_G(0,1)$, the root labels are incompatible, so
every corresponding full matching degree is zero and its term is omitted from the restricted
potential. If $v\in B_G(0,1)$, the root labels are compatible and the full matching degree
equals the child-pair degree. The conditional potential is therefore at most $Q(M)$.
Averaging over a random source root label, or retaining the prescribed source label $0$, does
not increase this bound.

By the definition of $A_\mu$, the constant terms are at most $\zeta$ and the three
coefficients of $Q(M)$ are at most $A_\mu$. Thus, for every pair of initial types under
consideration, the full potential satisfies
\[
 \cE_{h+1}(s,t)\leq\zeta+A_\mu Q(M).
\]

To bound $A_\mu$ in terms of $\zeta$, convexity gives
$x-x^\alpha\leq(\alpha-1)(1-x)$ for $0<x\leq1$. Dividing by $x^\alpha$ yields
\[
 x^{1-\alpha}\leq1+(\alpha-1)\phi(1-x).
\]
We apply this with $x=b(v)$ and average over $\mu$, or take $x=b(0)$ for a prescribed
source label. In both cases the averaged potential is at most $\zeta$, so
\[
 A_\mu\leq1+(2\alpha-1)\zeta<\infty.\qedhere
\]
\end{proof}

Consequently, the bound $M$ is preserved from one height to the next provided
\begin{equation}\label{eq:scalar-barrier-condition}
 \zeta+A_\mu Q(M)\leq M.
\end{equation}
Because $Q(M)\geq0$, it is enough to impose the following scalar condition, expressed in
terms of the one-site potential $\zeta$ and the child-error bounds $Z$ and $E(M)$ entering
$Q(M)$:
\[
 \zeta+\bigl(1+(2\alpha-1)\zeta\bigr)Q(M)\leq M.
\]

\subsection{Uniform matching bounds}\label{sec:completion-proof}

\begin{proof}[Proof of \cref{thm:markov-matching}]
We first suppose that $M\geq0$ satisfies \cref{eq:scalar-barrier-condition} and, when
$b(0)<1$, that $4S_H^2T^2(1-b(0))\leq1$. We prove by induction that
\[
 \cE_h(s,t)\leq M\qquad(h\geq0)
\]
for all pairs $s,t$ when $b(0)=1$, and all pairs in the same class otherwise. At height
zero the child subtrees are absent. Two roots with law $\mu$ have potential $\eta$, a
prescribed source root against law $\mu$ has potential $\phi(1-b(0))$, and a prescribed
target root has restricted potential zero. Each is at most $\zeta\leq M$.

Suppose the bound holds through height $h$. By
\cref{thm:one-step-closure,eq:scalar-barrier-condition},
\[
 \cE_{h+1}(s,t)\leq\zeta+A_\mu Q(M)\leq M.
\]

To recover the failure probability, we separate labellings with zero degree from those
with positive degree and use $1-r\leq\phi(1-r)$ for $r>0$. For each pair of starting
types under consideration,
\begin{align*}
 \bbP\bigl(\cM_h(s,t)^c\bigr)
 &=\rho_{s,h}\{x\in\cL_h\colon r_{t,h}(x)=0\}
  +\sum_{\substack{x\in\cL_h\\r_{t,h}(x)>0}}\rho_{s,h}\{x\}\bigl(1-r_{t,h}(x)\bigr)\\
 &\leq Z+\cE_h(s,t)\leq Z+M.
\end{align*}
When $b(0)<1$ and both $s,t\in\cI_\mu$, the first term vanishes by the positivity
condition~\labelcref{it:markov-positivity}, so their failure probability is at most $M$. When
$b(0)=1$, the same bound holds for all initial types by \cref{eq:positive-degrees}.

The events $\cM_h(s,t)$ decrease with height. The deterministic argument using K\H{o}nig's
infinity lemma in the proof of \cref{thm:reduction} applies to every pair of labellings and
gives
\[
 \cM_\infty(s,t)=\bigcap_{h\geq0}\cM_h(s,t).
\]
Continuity from above gives the same failure bounds at infinite height.

It remains to choose $M$ proportional to $\zeta$, uniformly over the stated families of
laws. The choices of $\beta_\ast$ and $u=1/2$ at the start of the section depend only on
$\alpha$, and $\lambda_\alpha<1$.

We restrict to $0\leq\zeta,M\leq1$. When $b(0)<1$ we also require
$4S_H^2T^2\zeta\leq1$. Then \cref{eq:explicit-zero-bound} gives
$Z\leq2S_H\zeta$. By \cref{eq:root-parameter-bounds}, the factor $G_H$ in
\cref{eq:explicit-weighted-error} is bounded on this range, its first term is at
most a constant times $\zeta$, and its second is at most a constant times
$\zeta^2+M^2\leq\zeta+M^2$. Hence
\[
 E(M)\leq C_1(\zeta+M^2),
\]
where $C_1$ depends only on $\alpha,H,T,B$.
Using $A_\mu\leq1+(2\alpha-1)\zeta$, we obtain
\[
 \zeta+A_\mu Q(M)\leq\lambda_\alpha M+C_2(\zeta+M^2),
\]
with $C_2>0$ depending on the same parameters. When $b(0)=1$, on the same range
$0\leq\zeta,M\leq1$, this bound follows directly from
$Q(M)=\lambda_\alpha M+C_\alpha(1/2)M^2$, with $C_2$ depending only on $\alpha$.

We choose a coefficient for the potential bound and then set
\[
 K_{\mathrm{pot}}>\max\left\{1,\frac{C_2}{1-\lambda_\alpha}\right\},
 \qquad M\coloneqq K_{\mathrm{pot}}\zeta.
\]
The right side of the preceding bound is at most $K_{\mathrm{pot}}\zeta$ whenever
\[
 C_2K_{\mathrm{pot}}^2\zeta
 \leq(1-\lambda_\alpha)K_{\mathrm{pot}}-C_2.
\]
The quantity $(1-\lambda_\alpha)K_{\mathrm{pot}}-C_2$ is positive by the choice of
$K_{\mathrm{pot}}$. We can therefore choose $\eps>0$ so that
\[
 K_{\mathrm{pot}}\eps\leq1,
 \qquad C_2K_{\mathrm{pot}}^2\eps
 \leq(1-\lambda_\alpha)K_{\mathrm{pot}}-C_2,
\]
and, when $b(0)<1$, also $4S_H^2T^2\eps\leq1$.
We take the constant in the theorem to be
\[
 K\coloneqq
 \begin{cases}
  K_{\mathrm{pot}},&b(0)=1,\\
  K_{\mathrm{pot}}+2S_H,&b(0)<1.
 \end{cases}
\]
For $0\leq\zeta\leq\eps$, all the preceding restrictions hold. The scalar criterion,
together with $Z\leq2S_H\zeta$ when $b(0)<1$, then gives failure probability at most
$K\zeta$ at every finite height and at infinite height. The constants $K$ and $\eps$ depend
only on the parameters asserted in the theorem: $T$ is the largest class size, and
\cref{eq:transition-budget} gives $B$ in terms of the transition probabilities. Prescribed
upper bounds on $H,T,B$ give the same uniform estimates. The restricted calculation after
\cref{eq:transition-budget} proves the selected-transition refinement stated in the theorem.

Finally, if $\mu(0)\geq p>0$, the summand at $0$ in $\eta$ gives
\[
 \eta\geq\mu(0)\phi(1-b(0))\geq p\phi(1-b(0)),\qquad
 \zeta\leq\frac\eta p.
\]
Thus $\eta\leq p\eps$ implies $\zeta\leq\eps$. The failure bound in terms of $\eta$ has
constant $K/p$.
\end{proof}

\ifSubfilesClassLoaded{

}{}

\end{document}

\section{Product laws with a common branching semigroup}\label{sec:common-presentations}

We now apply \cref{thm:markov-matching} to two product laws $\mu\times\nu_L$ and
$\mu\times\nu_R$ with the same label law $\mu$, where the arity laws $\nu_\sigma$ have nonempty
finite supports
\[
 S_\sigma\coloneqq\supp\nu_\sigma\subseteq\{2,3,\ldots\}
 \qquad(\sigma\in\{L,R\}).
\]
We further assume that both measures generate the same branching semigroup:
\[
 \Lambda\coloneqq\Lambda_{\nu_L}
 =\langle k-1\colon k\in S_L\rangle
 =\langle k-1\colon k\in S_R\rangle=\Lambda_{\nu_R}.
\]
As before, the exponent $\alpha\geq1$ s chosen large enough such that $\lambda_\alpha<1$.

\paragraph*{\bf Shared arities} We recall the encoding of \cref{sec:binary-presentations}. For each supported arity $k$,
we choose a binary profile $C_k$ with $k$ leaves. On side $\sigma$, the encoded field has the
recursive law of \cref{thm:profile-encoding-law}, with label law $\mu$, arity law $\nu_\sigma$
and distinguished label $0$. The two encodings will use a nonempty set of shared arities satisfying
\[
 S_{\mathrm{core}}\subseteq S_L\cap S_R,
 \qquad
 \Lambda=\langle k-1\colon k\in S_{\mathrm{core}}\rangle,
\]
which we call the \emph{common generating core}, or simply the \emph{common core}. We choose
the same profile $C_k$ on both sides for each $k\in S_{\mathrm{core}}$. Every other profile
is obtained by composing core profiles. We continue to write $C_k$ on either side, allowing
the noncore profiles to differ between sides. We fix these finite lists of profiles before
varying the arity probabilities or the label law.

\paragraph*{\bf Replacements} Replacing a leaf of an $a$-leaf profile by a $b$-leaf profile gives $a+b-1$ leaves. Thus
each replacement adds the shifted arity $b-1$. We first use this identity to construct the
profiles from common semigroup generators, then verify the Markov hypotheses.

\subsection{The common generators}\label{sec:common-generators}

An \emph{atom} of $\Lambda$ is a positive element which cannot be written as the sum
of two positive elements of $\Lambda$. We write $\cA$ for the set of atoms and note that it is
nonempty: the least positive element of $\Lambda$ cannot be a sum of two smaller positive
elements. We write $s(\Lambda)$ for this least positive element and put
\[
 s_0\coloneqq\min\cA=s(\Lambda)=\min(\Lambda\setminus\{0\}).
\]
We start with a simple combinatorial observation.
\begin{lemma}\label{thm:common-atoms}
The set $\cA$ is finite, generates $\Lambda$, and satisfies
\[
 \{a+1\colon a\in \cA\}\subseteq S_L\cap S_R.
\]
For every $k\geq2$ with $k-1\in\Lambda$, any expression of $k-1$ as a sum of atoms has
at most $\lfloor(k-1)/s_0\rfloor$ summands.
\end{lemma}

\begin{proof}
Express an atom as a sum of generators from either shifted support. There must be exactly
one summand, since two or more would decompose the atom into two positive elements of
$\Lambda$. Thus every atom belongs to both shifted supports, which also proves finiteness.

We prove generation by induction on positive elements of $\Lambda$. An element which is
not an atom is the sum of two smaller positive elements, each a sum of atoms by induction.
Finally, each atomic summand is at least $s_0$, which proves the bound on their number.
\end{proof}

For the main uniform statement below, we take
\(
 S_{\mathrm{core}}\coloneqq\{a+1\colon a\in\cA\}
\)
as the canonical common generating core and choose an $(a+1)$-leaf profile $C_{a+1}$ for
each atom $a$.
For example, the nonnested supports $S_L=\{4,6,7\}$ and $S_R=\{4,6,9\}$ have common
shifted semigroup $\langle3,5\rangle$ and atomic core $\{4,6\}$, so the uniform bounds below
require the masses of arities $4$ and $6$, but not those of $7$ and $9$, to remain bounded
away from zero.
For every supported arity $k$, choose an expression
\[
 k-1=a_1+\cdots+a_{n_k},\qquad a_i\in \cA.
\]
When $k=a+1$ is a core arity we use the one-term expression. Starting from a single leaf,
replace a current leaf by $C_{a_1+1}$, then a current leaf by $C_{a_2+1}$, and continue.
The resulting profile $C_k$ has $1+\sum_i a_i=k$ leaves.

We can fix all choices using only the supports. After choosing one core profile for each
atom, order the atoms, choose a shortest expression, order its summands, and at each step
replace a shallowest leaf, resolving ties lexicographically. We make the same choices for
arities present on both sides. The resulting profiles are independent of the arity masses
and the label law.

For a chain-regime offspring law $\theta$, the reduced law $\widetilde\nu$ is the original law
conditioned on arities $k\geq2$. Its shifted support differs from that of $\theta$ only by the
element $0$, and hence
\[
  \Lambda_{\widetilde\nu}=\Lambda_\theta.
\]
Thus chain laws with the same branching semigroup give reduced arity laws to which the preceding
construction applies. For a bushy-regime law, \cref{thm:full-support} gives
$\widetilde\nu_2>0$. Hence two bushy laws have the common core arity $2$, and repeated replacement
by the two-leaf profile constructs every required binary profile.

\subsection{The depth of the profiles}\label{sec:profile-depths}
The depth of the profiles directly determines the return bound $H$ used in \cref{thm:markov-matching}.

\begin{lemma}\label{thm:shallow-grafting}
Assume that every core profile $C_{a+1}$ has height at most $L$ and every replacement uses a
shallowest current leaf,
then
\[
 \operatorname{height}(C_k)\leq L+\lfloor\log_2(k-s_0)\rfloor.
\]
\end{lemma}

\begin{proof}
Before any replacement, there are at most $k-s_0$ leaves, since that replacement adds at least
$s_0$ leaves. If the chosen shallowest leaf has depth $d$, every vertex at smaller depth is
internal. There are therefore $2^d$ vertices at depth $d$, each with a leaf descendant, so
\[
 2^d\leq k-s_0.
\]
The replacement creates leaves at depths at most $d+L\leq L+\lfloor\log_2(k-s_0)\rfloor$.
Every final leaf is created in some replacement, giving the bound.
\end{proof}

\paragraph*{\bf Returns} For the simplest return argument, choose $C_{a+1}$ with one root child a leaf and the other
the root of a full binary tree with $a$ leaves. We construct the latter recursively: a tree
with one leaf is a single vertex, and a tree with $m\geq2$ leaves has two child subtrees with
$\lfloor m/2\rfloor$ and $\lceil m/2\rceil$ leaves, constructed in the same way. This gives
a tree of height $\lceil\log_2a\rceil$. For $a=1$, both root children of $C_{a+1}$ are leaves.
Thus each profile has a leaf of depth one and height at most
\[
 L\coloneqq1+\lceil\log_2\max\cA\rceil.
\]
With the same profile choice on both sides, put
\[
 N\coloneqq\max(S_L\cup S_R),\qquad
 h_C\coloneqq\max_{k\in S_L\cup S_R}\operatorname{height}(C_k).
\]
The lemma gives
\[
 h_C\leq L+\lfloor\log_2(N-s_0)\rfloor.
\]

\subsection{Verification of the Markov hypotheses}

We now verify the hypotheses of \cref{thm:markov-matching} for any finite choice of profiles
built from a common generating core as described at the start of this section. We use $h_C$
for the maximum profile height over both sides.

\subsubsection{Types and transitions}\label{sec:types-degrees}

On side $\sigma$, an original vertex has the type $F_\sigma$. For every profile $C_k$ on
side $\sigma$ and every internal vertex $w$ of $C_k$ other than its root, we introduce a
type recording the subtree of $C_k$ rooted at $w$. When two such subtrees are identical,
we may use one type for both. The type set $\cI$ consists of $F_L$, $F_R$ and these subtree
types. The root label of $F_\sigma$ has law $\mu$, and the root label of every subtree type
is $0$, so
\[
 \cI_\mu=\{F_L,F_R\}.
\]
The transitions are as follows. With probability $\nu_\sigma(k)$, the type $F_\sigma$ gives
its two children the types of the two children of the root of $C_k$, where a child that is
a leaf of $C_k$ receives the type $F_\sigma$. A subtree type gives its two children the
types determined in the same way by the subtree it records, with probability one.
Conditioned on the chosen child types, the two child labellings are independent, so these
root label laws and transitions define a process of the form in \cref{sec:process}, in which
the original vertices of the encoding are the vertices of type $F_L$ or $F_R$.

Counting both sides, we have
\begin{equation}\label{eq:type-count}
 \#\cI\leq
 2+\sum_{\sigma\in\{L,R\}}\sum_{k\in S_\sigma}(k-2),
\end{equation}
since a full binary $k$-leaf profile has $k-1$ internal vertices, of which only the root is
original. Identifying repeated subtrees often gives a substantially smaller count.

\subsubsection{Positive matching probability}

We assume $b(0)>0$, which holds when $\zeta_\alpha<\infty$. Every realised label belongs to
$\supp\mu\cup\{0\}$, and each label in this set has a compatible label of
positive $\mu$-mass. Indeed,
\[
 b(v)\geq\mu(v)>0\quad(v\in\supp\mu),\qquad b(0)>0.
\]
We use this fact to replace a composite target profile by successive draws of its core
profiles without losing the possibility of a match.

We use the following core-replacement construction twice. Prescribe the values of the label
and arity draws at finitely many original vertices, with every prescribed value in the
support of its law, and suppose that these draws determine an encoded labelling through a
given height. Fix a matching with a source labelling whose labels belong to
$\supp\mu\cup\{0\}$. For any chosen occurrence of $C_k$, replace the prescribed arity $k$
at its root by the root core arity in the fixed composition defining $C_k$. At each
additional core-profile root, prescribe its core arity and a label of positive $\mu$-mass
compatible with the source label paired to that vertex. Retain the label at the original
root, all other prescribed draws and the fixed matching. The binary profile is unchanged.
Every value in the new finite assignment has positive probability, so independence implies
that the event on which all these values are drawn has positive probability on the original
side. If every profile occurrence is replaced, all prescribed arities are core arities, and
the same event has positive probability on the other side as well.

\begin{lemma}\label{thm:fresh-positive}
For every $h\geq0$ and $s,t\in\{F_L,F_R\}$, each labelling of positive probability
under $\rho_{s,h}$ has positive matching degree against $\rho_{t,h}$.
\end{lemma}

\begin{proof}
Fix $s=F_\sigma$ and a labelling $x$ through height $h$ with $\rho_{s,h}\{x\}>0$. Under
$\rho_{s,h}$ the labelling is a function of the labels and arities drawn at the original
vertices of depth at most $h$. The countably many assignments of values to these draws
partition the event that the resulting labelling is $x$, so one assignment has positive
probability. Apply the core-replacement construction to every profile occurrence represented
in this assignment, use the same core profiles on the side of $t$, and retain the source
label at every original root. At each additional core-profile root, the source label is $0$,
so we prescribe a positive-mass label compatible with $0$. With the identity matching, the
resulting finite assignment determines a labelling $y$ on the side of $t$ such that
$\rho_{t,h}\{y\}>0$ and $x$ matches $y$. Hence
$r_{t,h}(x)\geq\rho_{t,h}\{y\}>0$.
\end{proof}

The same construction lets us use only core arities in the inverse-mixture bound of
\cref{eq:selected-inverse-mixture}. Fix a side $\sigma$ and a finite child height $h$.
Let $Y_1,Y_2$ be the two target child labellings through height $h$. Conditioned on the
root arity $k$, their laws are determined by the two branches of $C_k$. For a fixed source
child pair $x_1,x_2$, put
\[
 \begin{split}
 r_k(x_1,x_2)&\coloneqq
 \bbP\bigl((x_1,x_2)\mathrel{(\approx_h)^\square}(Y_1,Y_2)
       \mid\text{root arity}=k\bigr),\\
 r(x_1,x_2)&\coloneqq
 \sum_{k\in S_\sigma}\nu_\sigma(k)r_k(x_1,x_2).
 \end{split}
\]

\begin{lemma}\label{thm:atomic-normalisation}
If all source labels in $x_1,x_2$ belong to $\supp\mu\cup\{0\}$, then
\[
 r(x_1,x_2)>0\quad\Longleftrightarrow\quad
 r_j(x_1,x_2)>0\text{ for some }j\in S_{\mathrm{core}}.
\]
For every $\alpha\geq1$,
\begin{equation}\label{eq:core-inverse-mixture}
 \mathbf1_{\{r>0\}}r^{-\alpha}
 \leq\sum_{j\in S_{\mathrm{core}}}\nu_\sigma(j)^{-\alpha}
       \mathbf1_{\{r_j>0\}}r_j^{-\alpha},
\end{equation}
where each restricted inverse is defined to be zero at degree zero.
\end{lemma}

\begin{proof}
Suppose that $r_k(x_1,x_2)>0$. Since the finite-height target space is countable, there is
a target child pair of positive conditional probability that matches $(x_1,x_2)$. Fix the
child pairing and descendant automorphisms that give this match. Among the countably many
assignments of label and arity draws producing this child pair, choose one whose conditional
event has positive probability. Apply the core-replacement construction to the root
occurrence of $C_k$ in this assignment, and let $j\in S_{\mathrm{core}}$ be the core arity
at the root of its chosen decomposition. The new assignment has positive probability
conditioned on arity $j$ and preserves the fixed matching. Hence
\[
 r_k(x_1,x_2)>0\quad\Longrightarrow\quad r_j(x_1,x_2)>0
 \quad\text{for some }j\in S_{\mathrm{core}}.
\]
If $r(x_1,x_2)>0$, some $r_k(x_1,x_2)$ is positive, so the implication provides a positive
core degree. Conversely, if $r_j(x_1,x_2)>0$ for some $j\in S_{\mathrm{core}}$, then
$r(x_1,x_2)\geq\nu_\sigma(j)r_j(x_1,x_2)>0$. This proves the equivalence. Applying the
argument of \cref{eq:selected-inverse-mixture} with the sum restricted to the core arities
gives \cref{eq:core-inverse-mixture}.
\end{proof}

At $F_L$ and $F_R$, retain precisely the transitions induced by the core arities, and at
every other type retain its unique transition. By \cref{thm:atomic-normalisation}, these
sets satisfy the selected-transition hypothesis in \cref{thm:markov-matching}. At every other
type, the unique transition has probability one, so the sum in \cref{eq:transition-budget}
equals one. Since each nonempty core sum is at least one, we may use
\begin{equation*}
 B\coloneqq\max_{\sigma\in\{L,R\}}
              \sum_{k\in S_{\mathrm{core}}}\nu_\sigma(k)^{-\alpha}.
\end{equation*}
For the core obtained from atoms, we therefore
assume that $p_{\mathrm{core}}>0$ satisfies
\[
 \nu_\sigma(a+1)\geq p_{\mathrm{core}}\qquad
 (a\in\cA,\ \sigma\in\{L,R\}).
\]
Then
\[
 B\leq\#\cA\,p_{\mathrm{core}}^{-\alpha}.
\]
Thus this part of the matching bound depends only on the core probabilities. That is,
the noncore arity masses do not contribute.

\subsubsection{Return times to original vertices}\label{sec:return-times}

When a path reaches a profile leaf, it reaches the next original vertex and hence a type
in $\cI_\mu$. The possible distances for core profiles form the set
\[
 R_{\mathrm{core}}\coloneqq
 \{|v|\colon v\text{ is a leaf of }C_k,\ k\in S_{\mathrm{core}}\}.
\]
Here $|v|$ denotes depth from the profile root. We put
\[
 g\coloneqq\gcd R_{\mathrm{core}},\qquad
 \Gamma\coloneqq\langle \ell\colon \ell\in R_{\mathrm{core}}\rangle.
\]
Concatenating paths through core profiles realises every additional length in $\Gamma$.
Every leaf depth of a composite profile is also such a sum, and hence a multiple of $g$.

We use these congruences to assign the type classes directly. Put $F_L,F_R$ in $\cI_0$.
For an inserted type $t$, let $d$ be the depth of any leaf in its remaining finite profile,
and assign its class by
\[
 t\in\cI_i\quad\Longleftrightarrow\quad i+d\equiv0\pmod g
 \qquad(0\leq i<g).
\]
This does not depend on the chosen leaf. Indeed, the remaining depths are obtained by
subtracting the depth of the inserted vertex from full-profile leaf depths, all multiples
of $g$. It also does not depend on which occurrence of an identified subtree is
used. Descending to an inserted child decreases each remaining depth by one. If the child
is original, its remaining distance was one, so its parent's class is $\cI_{g-1}$.
At an original vertex, the same congruence applied to each root branch gives child class
$\cI_1$. Thus every transition advances both child types by one class, with indices read
modulo $g$, verifying the cyclic class condition in \cref{sec:finite-hypotheses}.

\begin{lemma}\label{thm:bounded-return}
There is an integer $H\geq0$ such that, whenever $s,t$ are types in the same class and
$x_0=s,x_1,\ldots,x_H$ is a possible source path, there is a depth $k\leq H$ for which
$x_k\in\cI_\mu$ and some possible target path $y_0=t,y_1,\ldots,y_k$ ends in $\cI_\mu$.
If some core profile has a leaf at depth one, then $g=1$ and one may take $H=2h_C$.
\end{lemma}

\begin{proof}
The generators $\ell/g$, for $\ell\in R_{\mathrm{core}}$, have greatest common divisor one. Hence the
numerical semigroup that they generate contains every sufficiently large integer; see
Niven--Zuckerman--Montgomery~\cite[Section~1.2]{NivenZuckermanMontgomery1991}. It follows that
there is an integer
$c_{\mathrm{ret}}\geq0$ such that every multiple of $g$ that is at least
$c_{\mathrm{ret}}$ lies in $\Gamma$. Set
\[
 H\coloneqq c_{\mathrm{ret}}+2h_C.
\]

Let $s,t$ be types in the same class, and let $x_0=s,x_1,\ldots,x_H$ be a possible source
path. Choose a possible target path from $t$ to an original vertex at depth $d\leq h_C$,
taking $d=0$ if $t\in\cI_\mu$. Appending paths through core profiles extends it to possible
target paths ending in $\cI_\mu$ at every depth in $\{d+m\colon m\in\Gamma\}$. If
$n\geq d+c_{\mathrm{ret}}$ and $n\equiv d\pmod g$, then $n-d$ is a multiple of $g$ at least
$c_{\mathrm{ret}}$, hence lies in $\Gamma$. Thus a possible target path of length $n$ ends
in $\cI_\mu$ for every such $n$.

On the source side, every possible path reaches $\cI_\mu$ within $h_C$ steps, and the gaps
between successive visits are at most $h_C$. The first visit of this path to $\cI_\mu$ at or
after depth $h_C+c_{\mathrm{ret}}$ therefore occurs at a depth $k$ with
\[
 h_C+c_{\mathrm{ret}}\leq k\leq2h_C+c_{\mathrm{ret}}=H.
\]
The initial types lie in the same class, so $k\equiv d\pmod g$, and
$k\geq h_C+c_{\mathrm{ret}}\geq d+c_{\mathrm{ret}}$. By the previous paragraph, some possible
target path of length $k$ ends in $\cI_\mu$.

If some core profile has a leaf at depth one, then $1\in R_{\mathrm{core}}$, so $g=1$ and
$\Gamma=\bbN_0$. We may then take $c_{\mathrm{ret}}=0$ in the construction above, which gives
$H=2h_C$.
\end{proof}

These available target paths verify the return condition~\labelcref{it:markov-returns}. The count in
\cref{eq:type-count} bounds the number of types in every class. The theorem therefore
applies, in particular, to the initial types $F_L$ and $F_R$.

We can now state the resulting bound for the two product laws.

\begin{proposition}[Uniform common-semigroup matching]
\label{thm:common-semigroup-matching}
Fix $\alpha\geq1$ with
$\lambda_\alpha<1$, finite supports $S_L,S_R$ with the same branching semigroup, and
$p_{\mathrm{core}}>0$. Choose the common atomic profiles as above. There are $K<\infty$ and
$\eps>0$, depending only on $\alpha,S_L,S_R,p_{\mathrm{core}}$, such that the following holds for
every pair of arity laws with these supports and every label law $\mu$. If
\[
 \nu_\sigma(a+1)\geq p_{\mathrm{core}}\quad(a\in\cA,\ \sigma\in\{L,R\}),\qquad
 \zeta_\alpha(\mu)\leq\eps,
\]
then
\[
 \bbP\bigl(\cM_h(F_L,F_R)^c\bigr)\leq K\zeta_\alpha(\mu)\quad(h\geq0),\qquad
 \bbP\bigl(\cM_\infty(F_L,F_R)^c\bigr)\leq K\zeta_\alpha(\mu).
\]
\end{proposition}

\begin{proof}
The profile heights, class sizes and return bounds depend only on the supports, while
$B\leq\#\cA\,p_{\mathrm{core}}^{-\alpha}$. The selected-transition refinement of
\cref{thm:markov-matching} therefore applies to the original types $F_L,F_R$.
\end{proof}

\paragraph*{\bf Choosing constants} For a fixed pair of arity laws, we may always take
\[
 p_{\mathrm{core}}\coloneqq\min_{\substack{\sigma\in\{L,R\}\\a\in\cA}}\nu_\sigma(a+1)>0,
\]
since the finite atomic core belongs to both supports. The same $K$ and $\eps$ apply while
these core probabilities remain bounded below by a common positive constant.

\subsection{An obstruction for preassigned profiles}\label{sec:profile-obstruction}

The common-profile requirement cannot be omitted: equal branching semigroups alone do not
ensure matching for arbitrary fixed binary encodings. The following counterexample works for
every fixed potential exponent. Take
\[
 \nu_L=\delta_4,\qquad
 \nu_R=\tfrac12\delta_4+\tfrac12\delta_7,
\]
and use balanced binary profiles at both arities. Both shifted supports generate
$3\bbN_0$. The four-leaf profile has all leaves at depth two, whereas the seven-leaf
profile has leaves at depths two and three. In the left process every odd-depth vertex
therefore has prescribed label $0$.

Use the path relation $0$--$1$--$2$, including reflexive pairs, and put
\[
 \mu(0)=\tfrac12,\qquad \mu(1)=\tfrac12-t,\qquad \mu(2)=t,
 \qquad 0<t<\tfrac12.
\]
Then
\[
 \eta_\alpha(\mu)
 =\frac{t}{2(1-t)^\alpha}+2^{\alpha-1}t\longrightarrow0
 \qquad(t\downarrow0).
\]

In the right process follow original vertices $u_0,u_1,\ldots$, choosing a fixed leaf at
depth two when the current arity is four and a fixed leaf at depth three when it is seven.
Given the process up to $u_i$ before its new arity is drawn, exactly one of the two
equiprobable arities makes $|u_{i+1}|$ odd. The label at $u_{i+1}$ equals $2$ with
independent probability $t$. Successive conditioning therefore gives
\[
 \bbP\bigl(X'_{u_i}\neq2\text{ whenever }|u_i|\text{ is odd},\ 1\leq i\leq n\bigr)
 =(1-t/2)^n.
\]
Every rooted automorphism preserves depth. A right label $2$ at odd depth cannot match any
left label at that depth, since all of them are $0$. Also $|u_i|\leq3i$. Hence
\[
 \bbP(\cM_{3n})\leq(1-t/2)^n\longrightarrow0,
 \qquad \bbP(\cM_\infty)=0.
\]

The common-profile construction avoids this obstruction. Here the only atom is $3$.
Choose a four-leaf profile $C_4$ and build $C_7$ by replacing a leaf of $C_4$ by another
copy of $C_4$. We may retain the balanced $C_4$, in which case every return to an original
vertex has even depth and we may use $g=2$. Alternatively, use a four-leaf profile with a
depth-one leaf. Both choices give a common presentation and the matching conclusion.

\section{Evaluation of the bounds}\label{sec:exponent-values}

The proof of \cref{thm:markov-matching} reduces the quantitative estimate to a scalar
barrier inequality for $M$. We first solve it when $b(0)=1$ and then record a constructive
uniform version for finite type models. We subsequently determine the admissible exponents,
give rational bounds at exponents $4/3$ and $2$, quantify the dependence on the largest
arity, and compare the upper bounds with two lower bounds.

\subsection{An explicit bound when the prescribed label is always compatible}

Fix $\alpha\geq1$ and $\beta\in[0,1]$ such that
\[
 \lambda_\alpha(\beta)=2\bigl(L_\alpha(\beta)+K_\alpha(\beta)\bigr)<1.
\]
By \cref{eq:mean-constants}, we have $\lambda_\alpha(\beta)\geq\lambda_\alpha$, with equality
when $\beta$ minimises the linear coefficient. We also fix $u\in(0,1)$, which controls the
quadratic coefficient $C_\alpha(u)$ in the four-law estimate \cref{eq:four-law-contraction}.

When $b(0)=1$, the solution is explicit. Then $\zeta=\eta$ and the zero-degree terms
vanish. Applying \cref{thm:four-law-contraction} with the chosen $\beta$ and $u$ bounds
the child-pair potential by $\lambda_\alpha(\beta)M+C_\alpha(u)M^2$. Put
\[
 d\coloneqq1+(2\alpha-1)\eta,
\]
so that $A_\mu\leq d$, and suppose that
\[
 d\lambda_\alpha(\beta)<1,\qquad
 \bigl(1-d\lambda_\alpha(\beta)\bigr)^2\geq4dC_\alpha(u)\eta.
\]
The smallest nonnegative solution of $M=\eta+d\bigl(\lambda_\alpha(\beta)M+C_\alpha(u)M^2\bigr)$ is
\begin{equation}\label{eq:zero-compatible-quadratic}
 M_\eta\coloneqq
 \frac{2\eta}{1-d\lambda_\alpha(\beta)
 +\sqrt{\bigl(1-d\lambda_\alpha(\beta)\bigr)^2-4dC_\alpha(u)\eta}}.
\end{equation}
Its denominator is positive, including at $\eta=0$, where it equals
$2\bigl(1-\lambda_\alpha(\beta)\bigr)$. Since
$A_\mu\leq d$, the one-step bound is at most
$\eta+d\bigl(\lambda_\alpha(\beta)M+C_\alpha(u)M^2\bigr)$, so the height induction closes
with $M=M_\eta$.
Consequently, for every countable type model with $b(0)=1$ and every pair $s,t\in\cI$,
\[
 \bbP\bigl(\cM_h(s,t)^c\bigr)\leq M_\eta\quad(h\geq0),
 \qquad
 \bbP\bigl(\cM_\infty(s,t)^c\bigr)\leq M_\eta.
\]

\subsection{A constructive uniform threshold}

\begin{proposition}[Constructive uniform threshold]\label{thm:constructive-threshold}
Fix $\alpha\geq1$, an integer return bound $H\geq0$, a class-size bound $T\geq1$ and a
transition bound $B\geq1$. Choose $\beta\in[0,1]$ such that
$\lambda_\alpha(\beta)<1$.
Fix also $u\in(0,1)$. There are constants $K_{\mathrm{match}}<\infty$ and $\eps>0$, given
explicitly in terms of $\alpha,\beta,u,H,T,B$, such that every finite type model satisfying the hypotheses of
\cref{thm:markov-matching}, with $b(0)<1$ and the indicated bounds, satisfies
\[
 \bbP\bigl(\cM_h(s,t)^c\bigr)\leq K_{\mathrm{match}}\zeta_\alpha(\mu)\quad(h\geq0),
 \qquad
 \bbP\bigl(\cM_\infty(s,t)^c\bigr)\leq K_{\mathrm{match}}\zeta_\alpha(\mu)
\]
for all types $s,t$ in the same class whenever $\zeta_\alpha(\mu)\leq\eps$. The constants
may be chosen so that these failure probabilities are at most $1/2$.
\end{proposition}

\begin{proof}
Fix the parameters in the proposition and put
\[
 C\coloneqq4+8(\alpha+1)B,
\]
as well as
\[
 S_H\coloneqq\sum_{d=0}^H2^d,
 \qquad
 G_H(x)\coloneqq\sum_{d=0}^Hx^d.
\]
For $0\leq t\leq t_0\coloneqq1/(4S_H^2T^2)$ and $M\geq0$, define
\[
 \overline Z(t)\coloneqq
 \frac{2S_Ht}{1+\sqrt{1-4S_H^2T^2t}}
\]
and
\[
 \overline E(t,M)\coloneqq
 G_H\bigl(4B(1+\alpha t)(1+\alpha M)\bigr)
 \left(2B(1+\alpha M)^2(\alpha+1)t
       +2B(1+\alpha t)\bigl(T\overline Z(t)+\alpha M\bigr)^2\right).
\]
These are the expressions defining $Z$ and $E(M)$ in
\cref{eq:explicit-zero-bound,eq:explicit-weighted-error}, with $1-b(0)$,
$I_{\mathrm{bad}}$ and $R_\mu$ replaced by $t$, $(\alpha+1)t$ and $1+\alpha t$,
respectively. The right sides of those bounds are nondecreasing in the three substituted
quantities. For a model covered by \cref{thm:constructive-threshold}, write
$\zeta\coloneqq\zeta_\alpha(\mu)$. By \cref{eq:root-parameter-bounds},
\[
 Z\leq\overline Z(\zeta),
 \qquad
 E(M)\leq\overline E(\zeta,M)
 \qquad(0\leq\zeta\leq t_0,\ M\geq0).
\]
Here $\zeta\leq t_0$ also implies the condition
$4S_H^2T^2(1-b(0))\leq1$ under which $Z$ is defined.

Put
\[
 \overline Q(t,M)\coloneqq
 \lambda_\alpha(\beta)M+C_\alpha(u)M^2+(2+2\beta+4\alpha M)\overline Z(t)
 +C(1+\alpha M)\overline E(t,M).
\]
Suppose that the restricted potentials through height $h$ are bounded by $M$. Applying
\cref{thm:four-law-contraction} with the chosen $\beta$ and $u$ in the proof of
\cref{thm:one-step-closure}, together with the preceding bounds on $Z$ and $E(M)$, bounds
the child-pair potential by $\overline Q(\zeta,M)$. The estimate
$A_\mu\leq1+(2\alpha-1)\zeta$ then gives
\[
 \cE_{h+1}(s,t)
 \leq\zeta+\bigl(1+(2\alpha-1)\zeta\bigr)\overline Q(\zeta,M)
 \qquad(0\leq\zeta\leq t_0,\ M\geq0).
\]

Fix $K>0$ and take $M=Kt$. For $0<t\leq t_0$, define
\begin{align*}
 J_K(t)&\coloneqq
 \frac{t+\bigl(1+(2\alpha-1)t\bigr)\overline Q(t,Kt)}{t}\\
 &=1+\bigl(1+(2\alpha-1)t\bigr)
 \biggl(\lambda_\alpha(\beta)K+C_\alpha(u)K^2t
 +(2+2\beta+4\alpha Kt)\frac{\overline Z(t)}{t}\\
 &\hspace{10cm}+C(1+\alpha Kt)\frac{\overline E(t,Kt)}{t}\biggr).
\end{align*}
The right side is at most $M=K\zeta$ whenever
$0<\zeta\leq t_0$ and $J_K(\zeta)\leq K$. The two quotients in $J_K$ are
\begin{align*}
 \frac{\overline Z(t)}{t}
 &=\frac{2S_H}{1+\sqrt{1-4S_H^2T^2t}},\\
 \frac{\overline E(t,Kt)}{t}
 &=G_H\bigl(4B(1+\alpha t)(1+\alpha Kt)\bigr)\\
 &\quad\cdot\left(2B(1+\alpha Kt)^2(\alpha+1)
 +2B(1+\alpha t)t
 \left(T\frac{\overline Z(t)}t+\alpha K\right)^2\right).
\end{align*}
They extend continuously to $t=0$, with values $S_H$ and
$2B(\alpha+1)G_H(4B)$, and every factor in these expressions and in $J_K$ is
nonnegative and nondecreasing on $[0,t_0]$. Hence $J_K$ extends to a continuous
nondecreasing function on $[0,t_0]$, with
\[
 J_K(0)=1+\lambda_\alpha(\beta)K+A_0,
 \qquad
 A_0\coloneqq(2+2\beta)S_H+2CB(\alpha+1)G_H(4B).
\]
Set
\[
 K\coloneqq1+\frac{1+A_0}{1-\lambda_\alpha(\beta)},
 \qquad
 K_{\mathrm{match}}\coloneqq K+2S_H,
 \qquad
 t_*\coloneqq\min\left\{t_0,\frac1{2K_{\mathrm{match}}}\right\}.
\]
Then $K-J_K(0)=1-\lambda_\alpha(\beta)>0$. Continuity and monotonicity therefore show that
\begin{equation}\label{eq:constructive-threshold}
 \eps_K\coloneqq
 \max\{t\in[0,t_*]\colon J_K(t)\leq K\}
\end{equation}
exists and is positive, and that $J_K(t)\leq K$ for $0\leq t\leq\eps_K$.

Suppose $\zeta\leq\eps_K$. If $\zeta>0$, the right side of the one-step bound is at most
$M=K\zeta$. If $\zeta=0$, then $\overline Q(0,0)=0$, and the same holds with $M=0$.
Moreover, $Z\leq2S_H\zeta$. The height induction and failure-probability argument in the
proof of \cref{thm:markov-matching} then give
\[
 \bbP\bigl(\cM_h(s,t)^c\bigr)
 \leq M+Z\leq K_{\mathrm{match}}\zeta
 \qquad(h\geq0)
\]
for types $s,t$ in the same class, with the same bound at infinite height. Finally,
$\eps_K\leq1/(2K_{\mathrm{match}})$ makes every stated failure probability at most $1/2$.
Taking $\eps\coloneqq\eps_K$ proves the proposition.
\end{proof}

\subsection{The exponent range}\label{sec:exponent-range}

We show that $\lambda_\alpha$ is continuous and strictly decreasing on $[1,\infty)$, so
that the exponents with $\lambda_\alpha<1$ form an open half-line.

For continuity, let $0<s\leq1$ and $1\leq\alpha\leq\alpha'$. Then
\[
 0\leq s^\alpha-s^{\alpha'}\leq\alpha'-\alpha,
\]
since $s^{\alpha'-\alpha}\geq1+(\alpha'-\alpha)\log s$, $s^\alpha\leq s$ and
$-s\log s\leq1-s$, and both sides vanish at $s=0$. We apply this to $s=(1+q)^{-1}$ in
the first summand and to $s=1-q$ in the second summand of the function maximised in
\cref{eq:mean-constants}. Since $q\leq1$, this gives
\[
 L_{\alpha'}(\beta)\leq L_\alpha(\beta)\leq L_{\alpha'}(\beta)+(1+\beta)(\alpha'-\alpha).
\]
Since $K_\alpha(\beta)=K_\alpha(0)(1-\beta)^{\alpha+1}$ and $K_\alpha(0)$ is the maximum of
$q(1-q)^\alpha$ over $[0,1]$, the same inequality applied to both factors gives
$K_{\alpha'}(\beta)\leq K_\alpha(\beta)
\leq K_{\alpha'}(\beta)+2(\alpha'-\alpha)$. Taking the minimum over $\beta$ gives
\[
 \lambda_{\alpha'}\leq\lambda_\alpha
 \leq\lambda_{\alpha'}+8(\alpha'-\alpha).
\]
Thus $\lambda_\alpha$ is nonincreasing on $[1,\infty)$.

For a strict decrease, fix $1\leq\alpha<\alpha'$ and a $\beta$ attaining the minimum
defining $\lambda_\alpha$. The choice $\beta=1$ does not attain it, since $q=0$ gives
$2(L_\alpha(1)+K_\alpha(1))\geq2$, whereas $q/(1+q)^\alpha\leq1/2$ and
$q(1-q)^\alpha\leq1/4$ give $2(L_\alpha(0)+K_\alpha(0))\leq3/2$. Thus $\beta<1$. The
maximum defining $K_{\alpha'}(0)$ is attained at an interior point $q$ of $[0,1]$, where
$(1-q)^{\alpha'}<(1-q)^\alpha$, so $K_{\alpha'}(0)<K_\alpha(0)$ and
$K_{\alpha'}(\beta)<K_\alpha(\beta)$. Consequently
\[
 \lambda_{\alpha'}\leq2\bigl(L_{\alpha'}(\beta)+K_{\alpha'}(\beta)\bigr)
 <2\bigl(L_\alpha(\beta)+K_\alpha(\beta)\bigr)=\lambda_\alpha.
\]

At $\alpha=1$, the values $q=1$ and $q=0$ give $L_1(\beta)\geq\max\{1/2,\beta\}$, and
$K_1(\beta)=(1-\beta)^2/4$, so $2(L_1(\beta)+K_1(\beta))\geq9/8$ for every $\beta$ and
$\lambda_1>1$. At $\alpha=2$, the choice $\beta=1/16$ in \cref{sec:two-exponent-values}
gives $\lambda_2\leq381/512<1$. Continuity and monotonicity therefore give a unique
$\alpha_*\in(1,2)$ such that
\[
 \lambda_{\alpha_*}=1,
 \qquad
 \lambda_\alpha<1\quad\Longleftrightarrow\quad\alpha>\alpha_*
 \quad(\alpha\geq1).
\]
The rational calculation in \cref{sec:exponent-four-thirds} shows that
$\alpha_*<4/3$, while numerical minimisation gives
$\alpha_*\approx1.325660330$.

\subsection{An explicit exponent of four thirds}\label{sec:exponent-four-thirds}

We verify $\lambda_{4/3}<1$ by a finite rational calculation, with $\beta=11/40$. Write
\[
 F_{4/3}(q)\coloneqq\frac{q}{(1+q)^{4/3}}+\frac{11}{40}(1-q)^{4/3}
 \qquad(0\leq q\leq1),
\]
so that $L_{4/3}(11/40)=\max F_{4/3}$. Its derivative is
\[
 F_{4/3}'(q)
 =\frac{1-q/3}{(1+q)^{7/3}}-\frac{11}{30}(1-q)^{1/3}.
\]
For $0\leq q<1$, put
\[
 R(q)\coloneqq
 \frac{1-q/3}{(1+q)^{7/3}(1-q)^{1/3}}.
\]
The sign of $F_{4/3}'(q)$ is the sign of $R(q)-11/30$, and logarithmic
differentiation gives
\[
 \frac{R'(q)}{R(q)}
 =\frac{-5q^2+30q-21}{3(3-q)(1-q)(1+q)}.
\]
The numerator on the right is strictly increasing on $[0,1]$ and has its unique zero there
in $(4/5,1)$. Hence $R$ first decreases and then increases. To locate its first crossing of
$11/30$, we cube the two positive terms in $F_{4/3}'(q)$. The derivative then has the same
sign as
\[
 1000(3-q)^3-1331(1-q)(1+q)^7.
\]
Direct substitution shows that this expression is positive at $q=589/1000$ and negative at
$q=59/100$.
Since $R(0)=1$ and $R(q)\to\infty$ as $q\uparrow1$, the derivative has exactly two zeros.
The first, denoted by $q_0$, belongs to $[589/1000,59/100]$ and is the only interior
maximum. The first summand in $F_{4/3}$ is increasing and the second is decreasing, so
\[
 \begin{split}
 F_{4/3}(q_0)
 &\leq\frac{59/100}{(159/100)^{4/3}}
   +\frac{11}{40}\left(\frac{411}{1000}\right)^{4/3}\\
 &<\frac{159}{500}+\frac{17}{200}=\frac{403}{1000}.
 \end{split}
\]
The two strict inequalities are verified by cubing, and
$F_{4/3}(1)=2^{-4/3}<2/5$. Thus
$L_{4/3}(11/40)<403/1000$. Moreover,
\[
 K_{4/3}(11/40)^3
 =\frac{465746660343}{527067520000000}
 <\left(\frac{12}{125}\right)^3,
\]
so $K_{4/3}(11/40)<12/125$. Consequently
\begin{equation*}
 \lambda_{4/3}\leq2\bigl(L_{4/3}(11/40)+K_{4/3}(11/40)\bigr)<\frac{499}{500}<1.
\end{equation*}
With $u=1/2$, the quadratic coefficient is $C_{4/3}(1/2)=272/9-2^{5/3}<28$, since
$L_{4/3}=2^{-4/3}$. Thus \cref{thm:four-law-contraction} gives the rational bound
\begin{equation*}
 \cE(\rho_1\times\rho_2,\tau_1\times\tau_2;R^\square)
 \leq\frac{499}{500}M+28M^2
 +\left(\frac{51}{20}+\frac{16}{3}M\right)\delta_1
 +4\left(1+\frac43M\right)\delta_2,
\end{equation*}
where $M$, $\delta_1$ and $\delta_2$ bound the eight directed potentials, the zero-degree probabilities
and the weighted zero-degree sums as in the hypotheses of that lemma.
For $\beta=11/40$, the exact value of $2(L_{4/3}(11/40)+K_{4/3}(11/40))$ is
approximately $0.995292$.

\subsection{The optimal coefficient at exponent two}\label{sec:optimal-coefficient-two}

At exponent $2$, the minimum in \cref{eq:mean-constants} can be computed exactly. The
derivative of the function maximised in $L_2(\beta)$ is
\[
 \frac{d}{dq}\left(\frac{q}{(1+q)^2}+\beta(1-q)^2\right)
 =(1-q)\bigl((1+q)^{-3}-2\beta\bigr),
\]
so for $1/16<\beta<1/2$ the maximum is attained at $q=(2\beta)^{-1/3}-1$. Moreover
$K_2(\beta)=(4/27)(1-\beta)^3$. The polynomial
\begin{equation}\label{eq:optimal-coefficient-two-quartic}
 3q^4+8q^3+6q^2-2
\end{equation}
is strictly increasing for $q>0$, negative at $q=1/3$ and positive at $q=1/2$, so it has
exactly one root $q$ in $(1/3,1/2)$. Expanding the product shows that
$(3q-1)(1+q)^3=1$ for this root. Put $\beta_0\coloneqq(3q-1)/2$. Then
$2\beta_0(1+q)^3=1$, so $\beta_0\in(1/16,1/2)$ and $L_2(\beta_0)=q/(1+q)^2+\beta_0(1-q)^2$.
For every $\beta'\in[0,1]$, the maximum defining $L_2(\beta')$ is at least the value at
$q$, so
\[
 L_2(\beta')+K_2(\beta')
 \geq\frac{q}{(1+q)^2}+\beta'(1-q)^2+\frac4{27}(1-\beta')^3.
\]
Since $(1-q)^2=(4/9)(1-\beta_0)^2$, the right side exceeds its value at $\beta'=\beta_0$ by
\[
 \frac4{27}(\beta'-\beta_0)^2(3-2\beta_0-\beta')\geq0.
\]
Thus $\beta_0$ attains the minimum in \cref{eq:mean-constants}, and the identity
$(3q-1)(1+q)^3=1$ reduces the minimum to
\[
 \lambda_2=2\bigl(L_2(\beta_0)+K_2(\beta_0)\bigr)=8q^3.
\]
The signs of the polynomial at $44403/100000$ and at $444036/1000000$ give
$0.7003<\lambda_2<0.7004$. Numerically,
\[
 q=0.444033223343499\ldots,\qquad
 \beta_0=0.166049835015249\ldots,\qquad
 \lambda_2=0.700384272171343\ldots.
\]
These decimal values describe the exact algebraic choice specified by
\cref{eq:optimal-coefficient-two-quartic}. The next subsection uses the simpler choice
$\beta=1/16$ to obtain a rational coefficient.

\subsection{Numerical bounds at exponent two}\label{sec:two-exponent-values}

We now take $\alpha=2$, $\beta=1/16$ and $u=1/2$, which give rational coefficients. The
identity
\[
 \frac14-\frac{q}{(1+q)^2}
 =\frac{(1-q)^2}{4(1+q)^2}\geq\frac{(1-q)^2}{16}
\]
shows that $L_2(1/16)=1/4$, attained at $q=1$, and $K_2(1/16)=125/1024$. Hence
$\lambda_2(1/16)=381/512$, and $C_2(1/2)=38$ since $L_2=1/4$.
\Cref{thm:four-law-contraction} gives
\begin{equation}\label{eq:mean-exponent-two}
 \cE(\rho_1\times\rho_2,\tau_1\times\tau_2;R^\square)
 \leq\frac{381}{512}M+38M^2
 +\left(\frac{17}{8}+8M\right)\delta_1+4(1+2M)\delta_2,
\end{equation}
with $M$, $\delta_1$ and $\delta_2$ as in the hypotheses of that lemma.

When $b(0)=1$, the scalar condition follows from
$\eta+(1+3\eta)(381M/512+38M^2)\leq M$, since $A_\mu\leq1+3\eta$. We take
$M=2\eta/(1-\lambda_2(1/16))=(1024/131)\eta$.
Substituting this value and dividing by $\eta$ reduces the
inequality to
\[
 \frac{40145354}{17161}\eta+
 \frac{119537664}{17161}\eta^2\leq1.
\]
The two coefficients are less than $2340$ and $7000$, and the left side is increasing in
$\eta$, so the inequality holds for $\eta\leq1/2500$, where
$2340/2500+7000/2500^2=5857/6250<1$. Consequently, for every countable type model
with $b(0)=1$ and every $s,t\in\cI$,
\[
 \eta_2(\mu)\leq\frac1{2500}
 \quad\Longrightarrow\quad
 \begin{aligned}
 \bbP\bigl(\cM_h(s,t)^c\bigr)
 &\leq\frac{1024}{131}\eta_2(\mu) &&(h\geq0),\\
 \bbP\bigl(\cM_\infty(s,t)^c\bigr)
 &\leq\frac{1024}{131}\eta_2(\mu).
 \end{aligned}
\]
The solution $M_\eta$ of \cref{eq:zero-compatible-quadratic} with $\lambda_2(1/16)=381/512$ and
$C_2(1/2)=38$ is defined for $\eta$ up to approximately $0.0004269348229$, and
\[
 \frac{M_\eta}{\eta}\longrightarrow\frac1{1-\lambda_2(1/16)}=\frac{512}{131}
 \qquad(\eta\downarrow0).
\]
The algebraic minimum $\lambda_2$ of \cref{sec:optimal-coefficient-two} in place of $\lambda_2(1/16)$
gives the limiting coefficient $1/(1-\lambda_2)=3.337608500\ldots$. These are evaluations
of the same sufficient criterion. We do not claim that they are optimal matching thresholds.

\subsection{Dependence on the largest arity}\label{sec:arity-dependence}

We make the constants of \cref{sec:common-presentations} explicit at exponent $2$ for a canonical
choice of core profiles. Fix a nonempty finite set $\cA$ of positive integers and a number
$p_{\mathrm{core}}>0$. For $a\in\cA$, choose the core profile $C_{a+1}$ by the recursive
construction in \cref{sec:profile-depths}. Let $\nu_L$ and $\nu_R$ be finitely
supported laws on $\{2,3,\ldots\}$ with supports $S_L$ and $S_R$, and suppose that
\[
 \nu_\sigma(a+1)\geq p_{\mathrm{core}}
 \qquad(a\in\cA,\ \sigma\in\{L,R\}),
 \quad\text{and}\quad
 k-1\in\langle\cA\rangle
 \qquad(k\in S_\sigma,\ \sigma\in\{L,R\}).
\]
Together, these conditions make $\{a+1\colon a\in\cA\}$ a common generating core for the
two supports. Hence the construction of \cref{sec:common-presentations} applies. Put
\[
 N\coloneqq\max(S_L\cup S_R),
 \qquad
 h_C\coloneqq\max_{\substack{\sigma\in\{L,R\}\\k\in S_\sigma}}
 \operatorname{height}(C_k),
\]
\[
 L\coloneqq1+\lceil\log_2\max\cA\rceil,
 \qquad
 B_0\coloneqq\max\{1,\#\cA\,p_{\mathrm{core}}^{-2}\},
\]
and
\begin{align*}
 p_{\mathrm{ar}}&\coloneqq2\log_2(16B_0)\geq8,
 & A_{\mathrm{ar}}&\coloneqq500B_0^2(16B_0)^{2L+1},\\
 K_{\mathrm{ar}}&\coloneqq\frac{1024}{131}A_{\mathrm{ar}},
 & \eps_{\mathrm{ar}}&\coloneqq\frac1{K_{\mathrm{ar}}^2 2^{4L+4}}.
\end{align*}
These constants depend only on $\cA$ and $p_{\mathrm{core}}$. Let $F_\sigma$ be the original-vertex
type in the encoding of the law $\nu_\sigma$. We claim that, for every pair of arity laws satisfying
the two conditions above, every label law $\mu$, and all $\sigma,\tau\in\{L,R\}$,
\[
 \zeta_2(\mu)\leq\eps_{\mathrm{ar}}N^{-2p_{\mathrm{ar}}}
 \quad\Longrightarrow\quad
 \begin{aligned}
 \bbP\bigl(\cM_h(F_\sigma,F_\tau)^c\bigr)
 &\leq K_{\mathrm{ar}}N^{p_{\mathrm{ar}}}\zeta_2(\mu) &&(h\geq0),\\
 \bbP\bigl(\cM_\infty(F_\sigma,F_\tau)^c\bigr)
 &\leq K_{\mathrm{ar}}N^{p_{\mathrm{ar}}}\zeta_2(\mu).
 \end{aligned}
\]
No lower bound is imposed on the probabilities of noncore arities. If
$\mu(0)\geq p_0>0$, then $\zeta_2\leq\eta_2/p_0$, and the claim gives
\[
 \eta_2(\mu)\leq p_0\eps_{\mathrm{ar}}N^{-2p_{\mathrm{ar}}}
 \quad\Longrightarrow\quad
 \begin{aligned}
 \bbP\bigl(\cM_h(F_\sigma,F_\tau)^c\bigr)
 &\leq\frac{K_{\mathrm{ar}}}{p_0}N^{p_{\mathrm{ar}}}\eta_2(\mu) &&(h\geq0),\\
 \bbP\bigl(\cM_\infty(F_\sigma,F_\tau)^c\bigr)
 &\leq\frac{K_{\mathrm{ar}}}{p_0}N^{p_{\mathrm{ar}}}\eta_2(\mu).
 \end{aligned}
\]

To prove the claim, we bound each parameter of the scalar induction in terms of $N$. Every core
profile has height at most $L$, so \cref{thm:shallow-grafting} gives $h_C\leq
L+\lfloor\log_2N\rfloor$. Every core profile has a leaf at depth one, so \cref{thm:bounded-return}
allows us to take
\[
 H=2h_C\leq2L+2\lfloor\log_2N\rfloor.
\]
The type count in \cref{eq:type-count}, the sum $S_H=2^{H+1}-1$ and the sum $G_H$ then
satisfy
\[
 T\leq2+2\sum_{k=2}^{N}(k-2)\leq N^2,
 \qquad
 S_H\leq2^{2L+1}N^2,
 \qquad
 G_H(16B_0)
 \leq\frac{(16B_0)^{2L+1}}{16B_0-1}N^{p_{\mathrm{ar}}}.
\]
Put
\[
 R_N\coloneqq(16B_0)^{2L+1}N^{p_{\mathrm{ar}}}.
\]
Since $16B_0-1\geq15$ and
$2^{2L+1}N^2\leq R_N$, the preceding bounds give
$S_H\leq R_N$ and $G_H(16B_0)\leq R_N/15$.
Let $\zeta\coloneqq\zeta_2(\mu)$ and $M\geq0$ satisfy $\zeta,M\leq1/2$ and $4S_H^2T^2\zeta\leq1$.
Then \cref{eq:root-parameter-bounds} gives $R_\mu\leq2$ and $I_{\mathrm{bad}}\leq3\zeta$, and
$T^2Z^2\leq4S_H^2T^2\zeta^2\leq\zeta$, so \cref{eq:explicit-zero-bound,eq:explicit-weighted-error}
give
\[
 Z\leq2S_H\zeta,
 \qquad
 E(M)\leq\frac{32}{15}B_0R_N(\zeta+M^2).
\]
We insert these bounds and the coefficients of \cref{eq:mean-exponent-two} into
$\zeta+A_\mu Q(M)$. The bound $A_\mu\leq1+3\zeta\leq5/2$, together with
$B\leq B_0$, $4+24B\leq28B_0$ and $\zeta M\leq(\zeta+M^2)/2$, gives
\begin{align*}
 \zeta
 &\leq B_0^2R_N(\zeta+M^2),\\
 3\frac{381}{512}\zeta M
 &\leq\frac32B_0^2R_N(\zeta+M^2),\\
 \frac52\,38M^2
 &\leq95B_0^2R_N(\zeta+M^2),\\
 \frac52\left(\frac{17}{8}+8M\right)Z
 &\leq\frac{245}{8}B_0^2R_N(\zeta+M^2),\\
 \frac52(4+24B)(1+2M)E(M)
 &\leq\frac{896}{3}B_0^2R_N(\zeta+M^2).
\end{align*}
The sum of the five coefficients is $10243/24<500$. Therefore
\[
 \zeta+A_\mu Q(M)
 \leq\frac{381}{512}M+A_{\mathrm{ar}}N^{p_{\mathrm{ar}}}(\zeta+M^2).
\]
With
\[
 M\coloneqq K_{\mathrm{ar}}N^{p_{\mathrm{ar}}}\zeta
 =\frac{2A_{\mathrm{ar}}}{1-381/512}N^{p_{\mathrm{ar}}}\zeta,
\]
the right side is at most $M$ exactly when $M^2\leq\zeta$, that is, when
$\zeta\leq K_{\mathrm{ar}}^{-2}N^{-2p_{\mathrm{ar}}}$. The hypothesis
$\zeta\leq\eps_{\mathrm{ar}}N^{-2p_{\mathrm{ar}}}$ implies this inequality, as well as $M\leq1/2$ and
$4S_H^2T^2\zeta\leq1$. Hence \cref{eq:scalar-barrier-condition} holds with this $M$.
Since $F_L$ and $F_R$ belong to $\cI_\mu$, the proof of \cref{thm:markov-matching} bounds
their failure probability by $M$, which is the claim.

\subsection{A lower bound for the linear coefficient}\label{sec:contraction-lower-bound}

We show that the linear coefficient of any uniform four-law contraction for the potential
$\phi(q)=q/(1-q)^\alpha$ is bounded below. This concerns the local estimate of
\cref{thm:four-law-contraction} alone and does not assert that the matching theorem
fails at the same exponent. Fix $p,q\in(0,1)$ and put
\[
 x\coloneqq
 \frac{1-(1-p)^\alpha}{1-(1-p)^\alpha(1-q)^\alpha},
 \qquad
 y\coloneqq
 \frac{1-(1-q)^\alpha}{1-(1-p)^\alpha(1-q)^\alpha}.
\]
These coefficients satisfy
\[
 x+(1-p)^\alpha y=(1-q)^\alpha x+y=1.
\]
Let $R$ be the symmetric relation on six distinct points $v_0,\ldots,v_5$ in which every
pair is compatible except $(v_0,v_3)$ and $(v_1,v_4)$, in either order. For small $s>0$,
set
\[
 m_s\coloneqq\frac{sx}{\phi(q)},\qquad
 n_s\coloneqq\frac{sy}{\phi(p)},
\]
and define
\[
 \rho_s\coloneqq m_s\delta_{v_0}+p\delta_{v_1}+(1-p-m_s)\delta_{v_2},
 \qquad
 \tau_s\coloneqq q\delta_{v_3}+n_s\delta_{v_4}+(1-q-n_s)\delta_{v_5}.
\]
These are probability laws for sufficiently small $s$. Take both source laws to be
$\rho_s$ and both target laws to be $\tau_s$. Every point of positive mass has positive
matching degree in both directions. The two directed potentials are
\begin{align*}
 \cE(\rho_s,\tau_s)
 &=m_s\phi(q)+p\phi(n_s)
 =s\bigl(x+(1-p)^\alpha y\bigr)+O(s^2)=s+O(s^2),\\
 \cE(\tau_s,\rho_s)
 &=q\phi(m_s)+n_s\phi(p)
 =s\bigl((1-q)^\alpha x+y\bigr)+O(s^2)=s+O(s^2).
\end{align*}
Since both source laws equal $\rho_s$ and both target laws equal $\tau_s$, the eight
directed potentials in \cref{thm:four-law-contraction} take these two values, and their
maximum is
\[
 M_s\coloneqq
 \max\{\cE(\rho_s,\tau_s),\cE(\tau_s,\rho_s)\}
 =s+O(s^2).
\]

\paragraph*{\bf First order contributions} There are two contributions of first order to the pair potential. A source pair with one
coordinate equal to $v_0$ fails when both target coordinates equal $v_3$, giving
$2m_s\phi(q^2)+O(s^2)$. A source pair equal to $(v_1,v_1)$ fails when at least one target
coordinate equals $v_4$, giving $2p^2n_s+O(s^2)$. 

\paragraph*{\bf Other contributions} All other contributions have order
$s^2$. More explicitly, the pair potential is
\begin{align*}
 &m_s^2\phi(2q-q^2)
 +2m_sp\phi(q^2+n_s^2)
 +2m_s(1-p-m_s)\phi(q^2)\\
 &\quad+p^2\phi(2n_s-n_s^2)
 +2p(1-p-m_s)\phi(n_s^2)
 =2m_s\phi(q^2)+2p^2n_s+O(s^2).
\end{align*}
Consequently
\[
 \lim_{s\downarrow0}
 \frac{\cE(\rho_s\times\rho_s,\tau_s\times\tau_s;R^\square)}{M_s}
 =\frac{2q}{(1+q)^\alpha}\,x+2p(1-p)^\alpha y.
\]
Suppose that $C\geq0$ and that
$\cE^\square\leq\lambda M+CM^2$ holds for every four-law comparison without zero degrees
whenever all eight directed potentials are at most $M$. We let $s$ tend to zero with
$p,q$ fixed, and then $q$ tend to one. This gives
\[
 \lambda\geq2^{1-\alpha}
 +\bigl(2p-2^{1-\alpha}\bigr)(1-p)^\alpha.
\]
The right side is maximised at
\[
 p=\frac{1+\alpha2^{-\alpha}}{\alpha+1}.
\]
Thus every such estimate must satisfy
\[
 \lambda\geq
 \underline{\lambda}_\alpha\coloneqq2^{1-\alpha}
 +2K_\alpha(0)(1-2^{-\alpha})^{\alpha+1}.
\]
In particular, $\underline{\lambda}_1=9/8$, so a linear coefficient at most one is
impossible at exponent one. The maximising value of $p$ lies in $(0,1)$ and is a stationary
point of the expression being maximised. The chain rule therefore shows that the derivative
of $\underline{\lambda}_\alpha$ is
\[
 2(\log 2)2^{-\alpha}\bigl((1-p)^\alpha-1\bigr)
 +2(p-2^{-\alpha})(1-p)^\alpha\log(1-p).
\]
The displayed maximiser satisfies $p>2^{-\alpha}$, so both terms are negative, and
$\underline{\lambda}_\alpha$ is strictly decreasing. A direct rational estimate gives
\[
 \underline{\lambda}_{4/3}
 \leq\frac45+2\left(\frac{10}{49}\right)\left(\frac8{25}\right)<1.
\]
Consequently, there is a unique $\underline\alpha_*\in(1,4/3)$ such that
$\underline{\lambda}_{\underline\alpha_*}=1$, and numerically
$\underline\alpha_*=1.192576572\ldots$.

\paragraph*{\bf The smallest exponent} Thus the smallest exponent admitting a uniform four-law contraction with linear
coefficient below one for this potential lies between
$\underline\alpha_*=1.192576572\ldots$ and $\alpha_*=1.325660330\ldots$ from
\cref{sec:exponent-range}. We do not determine it.

\subsection{The order of the failure probability}\label{sec:extension-limits}

The preceding subsection concerns the local estimate. We now show that the linear order of
the failure bound in \cref{thm:markov-matching} cannot be improved, even for i.i.d.\
labellings.

Take the label graph $0$--$1$--$2$ and let $\cM_h$ and $\cM_\infty$ be the matching events
for two independent i.i.d.\ labellings with law
\[
 \mu_\eps=(3/4-\eps)\delta_0+\tfrac14\delta_1+\eps\delta_2,
 \qquad 0<\eps\leq1/4.
\]
The incompatible ordered root pairs $(0,2)$ and $(2,0)$ have combined probability
$2(3/4-\eps)\eps$. Since every matching pairs the two roots,
\[
 \bbP(\cM_h^c)\geq2(3/4-\eps)\eps\quad(h\geq0),
 \qquad
 \bbP(\cM_\infty^c)\geq2(3/4-\eps)\eps.
\]
On the other hand,
\[
 \eta_\alpha(\mu_\eps)
 =(3/4-\eps)\eps
 \bigl((1-\eps)^{-\alpha}+(1/4+\eps)^{-\alpha}\bigr).
\]
Consequently, for every $h\geq0$,
\[
 \liminf_{\eps\downarrow0}
 \frac{\bbP(\cM_h^c)}{\eta_\alpha(\mu_\eps)}
 \geq\frac{2}{1+4^\alpha}>0.
\]
The same inequality holds with $\cM_\infty$ in place of $\cM_h$. Thus the linear order
$O(\eta_\alpha)$ cannot be improved to $o(\eta_\alpha)$ uniformly.

\ifSubfilesClassLoaded{

}{}

\end{document}

\section{Proof of the complete classification}\label{sec:trichotomy}
We now complete the classification and
first handle the full tree class, extend the geometric obstructions of
\cref{sec:trichotomy-simple} to general finite supports, and separate chain laws with different
branching semigroups. We then prove universality in the bushy and chain regimes. The relabellings
and binary encodings constructed in \cref{sec:general-classes,sec:general-chain} allow us to apply
the Markov matching theorem. Matching the encoded labels gives a quasi-isometry of the original
trees.

\subsection{Full trees and regime obstructions}
We start by proving quasi-isometry to the binary tree in the full tree regime. Recall that we call a rooted tree
\emph{full} if every vertex has at least two children. In the geometric group theory literature
a tree is called \emph{bushy} if every point lies within uniformly bounded distance of a vertex
with at least three unbounded complementary components. Full trees of bounded valence are bushy
in this geometric sense. Note however that in our classification we reserve \emph{bushy regime} for
the regime (B), whose realisations also contain finite, but generally unbounded bushes.

\begin{proposition}[{Full trees; Mosher--Sageev--Whyte~\cite[\S2.1]{MosherSageevWhyte2003}}]\label{thm:bushy}
  Let $N\in\bbN$ with $N\geq2$, and let $\cT\subseteq\cN(N)$ be a rooted subtree such that
  $v1,v2\in\cT$ for every $v\in\cT$. Then there is a root-preserving quasi-isometry
  $\cT\to\bbB$.
\end{proposition}

\begin{proof}
  The tree $\cT$ has bounded valence by construction. Further, every non-root vertex separates
  at least three unbounded complementary components because every vertex has at least two children,
  and the root lies at distance one from such a vertex. So $\cT$ is bushy in the terminology of
  Mosher--Sageev--Whyte, and \cite[\S2.1]{MosherSageevWhyte2003} gives a quasi-isometry
  $f\colon\cT\to\bbB$. Replacing $f(\troot)$ by the root of $\bbB$ changes $f$ by a bounded
  amount, so the resulting map remains a quasi-isometry and preserves the root.
\end{proof}

Recall the terminology of \cref{sec:trichotomy-simple}: rays, lines and bushes of depth $h$
at a vertex. We now collect the possible obstruction mechanisms.

\begin{lemma}[General-support regime obstructions]\label{thm:regime-obstructions-general}
  Let $\theta$ be a finitely supported supercritical offspring distribution, and let $\cT$ be
  a Galton--Watson tree with offspring distribution $\theta$, conditioned on having infinite
  diameter.
  \begin{enumerate}[(i),ref=(\roman*)]
    \item\label{it:gen-obstr-rays} Almost surely $\cT$ contains three rays that pairwise meet
      only in their common initial vertex.
    \item\label{it:gen-obstr-lines} If $\theta_0=0$, then almost surely there exists $\delta>0$
      such that every vertex of $\cT$ lies within distance $\delta$ of some line in $\cT$.
    \item\label{it:gen-obstr-bushes} If $\theta_0>0$, then almost surely $\cT$ has bushes of
      unbounded depth.
    \item\label{it:gen-obstr-thin} If $\theta_1>0$, then almost surely, for every
      $\ell\in\bbN$, there is a vertex $v_\ell$ whose ball $B_{\cT}(v_\ell,\ell)$ is
      isometric to a segment of $\bbZ$.
  \end{enumerate}
\end{lemma}

\begin{proof}
  Write $\cT^*$ for $\cT$ when $\theta_0=0$ and for its skeleton when $\theta_0>0$. By
  \cref{thm:harris-general}\labelcref{it:harris-general-skeleton}, the tree $\cT^*$ never dies
  out and is a subtree of $\cT$ carrying the induced metric. The proofs of
  \cref{thm:regime-obstructions}\labelcref{it:obstr-rays,it:obstr-lines} use only these facts and
  that every neck ends almost surely at a split. The latter follows from
  \cref{thm:harris-general}\labelcref{it:harris-general-reduced}. Those proofs therefore apply
  unchanged. At a split of arity greater than two, we simply select two children, which proves
  \cref{it:gen-obstr-rays,it:gen-obstr-lines}.

  For \cref{it:gen-obstr-bushes}, let $q\in(0,1)$ be the extinction probability and set
  $J=\max\supp\theta$. At a unary skeleton vertex, the event that the total offspring number is
  $J$ has positive conditional weight $J\theta_J(1-q)q^{\,J-1}$. Conditioned on this event there is a bush
  with offspring law $\theta^\dagger$. Since $\theta^\dagger_0=\theta_0/q>0$ and
  $\theta^\dagger_J=\theta_Jq^{\,J-1}>0$, this bush has positive probability of containing the
  complete $J$-ary tree of any prescribed height. Moreover, $\widetilde\theta_1=f'(q)>0$, so the
  conditional independence in
  \cref{thm:harris-general}\labelcref{it:harris-general-reduced,it:harris-general-bushes} allows
  the Borel--Cantelli argument from
  \cref{thm:regime-obstructions}\labelcref{it:obstr-bushes} to apply unchanged.

  Finally, suppose that $\theta_1>0$, and let $q$ be the extinction probability. At a unary
  skeleton vertex, the conditional probability that the total offspring number is one equals
  $\theta_1/f'(q)>0$. Thus, for every $L$, a given neck has positive probability of containing
  $L$ consecutive vertices with exactly one child. Conditioned on the infinite reduced skeleton,
  these events are independent by \cref{thm:harris-general}. The second Borel--Cantelli lemma
  therefore gives such chains of arbitrary length almost surely. The midpoint $v_\ell$ of a chain
  of length at least $2\ell+2$ has
  $B_{\cT}(v_\ell,\ell)$ contained in that chain and isometric to a segment of $\bbZ$, proving
  \cref{it:gen-obstr-thin}.
\end{proof}

\paragraph*{\bf Coarse branching counts} The branching semigroup also separates distinct chain classes. We now prove that the coarse
branching counts recorded by \cref{def:branching-semigroup} are quasi-isometry invariants of
chain-regime realisations.

\begin{proposition}[Separation within the chain regime]\label{thm:chain-separation}
  Let $\theta$ and $\theta'$ be finitely supported supercritical offspring distributions
  in the chain regime with $\Lambda_{\theta}\neq\Lambda_{\theta'}$, and let $\cT,\cT'$ be
  independent Galton--Watson trees with these offspring distributions. Then almost surely
  $\cT\not\simeq\cT'$.
\end{proposition}

\paragraph*{\bf Obstructions} Before proving \cref{thm:chain-separation}, we establish a deterministic obstruction lemma. Recall
that we write $\chi_x$ for the number of children of a vertex $x$.
We extend the notion of a \emph{branching semigroup} to rooted locally finite trees $\cS$
without leaves by setting
\[
  \Lambda(\cS)\coloneqq\left\langle \chi_x-1\colon x\in\cS\right\rangle.
\]
For the Galton--Watson trees considered below, this agrees almost surely with the branching
semigroup of the offspring law.
For $L\in\bbN$ and $n\geq3$, a vertex $v$ of a tree $T$ is the \emph{centre of an $n$-star of
length $L$} if $B_T(v,L)$ is obtained by identifying one endpoint of $n$ disjoint paths of
length $L$.

\begin{lemma}[Deterministic star obstruction]\label{thm:chain-star-obstruction}
  Let $\cS$ be a rooted locally finite tree without leaves, let $n\geq3$ satisfy
  $n-2\notin\Lambda(\cS)$, and let $D\in\bbN$ with $D\geq1$. There are $L,L'\in\bbN$ such that,
  if a rooted locally finite tree $\cS'$ contains two $n$-stars of length $L$ whose centres are
  at distance at least $L'$, then there is no $D$-quasi-isometry $\cS'\to\cS$.
\end{lemma}

\begin{proof}
  Write $S_{\cS}(w,r)$ for the set of vertices at distance $r$ from $w$.
  For $w\in\cS$ and $\rho\in\bbN_0$ with $\rho<d(w,\troot)$, counting the edges leaving
  $B_{\cS}(w,\rho)$ gives
  \begin{equation}\label{eq:sphere-count}
    \abs{S_{\cS}(w,\rho+1)}
    =2+\sum_{x\in B_{\cS}(w,\rho)}\bigl(\chi_x-1\bigr)
    \in 2+\Lambda(\cS).
  \end{equation}
  Choose integers $k,k',L,L'$ successively so that
  \[
    k>3D^3+2D,\qquad
    k'>D\bigl(D(k+3D)+D^2\bigr)+2D,\qquad
    L>D(k'+D)+D^2,\qquad
    L'>D(2k'+D).
  \]
  Suppose that $\psi\colon\cS'\to\cS$ is a $D$-quasi-isometry. The choice of $L'$ ensures that
  the image $w=\psi(v)$ of one of the two star centres satisfies $d(w,\troot)>k'$. Write $a_i$
  for the endpoint of its $i$th arm. For $i\neq j$, let $m$ be the median of
  $w,\psi(a_i),\psi(a_j)$.
  By \cref{thm:tree-stability}, we can choose $u_i\in[v,a_i]$ and $u_j\in[v,a_j]$ whose images
  lie within $D$ of $m$. The two arms meet only at $v$, so the quasi-isometry bounds give
  \[
    \begin{aligned}
      d(v,u_i)&\leq d(u_i,u_j)
      \leq D\bigl(d(\psi(u_i),\psi(u_j))+D\bigr)\leq3D^2,\\
      d(w,m)&\leq Dd(v,u_i)+2D\leq3D^3+2D<k.
    \end{aligned}
  \]
  The $n$ arms therefore determine distinct vertices of $S_{\cS}(w,k)$.

  Conversely, extend each $g\in S_{\cS}(w,k)$ along its geodesic from $w$ to a vertex $y$ at
  distance $k'$, and choose a coarse preimage $x$ with $d(\psi(x),y)\leq D$. The choice of $L$
  puts $x$ in $B_{\cS'}(v,L)\setminus\set v$, so $x$ lies on a unique arm. Suppose that distinct
  $g,g'$ are assigned to the same arm, with extensions $y,y'$ and coarse preimages $x,x'$.
  We may assume that $x$ is nearer to $v$, and put $t=d(v,x)$. Since
  $d(y,y')\geq2(k'-k)$, the median $m'$ of $w,\psi(x),\psi(x')$ satisfies
  \[
    2d(w,m')=d(w,\psi(x))+d(w,\psi(x'))-d(\psi(x),\psi(x'))\leq2k+4D.
  \]
  By \cref{thm:tree-stability}, there is $z\in[x,x']$ with $d(\psi(z),m')\leq D$. Since $x$
  lies between $v$ and $z$, the lower quasi-isometry bound gives
  \[
    t\leq d(v,z)\leq D\bigl(d(w,\psi(z))+D\bigr)\leq D(k+3D)+D^2.
  \]
  The upper bound also gives $k'=d(w,y)\leq d(w,\psi(x))+D\leq Dt+2D$, contradicting the
  choice of $k'$. Thus $\abs{S_{\cS}(w,k)}=n$, whereas
  \cref{eq:sphere-count} gives $\abs{S_{\cS}(w,k)}\in2+\Lambda(\cS)$. This contradicts
  $n-2\notin\Lambda(\cS)$.
\end{proof}

\begin{proof}[Proof of \cref{thm:chain-separation}]
  Exchanging the two laws if necessary, we may assume that
  $\Lambda_\theta\nsubseteq\Lambda_{\theta'}$. Hence there is some $k\geq2$ such that
  $\theta_k>0$ and $k-1\notin\Lambda_{\theta'}$. Almost surely,
  $\Lambda(\cT')=\Lambda_{\theta'}$: every supported offspring number occurs infinitely often
  along the ray $(1^d)_{d\geq0}$ by the second Borel--Cantelli lemma. Put $n\coloneqq k+1$.
  For each $L$, consider sufficiently separated vertices $1^d$. The event that $1^d$ has $k$
  children and that the $L$ vertices along its parent arm and each child arm are unary has a
  fixed positive probability, and these events are independent. The second Borel--Cantelli
  lemma therefore gives infinitely many $n$-stars of length $L$ almost surely. Intersecting over
  $L\in\bbN$, we obtain such stars with arbitrarily distant centres.

  Fix $D\geq1$. A pair with the lengths and separation given by
  \cref{thm:chain-star-obstruction} for $\cS=\cT'$ rules out a $D$-quasi-isometry
  $\cT\to\cT'$. Taking the countable intersection over $D\in\bbN$ proves the proposition.
\end{proof}

\subsection{Universality in the bushy and chain regimes}

We now compare two trees in the same regime. The relabelling constructions give a common
label law whose potential tends to zero as the geometric scale $D$ increases. We choose the
binary profiles independently of $D$, apply \cref{thm:markov-matching} at the admissible exponent
$5/2$, using $\lambda_{5/2}<1$ from \cref{sec:admissible-exponent}, and then glue the marked
comparisons between matched pieces. Since the profiles and arity laws are fixed, the matching
constants are independent of $D$.

\begin{proposition}[Universality in the bushy regime]\label{thm:hairy-general}
  Let $\theta$ and $\theta'$ be finitely supported supercritical offspring distributions
  with $\theta_0>0$ and $\theta_0'>0$, and let $\cT,\cT'$ be independent Galton--Watson trees
  with these offspring distributions, conditioned on having infinite diameter. Then almost
  surely there is a root-preserving quasi-isometry $\cT\to\cT'$.
\end{proposition}

\begin{proof}
  Write $J$ and $J'$ for the largest supported offspring numbers. Exchanging
  $(\theta,\cT)$ and $(\theta',\cT')$ if necessary, we may assume that $J\leq J'$. By
  \cref{sec:common-generators}, the reduced arity laws have the common core arity $2$. Choose the
  balanced $k$-leaf profile for every supported arity $k$, successively dividing its leaves into
  groups of $\lfloor k/2\rfloor$ and $\lceil k/2\rceil$. Every such profile is obtained by repeated
  leaf replacements with the two-leaf profile, so these choices give a common presentation
  as in \cref{sec:common-presentations}. Set
  $h\coloneqq\lceil\log_2\max\{J,J'\}\rceil$, which bounds their heights.

  Fix a sufficiently large integer $D$. Associate independent uniform random variables in $[0,1)$
  at the vertices on both sides, independent of one another and of the trees, to apply the maps of
  \cref{thm:cross-relabel}.
  By \cref{thm:cross-relabel}\labelcref{it:cross-relabel-law}, the maps on both sides satisfy the
  hypothesis of \cref{thm:product-form} with common label law $\mu_D$. Thus the label--arity
  pairs are i.i.d.\ with laws $\mu_D\times\widetilde\nu$ and
  $\mu_D\times\widetilde\nu'$, respectively. By \cref{thm:profile-encoding-law}, the binary
  encodings have the corresponding Markov laws, started at the types assigned to original
  vertices. We take the one-vertex label $v_0$ as the distinguished label $0$.

  The reduced arity laws and the chosen profiles satisfy the hypotheses of
  \cref{thm:common-semigroup-matching}. Their supports and the positive probabilities of
  arity $2$ are fixed by the offspring laws, so the proposition's constants are independent
  of $D$.
  Moreover, \cref{thm:cross-relabel}\labelcref{it:cross-relabel-eta} gives, for some $c>0$
  depending only on those laws,
  \[
    \eta_{G_D,5/2}(\mu_D)\leq e^{-cD^2},\qquad \mu_D(v_0)\geq\tfrac12,
  \]
  and therefore $\zeta_{G_D,5/2}(\mu_D)\leq2e^{-cD^2}$. Thus
  \cref{thm:common-semigroup-matching} applies for every sufficiently large $D$. The two
  entire binary fields can therefore be matched with probability at least $1-Ke^{-cD^2}$,
  for a constant $K$ depending only on $\theta$ and $\theta'$.

  For realisations in this this matching event, \cref{thm:cross-relabel}\labelcref{it:cross-relabel-qi} gives a
  $3^{31}D^{32}$-marked quasi-isometry from each source piece to its matched target piece.
  This includes pairs containing an inserted one-vertex piece. Hence
  \cref{thm:profile-transfer}, with $K=3^{31}D^{32}$, gives a root-preserving
  $72\cdot3^{62}(h+1)^2D^{64}$-quasi-isometry $\cT\to\cT'$.

  Let $A$ be the event that the original pair admits a root-preserving quasi-isometry. The
  matching event is contained in $A$, and $A$ depends only on the trees. Its probability
  therefore equals its probability on the space enlarged by the uniform variables, so
  $\bbP(A)\geq1-Ke^{-cD^2}$ for every sufficiently large $D$. Letting $D\to\infty$ gives
  $\bbP(A)=1$.
\end{proof}

\begin{proposition}[Universality within a chain class]\label{thm:chain-general}
  Let $\theta$ and $\theta'$ be finitely supported supercritical offspring distributions
  in the chain regime with $\Lambda_\theta=\Lambda_{\theta'}$, and let $\cT,\cT'$
  be independent Galton--Watson trees with these offspring distributions. Then almost surely
  there is a root-preserving quasi-isometry $\cT\to\cT'$.
\end{proposition}

\begin{proof}
  The reduced arity laws have the same branching semigroup as the original laws, see
  \cref{sec:common-generators}. Write $\cA$ for its atoms. By \cref{thm:common-atoms},
  both reduced supports contain every arity $a+1$ with $a\in\cA$, and the atoms generate
  every supported shift. Choose profiles for these common arities and construct the remaining
  profiles by the leaf replacements of
  \cref{sec:common-generators}. We fix these choices throughout the proof and write $h$ for their
  maximum height.

  At a sufficiently large integer scale $D$, label the first tree's necks by $\ell_D$ and
  the second tree's necks by the randomised map of \cref{thm:chain-coupling}. Use independent
  uniform variables at the vertices, independent also of the two trees. By
  \cref{eq:chain-joint}, the geometric neck lengths are independent of the arities. The two
  label maps therefore satisfy the hypotheses of \cref{thm:product-form}, with common label law
  $p^{(D)}$ by \cref{eq:chain-class-law,thm:chain-coupling}. Hence the label--arity pairs are
  i.i.d.\ with product laws $p^{(D)}\times\widetilde\nu$ and
  $p^{(D)}\times\widetilde\nu'$. \Cref{thm:profile-encoding-law} then identifies their binary
  encodings with the two corresponding Markov laws, started at the types assigned to original
  vertices.

  The reduced arity laws and the chosen profiles satisfy the hypotheses of
  \cref{thm:common-semigroup-matching}. They are fixed independently of $D$, and every
  common core arity has a fixed positive probability under each law. Consequently the
  proposition's constants are independent of $D$. By
  \cref{thm:quantised-law,thm:eta-bound}, we have $p^{(D)}_0\geq\tfrac12$ for all sufficiently
  large $D$ and
  \[
    \eta_{\mathsf P,5/2}(p^{(D)})\leq16\theta_1^{D(D-5/2)}\longrightarrow0.
  \]
  Since $\zeta_{\mathsf P,5/2}(p^{(D)})\leq2\eta_{\mathsf P,5/2}(p^{(D)})$, the proposition
  matches the two entire binary fields with probability at least
  $1-2K\eta_{\mathsf P,5/2}(p^{(D)})$, for the same constant $K$ at every such scale.

  For realisations in this event, necks with equal or adjacent labels have comparable lengths by
  \cref{thm:chain-coupling}\labelcref{it:chain-coupling-classes,it:chain-coupling-ratio}.
  We regard each inserted one-vertex piece as a neck of length $1$ with label $0$.
  Rounding the affine maps between the endpoints gives $C(\theta,\theta')D^2$-marked
  quasi-isometries in both directions, with constant maps covering one-vertex pieces. By
  \cref{thm:profile-transfer}, there is a root-preserving
  $72(h+1)^2C(\theta,\theta')^2D^4$-quasi-isometry of the original trees. Let $A$ be the event
  that the original pair admits a root-preserving quasi-isometry with some finite constant.
  This event does not depend on $D$. The matching event at each such scale is contained in $A$, and
  adjoining the auxiliary variables does not change $\bbP(A)$. Therefore
  $\bbP(A)\geq1-2K\eta_{\mathsf P,5/2}(p^{(D)})$ for every sufficiently large $D$, and hence
  $\bbP(A)=1$.
\end{proof}

\subsection{Completion of the classification}

Recall that the possible branching semigroups are the additive submonoids of $\bbN_0$ generated
by a nonempty finite set of positive integers, and that $(\mathrm{C}_\Lambda)$ denotes the chain
class with branching semigroup $\Lambda$.

\begin{restate}{thm:trichotomy}
  All finite Galton--Watson realisations form the quasi-isometry class $(\mathrm{Fin})$, and no
  finite tree is quasi-isometric to an infinite tree.

  Let $\theta$ and $\theta'$ be finitely supported offspring distributions, each either
  supercritical or with $\theta_1=1$, and let $\cT,\cT'$ be independent Galton--Watson trees
  with these offspring distributions, conditioned on having infinite diameter. Every such
  distribution belongs to exactly one of the following classes:
  \begin{itemize}
    \item[(R)] $\theta_1=1$, the \emph{ray} class, in which $\cT$ is the ray.
    \item[(F)] $\theta_0=\theta_1=0$, the \emph{full tree} class, in which $\cT$ admits
      a quasi-isometry to the binary tree.
    \item[$(\mathrm{C}_\Lambda)$] $\theta_0=0<\theta_1<1$ and
      $\Lambda_\theta=\Lambda$, the \emph{chain} class indexed by the branching
      semigroups $\Lambda$.
    \item[(B)] $\theta_0>0$, the \emph{bushy} class.
  \end{itemize}
  Two independent realisations belonging to the same class almost surely admit a root-preserving
  quasi-isometry.
  If they belong to different classes, then they are almost surely not quasi-isometric.
\end{restate}

\begin{proof}[Proof of \cref{thm:trichotomy}]
  All finite metric spaces are quasi-isometric. A quasi-isometry from a bounded space has bounded
  image, and coarse density then forces its target to be bounded. This proves the assertions about
  $(\mathrm{Fin})$. For finite trees, the constant map to the target root is root-preserving.

  Suppose first that $\theta$ and $\theta'$ belong to the same class. In class (R), the two
  trees are the ray. In class (F), \cref{thm:bushy} gives root-preserving quasi-isometries
  from both trees to the binary tree, and hence a root-preserving quasi-isometry between them.
  In a class $(\mathrm{C}_\Lambda)$, \cref{thm:chain-general} gives such a map almost surely,
  and \cref{thm:hairy-general} does so in class (B).

  Suppose next that the two laws belong to different classes. If they belong to
  $(\mathrm{C}_\Lambda)$ and $(\mathrm{C}_{\Lambda'})$ with $\Lambda\neq\Lambda'$, then
  \cref{thm:chain-separation} gives $\cT\not\simeq\cT'$ almost surely. It remains to consider
  laws in different coarse regimes.

  Suppose one law belongs to class (R). Its realisation is the ray, while the other law is
  supercritical. By
  \cref{thm:regime-obstructions-general}\labelcref{it:gen-obstr-rays}, the other realisation almost
  surely contains three rays that pairwise meet only in their common initial vertex. It follows
  from \cref{thm:three-rays} that $\cT\not\simeq\cT'$ almost surely.

  Suppose one law belongs to class (F) and the other to a chain class. The chain realisation almost
  surely satisfies the thin-ball hypothesis of \cref{thm:bottleneck}, by
  \cref{thm:regime-obstructions-general}\labelcref{it:gen-obstr-thin}. It is therefore not
  quasi-isometric to $\bbB$, whereas the realisation in the full tree class is quasi-isometric
  to $\bbB$ by \cref{thm:bushy}. Transitivity of quasi-isometry gives $\cT\not\simeq\cT'$ almost
  surely.

  In the remaining cases, one law belongs to class (B) and the other to class (F) or a chain
  class. After exchanging their names, we may assume that $\theta_0>0=\theta_0'$. By
  \cref{thm:regime-obstructions-general}\labelcref{it:gen-obstr-bushes,it:gen-obstr-lines}, the
  realisation $\cT$ almost surely has bushes of unbounded depth, while there is a finite $r$ such
  that every vertex of $\cT'$ lies within distance $r$ of a line. Hence
  \cref{thm:bush-separation} gives
  $\cT\not\simeq\cT'$ almost surely.
\end{proof}

\subsection{Dependence on common arity weights}

We conclude by using the separation of chain classes in \cref{thm:chain-separation} to show
that the constants in \cref{thm:common-semigroup-matching} cannot be chosen independently of
the probabilities of the common core arities. This obstruction already occurs for profiles
of height at most two.

\begin{proposition}[A common-weight obstruction]\label{thm:core-weight-obstruction}
Use the two-leaf profile and the three-leaf profile obtained by replacing the left leaf of
the two-leaf profile by another copy of it. For $0\leq t<1$, put
\[
 \nu_L=\delta_2,\qquad \nu_{R,t}=t\delta_2+(1-t)\delta_3.
\]
For a countable graph $G$ with distinguished vertex $0$ and a probability law $\mu$ on its
vertices, let $\bbP_t$ be the joint law of the two profile processes with common label law
$\mu$ and these arity laws. There are no constants $K<\infty$ and $\eps>0$ such that
\[
 \mu(0)\geq\tfrac12,\quad \eta_{G,5/2}(\mu)\leq\eps
 \quad\Longrightarrow\quad
 \bbP_t\bigl(\cM_h(F_L,F_R)^c\bigr)\leq K\eta_{G,5/2}(\mu)
\]
for every such $G$ and $\mu$, every $h\geq0$ and every $0<t<1$. Nor do such constants exist
with $\cM_\infty(F_L,F_R)$ in place of $\cM_h(F_L,F_R)$.
\end{proposition}

\begin{proof}
  Fix $a\in(0,1)$. Consider the chain laws
  \[
    \theta=a\delta_1+(1-a)\delta_2,\qquad
    \theta'_t=a\delta_1+(1-a)\bigl(t\delta_2+(1-t)\delta_3\bigr),
    \quad 0\leq t<1.
  \]
  The neck law is geometric with ratio $a$ on both sides, independently of $t$. At an
  integer scale $D$, label the necks on both sides by $\ell_D$ and write $p^{(D)}$ for their
  common law on the path graph $\mathsf P$. This agrees with the cross-law labelling of
  \cref{thm:chain-coupling}: the two unary probabilities are equal to $a$, so $\gamma=1$,
  the cut points are $D_k'=D^k$ and all boundary probabilities $\beta_k$ vanish. Hence the
  randomised map for the second law is exactly $\ell_D$. By
  \cref{thm:quantised-law,thm:eta-bound},
  \[
    p^{(D)}_0\longrightarrow1,\qquad
    \eta_{\mathsf P,5/2}(p^{(D)})\longrightarrow0.
  \]
  The arity laws are $\nu_L$ and $\nu_{R,t}$, so \cref{thm:profile-encoding-law} identifies
  their binary label fields with the two processes in the statement.

  At $t=0$, the branching semigroups of the original independent trees are $\bbN_0$ and
  $2\bbN_0$. They are almost surely not quasi-isometric by \cref{thm:chain-separation}.
  A matching of their entire binary fields would give a quasi-isometry by
  \cref{thm:profile-transfer}: the neck comparison in the proof of
  \cref{thm:chain-general} does not require the branching semigroups to agree and also covers
  inserted one-vertex pieces. Hence
  $\bbP_0\bigl(\cM_\infty(F_L,F_R)\bigr)=0$ at every fixed sufficiently large $D$.

  For a fixed height $h$, couple the right arities at $t$ and at $0$ at every binary address.
  Their disagreement probability is $t$, so the finite label fields agree with probability at
  least $1-(2^{h+1}-1)t$. In particular,
  \[
   \bbP_t\bigl(\cM_h(F_L,F_R)^c\bigr)
   \longrightarrow \bbP_0\bigl(\cM_h(F_L,F_R)^c\bigr)
   \qquad(t\downarrow0).
  \]

  Suppose the proposed constants existed. Choose $D$ with
  $p^{(D)}_0\geq\tfrac12$, $\eta_{\mathsf P,5/2}(p^{(D)})\leq\eps$ and
  $K\eta_{\mathsf P,5/2}(p^{(D)})<\tfrac12$. Passing to $t=0$ at each fixed $h$ gives
  \[
    \bbP_0\bigl(\cM_h(F_L,F_R)\bigr)
    \geq1-K\eta_{\mathsf P,5/2}(p^{(D)})>\tfrac12.
  \]
  By the finite-branching matching argument in \cref{sec:completion-proof},
  $\cM_\infty(F_L,F_R)=\bigcap_h\cM_h(F_L,F_R)$. Continuity from above therefore gives
  $\bbP_0\bigl(\cM_\infty(F_L,F_R)\bigr)\geq\tfrac12$, a contradiction. An
  infinite-matching bound would imply the same finite-height bounds, so it is impossible as
  well.
\end{proof}

\section{Quasi-isometric embeddings}\label{sec:embedding-hierarchy}

We now use the classification to compare trees by quasi-isometric embeddings. We recall the
theorem from the introduction.

\begin{restate}{thm:embedding-hierarchy}
  Let $\theta$ be a finitely supported supercritical offspring distribution, and let $\cT$ be a
  Galton--Watson tree with offspring distribution $\theta$, conditioned on having infinite
  diameter. Almost surely, there are quasi-isometric embeddings
  \[
    \cT\hookrightarrow\bbB
    \quad\text{and}\quad
    \bbB\hookrightarrow\cT.
  \]
\end{restate}

The embedding into $\bbB$ follows from the bounded offspring support. For the reverse direction,
the property $\bbB\preccurlyeq X$ is preserved under quasi-isometry of $X$, by composition.
We may therefore choose a convenient offspring law within each chain class. For a bushy
realisation, it suffices to embed into its survival skeleton, since the inclusion of the skeleton
with its original unit edges is isometric. We first prove the pruning lemma that supplies a binary
subtree for the chosen representatives, then complete the proof of
\cref{thm:embedding-hierarchy}.

\subsection{Pruning a concentrated offspring law}

Fix an integer $J\geq24$ and put $m=\floor{J/2}+1$. We use the independent offspring
variables $(\chi_v)_{v\in\cN(J)}$ on the ambient word tree, including at words outside the
Galton--Watson realisation. Define the sets of vertices surviving successive pruning rounds by
\begin{align*}
  V_0&=\cN(J),\\
  V_{n+1}&=\left\{v\in\cN(J)\colon \chi_v=J,\quad
    \#\set{j\in\set{1,\ldots,J}\colon vj\in V_n}\geq m\right\},\\
  V_\infty&=\bigcap_{n\geq0}V_n.
\end{align*}
We call the vertices in $V_\infty$ \emph{retained}. From a retained vertex that belongs to
the realisation, we form its \emph{retained descendant tree} by keeping all retained children
and repeating this rule at each successive vertex.

\begin{lemma}[A regular subtree]\label{thm:concentrated-regular-subtree}
  Let $\theta$ be an offspring law supported on $\set{1,\ldots,J}$ with
  $\theta_J\geq7/8$, and let $\cT$ be its Galton--Watson tree. Its root is retained with
  probability at least $3/4$. Conditioned on this event, its retained descendant tree is a
  Galton--Watson tree with offspring supported on $\set{m,\ldots,J}$, and it contains a
  regular $m$-ary subtree rooted at the original root. Every vertex of this subtree has $J$
  children in $\cT$. Almost surely, a retained vertex occurs in $\cT$ at finite depth. In
  particular, $\cT$ almost surely contains an isometric copy of the binary tree.
\end{lemma}

\begin{proof}
  The operation defining $V_{n+1}$ is nondecreasing in $V_n$. Since $V_1\subseteq V_0$,
  induction gives $V_{n+1}\subseteq V_n$ for every $n$. A vertex belongs to $V_n$ precisely
  when it begins a regular $m$-ary subtree of height $n$ whose vertices at depth less than
  $n$ all have offspring number $J$.

  Write $p_n=\bbP(\troot\in V_n)$, so $p_0=1$. Membership of $vj$ in $V_n$ depends only
  on offspring variables at words beginning with $vj$. These variables are independent for
  distinct children and independent of $\chi_v$. Each child belongs to $V_n$ with
  probability $p_n$, so
  \begin{equation}\label{eq:pruned-subtree-recursion}
    p_{n+1}=\theta_J\,\bbP\bigl(\operatorname{Bin}(J,p_n)\geq m\bigr).
  \end{equation}

  We show by induction that $p_n\geq3/4$. If $Y$ has the binomial law with parameters $J$
  and $3/4$, then its mean is $3J/4$ and its variance is $3J/16$. Since $Y<m$ implies
  $Y\leq J/2$, Chebyshev's inequality gives
  \[
    \bbP(Y<m)
    \leq\bbP\bigl(\abs{Y-3J/4}\geq J/4\bigr)
    \leq\frac{3J/16}{(J/4)^2}
    =\frac3J.
  \]
  The binomial upper tail is nondecreasing in its success probability. Thus, if
  $p_n\geq3/4$, \cref{eq:pruned-subtree-recursion} yields
  \[
    p_{n+1}\geq\theta_J\left(1-\frac3J\right)
    \geq\frac78\cdot\frac78
    =\frac{49}{64}>\frac34.
  \]
  This proves the induction.

  Continuity of probability for the decreasing events $\set{\troot\in V_n}$ gives
  \[
    p\coloneqq\bbP(\troot\in V_\infty)
    =\lim_{n\to\infty}p_n\geq\frac34.
  \]
  We next check that surviving every finite pruning round leaves enough children to continue
  indefinitely. For a fixed vertex $v$, the sets
  $\set{j\in\set{1,\ldots,J}\colon vj\in V_n}$ decrease and are finite, so their
  cardinalities converge to the cardinality of their intersection. Therefore
  \begin{equation}\label{eq:pruned-subtree-fixed-point}
    v\in V_\infty
    \quad\Longleftrightarrow\quad
    \chi_v=J\quad\text{and}\quad
    \#\set{j\in\set{1,\ldots,J}\colon vj\in V_\infty}\geq m.
  \end{equation}
  For the forward implication, each finite-round count is at least $m$. For the reverse
  implication, $m$ retained children belong to every $V_n$, so $v$ belongs to every
  $V_{n+1}$ and hence to $V_\infty$.

  By \cref{eq:pruned-subtree-fixed-point}, the retained descendant tree of any retained
  vertex in $\cT$ has at least $m$ children at every vertex. Choosing its first $m$ retained
  children at each vertex constructs a regular $m$-ary subtree. All chosen children belong
  to $\cT$, since their parents have $J$ children there.

  We also identify the branching law after pruning. The indicators
  $\mathbf1_{\{vj\in V_\infty\}}$, $1\leq j\leq J$, are independent Bernoulli variables
  with parameter $p$, and are independent of $\chi_v$. By
  \cref{eq:pruned-subtree-fixed-point}, conditioned on $v$ being retained, the number of
  retained children has the binomial law with parameters $J,p$ conditioned to be at least
  $m$. Its probabilities are
  \[
    \widehat\theta_r
    =\frac{\binom Jr p^r(1-p)^{J-r}}
      {\displaystyle\sum_{k=m}^J\binom Jk p^k(1-p)^{J-k}}
    \qquad(m\leq r\leq J).
  \]
  Conditioned on which children are retained, the descendant offspring fields of the retained
  children are independent, each with the original law conditioned on its own root being retained.
  The further condition that $v$ is retained only specifies $\chi_v=J$ and that at least
  $m$ children are retained. Consequently, after listing retained children in their original
  order, their retained descendant trees are independent copies of the retained tree
  conditioned on its root being retained. This proves the Galton--Watson assertion.

  Since every vertex has at least one child, all vertices $1^n$, $n\geq0$, belong to
  $\cT$. Consider the events
  \[
    A_n=\set{\chi_{1^n}=J,\ 1^n2\in V_\infty}\qquad(n\geq0).
  \]
  On $A_n$, the retained vertex $1^n2$ belongs to $\cT$. The events $A_n$ depend on
  disjoint sets of offspring variables: the respective spine vertices and their disjoint
  second-child descendant trees. They are therefore independent, with common probability
  $\theta_Jp\geq21/32$. By the second Borel--Cantelli lemma, almost surely infinitely
  many $A_n$ occur. More precisely, for
  $\tau=\min\set{n\geq0\colon A_n\text{ occurs}}$, with $\tau=\infty$ if none occurs,
  independence gives
  \[
    \bbP(\tau\geq N)=(1-\theta_Jp)^N
    \leq\left(\frac{11}{32}\right)^N\qquad(N\in\bbN_0).
  \]
  Thus a retained descendant tree occurs at the almost surely finite depth $\tau+1$.

  At any successful vertex, select two of the $m\geq2$ retained children at every successive
  vertex. The resulting binary tree is a subtree of $\cT$ with its original unit edges,
  so its inclusion is isometric.
\end{proof}

\subsection{The embeddings and strict hierarchy}

\begin{proof}[Proof of \cref{thm:embedding-hierarchy}]
  Choose $N\geq2$ bounding the offspring support. The word representation gives an isometric
  inclusion $\cT\hookrightarrow\cN(N)$, and \cref{thm:bushy} gives $\cN(N)\simeq\bbB$.
  Composing these maps proves $\cT\preccurlyeq\bbB$.

  For the reverse direction, the full tree case follows directly from \cref{thm:bushy}.
  Next suppose that $\cT$ belongs to a chain class $(\mathrm{C}_\Lambda)$, and let $\cA$
  be the finite set of positive atoms of $\Lambda$. These atoms generate $\Lambda$, by
  \cref{thm:common-atoms}.
  Since $\Lambda$ contains arbitrarily large positive multiples of any one of its atoms,
  we may choose an integer $J$ such that
  \[
    J\geq24,\qquad J>1+\max\cA,\qquad J-1\in\Lambda.
  \]
  Choose $0<\eps\leq1/8$ and a probability law $\nu$ giving positive mass to each element
  of $\set1\cup\set{a+1\colon a\in\cA}$ and no mass elsewhere. Define a representative
  offspring law by
  \[
    \theta'_J=1-\eps,\qquad
    \theta'_k=\eps\nu_k\quad(1\leq k<J),\qquad
    \theta'_0=0,
  \]
  with zero mass above $J$. Its branching semigroup is
  \[
    \Lambda_{\theta'}=\langle\cA\cup\set{J-1}\rangle=\Lambda,
  \]
  and $0<\theta'_1<1$, so it belongs to the chain class $(\mathrm{C}_\Lambda)$.
  In particular, only offspring numbers whose excess lies in $\Lambda$ receive mass.
  Since $\theta'_J\geq7/8$, \cref{thm:concentrated-regular-subtree} shows that its
  realisation $\cT'$ almost surely contains an isometric copy of $\bbB$.

  Take $\cT'$ independently of $\cT$. By \cref{thm:trichotomy}, the two realisations are
  almost surely quasi-isometric. On the joint probability-one event where both assertions hold,
  we compose the binary-tree inclusion with a quasi-isometry $\cT'\to\cT$. Thus almost
  every $\cT$ admits a quasi-isometric embedding of $\bbB$. Since $\Lambda$ was arbitrary,
  this covers all chain classes.

  Finally, suppose that $\cT$ is in the bushy regime and let $\cT^*$ be its survival
  skeleton with the original unit edges. By \cref{thm:harris-general,thm:full-support},
  the skeleton has a finitely supported supercritical offspring law $\widetilde\theta$ with
  $\widetilde\theta_0=0$, $0<\widetilde\theta_1<1$ and $\widetilde\theta_2>0$.
  Its branching semigroup is therefore $\bbN_0$, so the chain case gives
  $\bbB\preccurlyeq\cT^*$ almost surely. Composing with the isometric inclusion
  $\cT^*\hookrightarrow\cT$ gives $\bbB\preccurlyeq\cT$.
\end{proof}

For any fixed pair of finitely supported supercritical laws and independent survival-conditioned
realisations, \cref{thm:embedding-hierarchy} gives, almost surely,
\[
  \cT\preccurlyeq\bbB\preccurlyeq\cT',
  \qquad
  \cT'\preccurlyeq\bbB\preccurlyeq\cT.
\]
We finish by checking the strict comparisons with the finite and ray categories. Every finite
tree $F$ is quasi-isometric to a point, which embeds isometrically into the ray $\bbN_0$.
The lower metric bound prevents an unbounded space from embedding into a bounded space, so
$F\prec\bbN_0$.

Every infinite locally finite rooted tree contains an infinite ray from its root, whose inclusion
is isometric. Thus $\bbN_0\preccurlyeq\cT$. By
\cref{thm:regime-obstructions-general}\labelcref{it:gen-obstr-rays}, a supercritical realisation
conditioned on survival almost surely contains three rays meeting only at their common initial
vertex. The proof of \cref{thm:three-rays}, including its use of \cref{thm:tree-stability},
uses only the two metric inequalities and does not use coarse surjectivity. The same argument
therefore excludes a quasi-isometric embedding into $\bbN_0$. Consequently, almost surely,
\[
  F\prec\bbN_0\prec\cT.
\]
These comparisons give the strict hierarchy stated in the introduction.

\clearpage

\clearpage
\appendix
\crefalias{section}{appendix}
\crefalias{subsection}{subappendix}

\section{Glossary}\label{app:glossary}

The two tables below collect notation and terminology used beyond the passage in which they are
introduced. Each entry points to its definition and page. We omit standard notation and variables
local to one proof or calculation. When a symbol has different persistent roles, the group heading
and the entry itself distinguish the relevant contexts.

\makeatletter
\newcommand{\glossarysubsection}{%
  \@startsection{subsection}{2}%
    \z@{.5\linespacing\@plus.7\linespacing}{.3\linespacing}%
    {\normalfont\bfseries}}
\makeatother

\glossarysubsection{Symbols}

\begingroup
\footnotesize
\def\glgroup#1{\multicolumn{4}{@{}l@{}}{\rule{0pt}{3.2ex}\itshape#1}\\[2pt]}%
\tablehead{\textbf{Symbol} & \textbf{Meaning} & \textbf{Defined at} &
 \textbf{Page}\\[2pt] \hline}
\begin{supertabular}{@{}L{0.17\textwidth}|L{0.49\textwidth}|L{0.15\textwidth}|c@{}}
\shrinkheight{24pt}
\glgroup{Trees and coarse geometry}
$\cN=\cN(N)$ & the tree of finite words over $\set{1,\dots,N}$ &
 \cref{sec:words} & \pageref{sec:words}\\
$\bbB$, $\bbB_h$ & the binary tree $\cN(2)$ and the complete binary tree of height $h$ &
 \cref{sec:words,sec:matching} & \pageref{sec:words}, \pageref{sec:matching}\\
$\Aut(\bbB)$, $\Aut(\bbB_h)$ & the groups of rooted automorphisms of the infinite and
 height-$h$ binary trees & \cref{sec:words,sec:setup} &
 \pageref{sec:words}, \pageref{sec:setup}\\
$\chi_v$, $\theta$, $\theta_i$ & the number of children assigned to $v$, the offspring law,
 and its mass at $i$ & \cref{sec:gw-trees} & \pageref{sec:gw-trees}\\
$\cT$, $\cT(\omega)$ & a Galton--Watson tree, with the realisation displayed when needed &
 \cref{sec:gw-trees} & \pageref{sec:gw-trees}\\
$\partial\cT$, $d_{\mathrm{vis}}$ & the boundary of $\cT$ and its visual metric, given by
 $d_{\mathrm{vis}}(\xi,\zeta)=2^{-\abs{\xi\wedge\zeta}}$ for $\xi\ne\zeta$ and by
 $d_{\mathrm{vis}}(\xi,\xi)=0$ &
 \cref{sec:visual,eq:visual-metric} & \pageref{sec:visual}\\
$Y_1\simeq Y_2$ & the metric spaces $Y_1$ and $Y_2$ are quasi-isometric &
 \cref{def:qi} & \pageref{def:qi}\\
$X\preccurlyeq Y$, $X\prec Y$ & quasi-isometric embeddability, and embeddability with no
 embedding in the reverse direction & \cref{sec:qi-qs} & \pageref{sec:qi-qs}\\
$\eta$ & the distortion gauge of an $\eta$-quasisymmetric homeomorphism &
 \cref{def:qs} & \pageref{def:qs}\\
\glgroup{Classification}
$\Lambda_\theta$ & the branching semigroup of an offspring law satisfying $\theta_0=0$ &
 \cref{def:branching-semigroup} & \pageref{def:branching-semigroup}\\
$s(\Lambda)$ & the least positive element of a branching semigroup &
 \cref{sec:common-generators} & \pageref{sec:common-generators}\\
$(\mathrm{Fin})$ & the finite quasi-isometry class &
 \cref{sec:classification} & \pageref{sec:classification}\\
(R), (F), (B) & the ray, full-tree and bushy coarse regimes and quasi-isometry classes &
 \cref{sec:classification} & \pageref{sec:classification}\\
(C), $(\mathrm{C}_\Lambda)$ & the coarse chain regime and its semigroup-indexed
 quasi-isometry classes &
 \cref{sec:classification} & \pageref{sec:classification}\\
\glgroup{Graph matching}
$G=(V,E)$, $\mu$ (labels), $0$ & the countable label graph, the label law and, in the Markov
 setting, the prescribed label & \cref{sec:matching,sec:process} &
 \pageref{sec:matching}, \pageref{sec:process}\\
$B_G(v,1)$, $b(v)$ & the closed unit ball around $v$ and its $\mu$-mass, the probability
 that an independent label is compatible with $v$ &
 \cref{sec:matching} & \pageref{sec:matching}\\
$\eta_{G,\alpha}(\mu)$ & the graph potential of the label law &
 \cref{eq:potential} & \pageref{eq:potential}\\
$p_c$ & the probability that two independent labels are compatible &
 \cref{sec:iid-scalar-recursion} & \pageref{sec:iid-scalar-recursion}\\
$\alpha$, $\phi_\alpha(t)$ & the potential exponent and the weight
 $\phi_\alpha(t)=t/(1-t)^\alpha$ & \cref{eq:alpha} & \pageref{eq:alpha}\\
$\cX$, $R$, $q(x)$, $r(x)$ & an abstract label space with a symmetric reflexive relation,
 and the bad and good degrees of $x$ & \cref{sec:contraction} &
 \pageref{sec:contraction}\\
$\Phi(R,\mu)$ & the potential of the relation $R$ under $\mu$ &
 \cref{eq:iid-potential} & \pageref{eq:iid-potential}\\
$R^\square$ & the symmetrised square of a relation &
 \cref{eq:square} & \pageref{eq:square}\\
$L_\alpha$, $K_\alpha$, $c_\alpha$ & the constants in the pointwise potential estimates &
 \cref{eq:maxima,eq:chord} & \pageref{eq:maxima}, \pageref{eq:chord}\\
$A_\alpha$, $B_\alpha$, $M_\alpha$ & the constants in the i.i.d.\ contraction estimates &
 \cref{eq:contraction-constants} & \pageref{eq:contraction-constants}\\
$x\approx^{\mathrm{leaf}}_hy$, $x\approx_hy$ & leaf matching and full matching of two
 height-$h$ labellings & \cref{sec:setup} & \pageref{sec:setup}\\
$R_0$, $\Phi_0$ & the one-site matching relation and its potential &
 \cref{eq:Phi0} & \pageref{eq:Phi0}\\
$\cX_h$, $\mu_h$ & the set of height-$h$ labellings and its product law &
 \cref{eq:base} & \pageref{eq:base}\\
$\Phi_h$, $\rho$ & the leaf potential and its contraction ratio &
 \cref{eq:leaf-red} & \pageref{eq:leaf-red}\\
$\Theta_h$, $\Phi^\square_h$, $\mathfrak K$ & the full potential, the child-pair potential
 and the full-recursion constant & \cref{eq:sigma,eq:full-hyp} & \pageref{eq:sigma}\\
$\mathsf P$, $s_j$, $\eta_{\mathsf P,\alpha}(p)$ & the path graph on $\bbN_0$, the compatible
 mass at $j$ and the corresponding path potential &
 \cref{sec:integer} & \pageref{sec:integer}\\
\glgroup{Branching-time models}
$l$, $X_l$ & a positive length field on $\bbB$ and its associated metric tree &
 \cref{sec:branching-times} & \pageref{sec:branching-times}\\
$l(w)$ & the independent random length assigned to the binary branch vertex $w$ &
 \cref{sec:branching-times} & \pageref{sec:branching-times}\\
$f^+_\lambda$, $f^\pm_\lambda$ & the shifted exponential and two-sided branching-time
 densities & \cref{sec:branching-times} & \pageref{sec:branching-times}\\
$X^+_\lambda$, $X^\pm_\lambda$ & the random metric trees with the corresponding independent
 branching-time lengths & \cref{sec:branching-times} & \pageref{sec:branching-times}\\
$p_k$, $q^{(D)}$, $\mathsf P_{\bbZ}$ & the one-sided bin masses, the symmetric bin law on
 $\bbZ$ and the path graph used to compare its bins &
 \cref{thm:two-sided-potential} & \pageref{thm:two-sided-potential}\\
\glgroup{Chains and quantisation}
$\kappa(v,\omega)$, $\Omega_0$ & the chain length starting at $v$ and the full-probability
 event on which every offspring number lies in $\set{1,2}$ and all chain lengths are finite &
 \cref{sec:encoding} &
 \pageref{sec:encoding}\\
$\iota(w,\omega)$, $\lambda(w,\omega)$ & the initial vertex and length of the chain indexed
 by $w\in\bbB$ & \cref{eq:phi-def,eq:lambda-def} &
 \pageref{eq:phi-def}, \pageref{eq:lambda-def}\\
$(\bbB,\lambda)$ & the labelled binary tree encoding a binary-support chain-regime realisation &
 \cref{sec:encoding} & \pageref{sec:encoding}\\
$\ell_D$, $p^{(D)}$, $D_0(\theta_1)$ & the level map, the law of the quantised labels and the
 scale above which the potential estimate applies &
 \cref{sec:quantisation,thm:quantised-law,thm:eta-bound} &
 \pageref{sec:quantisation}, \pageref{thm:quantised-law}, \pageref{thm:eta-bound}\\
$\cC_k$ & the level class $\set{n\colon D^k\leq n<D^{k+1}}$ &
 \cref{sec:quantisation} & \pageref{sec:quantisation}\\
$\gamma$, $D_k'$, $\ell'$, $\cC_k'(u)$ & the exponent ratio, the cut points, the randomised
 level map for a second chain law and its classes &
 \cref{thm:chain-coupling} & \pageref{thm:chain-coupling}\\
$\gamma_*$ & the comparability constant
 $\max(1,\gamma)/\min(1,\gamma/2)$ &
 \cref{thm:cross-law} & \pageref{thm:cross-law}\\
\glgroup{Skeletons and shapes}
$\bbP^*$ & the law conditioned on having infinite diameter &
 \cref{sec:gw-trees} & \pageref{sec:gw-trees}\\
$J$, $f$ & the largest supported offspring number and the offspring generating function
 $f(s)=\sum_{j=0}^J\theta_js^j$ & \cref{sec:general-classes} &
 \pageref{sec:general-classes}\\
$q$ (extinction), $\widetilde\theta$, $\theta^\dagger$ & the extinction probability, the skeleton
 offspring law and, when $q>0$, the conjugate bush law; for general support,
 $\theta^\dagger_j=\theta_jq^{j-1}$ &
 \cref{thm:harris,thm:harris-general,eq:conjugate-general} &
 \pageref{thm:harris}, \pageref{thm:harris-general}, \pageref{eq:conjugate-general}\\
$\widetilde f$, $\widetilde\nu$ & the generating function of the skeleton law and the
 reduced-skeleton arity law & \cref{eq:hs-transform,eq:reduced-law} &
 \pageref{eq:hs-transform}, \pageref{eq:reduced-law}\\
$\widetilde\lambda(w)$ & the chain length at $w$ in the binary-support skeleton &
 \cref{sec:shape-harris} & \pageref{sec:shape-harris}\\
$\sigma$, $\cS$, $\abs\sigma$ & a shape, the countable shape space and the number of vertices
 in its realisation; in the binary-support case
 $\sigma=(m,(b_1,\dots,b_{m-1}))$ &
 \cref{def:shape,def:shape-general} & \pageref{def:shape}, \pageref{def:shape-general}\\
$m(w)$, $k(w)$ & the neck length and terminating arity of the general shape at $w$ &
 \cref{sec:general-shapes} & \pageref{sec:general-shapes}\\
$S(w)$ & the shape indexed by $w$; the binary-support shapes are i.i.d.\ with law $\mu$,
 while the general shapes have conditional law $\mu_{k(w)}$ given the arity field &
 \cref{thm:shape-iid,thm:conditional-iid} &
 \pageref{thm:shape-iid}, \pageref{thm:conditional-iid}\\
$\mu_k$, $\mu$ (shapes) & the general shape law conditioned on terminating arity $k$ and its mixture
 against $\widetilde\nu$ & \cref{def:conditional-laws,thm:conditional-explicit} &
 \pageref{def:conditional-laws}, \pageref{thm:conditional-explicit}\\
$\cR_D$, $\rep_D$, $G_D$ & the shape net, its representative map and its label graph &
 \cref{def:shape-net,sec:general-relabel} &
 \pageref{def:shape-net}, \pageref{sec:general-relabel}\\
$\mu_D$ & the pushforward shape-label law on $\cR_D$ &
 \cref{sec:shape-eta,thm:relabel} & \pageref{sec:shape-eta}, \pageref{thm:relabel}\\
$v_0$ & the label of the one-vertex shape, also assigned to inserted profile vertices &
 \cref{thm:shape-eta,thm:cross-relabel} &
 \pageref{thm:shape-eta}, \pageref{thm:cross-relabel}\\
$\mathsf Q$, $\mu_{\mathsf Q}$, $\ell$, $\ell'$ & the coupling graph, its label law and the two
 label maps in the binary-support shape coupling & \cref{thm:shape-coupling} &
 \pageref{thm:shape-coupling}\\
$\pi_k$, $\pi'_k$, $\ell_k$, $\ell'_k$ & the transport couplings of the conditional shape
 laws with the common reference mixture $\mu$ and the resulting randomised label maps &
 \cref{thm:relabel,thm:cross-relabel} &
 \pageref{thm:relabel}, \pageref{thm:cross-relabel}\\
$C_k$, $h_C$ & a binary profile with $k$ ordered leaves and the largest chosen profile height &
 \cref{sec:binary-presentations} & \pageref{sec:binary-presentations}\\
$X_C$, $d_C$, $j_C$, $p_C$ & a binary encoding, its metric, the inclusion of the original
 assembly and the collapse back to it & \cref{thm:profile-encoding} &
 \pageref{thm:profile-encoding}\\
\glgroup{Markov matching}
$\cI$, $\cI_\mu$, $P_t$ & the type set, the types receiving random labels and the child-type
 law at type $t$ & \cref{sec:process} & \pageref{sec:process}\\
$J_t$ & an optional nonempty set of retained transitions at type $t$ &
 \cref{thm:markov-matching} & \pageref{thm:markov-matching}\\
$\cL_h$, $\rho_{t,h}$ & the space of maps $\bbB_h\to V$ and the law on this space of the
 labelling started at type $t$ &
 \cref{sec:type-potentials} & \pageref{sec:type-potentials}\\
$\cM_h(s,t)$, $\cM_\infty(s,t)$ & the finite- and infinite-height matching events for
 independent processes started at $s$ and $t$ &
 \cref{sec:process} & \pageref{sec:process}\\
$\zeta_{G,\alpha}(\mu)$, $\zeta_\alpha$ & the maximum of the one-site potential and the
 root-defect term $(1-b(0))/b(0)^\alpha$, and its abbreviated form; by convention,
 $\zeta_\alpha=\infty$ when $b(0)=0$ &
 \cref{eq:root-defect} & \pageref{eq:root-defect}\\
$g$, $\cI_i$, $H$ & the number and cyclic partition of the type classes and the uniform
 return bound; in the common-profile application, $g=\gcd R_{\mathrm{core}}$ &
 \cref{sec:finite-hypotheses,sec:return-times} &
 \pageref{sec:finite-hypotheses}, \pageref{sec:return-times}\\
$T$, $B$ & the largest type-class size and the transition budget &
 \cref{sec:unweighted,eq:transition-budget} & \pageref{sec:unweighted}\\
$\lambda_\alpha(\beta)$, $\lambda_\alpha$ & the linear coefficient for a chosen $\beta$ and
 its minimum over $\beta$, which determines the admissible exponents &
 \cref{eq:mean-constants,eq:markov-exponent} &
 \pageref{eq:mean-constants}, \pageref{eq:markov-exponent}\\
$L_\alpha(\beta)$, $K_\alpha(\beta)$, $C_\alpha(u)$ & the terms in the linear coefficient
 and the quadratic coefficient in the four-law estimate &
 \cref{eq:mean-constants,eq:four-law-constants} &
 \pageref{eq:mean-constants}, \pageref{eq:four-law-constants}\\
$r_\tau(x)$, $q_\tau(x)$, $W_\tau(x)$ & the matching degree against $\tau$, its complement
 and the inverse-degree weight, which is zero when the matching degree vanishes &
 \cref{sec:restricted-potential} & \pageref{sec:restricted-potential}\\
$\cE(\rho,\tau;R)$ & the restricted potential for a source law $\rho$, target law $\tau$
 and relation $R$ & \cref{sec:restricted-potential} &
 \pageref{sec:restricted-potential}\\
$\cE_h(s,t)$, $W_{t,h}$, $\delta_{1,h}$ & the typed restricted potential, inverse-degree weight and
 largest zero-degree probability at height $h$ &
 \cref{sec:type-potentials} & \pageref{sec:type-potentials}\\
$\cU_h(\cJ)$ & for a source type in $\cI_i$ and a nonempty set $\cJ\subseteq\cI_i$, the event
 of zero matching degree against every type in $\cJ$ &
 \cref{sec:unweighted} & \pageref{sec:unweighted}\\
$S_H$, $Z$ & the number of possible stopping positions and the uniform zero-degree bound &
 \cref{thm:explicit-zero-bound} & \pageref{thm:explicit-zero-bound}\\
$I_{\mathrm{bad}}$, $R_\mu$ & the averaged incompatible-root weight and compatible-root
 factor & \cref{sec:weighted-mass} & \pageref{sec:weighted-mass}\\
$G_H(x)$, $E(M)$, $Q(M)$, $A_\mu$ & the finite geometric sum, weighted zero-degree bound,
 averaged child-pair bound and root coefficient used in the scalar induction &
 \cref{eq:explicit-weighted-error,sec:averaging,sec:independent-root} &
 \pageref{eq:explicit-weighted-error}, \pageref{sec:averaging},
 \pageref{sec:independent-root}\\
\glgroup{Common-profile applications and quantitative bounds}
$\nu_L$, $\nu_R$, $S_L$, $S_R$ & the two arity laws and their finite supports &
 \cref{sec:common-presentations} & \pageref{sec:common-presentations}\\
$S_{\mathrm{core}}$, $\cA$, $s_0$ & the common generating core, the atoms of its branching
 semigroup and the least positive atom &
 \cref{sec:common-presentations,thm:common-atoms} &
 \pageref{sec:common-presentations}, \pageref{thm:common-atoms}\\
$F_L$, $F_R$ & the types assigned to original vertices in the two binary encodings &
 \cref{sec:types-degrees} & \pageref{sec:types-degrees}\\
$p_{\mathrm{core}}$ & a common positive lower bound for the probabilities of the core
 arities & \cref{thm:atomic-normalisation} & \pageref{thm:atomic-normalisation}\\
$R_{\mathrm{core}}$, $\Gamma$ & the core leaf depths and the additive semigroup they
 generate & \cref{sec:return-times} & \pageref{sec:return-times}\\
$M_\eta$ & the scalar barrier used when the prescribed label is always compatible &
 \cref{eq:zero-compatible-quadratic} & \pageref{eq:zero-compatible-quadratic}\\
$K_{\mathrm{match}}$ & the multiplier in the constructive uniform failure bound &
 \cref{thm:constructive-threshold} & \pageref{thm:constructive-threshold}\\
$\alpha_*$, $\underline\alpha_*$ & the threshold exponents for $\lambda_\alpha$ and its
 universal lower bound $\underline\lambda_\alpha$, respectively &
 \cref{sec:exponent-range,sec:contraction-lower-bound} &
 \pageref{sec:exponent-range}, \pageref{sec:contraction-lower-bound}
\end{supertabular}
\endgroup

\glossarysubsection{Terminology}

\begingroup
\footnotesize
\def\glgroup#1{\multicolumn{4}{@{}l@{}}{\rule{0pt}{3.2ex}\itshape#1}\\[2pt]}%
\tablehead{\textbf{Term} & \textbf{Meaning} & \textbf{Defined at} &
 \textbf{Page}\\[2pt] \hline}
\begin{supertabular}{@{}L{0.17\textwidth}|L{0.49\textwidth}|L{0.15\textwidth}|c@{}}
\glgroup{Coarse geometry and branching processes}
quasi-isometry & a map distorting distances by at most an affine amount and having coarsely
 dense image & \cref{def:qi} & \pageref{def:qi}\\
root-preserving quasi-isometry & a quasi-isometry between rooted spaces that maps the source
 root to the target root & \cref{sec:qi-qs} & \pageref{sec:qi-qs}\\
$D$-quasi-isometry & a quasi-isometry whose multiplicative, additive and density constants
 are all $D$ & \cref{sec:qi-qs} & \pageref{sec:qi-qs}\\
quasisymmetry & a homeomorphism that distorts distance ratios by at most a fixed gauge &
 \cref{def:qs} & \pageref{def:qs}\\
visual boundary, visual metric & the rays from the root, metrised by the depth of their
 longest common prefix & \cref{sec:visual} & \pageref{sec:visual}\\
offspring distribution & the finitely supported law of the number of children of a vertex &
 \cref{sec:gw-trees} & \pageref{sec:gw-trees}\\
supercritical & having offspring mean greater than one &
 \cref{sec:gw-trees} & \pageref{sec:gw-trees}\\
Galton--Watson tree & the random subtree of $\cN$ grown from independent offspring variables &
 \cref{sec:gw-trees} & \pageref{sec:gw-trees}\\
full tree & a rooted tree in which every vertex has at least two children. Every such tree
 of bounded valence is quasi-isometric to $\bbB$ &
 \cref{thm:bushy} & \pageref{thm:bushy}\\
Harris decomposition & the decomposition of a surviving Galton--Watson tree into its
 infinite skeleton and finite bushes &
 \cref{thm:harris,thm:harris-general} &
 \pageref{thm:harris}, \pageref{thm:harris-general}\\
skeleton & the subtree formed by the vertices with an infinite line of descent &
 \cref{sec:shape-harris,sec:general-classes} &
 \pageref{sec:shape-harris}, \pageref{sec:general-classes}\\
split & a skeleton vertex with at least two skeleton children &
 \cref{sec:shape-harris,sec:general-classes} &
 \pageref{sec:shape-harris}, \pageref{sec:general-classes}\\
neck & a skeleton vertex with exactly one skeleton child &
 \cref{sec:shape-harris,sec:general-classes} &
 \pageref{sec:shape-harris}, \pageref{sec:general-classes}\\
neck length & the number of vertices from the start of a neck through its terminating split &
 \cref{sec:general-classes} & \pageref{sec:general-classes}\\
bush & a finite subtree rooted at a child outside the skeleton &
 \cref{sec:shape-harris,sec:general-classes} &
 \pageref{sec:shape-harris}, \pageref{sec:general-classes}\\
reduced skeleton & the tree whose vertices are the skeleton splits and whose edges record
 successive splits along skeleton descent &
 \cref{sec:general-classes} & \pageref{sec:general-classes}\\
branching semigroup & for $\theta_0=0$, the additive submonoid generated by the excesses
 $k-1$ of the supported arities & \cref{def:branching-semigroup} &
 \pageref{def:branching-semigroup}\\
possible branching semigroup & an additive submonoid of $\bbN_0$ generated by a nonempty
 finite set of positive integers & \cref{sec:classification} &
 \pageref{sec:classification}\\
chain class $(\mathrm{C}_\Lambda)$ & the quasi-isometry class of chain-regime laws with
 branching semigroup $\Lambda$ & \cref{thm:trichotomy} & \pageref{thm:trichotomy}\\
line & a subset of a tree isometric to $\bbZ$ &
 \cref{sec:trichotomy-simple} & \pageref{sec:trichotomy-simple}\\
bush of depth $h$ & a finite component $P$ of $T\setminus\set c$ contained in
 $B_T(c,h)$ and containing a vertex at distance $h$ from $c$ &
 \cref{sec:trichotomy-simple} & \pageref{sec:trichotomy-simple}\\
\glgroup{Matching}
compatible & at graph distance at most one in the label graph &
 \cref{sec:matching} & \pageref{sec:matching}\\
leaf labelling, full labelling & a labelling of the leaves of $\bbB_h$, respectively of all
 its vertices & \cref{sec:matching} & \pageref{sec:matching}\\
leaf matching, full matching & the existence of one rooted automorphism making all leaf
 labels, respectively all vertex labels, compatible &
 \cref{sec:setup} & \pageref{sec:setup}\\
graph potential & the one-site quantity $\eta_{G,\alpha}(\mu)$ whose smallness drives the
 matching theorems & \cref{eq:potential} & \pageref{eq:potential}\\
good degree, bad degree & the mass compatible, respectively incompatible, with a point &
 \cref{sec:contraction} & \pageref{sec:contraction}\\
matching degree & the target mass related to a fixed source point &
 \cref{sec:restricted-potential} & \pageref{sec:restricted-potential}\\
restricted potential & the source average of the potential weight over points of positive
 matching degree against a possibly different target law &
 \cref{sec:restricted-potential} & \pageref{sec:restricted-potential}\\
symmetrised square & the relation on pairs holding when the straight or crossed pairing
 relates both coordinates & \cref{eq:square} & \pageref{eq:square}\\
transversal & a good main diagonal or antidiagonal of a $2\times2$ array &
 \cref{thm:transversal} & \pageref{thm:transversal}\\
product relation & the relation on a product space holding coordinatewise &
 \cref{thm:product} & \pageref{thm:product}\\
one-site potential & the potential of the label relation itself, to which both matching
 recursions reduce & \cref{eq:Phi0} & \pageref{eq:Phi0}\\
local dominance & concentration near one label sufficient for the graph-potential condition &
 \cref{thm:dominance} & \pageref{thm:dominance}\\
\glgroup{Chains, shapes and profiles}
chain & the vertices $\iota(w)1^l$, $0\leq l\leq\lambda(w)-1$, from the initial vertex
 $\iota(w)$ to the terminating branch point & \cref{sec:encoding} & \pageref{sec:encoding}\\
chain label & the length of the chain traversed at a vertex of $\bbB$, abbreviated to
 ``label'' in the chain sections &
 \cref{eq:lambda-def} & \pageref{eq:lambda-def}\\
normal form & the unique decomposition of a tree vertex into completed chains and its position
 in the final chain & \cref{eq:normal-form} & \pageref{eq:normal-form}\\
associated tree & the tree encoded by a labelling of $\bbB$ through the chain recursion &
 \cref{sec:transfer} & \pageref{sec:transfer}\\
transfer & the passage from multiplicative label comparability to a quasi-isometry of the
 associated trees & \cref{thm:transfer} & \pageref{thm:transfer}\\
level & the value of the quantising level map &
 \cref{sec:quantisation} & \pageref{sec:quantisation}\\
level class, representative & a fibre of the level map and the canonical chain chosen for it &
 \cref{sec:quantisation} & \pageref{sec:quantisation}\\
chain coupling & the randomised requantisation of a second chain law to the class probabilities
 of the first & \cref{thm:chain-coupling} & \pageref{thm:chain-coupling}\\
shape, entry, exit & a finite piece from the root or a child of one split to the next split,
 with its finite bushes and its two marked endpoints &
 \cref{def:shape,def:shape-general} &
 \pageref{def:shape}, \pageref{def:shape-general}\\
assembly & the rooted tree obtained by joining a family of finite marked pieces along an
 indexing skeleton & \cref{sec:shape-harris,sec:binary-presentations} &
 \pageref{sec:shape-harris}, \pageref{sec:binary-presentations}\\
marked quasi-isometry & a quasi-isometry of finite pieces that maps their entries and exits
 within the stated bound of one another & \cref{sec:shape-net} & \pageref{sec:shape-net}\\
$K$-comparable & admitting a $K$-marked quasi-isometry in at least one direction &
 \cref{sec:shape-net} & \pageref{sec:shape-net}\\
binary profile & a finite full binary tree whose ordered leaves replace the children of a
 split & \cref{sec:binary-presentations} & \pageref{sec:binary-presentations}\\
binary profile process & the Markov label field obtained by placing a random label at each
 profile root, the prescribed label at its other internal vertices and independent copies at
 its leaves & \cref{thm:profile-encoding-law} & \pageref{thm:profile-encoding-law}\\
common generating core & shared arities whose shifted values generate the common branching
 semigroup & \cref{sec:common-presentations} & \pageref{sec:common-presentations}\\
common presentation & a choice of common profiles for the core arities and composite profiles
 for all remaining supported arities & \cref{sec:common-presentations} &
 \pageref{sec:common-presentations}\\
atom & a positive semigroup element that is not a sum of two positive elements &
 \cref{thm:common-atoms} & \pageref{thm:common-atoms}
\end{supertabular}
\endgroup

\clearpage
\section{Paper--Lean correspondence}\label{app:lean}

Below, we list the Lean declarations and their correspondences to objects and results in this
paper. Notes on the differences in formulation are recorded in \texttt{paper-correspondence.yaml}
in the accompanying formalisation~\cite{AthreyaTroscheit2026}.

\subsection{Audited headline statements}
\leavevmode\par\nobreak

\begingroup
\scriptsize
\tablehead{\textbf{Paper statement} & \textbf{Declaration in \texttt{Challenge.lean}}\\[2pt]
\hline}
\begin{supertabular}{@{}L{0.34\textwidth}|L{0.62\textwidth}@{}}
\cref{thm:matching}, leaf bound &
 \texttt{Challenge.audit\_graph\_leaf\_matching\_bound}\\
\cref{thm:matching}, full bound &
 \texttt{Challenge.audit\_graph\_full\_matching\_bound}\\
\cref{thm:matching}, infinite tree &
 \texttt{Challenge.audit\_exists\_infinite\_tree\_matching\_graphAut}\\
\cref{thm:trichotomy}, finite class: bounded graphs &
 \texttt{Challenge.audit\_bounded\_graph\_qi\_point}\\
\cref{thm:trichotomy}, finite class: extinction &
 \texttt{Challenge.audit\_extinction\_of\_not\_supercritical}\\
\cref{thm:trichotomy}, complete classification &
 \texttt{Challenge.audit\_full\_classification\_ae\_iff}\\
\cref{thm:embedding-hierarchy}, mutual embeddability &
 \texttt{Challenge.audit\_mutual\_embeddability}\\
\cref{thm:twovalue}, almost-sure statement &
 \texttt{Challenge.audit\_twovalue\_ae\_tree\_family}\\
\cref{thm:twovalue}, quantitative rate &
 \texttt{Challenge.audit\_twovalue\_rate\_tree}\\
\cref{thm:markov-matching}, finite alternative &
 \texttt{Challenge.audit\_markov\_matching\_finite}\\
\cref{thm:markov-matching}, zero-compatible alternative &
 \texttt{Challenge.audit\_markov\_matching\_zero}\\
\cref{thm:markov-matching}, finite alternative at infinite height &
 \texttt{Challenge.audit\_markov\_matching\_finite\_infinite}\\
\cref{thm:markov-matching}, zero-compatible alternative at infinite height &
 \texttt{Challenge.audit\_markov\_matching\_zero\_infinite}\\
\end{supertabular}
\endgroup

\subsection{Mathematical objects}
\leavevmode\par\nobreak

\begingroup
\scriptsize
\noindent
\begin{tabular}{@{}L{0.33\textwidth}|L{0.61\textwidth}@{}}
\textbf{Mathematical object} & \textbf{Principal Lean declaration(s)}\\[2pt] \hline
Finite rooted trees and general-support shapes &
 \texttt{ChainClasses.RTree}, \texttt{ChainClasses.GShape}\\
Reduced skeleton and reduced arity law &
 \texttt{BranchingProcess.skeleton}, \texttt{ChainClasses.reducedPMF}\\
Binary profile &
 \texttt{GraphMarkovMatching.Stopped.MTree}\\
Finite family of profiles &
 \texttt{ChainClasses.Profile.Family}\\
Binary encoding and encoded assembly &
 \texttt{ChainClasses.CascadeEnc}, \texttt{ChainClasses.SkelAssembly}\\
Common-core presentation &
 \texttt{GraphMarkovMatching.Stopped.Presentation}\\
Markov model and cyclic classes &
 \texttt{GraphMarkovMatching.Stopped.Model},
 \texttt{GraphMarkovMatching.Stopped.Model.Phase}\\
Selected transitions &
 \texttt{GraphMarkovMatching.Stopped.Model.Selection}\\
General assembly &
 \texttt{ChainClasses.GAssembly}\\
Finite star obstruction &
 \texttt{ChainClasses.Star}\\
Quasi-isometric embedding &
 \texttt{BranchingProcess.IsQIEmbWith}, \texttt{BranchingProcess.QIEmbeddable}\\
Pruning rounds and retained vertices &
 \texttt{BranchingProcess.Retained}, \texttt{BranchingProcess.RetainedInf}\\
Concentrated representative of a law &
 \texttt{BranchingProcess.Offspring.concentrate}\\
\end{tabular}
\endgroup

\subsection{Library declarations}

\subsubsection{The i.i.d.\ matching argument}

Unless a namespace is displayed, the declarations in this table lie in
\texttt{GraphMatching}.

\begingroup
\scriptsize
\tablehead{\textbf{Paper declaration} & \textbf{Primary Lean declaration(s)}\\[2pt] \hline}
\begin{supertabular}{@{}L{0.29\textwidth}|L{0.67\textwidth}@{}}
\cref{def:qi} &
 \texttt{BranchingProcess.IsQIWith}, \texttt{BranchingProcess.QuasiIsometric},
 \texttt{BranchingProcess.IsQIEmbWith}, \texttt{BranchingProcess.QIEmbeddable}\\
\cref{thm:matching} &
 \texttt{graph\_leaf\_matching\_bound}, \texttt{graph\_full\_matching\_bound},
 \texttt{exists\_infinite\_tree\_matching\_graphAut}\\
\cref{thm:full-probability-recursion} &
 \texttt{full\_probability\_recursion}\\
\cref{thm:iid-scalar-failure} &
 \texttt{iid\_scalar\_failure}\\
\cref{thm:toolkit} &
 \texttt{GraphMarkovMatching.Stopped.phi\_mul\_le\_beta},
 \texttt{GraphMarkovMatching.Stopped.sq\_le\_K\_phi\_add},
 \texttt{GraphMarkovMatching.Support.rpow\_neg\_alpha\_le},
 \texttt{GraphMarkovMatching.Stopped.chord\_gen},
 \texttt{GraphMarkovMatching.Support.phi\_prod\_le}\\
\cref{thm:contraction} &
 \texttt{Phi\_square\_leA}, \texttt{Phi\_square\_le\_of\_smallA}\\
\cref{thm:transversal} &
 \texttt{transversal}, \texttt{bad\_row\_or\_bad\_col}\\
\cref{thm:product} &
 \texttt{PhiA\_prodPMF\_le}, \texttt{phiA\_prod\_le}\\
\cref{thm:reduction} &
 \texttt{leaf\_matching\_boundA}, \texttt{full\_matching\_boundA}\\
\cref{thm:dominance} &
 \texttt{etaGA\_le\_localDominance}\\
\cref{thm:double-exp} &
 \texttt{double\_exp\_matching}\\
\end{supertabular}
\endgroup

\subsubsection{Galton--Watson trees, shapes and classification}

Unless another namespace is displayed, the declarations in this table lie in
\texttt{ChainClasses}.

\begingroup
\scriptsize
\tablehead{\textbf{Paper declaration} & \textbf{Primary Lean declaration(s)}\\[2pt] \hline}
\begin{supertabular}{@{}L{0.29\textwidth}|L{0.67\textwidth}@{}}
\cref{def:branching-semigroup} &
 \texttt{shiftSupp} (generators)\\
\cref{thm:trichotomy} &
 \texttt{full\_classification\_rooted\_ae\_iff},
 \texttt{full\_classification\_ae\_iff}\\
\cref{thm:twovalue} &
 \texttt{twovalue\_ae\_tree\_family}, \texttt{twovalue\_rate\_tree}\\
\cref{thm:chains} &
 \texttt{tree\_eq\_chains}, \texttt{chain\_unique},
 \texttt{ray\_chain\_disjoint}\\
\cref{thm:geometric} &
 \texttt{chains\_ae}, \texttt{label\_marginal}, \texttt{label\_prod},
 \texttt{label\_iIndepFun}, \texttt{exists\_chain\_field}\\
\cref{thm:isometry} &
 \texttt{isometry\_of\_relabel}\\
\cref{thm:transfer} &
 \texttt{transfer}, \texttt{transferR}\\
\cref{thm:level} &
 \texttt{level\_comparable}, \texttt{level\_close\_comparable}\\
\cref{thm:quantised-law} &
 \texttt{qPMF}, \texttt{class\_sum},
 \texttt{quantised\_label\_iIndepFun}\\
\cref{thm:chain-coupling} &
 \texttt{coupling\_law}\\
\cref{thm:harris} &
 \texttt{IsBushy.survival\_decomposition}, \texttt{IsBushy.conjugate\_tilt},
 \texttt{IsBushy.conjugate\_subcritical},
 \texttt{BranchingProcess.survivalMeasure\_skelField\_pattern}\\
\cref{def:shape} &
 \texttt{Shape}, \texttt{Shape.realise}\\
\cref{thm:shape-iid} &
 \texttt{assembly\_isometric\_sample}, \texttt{survivalMeasure\_shapes},
 \texttt{shape\_decomposition}\\
\cref{thm:shape-mass} &
 \texttt{shape\_mass\_point}, \texttt{shape\_mass\_tail}\\
\cref{def:shape-net} &
 \texttt{netMem}, \texttt{repIdx}, \texttt{shapeFamily}\\
\cref{thm:shape-net} &
 \texttt{exists\_rep}, \texttt{repIdx\_le},
 \texttt{markedQI\_of\_repIdx\_adjacent}, \texttt{size\_repIdx\_le}\\
\cref{thm:shape-connected} &
 \texttt{NetLink}, \texttt{netLink\_connected\_shapeFamily}\\
\cref{thm:dilution} &
 \texttt{dilution}, \texttt{markedQI\_contract\_shape},
 \texttt{dilution\_count\_shape}\\
\cref{thm:glued-transfer} &
 \texttt{Assembly.glued\_transfer}\\
\cref{thm:eta-bound} &
 \texttt{qF\_le}, \texttt{quantised\_eta\_le}, \texttt{forall\_scale}\\
\cref{thm:cross-law} &
 \texttt{crosslaw\_rate\_tree}, \texttt{exists\_crosslaw\_rate\_tree},
 \texttt{crosslaw\_ae\_tree}\\
\cref{thm:shape-shrink} &
 \texttt{markedQI\_shape\_shrink}, \texttt{markedQI\_shape\_shrink\_of\_le},
 \texttt{markedQI\_shape\_fix}\\
\cref{thm:shape-eta} &
 \texttt{shape\_eta\_final}\\
\cref{thm:shape-coupling} &
 \texttt{exists\_isShapeCoupling\_etaS}\\
\cref{thm:hairy} &
 \texttt{bushy\_rate}, \texttt{bushy\_ae\_tree},
 \texttt{bushy\_rate\_two\_law},
 \texttt{bushy\_ae\_shape\_tree\_two\_law}\\
\cref{thm:trichotomy-simple} &
 \texttt{simple\_classification}, \texttt{simple\_classification\_ae\_iff}\\
\cref{thm:tree-stability} &
 \texttt{BranchingProcess.exists\_between\_dist\_le}\\
\cref{thm:bush-separation} &
 \texttt{BranchingProcess.not\_quasiIsometric\_of\_bush\_of\_line}\\
\cref{thm:three-rays} &
 \texttt{BranchingProcess.not\_quasiIsometric\_rayGraph}\\
\cref{thm:regime-obstructions} &
 \texttt{hasThreeRays\_sampleWord}, \texttt{nearLine\_inTree},
 \texttt{unbounded\_bushes\_ae}\\
\cref{thm:bottleneck} &
 \texttt{not\_isQIWith\_binary\_of\_thinBalls}\\
\cref{thm:converse} &
 \texttt{converse\_binary\_ae}\\
\cref{thm:harris-general} &
 \texttt{transform\_expansion}, \texttt{reduced\_sum},
 \texttt{conjugate\_mean\_general},
 \texttt{survivalMeasure\_gArities},
 \texttt{BranchingProcess.survivalMeasure\_skelField\_pattern},
 \texttt{BranchingProcess.survivalMeasure\_decorations\_treeLaw}\\
\cref{thm:full-support} &
 \texttt{reducedLaw\_pos}, \texttt{reducedWeight\_pos}\\
\cref{def:shape-general} &
 \texttt{GShape}, \texttt{GShape.realise}\\
\cref{thm:split-joint} &
 \texttt{splitJoint}, \texttt{splitJoint\_not\_product}\\
\cref{def:conditional-laws} &
 \texttt{gCondPMF}, \texttt{gMixMass}\\
\cref{thm:conditional-explicit} &
 \texttt{gCondMass}, \texttt{gCondMass\_def},
 \texttt{survivalMeasure\_gShapeSplit}, \texttt{gMixMass\_eq\_measure}\\
\cref{thm:conditional-iid} &
 \texttt{conditional\_iid}\\
\cref{thm:mass-uniform} &
 \texttt{ofReal\_pow\_le\_gCondMass}, \texttt{ofReal\_pow\_le\_gMixMass},
 \texttt{exists\_gShape\_size\_tail}\\
\cref{thm:general-dilution} &
 \texttt{RTree.cutCount\_le\_div}, \texttt{markedQI\_gContractTree},
 \texttt{gDilution}, \texttt{gDilution\_netMem}\\
\cref{thm:general-glued} &
 \texttt{GAssembly.glued\_transfer}\\
\cref{thm:relabel} &
 \texttt{relabel\_law}, \texttt{markedQI\_relabel},
 \texttt{exists\_gShapeCoupling\_relabel},
 \texttt{exists\_etaG\_gNetGraph\_small}\\
\cref{thm:product-form} &
 \texttt{product\_of\_constant\_conditional},
 \texttt{labelMeasure\_label\_pattern},
 \texttt{labelMeasure\_label\_marginal}\\
\cref{thm:cross-relabel} &
 \texttt{crossRelabel\_qi\_compat},
 \texttt{exists\_gShapeCoupling\_cond\_mix},
 \texttt{exists\_etaG\_gNetGraph\_small}\\
\cref{thm:profile-encoding} &
 \texttt{CascadeEnc.skelToB\_isQIMap},
 \texttt{Profile.profileEncHeight}\\
\cref{thm:profile-encoding-law} &
 \texttt{Profile.map\_encLab\_of\_pattern},
 \texttt{Profile.map\_encodedStates}\\
\cref{thm:profile-transfer} &
 \texttt{Profile.exists\_portrait\_of\_infMatchK},
 \texttt{Profile.sample\_qi\_of\_profile\_match}\\
\cref{thm:bushy} &
 \texttt{fullTree\_quasiIsometric\_binary}\\
\cref{thm:regime-obstructions-general} &
 \texttt{threeRays\_ae}, \texttt{nearLine\_gSample\_ae},
 \texttt{unbounded\_bushes\_gSample\_ae}, \texttt{thinBalls\_gSample\_ae}\\
\cref{thm:chain-separation} &
 \texttt{chainSeparation}\\
\cref{thm:chain-star-obstruction} &
 \texttt{Star}, \texttt{star\_pair\_no\_qi}\\
\cref{thm:hairy-general} &
 \texttt{bushy\_general\_ae}\\
\cref{thm:chain-general} &
 \texttt{chain\_general\_ae}\\
\cref{thm:core-weight-obstruction} &
 \texttt{Profile.no\_uniform\_core\_weight\_bound},
 \texttt{Profile.no\_uniform\_core\_weight\_infinite\_bound}\\
\cref{thm:concentrated-regular-subtree} &
 \texttt{BranchingProcess.retainedInf\_iff},
 \texttt{BranchingProcess.retainedProb\_succ},
 \texttt{BranchingProcess.ofReal\_le\_retainedProb},
 \texttt{BranchingProcess.ofReal\_le\_sampleMeasure\_retainedInf},
 \texttt{BranchingProcess.retainedTreeLaw\_eq\_treeLaw},
 \texttt{BranchingProcess.exists\_regular\_embedding\_of\_retainedInf},
 \texttt{BranchingProcess.exists\_binary\_embedding\_of\_retainedInf},
 \texttt{sampleMeasure\_spine\_fail}, \texttt{sampleMeasure\_spine\_fail\_le},
 \texttt{ae\_exists\_spine\_retained},
 \texttt{ae\_exists\_retained\_vertex},
 \texttt{ae\_qiEmbeddable\_binary\_of\_concentrated}\\
\cref{thm:embedding-hierarchy} &
 \texttt{qiEmbeddable\_sample\_binary},
 \texttt{ae\_qiEmbeddable\_binary\_sample},
 \texttt{embedding\_hierarchy\_ae},
 \texttt{mutual\_embeddability\_ae}\\
Strict hierarchy in \cref{sec:embedding-hierarchy} &
 \texttt{qiEmbeddable\_rayGraph\_of\_not\_survives},
 \texttt{not\_qiEmbeddable\_rayGraph\_of\_not\_survives},
 \texttt{qiEmbeddable\_rayGraph\_of\_survives},
 \texttt{ae\_not\_qiEmbeddable\_rayGraph},
 \texttt{BranchingProcess.not\_qiEmbeddable\_rayGraph}\\
\end{supertabular}
\endgroup

\subsubsection{The Markov matching argument}

Unless a nested namespace is displayed, the declarations in this table lie in
\texttt{GraphMarkovMatching.Stopped}.

\begingroup
\scriptsize
\tablehead{\textbf{Paper declaration} & \textbf{Primary Lean declaration(s)}\\[2pt] \hline}
\begin{supertabular}{@{}L{0.29\textwidth}|L{0.67\textwidth}@{}}
\cref{thm:markov-matching} &
 \texttt{markov\_matching\_finite}, \texttt{markov\_matching\_zero},
 \texttt{markov\_matching\_finite\_infinite},
 \texttt{markov\_matching\_zero\_infinite}, \texttt{finiteEps},
 \texttt{finiteKmatch}\\
\cref{thm:markov-matching}, retained transitions &
 \texttt{Model.Selection}, \texttt{markov\_matching\_finite},
 \texttt{markov\_matching\_finite\_infinite}\\
\cref{thm:markov-matching}, $\eta_\alpha$ bounds &
 \texttt{markov\_matching\_finite\_eta}, \texttt{Model.zeta\_le\_eta\_div}\\
\cref{thm:four-law-contraction} &
 \texttt{fourLaw\_contraction}\\
\cref{thm:zero-degree-dichotomy} &
 \texttt{Model.zeroEv\_succ\_imp}, \texttt{Model.zeroMass\_succ\_le}\\
\cref{thm:explicit-zero-bound} &
 \texttt{Model.zeroMass\_le\_of\_stops}, \texttt{Model.z\_le\_of\_returns},
 \texttt{Zbar\_fixed}\\
\cref{thm:explicit-weighted-bound} &
 \texttt{Model.wZero\_le\_of\_stops}, \texttt{Model.wZero\_le\_Efun}\\
\cref{thm:one-step-closure} &
 \texttt{Model.P\_succ\_le}, \texttt{Model.DmuC\_le},
 \texttt{Model.childPair\_le\_Qfun}, \texttt{Model.fourLawAt\_of\_params}\\
\cref{thm:zero-mixture-convexity} &
 \texttt{zero\_mixture\_convexity}\\
\cref{thm:averaged-child-pair} &
 \texttt{Model.childPair\_le\_Qfun}\\
\cref{thm:common-atoms} &
 \texttt{isAtom\_mem\_generators}, \texttt{atoms\_subset\_inter},
 \texttt{mem\_closure\_atoms}, \texttt{length\_le\_of\_sum\_atoms},
 \texttt{exists\_atoms\_finset}\\
\cref{thm:shallow-grafting} &
 \texttt{MTree.height\_composite\_le}, \texttt{MTree.minLeafDepth\_le\_log}\\
\cref{thm:fresh-positive} &
 \texttt{Presentation.freshPositive}\\
\cref{thm:atomic-normalisation} &
 \texttt{Presentation.selection}, \texttt{Presentation.inverseSum\_le}\\
\cref{thm:bounded-return} &
 \texttt{Presentation.commonReturns}\\
\cref{thm:common-semigroup-matching} &
 \texttt{two\_laws\_matching}, \texttt{atomicPresentation},
 \texttt{Presentation.presentation\_matching\_uniform},
 \texttt{Presentation.presentation\_matching\_uniform\_infinite}\\
\cref{thm:constructive-threshold} &
 \texttt{epsK\_spec}, \texttt{epsK0\_spec},
 \texttt{barrier\_of\_le\_epsK}\\
\end{supertabular}
\endgroup

\subsection{Statements without formal counterparts}
\leavevmode\par\nobreak

\begingroup
\scriptsize
\tablehead{\textbf{Paper declaration} & \textbf{Subject}\\[2pt] \hline}
\begin{supertabular}{@{}L{0.29\textwidth}|L{0.67\textwidth}@{}}
\cref{def:qs,thm:boundary} & Quasisymmetry and visual boundaries\\
\cref{thm:fractal-percolation} & Fractal-percolation application\\
\cref{thm:two-sided-potential,thm:branching-times} & Random branching times\\
\end{supertabular}
\endgroup


\begin{thebibliography}{99}
  \bibitem{Aldous1991}
    D.~Aldous.
    The continuum random tree.~I.
    \emph{Ann. Probab.}, \textbf{19} (1991), 1--28.

  \bibitem{Aldous1993}
    D.~Aldous.
    The continuum random tree.~III.
    \emph{Ann. Probab.}, \textbf{21} (1993), 248--289.

  \bibitem{AldousContatCurienHenard2023}
    D.~Aldous, A.~Contat, N.~Curien, and O.~H\'enard.
    Parking on the infinite binary tree.
    \emph{Probab. Theory Related Fields}, \textbf{187} (2023), 481--504.

  \bibitem{AnttilaErikssonBiquePyorala2025}
    R.~Anttila, S.~Eriksson-Bique, and A.~Py\"or\"al\"a.
    Quasisymmetric mappings on two variants of fractal percolation.
    Preprint (2025), \texttt{arXiv:2510.06900}.

  \bibitem{ArmstrongKempe2026}
    S.~Armstrong and J.~Kempe.
    Formalization of De Giorgi--Nash--Moser theory in Lean.
    Preprint (2026), \texttt{arXiv:2604.05984}.

  \bibitem{ArmstrongKuusi2025}
    S.~Armstrong and T.~Kuusi.
    Renormalization group and elliptic homogenization in high contrast.
    \emph{Invent. Math.}, \textbf{242} (2025), 895--1086.

  \bibitem{AthreyaNey1972}
    K.~B.~Athreya and P.~E.~Ney.
    \emph{Branching Processes}.
    Grundlehren der mathematischen Wissenschaften \textbf{196}, Springer, Berlin, 1972.

  \bibitem{AthreyaTroscheit2026}
    J.~S.~Athreya and S.~Troscheit.
    The quasi-isometry classes of Galton--Watson trees: Lean formalisation.
    Palomar, \texttt{PALOMAR-2026-09-25-000002} (2026).
    \url{https://palomar-registry.org/entry?id=PALOMAR-2026-09-25-000002}.

  \bibitem{BasuSidoraviciusSly2018}
    R.~Basu, V.~Sidoravicius, and A.~Sly.
    Lipschitz embeddings of random fields.
    \emph{Probab. Theory Related Fields}, \textbf{172} (2018), 1121--1179.

  \bibitem{BasuSly2014}
    R.~Basu and A.~Sly.
    Lipschitz embeddings of random sequences.
    \emph{Probab. Theory Related Fields}, \textbf{159} (2014), 721--775.

  \bibitem{Benjamini2013}
    I.~Benjamini.
    \emph{Coarse geometry and randomness}.
    Lecture Notes in Mathematics \textbf{2100}, Springer, Cham, 2013.

  \bibitem{BenjaminiPeres1994}
    I.~Benjamini and Y.~Peres.
    Markov chains indexed by trees.
    \emph{Ann. Probab.}, \textbf{22} (1994), 219--243.



  \bibitem{Bonk2011}
    M.~Bonk.
    Uniformization of Sierpi\'nski carpets in the plane.
    \emph{Invent. Math.}, \textbf{186} (2011), 559--665.

  \bibitem{BonkMerenkov2013}
    M.~Bonk and S.~Merenkov.
    Quasisymmetric rigidity of square Sierpi\'nski carpets.
    \emph{Ann. of Math. (2)}, \textbf{177} (2013), 591--643.

  \bibitem{BonkMeyer2022}
    M.~Bonk and D.~Meyer.
    Uniformly branching trees.
    \emph{Trans. Amer. Math. Soc.}, \textbf{375} (2022), 3841--3897.

  \bibitem{BonkSchramm2000}
    M.~Bonk and O.~Schramm.
    Embeddings of Gromov hyperbolic spaces.
    \emph{Geom. Funct. Anal.}, \textbf{10} (2000), 266--306.
%

  \bibitem{BridsonHaefliger1999}
    M.~R.~Bridson and A.~Haefliger.
    \emph{Metric spaces of non-positive curvature}.
    Grundlehren der mathematischen Wissenschaften 319, Springer, Berlin, 1999.

  \bibitem{BuragoKleiner1998}
    D.~Burago and B.~Kleiner.
    Separated nets in Euclidean space and Jacobians of biLipschitz maps.
    \emph{Geom. Funct. Anal.}, \textbf{8} (1998), 273--282.

  \bibitem{BuyaloSchroeder2007}
    S.~Buyalo and V.~Schroeder.
    \emph{Elements of asymptotic geometry}.
    EMS Monographs in Mathematics, European Mathematical Society, Z\"urich, 2007.

  \bibitem{ChrontsiosGaritsisIoannidisVellis2026}
    E.-K.~Chrontsios-Garitsis, F.~Ioannidis, and V.~Vellis.
    Universal quasiconformal trees.
    \emph{Adv. Math.}, \textbf{501} (2026), Article~111107.

  \bibitem{DavidSemmes1997}
    G.~David and S.~Semmes.
    \emph{Fractured fractals and broken dreams: self-similar geometry through metric and
    measure}.
    Oxford Lecture Series in Mathematics and its Applications 7, Oxford University Press, 1997.

  \bibitem{DeLaHarpe2000}
    P.~de la Harpe.
    \emph{Topics in geometric group theory}.
    Chicago Lectures in Mathematics, University of Chicago Press, Chicago, IL, 2000.

  \bibitem{Durrett2019}
    R.~Durrett.
    \emph{Probability: theory and examples}.
    Fifth edition. Cambridge Series in Statistical and Probabilistic Mathematics \textbf{49},
    Cambridge University Press, Cambridge, 2019.

  \bibitem{FlajoletSedgewick2009}
    P.~Flajolet and R.~Sedgewick.
    \emph{Analytic combinatorics}.
    Cambridge University Press, Cambridge, 2009.


  \bibitem{FraserMiaoTroscheit2018}
    J.~M.~Fraser, J.~J.~Miao, and S.~Troscheit.
    The Assouad dimension of randomly generated fractals.
    \emph{Ergodic Theory Dynam. Systems}, \textbf{38} (2018), no.~3, 982--1011.

  \bibitem{FraserTyson2026}
    J.~M.~Fraser and J.~T.~Tyson.
    Sobolev and quasiconformal distortion of intermediate dimension with applications to conformal
    dimension.
    \emph{J. Lond. Math. Soc. (2)}, \textbf{113} (2026), paper no. e70445.

  \bibitem{Hall1935}
    P.~Hall.
    On representatives of subsets.
    \emph{J. London Math. Soc.}, \textbf{10} (1935), 26--30.

  \bibitem{Harris1948}
    T.~E.~Harris.
    Branching processes.
    \emph{Ann. Math. Statistics}, \textbf{19} (1948), 474--494.

  \bibitem{Harris1963}
    T.~E.~Harris.
    \emph{The theory of branching processes}.
    Die Grundlehren der Mathematischen Wissenschaften \textbf{119}, Springer, Berlin, 1963.

  \bibitem{Heinonen2001}
    J.~Heinonen.
    \emph{Lectures on analysis on metric spaces}.
    Universitext, Springer-Verlag, New York, 2001.

  \bibitem{KolossvaryTroscheit2025}
    I.~Kolossv\'ary and S.~Troscheit.
    Recent progress on fractal percolation.
    To appear in \emph{Fractal Geometry and Stochastics VII}, Birkh\"auser;
    \texttt{arXiv:2508.08150}.

  \bibitem{Konig1927}
    D.~K\H{o}nig.
    \"Uber eine Schlussweise aus dem Endlichen ins Unendliche.
    \emph{Acta Sci. Math. (Szeged)}, \textbf{3} (1927), 121--130.

  \bibitem{LeGall2013}
    J.-F.~Le~Gall.
    Uniqueness and universality of the Brownian map.
    \emph{Ann. Probab.}, \textbf{41} (2013), 2880--2960.

  \bibitem{Lindquist2018}
    J.~Lindquist.
    Bilipschitz equivalence of trees and hyperbolic fillings.
    \emph{Conform. Geom. Dyn.}, \textbf{22} (2018), 225--234.

  \bibitem{LiYuZheng2026}
    Z.~Li, R.~Yu, and T.~Zheng.
    Quasi-isometric rigidity for random subsets in products of trees.
    Preprint, \texttt{arXiv:2606.05089} (2026).

  \bibitem{Lyons1992}
    R.~Lyons.
    Random walks, capacity and percolation on trees.
    \emph{Ann. Probab.}, \textbf{20} (1992), no.~4, 2043--2088.

  \bibitem{LyonsPeres2016}
    R.~Lyons and Y.~Peres.
    \emph{Probability on trees and networks}.
    Cambridge Series in Statistical and Probabilistic Mathematics, Cambridge University Press,
    New York, 2016.

  \bibitem{MackayTyson2010}
    J.~M.~Mackay and J.~T.~Tyson.
    \emph{Conformal dimension: theory and application}.
    University Lecture Series 54, American Mathematical Society, Providence, 2010.

  \bibitem{McMullen1998}
    C.~T.~McMullen.
    Lipschitz maps and nets in Euclidean space.
    \emph{Geom. Funct. Anal.}, \textbf{8} (1998), 304--314.

  \bibitem{Miermont2013}
    G.~Miermont.
    The Brownian map is the scaling limit of uniform random plane quadrangulations.
    \emph{Acta Math.}, \textbf{210} (2013), 319--401.

  \bibitem{MillerTian2026a}
    J.~Miller and Y.~Tian.
    The conformal dimension of the Brownian sphere is two.
    Preprint (2026), \texttt{arXiv:2603.24473}.

  \bibitem{MillerTian2026b}
    J.~Miller and Y.~Tian.
    The conformal dimension of the Brownian tree is one.
    Preprint (2026), \texttt{arXiv:2604.25769}.

  \bibitem{MillerTian2026c}
    J.~Miller and Y.~Tian.
    Quasisymmetric rigidity of the Brownian sphere.
    Preprint (2026), \texttt{arXiv:2606.10973}.

  \bibitem{MosherSageevWhyte2003}
    L.~Mosher, M.~Sageev, and K.~Whyte.
    Quasi-actions on trees I. Bounded valence.
    \emph{Ann. of Math. (2)}, \textbf{158} (2003), 115--164.

  \bibitem{NivenZuckermanMontgomery1991}
    I.~Niven, H.~S.~Zuckerman, and H.~L.~Montgomery.
    \emph{An introduction to the theory of numbers}.
    Fifth edition, Wiley, New York, 1991.

  \bibitem{Pansu1989}
    P.~Pansu.
    Dimension conforme et sph\`ere \`a l'infini des vari\'et\'es \`a courbure n\'egative.
    \emph{Ann. Acad. Sci. Fenn. Ser. A I Math.}, \textbf{14} (1989), 177--212.

  \bibitem{Papasoglu1995}
    P.~Papasoglu.
    Homogeneous trees are bilipschitz equivalent.
    \emph{Geom. Dedicata}, \textbf{54} (1995), 301--306.

  \bibitem{Peled2010}
    R.~Peled.
    On rough isometries of Poisson processes on the line.
    \emph{Ann. Appl. Probab.}, \textbf{20} (2010), 462--494.

  \bibitem{RossiSuomala2021}
    E.~Rossi and V.~Suomala.
    Fractal percolation and quasisymmetric mappings.
    \emph{Int. Math. Res. Not. IMRN}, \textbf{2021}, no.~10, 7372--7393.

  \bibitem{Sevastyanov1951}
    B.~A.~Sevast'yanov.
    The theory of branching random processes (in Russian).
    \emph{Uspekhi Mat. Nauk}, \textbf{6} (1951), no.~6(46), 47--99.

  \bibitem{SmilanskySolomon2026}
    Y.~Smilansky and Y.~Solomon.
    BiLipschitz and bounded displacement equivalence of Delone sets.
    Preprint (2026), \texttt{arXiv:2609.01763}.

  \bibitem{Troscheit2021}
    S.~Troscheit.
    On quasisymmetric embeddings of the Brownian map and continuum trees.
    \emph{Probab. Theory Related Fields}, \textbf{179} (2021), 1023--1046.

  \bibitem{Whyte1999}
    K.~Whyte.
    Amenability, bilipschitz equivalence, and the von Neumann conjecture.
    \emph{Duke Math. J.}, \textbf{99} (1999), 93--112.
\end{thebibliography}
\end{document}